\documentclass[12pt, a4]{amsart}
\usepackage{epsf,amscd,amssymb}
\usepackage{amsmath}
 \usepackage{graphicx,amscd,amssymb}
\usepackage{graphics}
\usepackage{rotating}

\newtheorem{thm}{Theorem}[section]
\newtheorem{lem}[thm]{Lemma}

\newtheorem{prop}[thm]{Proposition}

\newtheorem{thmA}{Theorem A}

\newtheorem{thmB}{Theorem B}

\newtheorem{rem}[thm]{Remark}

\newtheorem{defn}[thm]{Definition}

\newcommand\bR{{\mathbb{R}}}
\newcommand\bC{{\mathbb C}}
\newcommand\bZ{{\mathbb Z}}

\newcommand\rp{\bR P}

\newcommand\SL{{\rm SL}}

\newcommand\Hom{{\rm Hom}}

\newcommand\SI{{\bf S}}

\newcommand\clo{{\rm Cl}}

\newcommand\ra{\rightarrow}

\newcommand\eps{\epsilon}

\newcommand\ovl{\overline}

\newcommand\Idd{{\rm I}}

\newcommand\CP{{\rm CP}}

\newcommand\tri{\triangle}

\newcommand\SOThr{\mathrm{SO(3)}}
\newcommand\SUTw{\mathrm{SU(2)}}
\newcommand\rep{\mathrm{rep}}
\newcommand\SLThr{\mathrm{SL(3,\mathbb{R})}}

\begin{document}

\markboth{The $\SOThr$-character spaces}
{The $\SOThr$-character spaces}

\title[The $\SOThr$--representations of the surface groups]
{Spherical triangles and the two components of the $\SOThr$-character space of the fundamental group of a closed surface
of genus $2$}
\author{Suhyoung Choi}
\address{Department of Mathematics \\ KAIST \\
Daejeon 305-701, South Korea}

\maketitle 



\begin{abstract}
We use geometric techniques to explicitly find the topological
structure of the space of
$\SOThr$-representations of the fundamental group of a closed surface of
genus $2$ quotient by the conjugation action by $\SOThr$. 
There are two components of the space. 
We will describe the topology of both components and describe the corresponding 
$\SUTw$-character spaces by parametrizing them by spherical triangles. 
There is the sixteen to one branch-covering for each component,
and the branch locus is a union of $2$-spheres or $2$-tori.
Along the way, we also describe the topology of both spaces.
We will later relate this result to future work into
higher-genus cases and the $\SL(3,\bR)$-representations.

\end{abstract}




It is well-known that the set of homomorphisms of a finitely presented group $\Gamma$ 
to a Lie group $G$ has a topology as a subspace of $G^n$, where $n$ is the number of generators of $\Gamma$, 
defined by a number of polynomial equations corresponding to the relations.
Let $\pi$ be the fundamental group of a compact surface and $G$ a compact Lie group.  
Then the space of homomorphisms $\mathrm{Hom}(\pi,G)$ is an algebraic set for which $G$ acts by conjugation.  
Generally, the quotient space by this action $\mathrm{Hom}(\pi,G)/G$ is a semi-algebraic set, 
called the {\em $G$-character space} of $\pi$. 

The study of this space constitutes an interesting subject with a long history and involves many other principal 
fields such as algebraic geometry, topology, geometry, and mathematical physics,
particularly when the Lie group is the special unitary group $\SUTw$. (See Goldman \cite{Grep} for a part of the beginning 
of the subject.)

We denote by $\SL(3, \bR)$ the special linear group, and $\SOThr$ 
the special orthogonal group for $\bR^3$ with a fixed Euclidean metric. 
We are more interested in
how this space is related to the $\SLThr$-character space. 
The $\SOThr$-character space will embed inside the corresponding $\SLThr$-character 
space. The two components will embed in the two components of the 
later space respectively and the Hitchin-Teichmuller component will not contain
any of the $\SOThr$-characters. 
We plan to use our work in this paper to describe the topology of 
the $\SLThr$-character space.


In this paper, we will find a detailed cell-decomposition of 
the $\SOThr$-character space of the fundamental group $\pi_1(\Sigma)$ of 
a closed orientable genus $2$-surface $\Sigma$ by parametrizing the representations 
by spherical triangles. 
The cell-decomposition is one adopted to the later study 
of the $\SLThr$-character space for $\pi_1(\Sigma)$, which is 
the purpose of this paper. This cellular decomposition allows us to obtain 
the topological characterization of the $\SOThr$-character space of $\pi_1(\Sigma)$. 
The topology of the two components of this space is described.

The classical work of Narashimhan, Ramanan, Seshadri, Newstead and so on 
\cite{NR}, \cite{NS},\cite{New2} show that 
the space of $\SUTw$-characters for a genus-two closed surface is diffeomorphic to $\CP^3$. 
We will show that this space $16$ to $1$ branch-covers the component of our space 
containing the identity representation.
The other component is also branch-covered by the space of $\SUTw$-characters for 
a punctured genus two surface with the puncture holonomy $-\Idd$ 
in a $16$ to $1$-manner. We will show that this second space is a real $6$-dimensional manifold, 
so-called an octahedral manifold,
obtained by a completion of a toric variety over the octahedron.
We will obtain the precise $\bZ_2^4$-group action inducing the branch coverings.

Newstead \cite{New} and others worked on determining cohomology rings of moduli spaces arising 
in this manner. Huebschmann, Jeffrey and Weitsmann \cite{JW} \cite{JW2} researched extensively on the
spaces of characters to $\SUTw$, and showed that they have toric structures by finding 
the open dense sets where $n$-dimensional torus act on. 
Huebshmann \cite{Hueb} also showed that this space branch-covers the $\SOThr$-character spaces; 

One can study this area using the Yang-Mills and Higgs bundle techniques as initiated by 
Atiyah-Bott \cite{AB}, Donaldson \cite{Do}, \cite{Do2}, 
Hitchin \cite{Hit}, Simpson \cite{Sim}, and Corlette \cite{Cor}. 
There are now extensive accomplishments in this area using these techniques. 
(See Bradlow, Garcia-Prada, and Gothen \cite{BGG1} and \cite{BGG2}.)

These spaces were originally studied by Ramanan (See Theorem 1 of  \cite{Ra} and \cite{DR}) following Newstead's work 
\cite{New2} and describes the spaces as the intersections of two quadrics but without 
details of the topology of the space as was pointed out by Oscar Garcia-Prada to the author. 
If we compare this paper to the papers using algebraic geometry or Yang-Mills fields, our 
method is more elementary and geometric. However, the steps and the details
to check seem more here since we are not using already established theories, and also
the arguments are not totally geometrical yet. 
The main point of our method seems to be that we have more direct way to 
relate the $\SOThr$-character space with the $\SUTw$-character space
with cell-structures preserved under the branching map. 

For $G$-character spaces of free groups for many Lie groups $G$, 
see Florentino and Lawton \cite{FL}. In some cases, the cohomology groups and explicit topological structures are known
when the fundamental group are free groups of rank two and three due to T. Goodwillie \cite{Good}
(see Baird \cite{Ba} also). 

We also remark that 
the branching property is of course known by Ramanan \cite{Ra} and Huebshmann \cite{Hueb}; but we find 
the exact branch locus here and the finite group action. 
One can find a much shorter proof but with no results on the cell structures or the branch loci using 
a generalization of the arguments of the recent paper of Florentino and Lawton \cite{FL} for the free-group representations; 
however, the author is not able to obtain the proof. 

A major disadvantage of our method is that we need to make use of the smoothness result of Huebshmann \cite{Hueb}. 
However, we have worked on a geometric way to avoid this problem and directly prove our theorems. 
We will continue our geometric methods and present the method if successfully accomplished. 
Our method seems to be able to replicate these results. We are currently interested 
in working on free groups of rank four and so on.




By a {\em representation}, we mean a group homomorphism into a specified Lie group such as $\SOThr$ or 
$\SUTw$. By a {\em character}, we mean an equivalence class of the representation by conjugations by elements of the Lie group. 
Also, any adjective modifying a representation modifies the corresponding character and vice versa in our definitions
by the convention that a character has a property if and only if every representation in the character has the property. 

Let $\Sigma$ be a closed surface of genus $2$ and $\pi_1(\Sigma)$ its
fundamental group and let $\SOThr$ denote the group of
special orthogonal matrices with real entries.
The space of homomorphisms $\pi_1(\Sigma) \rightarrow \SOThr$ admits an
action by $\SOThr$ given by
\[ h(\cdot) \mapsto g\circ h(\cdot) \circ g^{-1}, \hbox{ for } g \in \SOThr \mbox{ and a homomorphism } h. \] 
We can realize the set of homomorphisms
$\Hom(\pi_1(\Sigma), \SOThr)$ as an algebraic subset of $\SOThr^{4}$ where $4$ is
the number of generators of $\pi_1(\Sigma)$.
We denote by $\rep(\pi_1(\Sigma), \SOThr)$ the topological quotient space of the space
under the action of $\SOThr$. This is called the {\em space of $\SOThr$-characters} of $\pi_1(\Sigma)$. 
Sometimes, we say that a representation {\em belong to} it if the corresponding character belongs to it.

There are two topological components of $\rep(\pi_1(\Sigma), \SOThr)$. 
There is the identity component ${\mathcal C}_0$ containing the equivalence 
class of the identity representation, i.e., the one sending every element of 
$\pi_1(\Sigma)$ to the identity. The representations in 
the other component ${\mathcal C}_{1}$ correspond to the flat $\SOThr$-bundles 
with the Stiefel-Whitney numbers equal to $1 \in \bZ_2$.


We define a solid tetrahedron $G$ in the positive octant of $\bR^{3}$
by the equation $x+y+z \geq \pi$, $x \leq y+z - \pi$,
$y \leq x+z -\pi$, and $z \leq x+z -\pi$.
There is a natural action of the Klein four-group on $G$ by isometries
generated by three involutions each fixing a maximal segment in $G$
(See Figure \ref{fig:tetr1}.)

For a natural number $n$, 
let $T^{n}$ denote the $n$-dimensional torus, i.e., a product of $n$ circles $\SI^1$, 
and let $\bZ_{2}^{n}$ denote the direct sum of $n$ copies of
the cyclic group $\bZ_{2}$ of order $2$. For each $n$,
let $\SI^n$ denote the unit sphere in Euclidean space $\bR^{n+1}$. 

We will denote the Klein four-group by $V$ whenever possible. 
Actually $V$ is isomorphic to $\bZ_2^2$, which will be called a Klein four-group as well.
By a {\em double Klein four-group}, we mean a product of two copies of the Klein four-group,
which is isomorphic to $\bZ_2^4$.

We will realize $\CP^3$ as a torus fibration over $G$ where fibers over 
the interior points are $3$-dimensional tori, fibers over the interior of faces $2$-dimensional
tori, fibers over the interior of edges circles, and the fiber over a vertex a point. 
(See Section \ref{subsec:cp3} for more details.)







\begin{thmA}\label{thm:A}
Let $\pi_1(\Sigma)$ the fundamental group of a closed surface $\Sigma$ of
genus $2$.
\begin{itemize}
\item[(i)] The identity component ${\mathcal C}_0$ of $\rep(\pi_1(\Sigma),\SOThr)$ is
homeomorphic to the quotient space of $\rep(\pi_1(\Sigma),\SUTw)$ 
by a double Klein four-group action with nonempty fixed points.
\item[(ii)] $\rep(\pi_1(\Sigma),\SUTw)$ is homeomorphic to $\CP^3$.
\item[(iii)] The quotient by the double Klein four-group induces 
a 16-to-1 branch-covering of $\rep(\pi_1(\Sigma),\SUTw)$ 
onto ${\mathcal C}_0$. 
\item[(iv)] $\rep(\pi_1(\Sigma),\SOThr)$ has an orbifold structure with singularities in a union of 
six $2$-spheres meeting transversally. 
\end{itemize}
\end{thmA}


More precisely, we will see $\rep(\pi_1(\Sigma),\SUTw)$ as $\CP^3$ by 
inserting into $\CP^3$ the four $3$-balls corresponding to the vertices 
and inserting solid tori at the circles over the interior of edges. 
The parameters of solid tori over the open edges 
will converge to $3$-balls as they approach the fibers above the vertices.
The subspace of abelian characters consist of $2$-tori over 
the interior of faces and the boundary $2$-tori of the solid tori over edges and 
the boundary sphere of the vertex $3$-balls.

\begin{rem}\label{rem:mainA}
We can describe the branch locus 
exactly with two Klein four-groups acting differently. The two generates the double Klein four-group action. 
\begin{itemize}
\item[(i)] 
The Klein four-group $\bZ_2^2$ is acting on each of the torus fibers with
the branch locus equal to the union of six $2$-spheres over the six edges of the tetrahedron $G$. 
Two of the six $2$-spheres meet
at a point if and only if their edges do so at a vertex. Three of them meet at a unique point 
if and only if their edges do so.
\item[(ii)] The next Klein four-group  $\bZ_2^2$ 
acts on the tetrahedron $G$ and inducing the corresponding action on $\rep(\pi_1(\Sigma),\SUTw)$
preserving the torus fibration. The branch locus is a union of six $2$-spheres over three
mutually disjoint maximal segments in $G$. 
There are two disjoint $2$-spheres over each maximal segment. 
Over the same segment, the $2$-spheres are disjoint, and 
the spheres over distinct axes meet at a point if their segments do so. 
\item[(iii)] The spheres in (i) and the spheres in (ii) meet at two points if and only 
if the corresponding edges of $G$ and the maximal segments meet.
This gives the complete intersection pattern.
\item[(iv)] The images of the spheres are three for the type (ii) spheres and three for the type (iii) ones.
\end{itemize} 
\end{rem}

See Remark \ref{rem:finalorbstr} for more details. 

Consider a solid octahedron in $\bR^3$. Let $B^3$ be the $3$-dimensional unit ball 
and $\SI^3$ the unit sphere in $\bR^4$.
An {\em octahedral manifold} is a manifold obtained from the toric variety over 
the solid octahedron by removing the six singularities 
and gluing in six submanifolds homeomorphic to $B^3\times \SI^3$ 
to make it into a manifold. 

The octahedral manifold is a torus fibration over 
an octahedron so that over the interior of the octahedron the fibers are 
$3$-dimensional tori and over the interior of faces the fibers are $2$-dimensional tori 
and over the interior of the edges the fibers are circles and the over 
the vertex the fibers are $3$-spheres. We will define this precisely later. 

Let $\Idd$ be the identity matrix in $\SUTw$ and define $-\Idd = (-1)\Idd$. 
Let $\Sigma_1$ denote a surface of genus two with one puncture, and 
$\rep_{-I}(\pi_1(\Sigma), \SUTw)$ be the quotient space of the subspace of 
$\Hom(\pi_1(\Sigma_1),\SUTw)$ determined by the condition that 
the holonomy of the boundary curve $-I$ under the usual conjugation action.

\begin{thmB}\label{thm:B}
\begin{itemize}
\item[(i)] ${\mathcal C}_1$ is homeomorphic to the double Klein four-group quotient of 
an octahedral manifold.
\item[(ii)] $\rep_{-I}(\pi_1(\Sigma_1),\SUTw)$ is homeomorphic to the octahedral manifold 
and branch-covers ${\mathcal C}_1$ in a $16$ to $1$ manner by an action of 
$\bZ_2^4$ and has a cell structure. 
\item[(iii)] There is a Klein four-group $\bZ_2^2$-action preserving each of the torus fibers of the space, where
the branch locus is a union of three mutually disjoint $2$-tori.
There is another Klein four-group  $\bZ_2^2$-action on the space where 
the branch locus is the union of the three $2$-tori. 
They generate a double Klein group action with the branch locus the union of six $2$-tori. 
(See  Section \ref{subsec:axes}.)
\item[(iv)] The quotient of $\rep_{-I}(\pi_1(\Sigma_1,\SUTw)$ under the double Klein four-group action 
gives us ${\mathcal C}_1$ and has an equivariant $T^3$-cell decomposition. 
\end{itemize}
\end{thmB}



We remark that the space ${\mathcal C}_1$ has been already studied by Ramanan \cite{Ra}
to some depth and he describes the branch maps. 
For the $\SUTw$-character variety $\rep_{-I}(\pi_1(\Sigma_1),\SUTw)$, 
Newstead \cite{New2} studied it as the space of stable bundles of rank two and odd degrees over a curve of genus 2
and he show that it is homeomorphic to the intersection of two quadrics. 
(See also Desale and Ramanan \cite{DR}.) Miles Reid communicated to the author by an e-mail that 
the author's description as a toric variety with six singular points replaced by three dimensional spheres  
is homeomorphic to the intersection description. (See his thesis \cite{Re} to understand the issues here.)


We briefly outline this paper: 
We start with $\rep(\pi_{1}(P), \SOThr)$ for a pair of pants $P$.
We show that there is a dense set where the characters are
obtainable by gluing two isometric triangles. Every
character is obtained by gluing two isometric triangles which
may be degenerate ones. The choices involved give us the
$V$-action.

The surface $\Sigma$ decomposes into two pairs of pants.
The space $\rep(\pi_1(\Sigma), \SOThr)$ is then constructed using
the gluing maps between two pairs of pants.
We obtain the two components in this manner
and realize them as quotient spaces of $T^{3}$-bundles over a
blown-up tetrahedron or a blown-up octahedron
under a $V$-action.

In Section \ref{sec:glimit}, we compactify the space of isometry classes of spherical triangles
 by adding generalized triangles including hemispheres, segments, 
and points. This space can be viewed as a tetrahedron truncated at vertices and edges
in a parallel manner. 

In Section \ref{sec:rpp},  we will describe the $\SOThr$-character space of 
the fundamental group of a pair of pants by relating spherical triangles with 
the $\SOThr$-characters. Each of them is a tetrahedron with 
the above described action by $V$.
(This is an old idea which also appeared in 
Thurston's book \cite{Thbook} and similar to Goldman \cite{Gconv}.)

In Section \ref{sec:sosu}, we discuss the relationship of $\SUTw$ with $\SOThr$ and find 
a useful way to describe elements. 
We redescribe the $\SOThr$- and $\SUTw$-character spaces of 
a free group of rank 2 for the purposes of using the idea in the later sections.

In Section \ref{sec:ch}, we describe the $\SOThr$-character space for the 
fundamental group of the surface $\Sigma$ of genus $2$, which has 
two components ${\mathcal C}_0$, containing the identity representation, and 
the other component ${\mathcal C}_1$.   

In major Section \ref{sec:C0}, we show how we can consider ${\mathcal C}_0$ as a quotient space of the 
$T^3$-bundle over the blown-up tetrahedron above, and
we describe the explicit quotient relations for ${\mathcal C}_0$ by going over 
each of the faces of the blown-up tetrahedron. This will be a long section; however, the ideas here 
are very repetitive and we write them for the convenience of the readers and for future uses.

In Section \ref{sec:SUT}, we move to the $\SUTw$-character space of the fundamental group of 
$\Sigma$ and we consider the geometric representations of such characters using 
the spherical triangles again. The character space is again considered a quotient 
space of the $T^3$-bundle over the blown-up tetrahedron. 
We determine the topology using the quotient relation. Here, we make use of the smoothness
result of this space. (See Huebshmann \cite{Hueb}.)

In Section \ref{sec:topC0}, we determine the topology of ${\mathcal C}_0$ and describe the $\bZ_2^4$-action 
on the $\SUTw$-character space of the fundamental group of 
$\Sigma$ to branch-cover ${\mathcal C}_0$. 

In Section \ref{sec:other}, we now turn our attention to ${\mathcal C}_1$. Here, we show that our space 
is the quotient space of an octahedron blown-up at vertices times $T^3$. 
We describe the equivalence relations. 

In Section \ref{sec:su2o}, we show that the space $\rep_{-I}(\pi_1(\Sigma_1), \SUTw)$ of $\SUTw$-characters of 
the fundamental group $\pi_1(\Sigma_1)$ of the once-punctured genus two surface $\Sigma_1$ with 
the puncture holonomies equal to $-\Idd$ is homeomorphic to an octahedral manifold.

In Section \ref{sec:C1}, we determine the topology of ${\mathcal C}_1$ and describe the $\bZ_2^4$-action 
on the above manifold to branch-cover ${\mathcal C}_1$.

\section{The geometric limit configuration space}\label{sec:glimit}

In this section, we will define the solid tetrahedron $G$ precisely and the Klein four-group action on it.
The above tetrahedron is not enough 
for our purpose. We need to consider the sequence of spherical triangles converging 
to the boundary. We will use the geometric-limit configuration spaces for the next
section when we are discussing the surface of genus $2$.

First, we describe a tetrahedron $G$ and the blown-up tetrahedron $\tilde G$.
Next, we will describe the geometric-limit configuration space of 
generalized triangles and its metric and topology. 
We show that the geometric-limit configuration space can be identified with 
a blown-up tetrahedron by describing each face of $\tilde G$ with explicit space of 
generalized triangles. We show that the Klein four-group $V$ acts on $\tilde G$ extending the action on $G^o$.
Finally, we present the standard position for generalized triangles. 

We remark that the author cannot find a way to identify $\tilde G$ as a blown-up tetrahedron in 
a direct and more algebraic way without using the topological methods below. 
We believe that there should naturally be a known real algebraic blow-up construction along some divisers 
to see $\tilde G$ as being blown-up from $G$. 
\subsection{A blown-up tetrahedron}

By a {\em rotation}, we mean a transformation which can be locally conjugated to be 
an orientation-preserving isometry in the unit $n$-sphere or a Euclidean $n$-space 
with a set of fixed point a subspace of codimension two for $n \geq 2$.
A {\em rotation angle} is measured by taking angle from the fixed subspace in counterclockwise direction
once the fixed subspace is given an orientation. 

We define a tetrahedron $T$ in the positive octant of $\bR^{3}$
by the equation $x+y+z \geq \pi$, $x \leq y+z - \pi$,
$y \leq x+z -\pi$, and $z \leq x+z -\pi$.
This has vertices $(\pi,0,0)$, $(0,\pi,0)$,$(0,0,\pi)$, and
$(\pi,\pi,\pi)$ and is a regular tetrahedron with
edge lengths $\sqrt{2}\pi$.

Recall that the {\em Klein four-group} $V$ is a four element group $\{\Idd, i, j, k\}$
where $ij=k,jk=i,ki=j$ isomorphic to $\bZ_2 \oplus \bZ_2$.
We can identify each element of $V$ with an involutionary rotation with an axis through
$(\pi/2, \pi/2, \pi/2)$ parallel to one of the axis (see Figure \ref{fig:tetr1}).
(Here an {\em axis} means a maximal segment fixed by an involution)
We denote by $I_A$ the involution about the axis parallel to the $z$-axis,
$I_B$ the one about the axis parallel to the $x$-axis, and $I_C$ the one about 
the axis parallel to the $y$-axis. 

Actually $V$ is isomorphic to $\bZ_2^2$, which will be called a Klein four-group as well.
We will denote the Klein four-group by $V$ whenever possible. 
By a {\em double Klein four-group}, we mean a product of two Klein four-groups,
which is isomorphic to $\bZ_2^4$.

We will realize $\CP^3$ as a torus fibration over $G$ where fibers over 
the interior points are $3$-dimensional tori, fibers over the interior of faces $2$-dimensional
tori, fibers over the interior of edges circles, and the fiber over a vertex a point. 
(See Section \ref{subsec:cp3} for more details.)

\begin{figure}
\centerline{\includegraphics[height=7cm]{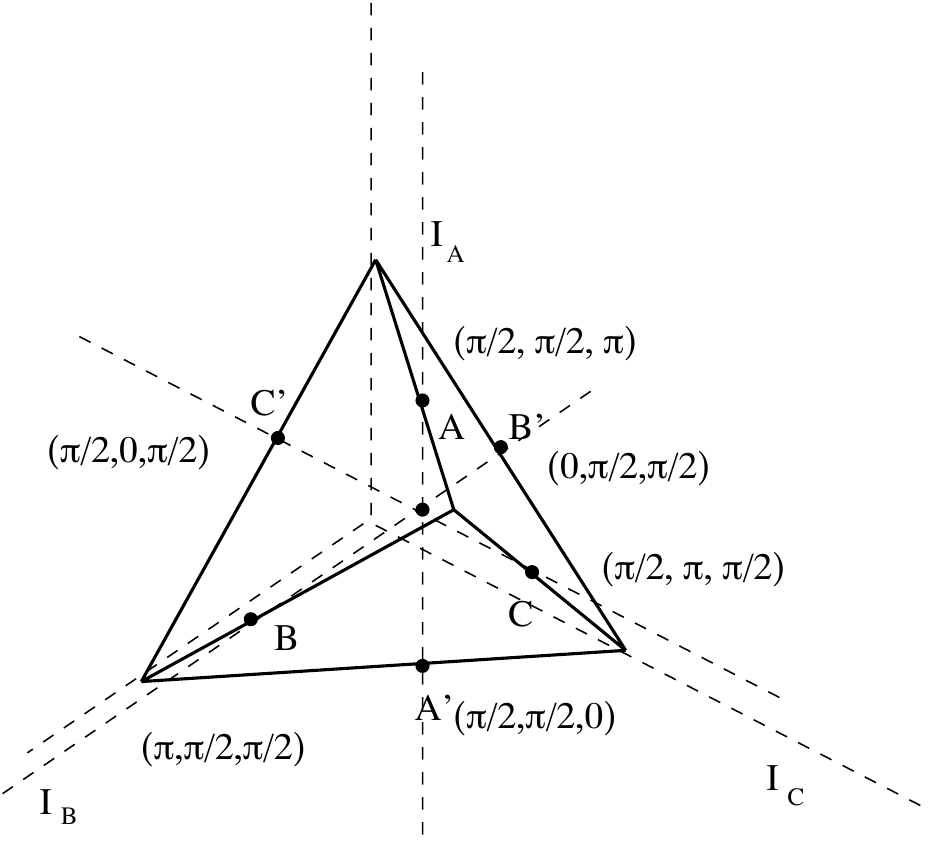}}
\caption{The tetrahedron and the Klein four-group-symmetries. 
The three edges in front are labeled $A,B$, and $C$ 
in front and the three opposite edges are labeled $A',B'$, and $C'$.}
\label{fig:tetr1}
\end{figure}

A {\em blown-up tetrahedron} $\tilde T$ is combinatorially 
obtained by first truncating small tetrahedron $T$ at each of the vertices of $T$
with base parallel to the respective opposite side of the vertex
and then truncating a thin prism along each of the remaining six edges by a plane parallel to 
one of the $xy$-, $yz$-, or $zx$-plane
so that a vertex of one of the six newly formed rectangles
is a vertex of another rectangle:
\begin{multline}
 x+y+z \geq \pi, y \leq x+z -\pi,  x \leq y+z - \pi,  z \leq x+z -\pi, \\
 \eps \leq x \leq \pi-\eps, \eps \leq y \leq \pi-\eps, \eps \leq z \leq \pi-\eps \\
 x+y+z \leq 3\pi- 4\eps, x \leq y+z +\pi - 4\eps, y \leq x+z +\pi - 4\eps, z \leq x+y+\pi-4\eps 
\end{multline}
where $\eps >0$ can be any number less than say $1/10$.
We label as follows (see Figure \ref{fig:trtet}): 
\begin{multline} 
a: x+y+z=\pi, b: y=x+z-\pi, c: x+y+z -\pi, d: z = x+y -\pi \\
A: z = \pi-\eps, A': z = \eps, B: x =\pi-\eps, B':x =\eps, C: y = \pi- \eps, C': y = \eps \\  
I: x+y+z = 3\pi - 4\eps, II: x = y+z + \pi - 4\eps, III: y=x+z+\pi -\eps, IV: z=x+y+\pi - 4\eps 
\end{multline}
The faces are given by $\tilde T$ intersected with the planes of the respective equations.
(Here of course the geometry does not matter and only the combinatorial structure is needed for 
our purposes. More precisely, $\tilde T$ is the combinatorial polytope obtained from the region )


We can think of $\tilde T$ as $T$ 
with each vertex replaced by a triangle and 
each point of the interior of edge replaced by a segment. 
This replacement is reminiscent of algebraic blowing-up operations: 
(Unfortunately, the author cannot really identify the process with algebraic geometric blow-up processes 
presently.)

We will use the term blown-up tetrahedron.
(See Figure \ref{fig:trtet}.) By a {\em region}, we will mean 
the subdomain in the boundary of $\tilde T$ or 
$\tilde T$ times a three-dimensional torus and their quotient images under 
the quotient maps that we will introduce later.
\begin{figure}

\centerline{\includegraphics[height=6cm]{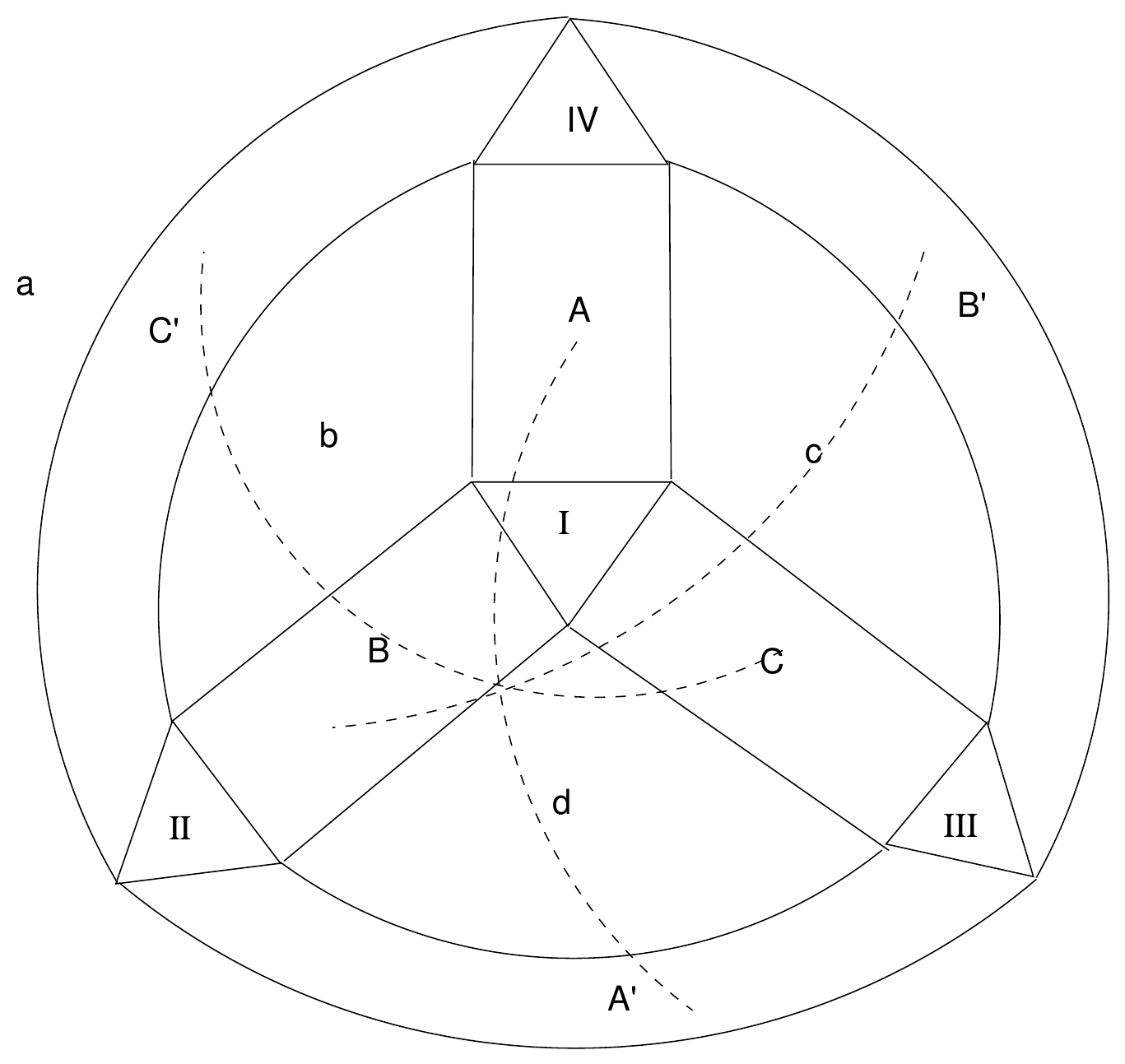}}
\caption{The face diagram of blown-up solid tetrahedron and
regions to be explained later. 
}
\label{fig:trtet}
\end{figure}


\subsection{Generalized spherical triangles and geometric-limit configuration spaces}

We will be using the Hausdorff distances between closed subsets of $\SI^2$: 
Denote by $N_\eps(A)$ the $\eps$-neighborhood of a subset $A$ for a real number $\eps> 0$. 
Let $K_1$ and $K_2$ be the two closed subsets of $\SI^2$, and hence they are compact. 
We define the {\em Hausdorff distance} between $K_1$ and $K_2$ by the condition
$d_H(K_1, K_2) < \eps$ if $K_1 \subset N_\eps(K_2)$ and $K_2 \subset N_\eps(K_1)$.

A geodesic in $\SI^2$ is an arc in a great circle in $\SI^1$. 
A {\em short geodesic} is a geodesic of length $\leq \pi$. 
A {\em great segment} is a geodesic segment of length equal to $\pi$. 

By a {\em triangle} in $\SI^2$, we mean a disk 
bounded by three short geodesics and with vertices to be denoted by 
$v_0, v_1, v_2$ and edges $l_0, l_1, l_2$ opposite to these respectively. 
They are spherical triangles in the usual terminology outside this paper. 
We will always assume that $v_i$s and $l_i$s appear in the clockwise direction. 

A {\em lune} is the closed domain in $\SI^2$ bounded by two segments
connecting two antipodal points forming an angle $< \pi$.
A {\em hemisphere} is the closed domain bounded by a great circle
in $\SI^{2}$.

We say that ordinary triangles to be {\em nondegenerate triangles}. 
We define pointed-lunes, pointed-hemispheres,
pointed segments, and pointed-points here. These are {\em degenerate triangles}.
\begin{itemize}
\item A {\em pointed-lune} is a lune with three ordered points where two of them
are antipodal vertices of the lune, and the third one is either
on an edge or identical with one of the vertices.
If, in the boundary, the three points appear in clockwise
direction, then the pointed lune is positively oriented. If
two points are identical, the pointed-lune is considered {\em degenerate}.
\item A {\em pointed-hemisphere} is a hemisphere with three ordered points
on the boundary great circle where a segment between any two not
containing the other is of length $\leq \pi$.
Again, a {\em pointed-hemisphere} is positively oriented if
three points appear in the clockwise direction.
Note that two of the points could be identical and
now two distinct points are antipodal. In this case,
the pointed-hemisphere is said to be {\em degenerate}.
\item A {\em pointed-segment} is a segment of length $\leq \pi$
with three ordered points where two
are the endpoints and one is on the segment.
Here again, the third point could be identical with one of the endpoint,
and the {\em pointed-segment} is {\em degenerate}.
\item A {\em pointed-point} is a point with three identical vertices.
\end{itemize}

Sometimes, we will drop the ``pointed'' in using the terminology.

A nondegenerate or degenerate triangle is a {\em generalized triangle}. 
A {\em triangle} is just a generalized triangle from now on.
The three ordered points are said to be {\em vertices} and 
are denoted by $v_0, v_1, v_2$.

The notion of edges for nondegenerate triangles is the same as in the ordinary spherical 
geometry. One can define notions of edges of the degenerate triangles 
as the closure of the appropriate component of the complement of the 
pair of vertices if the pair of vertices is distinct, and the set of the pair if 
the pair is not distinct. 
Thus, we can form generalized edges $l_0, l_1$, and $ l_2$ 
as being opposite $v_0,v_1$, and $v_2$ respectively.
Again, the orderings must follow the clockwise direction.

The notion of angles for nondegenerate triangles is the same as the ordinary notions.
We now associate angles to each of the three vertices of degenerate triangles 
by the following rules, where one can assign more than one set of angles to a given degenerate triangle.
The angles are numbers in $[0,\pi]$. Let us use indices in $\bZ_3$:
\begin{itemize} 
\item If a vertex $v_i$ has two nonzero length edges $l_{i-1}$ and $l_{i+1}$ ending 
at $v_i$, then we define the angle $\theta_i$ at $v_i$ to be the interior angle between the edge
vectors oriented away from $v_i$. 
\item If a vertex $v_i$ is such that exactly one of $l_{i-1}$ or $l_{i+1}$ has a zero length, say $l_{i-1}$ without loss of generality, 
then we choose an arbitrary great circle $\SI^1_{i-1}$ containing $v_i$ so that the generalized triangle is contained  
in the closure of one of the hemisphere $H$ bounded by it. 
We take the counter-clockwise unit tangent vector for $\SI^1_{i-1}$, to be called the {\em direction vector} at $v_i$ for $l_{i-1}$, 
and we take the unit tangent vector for $l_{i+1}$ at $v_i$ and measure the angle between them. 
The a degenerate segment $l_{i-1}$ with the projective class of the direction vector is said to be an {\em infinitesimal edge}.
If $l_{i+1}$ has a zero length, then we choose $\SI^1_{i+1}$ in the above manner and the direction vector at $v_i$
is chosen to be in the clockwise direction.
\item If a vertex $v_i$ is such that both of $l_{i-1}$ or $l_{i+1}$ are zero lengths, then
we have a pointed-point, and the angles to the three vertices are given arbitrarily so that 
they sum up to $\pi$.
\end{itemize}


The following are the consequences of the above angle assignment method:
\begin{itemize}
\item For the pointed-lune with three distinct vertices, 
the angles at the vertices of the lune are the ones
given by the segments, i.e., the angle of the lune itself,
and the vertex in the interior of
a segment is given the angle $\pi$.
\item For the pointed-lune with two vertices equal,
the vertex which are distinct from others are given
the angle of the lune, and the two vertices are given
angles greater than or equal to the lune-angle by the above procedure.
We can show that the three angles have a sum equal to $\pi$ plus $2$ times the angle of the lune.
\item For the pointed hemisphere, the vertices are given angles $\pi$.
\item For the pointed segment with three distinct vertices,
the angles $0, 0, \pi$ are given to the endpoints of the segment
and the vertex in the middle.
\item For the pointed segment with two vertices equal,
the vertex which are distinct from the others is given the angle $0$
and the other two are given angles which sums to $\pi$.
\item For a pointed-point, the three vertices are given three angles summing up to
$\pi$. These configurations with angles assigned are said to be {\em degenerate triangles}.
\end{itemize}

Let $\hat G$ denote the space of all triangles, pointed-lunes,
pointed-hemispheres, pointed-segments and pointed-points on $\SI^{2}$
with angles assigned.
Let $\hat G$ be given a metric defined by
letting $D(L, M)$ to be maximum of
the Hausdorff distance between the closed domains of elements $L$ and $M$ of $\hat G$
and the Hausdorff distances between corresponding points and segments of $L$ and $M$
and the absolute values of the differences between the corresponding angles
respectively. 

\begin{prop}\label{prop:dense0} 
$\hat G$ is compact under the metric,
and the subspace of nondegenerate triangles are dense in $\hat G$.
\end{prop}
\begin{proof}
The space $\mathcal K$ of all compact subsets of $\SI^2$ with the Hausdorff metric is compact. 
$\hat G$ is a closed subset of the compact metric space $\mathcal K \times (\SI^2)^3 \times \mathcal K^3 \times [0, \pi]^3$.
Thus it is compact. 

Given a degenerate triangle $T'$, we take three great circles $\SI^1_i$ containing $l_i$ for $i=0,1,2$ 
and by a perturbation, we can obtain a nondegenerate triangle $T''$ in any small neighborhood of $T'$.
\end{proof} 

\begin{defn}\label{defn:edgecircle}
Given a generalized triangle, for each edge, degenerate or not, is contained in
a unique great circle determined by the direction vector and the vertices of the edge.
When we say an edge of a triangle is {\em contained} in a great circle, we mean only this situation.
\end{defn}

An element of $\hat G$ can be put using
an orientation-preserving isometry to a {\em polar position}: 
Given a generalized triangle with vertices $v_0, v_1,$ and $v_2$, 
if $v_0\ne v_1$, then by
setting $v_0$ be the north pole $[0,0,1]$ and the edge $l_2$ connecting
$v_0$ to $v_1$ in the boundary to lie on the $xz$-plane intersected with $\SI^2$.
If $v_0 = v_1$, and not equal to $v_2$, then we move $v_0$ to
the north pole and move the edge $l_1$ to lie in the $yz$-plane intersected with
$\SI^2$ by isometries. If $v_0=v_1=v_2$, then we simply put every thing to the north pole.
This agrees with the clock-wise orientation of $v_0,v_1$, and $v_2$ on 
the boundary of the triangle.

The isometry group $\SOThr$ acts properly on $\hat G$.
The quotient topological space is denoted by $\tilde G$.
This is a compact metric space with metric induced from $\hat G$
by taking the Hausdorff distances between the orbits.
We will denote by $G^o$ the quotient space of the space of 
nondegenerate triangles by the $\SOThr$-action.

\begin{prop}\label{prop:Go} 
Let $G^o$ be the set of isometry classes of nondegenerate triangles. 
Then there is a map sending 
a spherical triangle to its angles: $f(t)= (\theta_1, \theta_2, \theta_3)$ for 
the angles $\theta_1, \theta_2,$ and $\theta_3$ of $v_1, v_2,$ and $v_3$.
The space $G^o$ of isometry classes of nondegenerate
triangles is homeomorphic with  interior of $T$ by $f$.
\end{prop}
\begin{proof} 
This map is one-to-one since angles determine a triangle up to isometry by classical spherical geometry.
Furthermore, $f$ is continuous since the metric $D$ contains the angle differences. 

Also, the image of $f$ is in $T^o$:
Let us choose an element in $G^o$ and realize it by a nondengenerate triangle $t$.
We know that the angles $\theta_1, \theta_2$, and $\theta_3$ sum to a number $\geq \pi$
by spherical geometry.

Also, constructing a lune $L$ whose sides contain two sides $l_1$ and $l_2$ of the triangle $t$, we see that the closure of 
$L-t$ is another nondegenerate triangle $t'$ with angles $\theta_0, \pi-\theta_1, \pi-\theta_2$. 
Thus $\theta_0+ \pi-\theta_1+ \pi-\theta_2 > \pi$ or $\theta_0 < \theta_1 + \theta_2 -\pi$. 
By considering the indices to be in $\bZ_3$ and the symmetry of indices, we see that the image is in $T^o$. 

Conversely, given an element of $T^o$, we first construct a lune $L$ with angle $\theta_0$ at a point and let the point be the vertex $v_0$.
Then we choose a point $v'_1$ on a right edge $e_1$ of $L$ and take a maximal segment $s_1$ in $L$ from $v'_1$ of angle $\theta_1$. 
Let $v'_2$ to be the endpoint of $s_1$ in $e_2$ and let $\theta'_2$ be the angle at $v'_2$ 
in the triangle with vertices $v_0, v'_1, v'_2$. 

We first assume $\theta_0 \geq \theta_1$ by symmetry. 
We temporarily assume that $\theta_0+\theta_1 < \pi$
Moving $v_1$ on $e_1^o$, we can realize the variable $\theta'_2$ can realize 
any value in the open interval $(\pi-\theta_0-\theta_1, \pi-\theta_0+\theta_1)$. 
We now suppose that $\theta_0 + \theta_1 \geq \pi$. 
Moving $v_1$ on $e_1^o$, we can realize the variable $\theta'_2$ can realize 
any value in $(-\pi+\theta_0+\theta_1, \pi-\theta_0+\theta_1)$.
We conclude that over the region $[0, \pi] \times [0, \pi]$ with $\theta_0 \geq \theta_1$ 
$\theta'_2$ can achieve any angle in $T^o$.

By symmetry between $0$ and $1$, we obtain that any point of $T^o$ is realizable as a triangle. 
Also, our argument constructs a continuous map $T^o \ra G^o$ by construction of a triangle from angles. 
This is a homeomorphism as can be verified by the above construction. 
\end{proof}


As discussed above, each triangle or degenerate triangle has an angle or length associated with 
each vertex and edge. We denote by $v(0),v(1), v(2)$ the angles 
at the vertices $v_0, v_1, v_2$ respectively and $l(0), l(1), l(2)$ the length of 
the edges $l_0, l_1, l_2$ respectively. 
	
\begin{prop}\label{prop:conv}
A sequence of elements of $\tilde G$ converges to
an element of $\tilde G$ if and only if the corresponding sequences of angles and
lengths converge to those of the limit. Hence, we can imbed $\tilde G$ as
the compact closure of the image of $G^o$ in $\bR^6$ by the function
$\iota: \tilde G \ra \bR^6$ given by sending an isometry class of 
a generalized triangle to the associated angles and lengths
$(v(0),v(1),v(2), l(0),l(1), l(2))$.
\end{prop}
\begin{proof}
If a sequence of 6-tuples of angles and lengths converges and
the corresponding sequence of lengths denoted by $l(2)$ is bounded below by
a small positive number, then
the corresponding elements of $\hat G$ can be put in $\SI^2$
in polar positions so that the sequence of 6-tuples converges geometrically
to an element of $\hat G$ by spherical geometry.

If the sequence of lengths denoted $l(2)$ is going to zero but
the sequence of lengths denoted $l(1)$ is not, then we change
the polar position to ones where the segments denoted by $l(1)$ are put on
a great circle defined by $x=0$ and putting the vertex denoted by $v_0$ on the north pole.
The arguments are the same. If the sequences of lengths $l(1)$ and $l(2)$ both go to 
zero, then the sequence of triangles converges to a pointed point, and 
the sequences of angles and lengths all converge.  Hence, we have a convergence 
in $\tilde G$.

Conversely, if a sequence of elements of $\tilde G$ converges
to say $e$, then we can again put the elements to polar forms.
Since $\hat G$ is compact, the sequence has a subsequence
converging in $\hat G$. Since $\SOThr$ is compact, the limit is in the equivalence class of
the limit $e$. The corresponding sequences of lengths and angles converge also clearly.

The map $\iota$ is continuous because of the above paragraph. 
$\iota$ is injective since two degenerate or nondegenerate triangles are clearly 
not isometric if they have different angles or different edge lengths.  
 Since $\tilde G$ is compact and $\bR^6$ is Hausdorff, the map $\iota$ is an imbedding.

\end{proof}

The main aim of the section is to prove the following proposition. 

\begin{thm}\label{thm:T}
    The geometric-limit configuration space $\tilde G$
     is homeomorphic to a blown-up solid tetrahedron.
    \end{thm}
    
    We prove this in Section \ref{subsec:proofT}.
But first, we need to describe the $2$-dimensional domains classifying the above degenerate triangles: 
  
  \subsection{Parameterizing the degenerate triangles}

We will classify the degenerate triangles according to their types and show that the
collection form nice topology of triangles and rectangles, i.e., $2$-cells. 
(See Figures \ref{fig:figabcd},  \ref{fig:figABC}, and \ref{fig:regI}  as reference.)
Let us denote by $l(i)$ the coordinate function measuring length of $l_i$ for $i=0,1,2$,
and $v(j)$ the coordinate function measuring the angle of $v_i$ for $i=0,1,2$.

\begin{figure}

\centerline{\includegraphics[height=6cm]{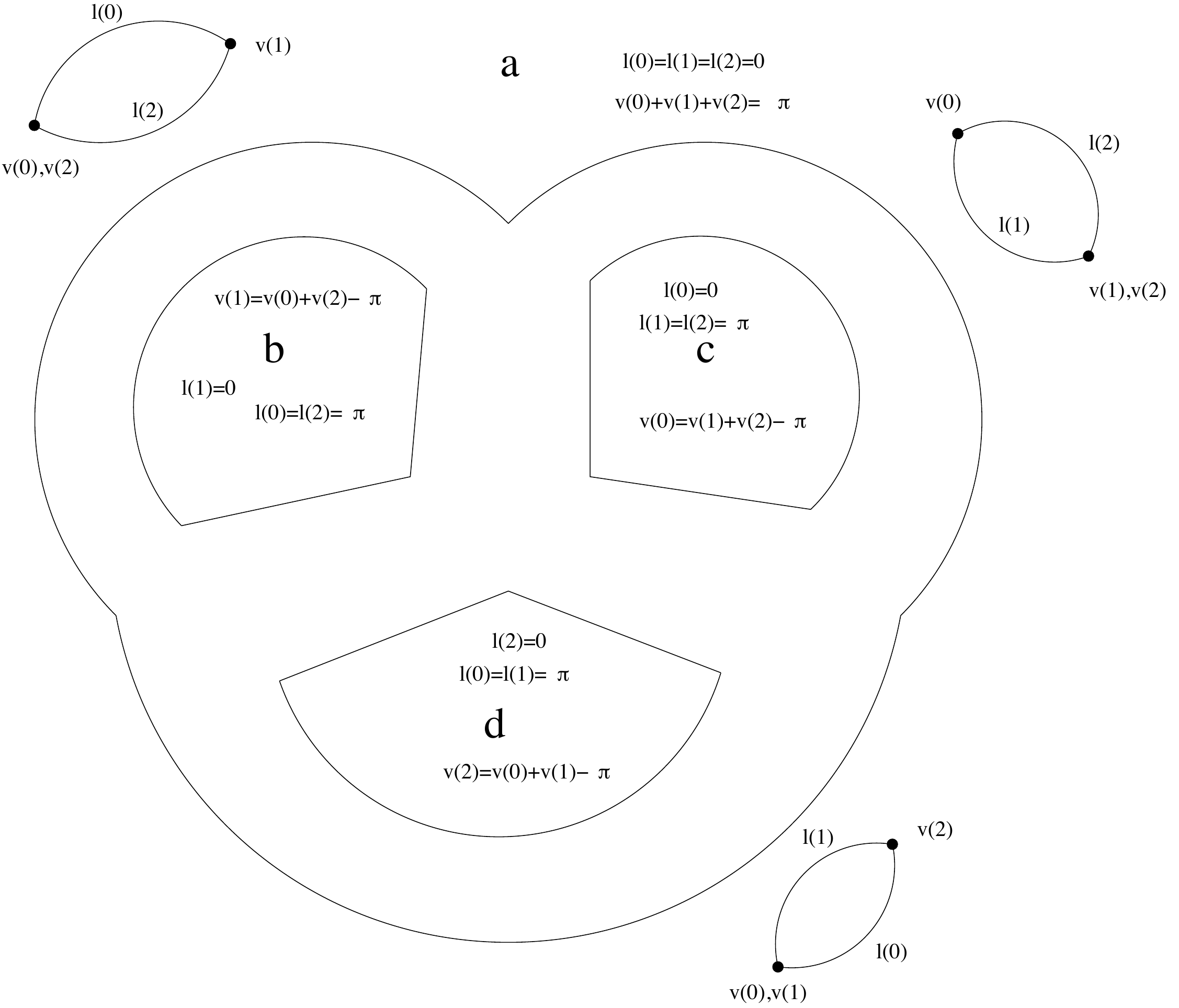}}
\caption{Regions $a,b,c$, and $d$ }
\label{fig:figabcd}
\end{figure}

First, we discuss regions $a, b, c$, and $d.$ (See Figure \ref{fig:figabcd}.)
These are formed by degenerate triangles where two vertex are the same and the other 
is either equal to the first ones or antipodal to them. These are pointed points or pointed lunes.
There are four classes given by $l(0)=l(1)=l(2)=0$ (the face $a$), 
$l(0)=l(2)=\pi$ and $l(1)=0$ (the face $b$), $l(1)=l(2)=\pi$ and $l(0)=0$ (the face $c$), 
and $l(0)=l(1)=\pi$ and $l(2)=\pi$ (the face $d$). 


Each point of region $a$
corresponds to a pointed-point with three vertices which are
all equal to one another. Here lengths $l(0), l(1)$, and $l(2)$ are all zero
and vertices $v_0, v_1$, and $ v_2$ are given angles $v(0),v(1)$, and $v(2)$
summing up to $\pi$. Two of the angles parametrize the space as we can
see from the definition of our metric $D$ above. 
They satisfy $v(0)+v(1)+v(2)=\pi$ and $0 \leq v(i) \leq \pi$ for $i=0,1,2$.

Each point of region $b$ corresponds to a lune with
vertices $v_1$ and $v_0=v_2$. The vertices are
assigned angles $v(0),v(1),v(2)$ satisfying
$v(1)=v(0)+v(2) -\pi$ with $0 \leq v(i) \leq \pi$ for $i=0,1,2$.

Regions $c$ and $d$ are described in the similar way in Figure \ref{fig:figabcd}.

The regions $a,b,c$, and $d$ are still said to be {\em faces}.

The angles parameterize each of $a, b, c, d$ and produce a one-to-one correspondences. 
Since the parametrizing rectangles are compact and $\tilde G$ is Hausdorff, we conclude that 
they are closed imbedded $2$-dimensional disks in $\tilde G$ and are mutually disjoint. 

\begin{figure}

\centerline{\includegraphics[height=10cm]{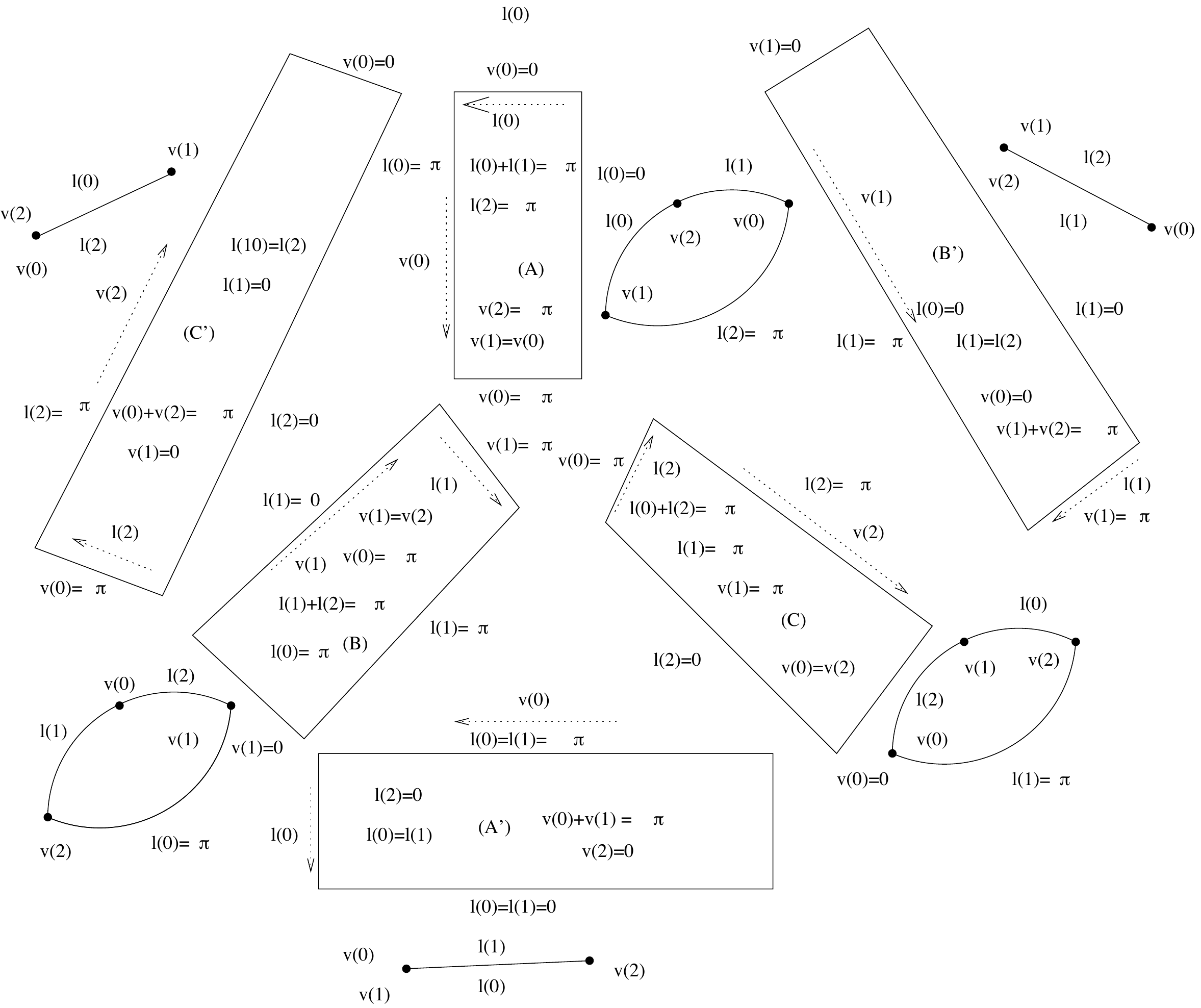}}
\caption{Regions $A, A', B, B', C$, and $C'$ }
\label{fig:figABC}
\end{figure}

For regions $A, B,$ and $C$, we classify lunes with one vertex of angle $\pi$
and the lune angles in $(0, \pi)$; a segment with one vertex of angle $\pi$ and two antipodal vertices or;
a hemisphere with one angle $\pi$ and the other two vertices are antipodal. 
For regions $A', B'$ and $C'$, we classify a segment with one vertex of angle $0$ and 
two other vertices identical including a pointed point with one vertex angle $0$, or 
a segment of length $\pi$, with one end vertex  angle $0$ and the other two vertices equal:

Each point of ${A}$ corresponds to a lune with vertices $v_0$ and $v_1$ and
the vertex $v_2$ is on a segment of length $\pi$ between them.
Thus $l(2)=\pi$ and $v(2)=\pi$ here. 
For region ${A}$, there are two coordinate functions $l(0)$ increasing from $0$ to $\pi$
as indicated by the top dotted arrow and $v(0)$ increasing from $0$ to $\pi$
indicated by the down dotted arrow in the left of the square ${A}$.
The values of $l(0)$ and $v(0)$ determine the configuration
uniquely. The coordinate system is given by $(v(0),l(0))$.

A {\em tie} is a segment in the edge region defined by the lune angles set constant.
For example, in $A$, it is given by the equation $v(0)=\theta$ for $\theta \in [0, \pi]$
and $l(0)$ can take any values in $[0, \pi]$. An {\em end tie} is a tie in 
the end of the edge region. A {\em mid-tie} is a tie given by setting 
the angle parameter exactly in the middle. For $A$, it is given by $v(0)=\pi/2$.

For region ${A}'$, each point on the square corresponds
to a segment of length $l(0)$ with vertices $v_0 = v_1$ and $v_2$.
We have $l(2)=0$ and the two coordinate functions are $l(0)$ and $v(0)$
represented by top dotted arrow pointing right and
the left dotted arrow pointing up. They take values in $[0,\pi]$ as above. 
$v_0$ is given an angle $v(0)$ and $v_1$  given the angle $v(1)$ where
$v(0)+v(1) = \pi$ and $v_2$ is given the angle $0$. Here $l_2$ is
a degenerate segment, i.e., a point. 
The coordinate system is given by $(v(0),l(0))$.

There is actually a one-to-one correspondence between ${A}$ and ${A}'$
by taking the lune and moving $v_0$ to $v_1$ and
taking the segment between $v_0$ and $v_2$.
This changes the associated angle at $v_0$ to be $\pi$ minus the original angle
and similarly for the angle at $v_2$.

For region $B$, each point of the rectangle corresponds to a lune with
vertices $v_1$ and $v_2$ with $v_0$ on the segment between
$v_1$ and $v_2$ as a generalized triangle. We have  $v(0)=\pi$. 
We have two parameterizing functions $l(1)$ and $v(1)$.

For region $B'$, each point of the rectangle, we associate
a segment with vertices $v_0$ and $v_1=v_2$ of length $l(1)$. 
For regions $B$ and $B'$, the coordinate systems are given by
$(v(1), l(1))$.

There is a one-to-one correspondence between $B$ and $B'$ given
by sending $v_1$ to $v_2$ in the lune for $B$ and changing the
angles at $v_1$ and $v_0$ to $\pi$ minus the original angles
respectively.

The analogous descriptions hold for regions $C$ and $C'$.
The coordinates are given by $(v(2),l(2))$.
Equivalently, we can use $(v(0),l(0))$.

These domains from the edges of $G$ are said to be {\em edge regions}.

We note that the edge regions are disjoint except at the vertices. 
For each of the three cycles in $A, B, C, A', B', C'$, we collapse six vertices to three vertices 
as indicated in Figure \ref{fig:trtet}. This can be easily verified using 
the metric $D$ and looking at the configurations. 


For region $I$, pointed-hemispheres are classified here. Here $v(0)=\pi, v(1)=\pi, v(2)=\pi$.
We identify each point $(l(0), l(1), l(2))$ with a closed hemisphere with
vertices $v(0), v(1), v(2)$ where $l(i)$ is the distance between
$v(i+1)$ and $v(i+2)$ for each $i=0,1,2$.
The functions $l(0), l(1)$, and $l(2)$ are extended as
in the Figure \ref{fig:regI} satisfying relations $l(0)+l(1)+l(2)=2\pi$
where $0 \leq l(i) \leq \pi$ for each $i=0,1,2$.
Actually, the equation gives a plane in $(l(0),l(1),l(2))$-space intersected
with the closed positive octant. 

\begin{figure}

\centerline{\includegraphics[height=9cm]{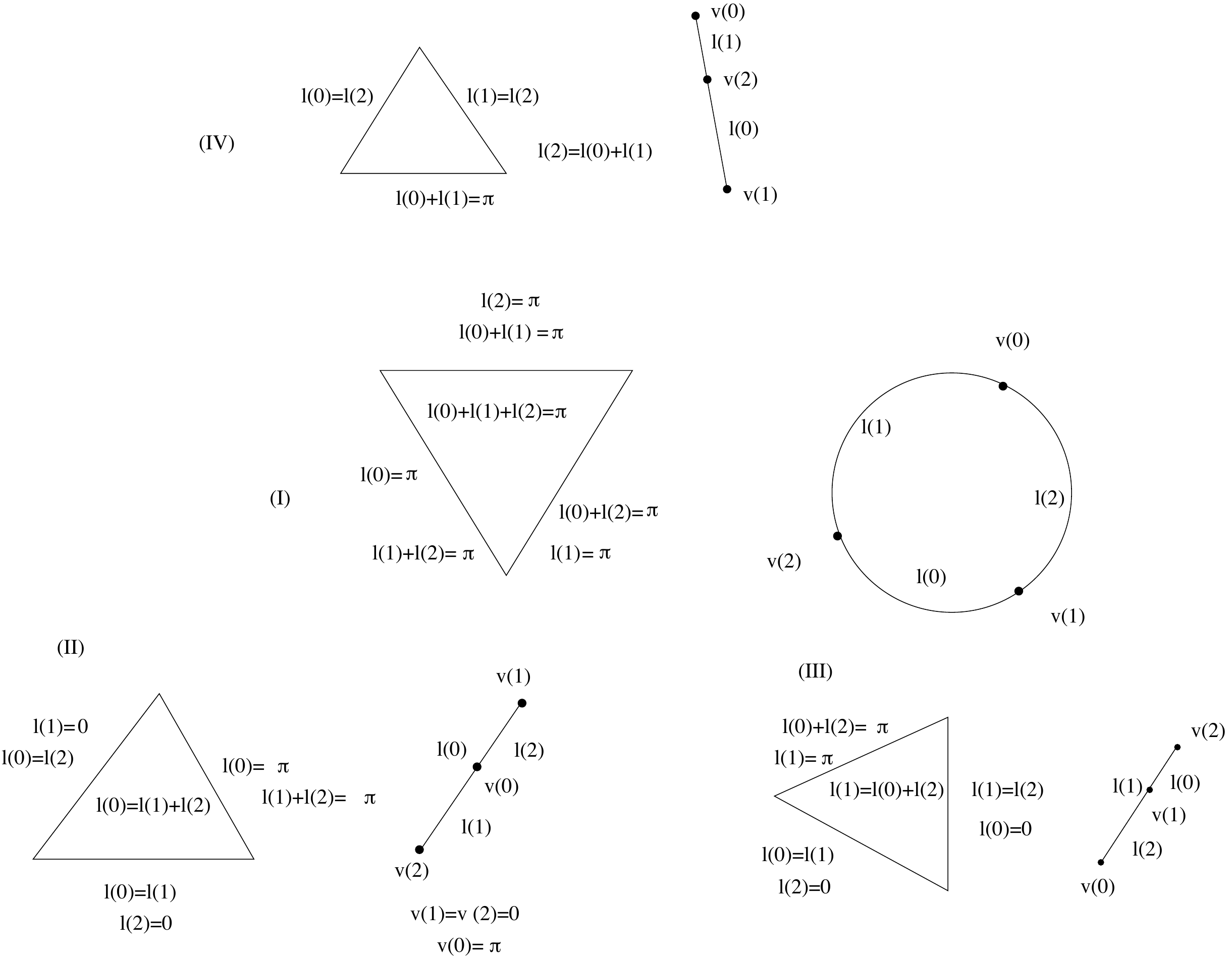}}
\caption{Regions $I$, $II$, $III$, and $IV$ with configurations in the right sides}
\label{fig:regI}
\end{figure}

For region $II$, we have $v(0)=\pi, v(1)=v(2)=0$. 
$l(0) = l(1)+l(2)$ and the bounds as above for each $l(i)$.
Each point $(l(0), l(1), l(2))$ is identified with
a segment of length $l(0)$ with vertices $v(1), v(2)$
and $v(0)$ on it. Clearly $l(1)$ is the distance between
$v(0)$ and $v(2)$ and $l(2)$ one between $v(0)$ and $v(1)$.

For region $III$, $v(1)=\pi, v(0)=v(2)=0$ and $l(1)=l(0)+l(2)$.
Each point $(l(0), l(1), l(2))$ is identified with
a segment of length $l(1)$ with vertices $v(0), v(2)$
and $v_1$ on it.

For region $IV$, $l(2) = l(0) + l(1)$. Each point
$(l(0), l(1), l(2))$ is identified with a segment of length $l(2)$
with vertices $v_0, v_1$ and $v_2$ on it.

Notice for regions $I$, $II$, $III$, and $IV$,
the points are coordinatized by $(l(0), l(1))$
for $0\leq l(0),l(1) \leq \pi$. These are said to be 
{\em vertex regions}.
The regions $I, II, III,$ and $IV$ are mutually disjoint disks as we can see from $D$. 

Finally, the union of the regions, $a, b, c, d, A, B, C, A', B', C', I, II, III,$ and $IV$, 
is a $2$-sphere being identified as in Figure \ref{fig:trtet}:
This follows since each region is homeomorphic to disks and 
the identifications are along arcs as in Figure \ref{fig:trtet}.
The union is homeomorphic to the identification space of the regions. 
This fact can be seen by looking at $D$ and the configurations themselves
as some of the pointed degenerate triangles are identical.

\subsection{The proof of Theorem \ref{thm:T}}\label{subsec:proofT}


  \begin{proof}
  Since $\tilde G$ is compact and $\bR^6$ is Hausdorff and $\iota$ is continuous, 
    $\tilde G$ is imbedded as the closure of $\iota(G^o)$. We only need to show that 
    $\tilde G$ has the structure of a blown-up solid  tetrahedron.
   To do that, we show that $\tilde G$ is a manifold with boundary consisting of 
   degenerate triangles or the union of the above regions. 
   
   First, the union of $a, b, c, d, A, B, C, A', B', C', I, II, III,$ and $IV$ is homeomorphic to $\SI^2$. 
   Each point of the interior of each region has a neighborhood homeomorphic to 
   an open ball intersected with closed upper half-space. 
   For region $I$, we can use coordinate functions $l(0), l(1),$ and $l(2)$ now 
   satisfying $l(0)+l(1)+l(1)\leq 2\pi$, $l(0) \geq 0, l(1) \geq 0, l(2) \geq 0$, $l(i)+l(i+1) \geq l(i+2)$ for a cylic index $i=0,1,2 \mod 3$. 
   Then we see that they parameterize a sufficiently small neighborhood $N$ of 
   a point of $I$. More precisely, $(l(0), l(1), l(2))$ form an injective continuous map $\eta$ from 
   a neighborhood $N$ to an open subset $N'$ of $\bR^3$ defined by $l(0)+ l(1) + l(2) \leq 2\pi$, $l(0)\geq 0, l(1) \geq 0, l(2) \geq 0$,  
   $l(i)+l(i+1) \leq l(i+2)$ for a cyclic  index $i=0,1,2 \mod 3$. 
   (Here, $\bR^3$ has coordinates $(l(0), l(1), l(2))$.)
   Conversely, from a sufficiently small open set in $N'$, 
   we can construct elements of $\tilde G$ in $N$, using given the length informations and polar positions. 
   This forms a continuous inverse map to $\eta$ as we can see from the polar position conditions. 
     
   Hence a neighborhood of $I^o$ is a manifold. 
   The similar consideration holds for each of $II^o, III^o,$ or $IV^o$. 
   
   For the region $A^o$, there are coordinate functions $l(0)$ and $v(0)$. 
   We use coordinate functions $l(0), v(1), v(2)$. Then these functions parameterize
   a neighborhood of an interior point of $A^o$. 
   Let $\bR^3$ have coordinates $(l(0), v(1), v(2))$.
   More precisely, 
   $(l(0), v(1), v(2))$ defines a homeomorphism $\eta$ from a neighborhood $N$ of 
   a point of $A^o$ to an open subset in the subset of $\bR^3$ defined by $ 0< v(2)\leq \pi, 0< l(0)< \pi, 0< v(1) < \pi$
   since $l(0), v(1), v(2)$ determine a generalized triangle uniquely.
   This shows that a neighborhood of $A^o$ is a manifold. 
   (Here, $\bR^3$ has coordinates denoted by $l(0), v(1), v(2)$. The homeomorphism 
   property of $\eta$ again follows as above by geometric constructions to obtain $\eta^{-1}$.)

   Similar considerations shows that some respective neighborhoods of 
   $B^o, C^o, A^{\prime o}, B^{\prime o}, C^{\prime o}$ are manifolds.
   
   For the region $a^o$, we use coordinate functions $(v(0), v(1), v(2))$ 
   to find a function $\eta$ from a neighborhood of a point of $a^o$ 
   to a small open subset of  the subspace of $\bR^3$ defined by $v(0)+v(1)+v(2) \geq \pi$, $0< v(i) < \pi$ for $i=0,1,2$. 
   This implies as above 
   that $a^o$ has a manifold neighborhood. 
   (Here $\bR^3$ has coordinates denoted by $v(0), v(1), v(2)$. The homeomorphism property 
   of $\eta$ is similar to above.)
   
   Similar consideration shows that some respective neighborhoods of 
   $b^o, c^o, d^o$ are manifolds. 
   
   We now discuss the interiors of the intersecting edges and show that 
   they have manifold neighborhoods. First, consider the edge $I_{A, b}:=A \cap b$.
   Let $p \in I_{A,b}^o$. Here $l(0)=l(2)=\pi$ and $l(1)=0$ and $v(1)=v(0)+v(2) -\pi$ from $b$ 
   and $l(0)+l(1)=\pi$, $l(2)=\pi$, $v(0)=v(1)$, $v(2)=\pi$ from $A$. 
   Thus, we have $l(1)=0, l(0)=l(2)=\pi$ and $v(0)=v(1), v(2)=\pi$ on the edge. 
   We use $(l(1), v(0), v(2))$ as coordinate functions in a sufficiently small neighborhood $N$ of $p$. 
   That is, they define an injective function $\eta$ from $N$ to a small neighborhood of a corner point of
   a subset $\bR^3$ defined by $l(1) \geq 0$ and $v(2) \leq \pi$.   
   Then $b\cap N$ is characterized by $l(1)=0$, and $A\cap N$ is characterized by $v(2)=\pi$. 
     (Here $\bR^3$ has coordinates denoted by $l(1), v(0), v(2)$. The homeomorphism
     property of $\eta$ is similar to above.)
This implies that $I_{A, b}^o$ has a manifold neighborhood.  
   
   Next, consider the edge $I_{A, I} := I \cap A$. Here, we use coordinate functions $l(0), l(1), l(2)$
   on a small neighborhood $N$ of a point of $I_{A, I}^o$. 
   $A$ is determined by $l(2)=\pi$ and $I$ is determined by $l(0)+l(1)+l(2)=\pi$. 
   There is a homeomorphism $\eta$ from $N$ to a small open subset of 
   $\bR^3$ defined by $l(2) \leq \pi, l(0)+l(1)+l(2) \leq 2\pi$. 
   
   For all other edges, similar considerations show that there are some respective 
   manifold neighborhoods. 
   
   Next, we consider $12$ vertices, i.e., points where three regions meet. 
   Remove these points from $\tilde G$.
   For each point, we can find a disk $D$ with $\partial D$ in $\tilde G - G^o$ 
   to obtain a contactible manifold $N$ homeomorphic to compact $3$-ball $B^3$ with one point removed.  
   Hence, with the point added back, the manifold $N$ becomes a compact manifold homeomorphic to $B^3$. 
   The disks can be chosen to be mutually disjoint and each containing one of $12$ vertices in one side.
   Doing these for each point, we have shown that $\tilde G$ is homeomorphic to $B^3$ and has the boundary structure of 
   a blown-up tetrahedron.

			\end{proof}
			

\begin{rem} 
From now on, we can identify the blown-up tetrahedron $T$ with $\tilde G$
and the interior of $T$ with the interior $G^o$. 
Thus, when we say $\tilde G$, we will think of a blown-up tetrahedron. 
\end{rem}			
	
\begin{rem} \label{rem:rays}
Let us give some ``intuitive idea'' here. 

Given each smooth arc $\alpha: [0, 1) \ra G^o$ starting from a point of $G^o$ and 
does not have a compact closure in $G^o$ and has well-defined limits of derivatives as $t \ra 1$,
it follows that $\alpha$ ends at a unique point of $\tilde G$. 
This can be seen from the angle and length formulas of spherical geometry
and L'Hospital's rules. Setting $\alpha$ as a half-open segment is most illustrative here. 

This gives us a natural way to view $\tilde G$ as a blown-up tetrahedron. 
Each ray ending at a common point the interior of an edge of $G$ will have different limits in $\tilde G$ according 
to its direction. Hence, we see that a point is replaced by an arc, i.e. a tie. 
Hence, an open edge is replaced by an open rectangle. 

Each ray ending at a vertex of $G$ will have different limits in $\tilde G$ according to 
its direction. Here, we see that a triangle replaces a point. 

The author is not yet able to find natural way to use these facts to prove Theorem \ref{thm:T} but probably 
our process is a natrual blow-up construction in real algebraic scheme theory. 
\end{rem}


\subsection{An extension of the Klein four-group action}

We now extend the discussions to the boundary
and study identifications in the boundary:
The map $I_{{A}}$ in $\tilde G^o$ can be described as 
first find an element $\mu$ in $\tilde G^o$ and representing it as a triangle with vertices $v_0, v_1, v_2$ 
and taking a triangle with vertices $v'_0=-v_0$ and $v'_1=-v_1$ and $v'_2=v_2$. 
That is, we change our triangle with vertices $v_0, v_1$, and $v_2$ and edges $l_0, l_1,$ and $l_2$ 
to one with vertices $-v_0, -v_1$, and $v_2$ and edges $l'_0$ with endpoints $-v_1, v_2$ 
in the great circle containing $l_0$ and 
$l'_1$ endpoints $-v_0, v_2$ in the one containing $l_1$ and $l'_2=-l_2$ with endpoints 
$-v_0, -v_1$ in the one containing $l_2$. 
(Here, the word ``contain'' is according to Definition \ref{defn:edgecircle}.)
Then take the isometry class of the second triangle as the image $I_A(\mu)$. 
This gives the changes in parameter as follows 
\[(v(0), v(1), v(2), l(0), l(1), l(2)) \ra (\pi-v(0),\pi-v(1), v(2), \pi-l(0),\pi-l(1),l(2)).\]

Similarly, the map $I_{{B}}$ changes the triangle with vertices $v_0, v_1$, and $v_2$ 
to one with $v_0, -v_1$, and $-v_2$:
\[(v(0), v(1), v(2), l(0), l(1), l(2)) \ra (v(0),\pi-v(1), \pi-v(2), l(0),\pi-l(1), \pi-l(2)).\]

Similarly, the map $I_{{C}}$ changes the triangle 
with vertices $v_0, v_1$, and $v_2$ to one with 
$-v_0, v_1,$ and $-v_2$:
\[(v(0), v(1), v(2), l(0), l(1), l(2)) \ra (\pi-v(0),v(1), \pi-v(2), \pi-l(0),l(1),\pi-l(2)).\]

For our geometric degenerate triangles, we do the same. 
For regions $a, b, c$, and $d$, the transformations are merely the linear extensions 
or equivalently extensions with respect to the metrics. 

We first describe $I_{{A}}$: 
\begin{itemize} 
\item[${A}, {A}'$] Regions ${A}$, ${A}'$ are preserved under $I_{{A}}$,
and $(l(0), v(0))$ is mapped to \break $(\pi-l(0), \pi-v(0))$.
\item[${B},{B}'$] The region $B$ is mapped to $B'$ and vice versa. $(l(1), v(1))$ is mapped 
to $(\pi-l(1), \pi-v(1))$.
\item[$C,C'$] The region $C$ is mapped to $C'$ and vice versa. $(l(2), v(2))$ is 
mapped to $(l(2), v(2))$.
\item[$I,IV$] The region $I$ is mapped to $IV$ and vice versa. $(l(0),l(1),l(2))$ 
is mapped to $(\pi-l(0),\pi-l(1),l(2))$. 
\item[$II,III$] The region $II$ is mapped to $III$ and vice versa. 
$(l(0),l(1),l(2))$ is mapped to $(\pi-l(0),\pi-l(1), l(2))$. 
\end{itemize}

The transformations $I_{{B}}$ and $I_{{C}}$ are analogous. 
In fact, recall that one can identify $\tilde G$ by the tetrahedron that closure of 
$\tilde G^o$ truncated by planes parallel to sides at vertices and edges: 
more precisely, at a vertex $v$, we take a plane parallel to the 
opposite side and passing through the tetrahedron of distance $\eps$ to $v$ 
for a sufficiently small $\eps> 0$. 
We cut off the neighborhood of $v$ by the plane. 
We do the same steps at each vertex for the same $\eps$. 
At an edge $e$, we take a plane  parallel to $e$ and the edge opposite $e$
which is at the distance $\eps/3$ from the line containing $e$.
We truncated every edge. This gives us an imbedding of 
$\tilde G$ with incorrect coordinates. 

Each face of $\tilde G$ has a metric induced by the parameterization maps into $\bR^2$ or $\bR^3$ with 
Euclidean metrics. In this identification, 
$I_{{A}}$ is realized as an involutive isometry of $G^o$ and each of the faces of $\tilde G$ of order two 
with the set of fixed points the axis trough the center of ${A}$ and ${A}'$. 
$I_{{B}}$ is realized as one with the fixed axis though the center of $B$ and $B'$
and $I_{{C}}$ one with the fixed axis though the center of $C$ and $C'$.


\subsection{Standard positions for generalized triangles}

For later purposes, we introduce {\em standard positions} for 
generalized triangles. Heuristically speaking, although the edges maybe degenerate, we can 
still read angles with respect to the degenerate edges with infinitesimal directions only. 
We use the idea: 

Let $H$ be the hemisphere containing $[1,0,0],[0,1,0]$ in the boundary 
great circle $\SI^1$ and $H^o$ containing $[0,0,1]$.

We are given a generalized triangle with vertices $v_0, v_1, v_2$, and we find 
the standard position to put this triangle in $\SI^1$: By an isometry, 
the vertex $v_0$ is to be put on $[1,0,0]$ and 
$v_1$ on the $\SI^1$ in the closure $L$ of the component of 
$\partial H -\{[1,0,0],[-1,0,0]\}$ containing $[0,-1,0]$, and
$v_2$ on $H$. As a consequence, 
the interior of the triangle to be in $H^o$ and $l_2 \subset L$.
(Denote by $-L$ the antipodal segment to $L$, i.e., the set of antipodes to points of $L$.)
(We remark that the triangle is then put in the right orientation.) 

For any nondegenerate triangle, our process determines the triangle 
in a standard position uniquely. 
For generalized triangles in the regions $I$, $II$, $III$, and $IV$, the standard position 
is determined by the above conditions as well. In fact, in $III, IV$, the triangle is a subset of $L$ 
and in $II$, our triangle is a segment in $\partial H$ with a point of it $[1,0,0]$.

For a generalized triangle in the regions $a,b,c$, and $d$, if the vertices satisfy $v_0=v_1=v_2$, 
then we let these be $[1,0,0]$ all (the case $a$).

Suppose that $v_0=v_2$ and are antipodal to $v_1$  (the case $b$). Then we 
let $v_0=v_2=[1,0,0]$ and $v_1=[-1,0,0]$ and let 
$l_1$ be the degenerate segment equal to $\{v_0=v_2\}$ 
and $l_0$ is  the segment of length $\pi$ determined by angle $\theta_0$ with $L$ at $v_0$ 
and $l_2 = L$.

Suppose $v_1=v_2$ and are antipodal to $v_0$ (the case $c$). Then we let 
$v_0=[1,0,0]$ and $v_1=v_2=[-1,0,0]$. 
We let $l_2 = L$ and $l_1$ be the segment of length $\pi$ determined by the angle $\theta_0$ 
from $L$ at $v_0$ and let $l_0$ be the degenerate segment 
equal to $\{v_1=v_2\}$.

Finally, suppose that $v_0=v_1$ and $v_2$ 
is antipodal to these (the case $d$). Then we let $v_0=v_1=[1,0,0] $ and $v_2=[-1,0,0]$ 
and let the edge $l_1$ be the segment of length $\pi$ determined by the angle $\theta_0$ 
from $L$ at $v_0$, $l_0$ be one determined by the angle 
$\pi-\theta_1$ from $L$ at $v_1$, and $l_2$ is the degenerate segment 
equal to $\{v_0=v_1\}$.

Now we consider the interiors of regions ${A}$,${A}'$,$B$,$B'$,$C$,$C'$: for triangles there we assign
$v_0=[1,0,0]$.
\begin{itemize}
\item[(${A}$)] We put $v_1 = [-1,0,0]$ and $l_2 = L$, and $l_0\cup l_1$ is  the segment of length $\pi$ determined by 
$\theta_0$ with $L$ at $v_0$. 
\item[(${A}'$)] We put $v_1=[1,0,0]$ and $l_2=\{v_0=v_1\}$, and 
$l_1=l_0$ is  the segment of length $\leq \pi$ determined by angle $\theta_0$ from $L$ at $v_0$
as above. 
\item[(B)] We put $v_1 \in L$, $v_2=-v_1 \in -L$, and $l_2 \subset L$, and 
$l_1\subset -L$. $l_0$ is determined by $\theta_1$ 
at $v_1$. 
\item[(B')] We put $v_1=v_2$ and $l_1=l_2$ are in $L$ and $l_0=\{v_1=v_2\}$. 
\item[(C)] $v_1$ is on $L$ and $v_2=[-1,0,0]$ and $l_0$ and $l_2$ are in $L$,
and $l_1$ is the segment connecting $[1,0,0]$ to $[-1,0,0]$ of angle $\theta_0$ 
with $L$ at $v_0$. 
\item[(C')] We have $v_2=v_0$, and $l_0=l_2$ are in $L$, and $l_1=\{v_0=v_2\}$.
 \end{itemize}
 
 Let $\mathcal S \subset \hat G$ denote the subspace of generalized triangles in standard positions.
 \begin{prop}\label{prop:stdp} 
 There is a homeomorphism $\tilde G \ra \mathcal S \subset \hat G$.
 \end{prop}
 \begin{proof} 
 The injectivity is clear. The continuity follows by showing that puting a generalized 
 triangle $t$ to one in $\mathcal S$, we need an isometry $i_t$ that depends continuously on $t$.
 A sequence argument will show. 
 \end{proof}


\section{The character space of the fundamental group of
a pair of pants}\label{sec:rpp}

The character space of the fundamental group of a pair of pants will be analyzed geometrically using triangles 
in this section and will be shown to be the quotient of a solid tetrahedron with some Klein four-group $V$ axial 
action. That is, for each representation we obtain a triangle determined up to the $V$-action. 
The basic idea is similar to that of Goldman \cite{Gconv} modified to our setting and is 
related to matrix-multiplication by geometry. 
(In the next Section \ref{sec:sosu}, we generalize these results to $\SUTw$.)

\subsection{Matrix-multiplications by geometry}
We will coordinatize $\SI^1$ by $[0,2\pi]$ with $0$ and $2\pi$ identified. 
In fact, we will think of a number in $[0,2\pi]$ as a number mod $2\pi$ if our group is $\SOThr$ 
and without mod $2\pi$ if our group is $\SUTw$. 

An isometry $\iota$ in $\SOThr$ fixes at least one point $v$ in
$\SI^{2}$. Its antipodal point $-v$ is also fixed and
with respect to a coordinate system with $v$ as one of
the frame vector, a matrix of $\iota$ equals
\[ \left(\begin{matrix}
1 & 0 & 0 \\
0 & \cos(\theta) & -\sin(\theta) \\
0 & \sin(\theta) & \cos(\theta)
\end{matrix}\right), 0 \leq \theta \leq 2\pi. \]
Here, $\theta$ is the counter-clockwise rotation angle of $\iota$ and is determined
by the choice of the fixed point $v$
and will be $2\pi-\theta$ if we chose the other
fixed point. Notice that the rotation angle is measured
in the counter-clockwise direction with respect to
the chosen fixed point.

An element of $\SOThr$ can be written as $R_{x,\theta}$ where 
$x$ is a fixed point and $\theta$ is the angle of rotation
in the counter-clockwise direction where $0\leq \theta \leq 2\pi \mod 2\pi$.  
For the identity element, $x$ is not determined but $\theta=0$. 
For any nonidentity element, $x$ is determined up to antipodes:
$R_{x,\theta} = R_{-x,2\pi-\theta}$. 

We begin with ``multiplication by geometry": 
Let $w_0,w_1,$ and $w_2$ be vertices of a triangle oriented in
the clockwise direction. 
Let $e_0, e_1,$ and $ e_2$ denote the opposite edges. 
Let $\theta_0,\theta_1,$ and $\theta_2$ 
be the respective angles. Then 
\[R_{w_2,2\theta_2} R_{w_1,2\theta_1} R_{w_0,2\theta_0} = \Idd :\]
Draw three adjacent isometric triangles obtained by reflecting isometrically on 
the great circles containing edges $e_0,e_1,$ and $e_2$ respectively. (See Definition \ref{defn:edgecircle}.)
each edge. Call them $T_0,T_1,$ and $T_2$ while they are obtained from reflecting. 
Since their edge lengths are the same, they are congruent. Moreover, 
by considering the congruent edges and the angles between them, we obtain that 
$R_{w_0,2\theta_0}$ fixes $w_0$ and sends $T_1$ to $T_2$,
$R_{w_1,2\theta_1}$ fixes $w_1$ and sends $T_2$ to $T_0$, and 
$R_{w_2,2\theta_2}$ fixes $w_2$ and sends $T_0$ to $T_1$. 
The composition the three isometries sends each vertex to itself and hence is 
the identity map. 
\begin{figure}
\centerline{\includegraphics[height=7cm]{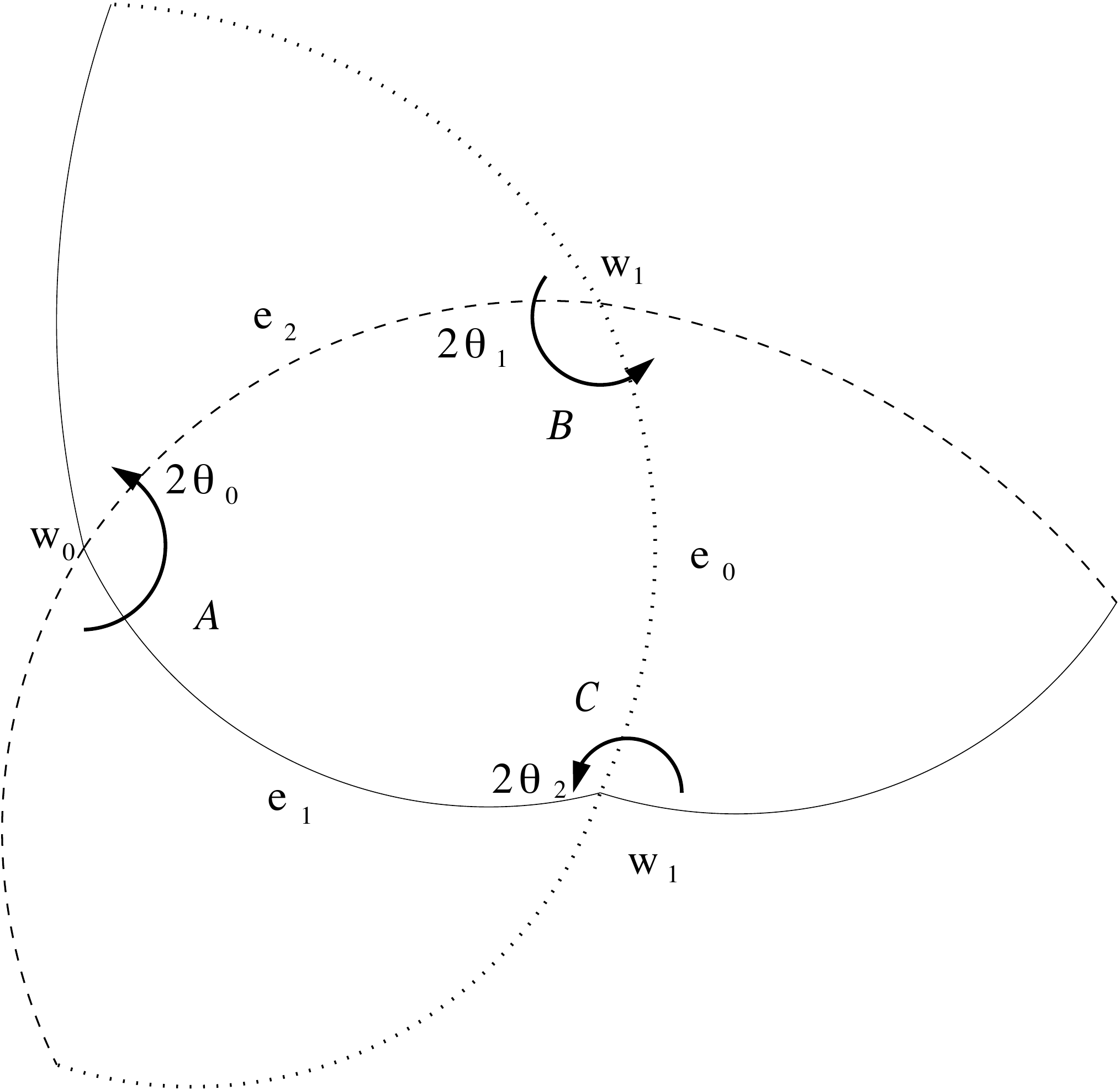}}
\caption{Triangular representations}
\label{fig:triangleisom}
\end{figure}

Finally, denoting the rotation at $w_0,w_1,w_2$ by $\mathcal{A}, \mathcal{B}, \mathcal{C} $ respectively,
we obtain 
\begin{equation}\label{eqn:cba}
{\mathcal{CBA}}=\Idd, {\mathcal C}^{-1}={\mathcal{BA}}, {\mathcal C}={\mathcal A}^{-1}{\mathcal B}^{-1}.
\end{equation}
We will generalize the idea when the triangles are degenerate.
(See Section \ref{subsec:gmult}.)

Of course, careful considerations of orientations 
are important: the triangle is ordered clockwise and the rotation directions 
are counter-clockwise.

\subsection{Mutiplications by geometry extended} \label{subsec:gmult}
We now extend the idea of ``multiplication by geometry" to degenerate cases
(which is applicable in $\SUTw$ as well):
Let $w_0,w_1$, and $w_2$ be vertices of a degenerate triangle oriented in
the clockwise direction. 
Let $e_0, e_1$, and $e_2$ denote the opposite edges. 
Even in a degenerate case, since we can put the generalized triangle in
the standard position, we have a well-defined great circle containing 
each of the edges for any generalized triangle. Thus, we can reflect these 
by the great circles containing the edges and obtain the same constructions as above. 

Let $\theta_0,\theta_1$, and $\theta_2$ be the respective angles in $[0,\pi]$. Then 
\[R_{w_2,2\theta_2} R_{w_1,2\theta_1} R_{w_0,2\theta_0} = \Idd\]
will hold. This is checked by using the same argument as above. 


Using this equation, we can compute 
\[R_{w_2,2\pi-2\theta_2}= R_{w_1,2\theta_1} R_{w_0,2\theta_0},\]
for example. 
Finally, denoting the rotation at $w_0,w_1,w_2$ by $\mathcal{A},\mathcal{B},\mathcal{C}$ respectively,
we obtain 
\begin{equation}\label{eqn:gcba}
{\mathcal{CBA}}=\Idd, {\mathcal C}^{-1}={\mathcal{BA}}, {\mathcal C}={\mathcal A}^{-1}{\mathcal B}^{-1}.
\end{equation}

We should remark that when $\theta_2=0,\pi$, then $R_{x,2\theta}$ is the identity map. 
However, the equation still holds. 

When the angle at $w_2$ is $0$ or $\pi$ so that $\mathcal C$ equals $I$ (respectively $I$ or $-I$ in $\SUTw$), 
the vertex $w_2$ can be chosen arbitrarily on the boundary of the triangle.


We note that under quotient relation, $\mathcal C$ is a smooth function of $\mathcal A$ and $\mathcal B$.

As an illustration, we give some examples here. The readers can fill in all 
remaining cases using permutations. The triangle is clockwise
oriented in the boundary by the order of $v_0, v_1,$ and $v_2$ 
appearing with respective angles $\theta_0, \theta_1,$ and $ \theta_2$.
\begin{itemize} 
\item The triangle is a point with $v_0=v_1=v_2$. Then $\theta_2 = \pi-\theta_0-\theta_1$ 
and $\mathcal C$ is a rotation at $v_2$ of the angle $2\theta_2$.  
\item The triangle is a segment with $v_0=v_1$ and $v_2$ with nonzero angles 
$\theta_0,\theta_1$ and $\theta_2=0$. Then $\theta_1=\pi-\theta_0$ naturally. 
Thus ${\mathcal B}={\mathcal A}^{-1}$ and ${\mathcal C}=I$. Actually, we can choose $v_2$ to be any point in $\SI^2$.
\item The triangle which is a lune with $v_0$ and $v_2$ as two vertices of the lune and $v_1$ in the interior of an edge. 
Then the lune angle $\theta$ equals $\theta_0$
and $\mathcal C$ has a fixed point $v_2=-v_0$ and the angle of rotation is $2\theta_0$. 
\item The triangle which is a lune with $v_0$ and $v_1$ as the two vertices of the lune 
with lune angle $\theta$. Then
the lune angle $\theta$ equals smaller of $\theta_0$ or $\theta_1$
which can be seen from the classification of generalized triangles.
\begin{itemize}
\item If $\theta_0=\theta$ and $\theta_1=\theta$, 
then $v_2$ can be any point on the segment of a lune and
$\theta_2=\pi$ and $\mathcal C$ is $I$.
\item If $\theta_0=\theta$ and $\theta_1> \theta$, 
then $v_2=v_1$ and $\theta_2=\theta_0+\pi-\theta_1$.
\item If $\theta_1=\theta$ and $\theta_0 > \theta$, then 
$v_2=v_0$ and $\theta_2=\theta_1+\pi-\theta_0$. 
\end{itemize}
\item If $\theta_0=\theta_1=\pi$, then the triangle is a hemisphere, and $v_2$ can be any point on 
a segment with endpoints $-v_0,-v_1$ and $\theta_2=\pi$. 
In this case, ${\mathcal{A}}, {\mathcal{B}}$, and $\mathcal C$ are all identity. 
\end{itemize}
(In $\SUTw$ case, we only need some obvious sign changes.)

\begin{prop}\label{prop:multrule}
From above, we obtain the following rules. Suppose that we have a generalized triangle with vertices $v_0, v_1$, and $v_2$
occurring in the clockwise order with angles $\phi_0, \phi_1, \phi_2$ respectively.
Let $\mathcal A$ be a transformation with a fixed point $v_0$ and rotation angle $\phi_0$, 
$\mathcal B$ be one with $v_1$ and $\phi_1$, and $\mathcal C$ be one with $v_2$ and $\phi_2$ respectively
satisfying $\mathcal C \mathcal B \mathcal A =\Idd$.
\begin{itemize}
\item  Suppose that $v_0$ and $v_1$ do not coincide and do not form an antipodal pair. 
Then $\pm v_0, \pm v_1$ form four points on a great circle. 
\begin{itemize}
\item Suppose that $\phi_0, \phi_1$ are both not zero nor $\pi$. Then so is $\phi_2$ and $v_0, v_1$, and $v_2$ form the vertices of 
a nondegenerate triangle. 
\item Suppose that $\phi_0=0$ and $\phi_1 \ne 0, \pi$. Then $v_2=v_1$ and $\phi_2=\pi-\phi_1$. 
\item Suppose that $\phi_0=\pi$ and $\phi_1\ne 0, \pi$. Then $v_2=-v_1$ and $\phi_2=\phi_1$. 
\item Suppose that $\phi_0=0$ and $\phi_1 =0$. Then $\mathcal C=\Idd$ and $v_2$ can be any point on the segment $\ovl{v_0v_1}$ and $\phi_2=\pi$.
\item Suppose that $\phi_0=0$ and $\phi_1=\pi$. Then $\mathcal C=\Idd$  and $v_2$ can be any point on the segment $\ovl{v_0({-v_1})}$ and $\phi_2=0$.
\item Suppose that $\phi_0=\pi$ and $\phi_1=0$. Then $\mathcal C=\Idd$ and  $v_2$ can be any point on the segment $\ovl{(-v_0)v_1}$ and $\phi_2=0$.
\item Suppose that $\phi_0=\pi, \phi_1=\pi$. Then $\mathcal C=\Idd$ and  $v_2$ can be any point on the segment $\ovl{(-v_1)(-v_2)}$, $\phi_2=\pi$, 
and $v_0,v_1$, and $v_2$ are the vertices of a pointed hemisphere. 
\end{itemize}
\item Suppose that $v_0= v_1$. Then $v_2=v_0=v_1$ and the angle $\phi_2=\pi-\phi_0-\phi_1$
if $\phi_0 \ne 0, \pi, \phi_1 \ne 0, \pi$. If $\phi_0=0, \phi_1=\pi$ or $\phi_0=\pi, \phi_1=0$, then $\phi_0=0$
and $\mathcal C = \Idd$ and $v_2$ can be any point of $\SI^2$.
If $\phi_0=\phi_1=0$, then $\phi_2=\pi$ and $v_2=v_0$. 
if $\phi_0=\phi_1=\pi$, then $\phi_2=\pi$ and $v_2=-v_0$. 
\item Suppose that $v_0 = -v_1$. Then 
\begin{itemize}
\item If $\phi_0 > \phi_1$, then $v_2=v_0$ and $\phi_2 = \pi-\phi_0+\phi_1$. 
\item If $\phi_1>\phi_0$, then $v_2=v_1$ and $\phi_2=\pi+\phi_0-\phi_1$. 
\item If $\phi_0=\phi_1$, then $\mathcal C=\Idd$ and $\phi_2=\pi$ and $v_2$ can be any point of $\SI^2$. 
\end{itemize}
\end{itemize}
\end{prop}



\subsection{The $\SOThr$-character space of the fundamental group of a pair of pants}
Let $P$ be a pair of pants and let $\tilde P$ be the universal cover.
Let $c_0, c_1,$ and $c_2$ denote three boundary components of
$P$ oriented using the boundary orientation.
Let $\pi_{1}(P)$ denote the fundamental group of $P$ seen as a group
of deck transformations generated by three elements
$\mathcal A, \mathcal B,$ and $\mathcal C$ parallel to the boundary components of $P$
satisfying ${\mathcal{CBA}} = I$. $\mathcal A$ acts on a lift $\tilde c_0$ of $c_0$
in $\tilde P$ and $\mathcal B$ on $\tilde c_1$ a lift of $c_1$ and
$\mathcal C$ on $\tilde c_2$ a lift of $c_2$.

Take a triangle on the sphere $\SI^{2}$ with geodesic edges
so that each edge has length $< \pi$ so that the
vertices are ordered in a clockwise manner in the boundary
of the triangle.
Such a triangle is classified by their angles
$\theta_{0}, \theta_{1}, \theta_{2}$ satisfying
\begin{eqnarray}
    \theta_{0}+\theta_{1}+\theta_{2} & >&  \pi \\
    \theta_{i} & < & \theta_{i+1} + \theta_{i+2} - \pi, i =0,1,2.
    \label{e:tri}
    \end{eqnarray}
where the indices are considered to be in $\bZ_{3}$.
The region gives us an open tetrahedron in the positive octant of
$\bR^{3}$ with vertices
\[(\pi,0,0), (0,\pi,0), (0,0,\pi), (\pi,\pi,\pi)\]
and thus we have $0< \theta_{i} < \pi$.
This is a regular tetrahedron with edge lengths all equal to
$\sqrt{2}\pi$. 

Let $\tri_{1}$ be a nondegenerate spherical triangle with vertices removed
and $\tri_{2}$ an isometric triangle with vertices also removed. We can glue $\tri_{1}$
and $\tri_{2}$ along corresponding open edges
and obtain an open pair of pants $P'$ with an elliptic metric.
We can attach the boundary component circles to $P'$ with
labeling ${\mathcal{A}}, {\mathcal B},$ and $\mathcal C$ corresponding to the vertices.
This gives us a geometric structure on $P'$ modeled on
$(\SI^{2}, \SOThr)$. Hence, there is an associated representation
$h: \pi_{1}(P') \ra \SOThr$ determined up to conjugation.
Such a representation is said to be a {\em triangular representation}.
This approach was used in Thurston's book \cite{Thbook}
and Goldman \cite{Gconv}.
\begin{figure}

\centerline{\includegraphics[height=7cm]{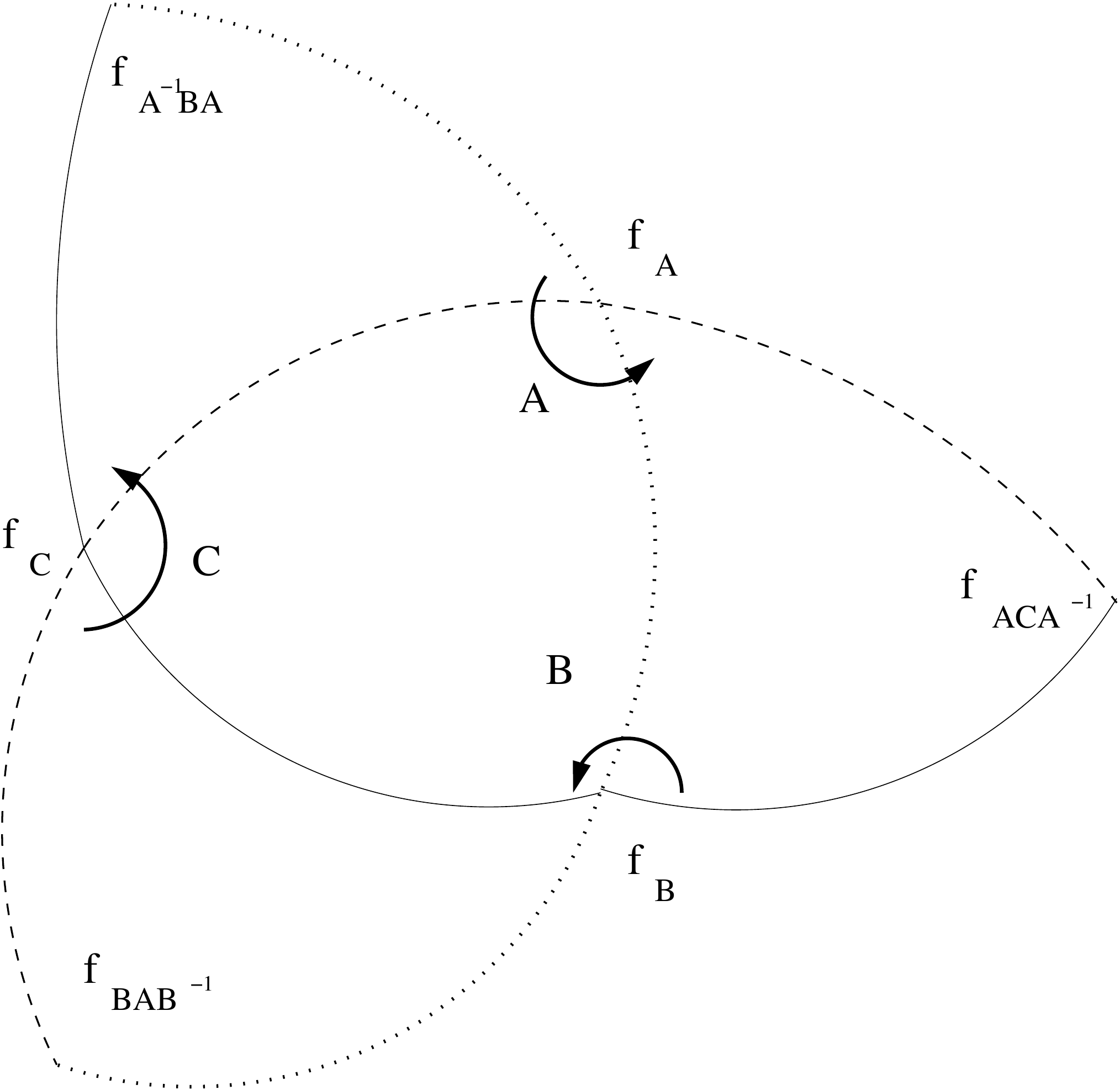}}
\caption{Triangular representations}
\label{fig:triangleisom2}
\end{figure}

\begin{lem}\label{lem:perturb}
    $\rep(\pi_{1}(P), \SOThr)$ contains a dense open set where
    each character is triangular. In fact, one can always 
    find a path from any character to triangular characters 
    where each point except the beginning point of the path is triangular. 
    \end{lem}
\begin{proof}
    Let $h$ be a representation $\pi_{1}(P) \ra \SOThr$.
    Then $h({\mathcal A}), h({\mathcal B})$, and $h({\mathcal C})$ are elements of $\SOThr$ and hence
    contains a fixed point in $\SI^{2}$.
    Let $f_{\mathcal A}$ and $-f_{\mathcal A}$ denote the fixed points of $h({\mathcal A})$
    and $f_{\mathcal B}$ and $-f_{\mathcal B}$ those of $h({\mathcal B})$ and
    $f_{\mathcal C}$ and $-f_{\mathcal C}$ those of $h({\mathcal C})$.
    The fixed points are chosen arbitrarily if there are more than
    one fixed points up to antipodal transformations, i.e.,
    only when the matrix is the identity matrix.

    We can choose one among each  of the three antipodal pairs and
    obtain a triangle which might be degenerate. (The choice of what 
    degenerate triangle can be arbitrary. See Section \ref{sec:glimit} 
    for example of degenerate triangles. Essentially, if the holonomy of one of the generators is
    the identity, we fall into the edge regions, and if two or more are so, then we fall into to the vertex regions, 
    and finally if the holonomy is abelian with no identity for generators, then we fall into the face regions.)
    Let $\tri(f_{{\mathcal A}}, f_{{\mathcal B}},f_{{\mathcal C}})$ denote such a triangle, and
    the vertices are oriented in a clockwise direction (See
    Figure \ref{fig:triangleisom2}).  
    \begin{itemize}
    \item    $h({\mathcal A})$ fixes $f_{{\mathcal A}}$ and send $f_{{\mathcal C}}$ to 
    $h({\mathcal A})(f_{{\mathcal C}})$ chosen to be $f_{{\mathcal A}{\mathcal C}{\mathcal A}^{-1}}$ and 
    $h({\mathcal A})^{-1}(f_{{\mathcal B}})$ chosen to be $f_{{\mathcal A}^{-1}{\mathcal B}{\mathcal A}}$ to $f_{{\mathcal B}}$.
    Hence, $h({\mathcal A})$ send $\tri(f_{{\mathcal A}}, f_{{\mathcal C}}, f_{{\mathcal A}^{-1}{\mathcal B}{\mathcal A}})$ to
    $\tri(f_{{\mathcal A}}, f_{{\mathcal A}{\mathcal C}{\mathcal A}^{-1}}, f_{{\mathcal B}})$.
 \item   $h({\mathcal B})$ fixes $f_{\mathcal B}$ and send 
 $f_{\mathcal A}$ to $h({\mathcal B})(f_{\mathcal A})$ chosen to be $f_{{\mathcal B}{\mathcal A}{\mathcal B}^{-1}}$
    and send $f_{{\mathcal A}{\mathcal C}{\mathcal A}^{-1}}$ 
    to $h({\mathcal B})(f_{{\mathcal A}{\mathcal C}{\mathcal A}^{-1}})
    =h(\mathcal B) h(\mathcal A)(f_{\mathcal C}) = h(\mathcal C)^{-1}(f_{\mathcal C})=f_{\mathcal C}$
    as ${\mathcal B}{\mathcal A}={\mathcal C}^{-1}$ 
    and hence sends $\tri(f_{{\mathcal B}}, f_{{\mathcal A}}, f_{{\mathcal A}{\mathcal C}{\mathcal A}^{-1}})$
    to $\tri(f_{{\mathcal B}}, f_{{\mathcal B}{\mathcal A}{\mathcal B}^{-1}}, f_{{\mathcal C}})$.    
     \item   $h({\mathcal C})$ fixes $f_{\mathcal C}$ and sends $f_{\mathcal B}$ 
     to $h({\mathcal C})(f_{\mathcal B}) = h({\mathcal A})^{-1}h({\mathcal B})^{-1}(f_{\mathcal B}) 
     = h({\mathcal A})^{-1}(f_{\mathcal B}) 
     =f_{{\mathcal A}^{-1}{\mathcal B}{\mathcal A}}$
     as ${\mathcal C}={\mathcal A}^{-1}{\mathcal B}^{-1}$ and 
     sends $f_{{\mathcal B}{\mathcal A}{\mathcal B}^{-1}}$ to 
     $h({\mathcal C})(f_{{\mathcal B}{\mathcal A}{\mathcal B}^{-1}}) = h({\mathcal C})h({\mathcal B})(f_{\mathcal A})
     =h({\mathcal A})^{-1}(f_{\mathcal A})=f_{\mathcal A}$
       and hence sends  $\tri(f_{{\mathcal C}},f_{{\mathcal B}}, f_{{\mathcal C}^{-1}{\mathcal B}{\mathcal C}})$ to
    $\tri(f_{{\mathcal C}}, f_{{\mathcal A}^{-1}{\mathcal B}{\mathcal A}}, f_{{\mathcal A}})$.  
    
    \end{itemize}
    Thus we obtain the diagram in Figure
    \ref{fig:triangleisom2}.

    Since $h({\mathcal A}), h({\mathcal B})$, and $h({\mathcal C})$ are isometries, we see that
    the three adjacent triangles are congruent triangles.
    By the length consideration, the four triangles are congruent to one another.
    
    If one of the triangle is nondegenerate, then all are not and is congruent to 
    $\tri(f_{{\mathcal A}}f_{{\mathcal B}}f_{{\mathcal C}})$.

    Thus, if one of the triangle is degenerate, then all other are
    degenerate also. 
    $h$ corresponds to an abelian representation and it fixes a pair of antipodal points $x, -x$.
    A nonidentity abelian representation is determined by a three angles $\theta_0, \theta_1,$ and $\theta_2$ at a fixed point $x$ of 
    $h({\mathcal A}), h({\mathcal B})$, and $h({\mathcal C})$ where $\theta_0+\theta_1+\theta_2 = 2n\pi$ for $n=1,2$.
    If $\theta_0+\theta_1 +\theta_2 = 2\pi$, then we see that the region $a$ gives us all such abelian representation.
    If $\theta_0+\theta_1 +\theta_2=4\pi$, then we may assume without loss of generality that $\theta_0 > \pi$
    and we take as the fixed point of $h(\mathcal A)$ as $-x$. Then we see that a point of the face $c$ represents it. 
        Thus, every representation comes from a generalized triangle construction. 
    
    If the triangle corresponding to a representation is degenerate, then we can continuously change the lengths and angles
    of the triangles to make them nondegenerate by Proposition \ref{prop:dense0} and the deformation creates a
    representation near $h$ as possible by the density of the nondegenerate triangles in $\tilde G$. 
    Hence, the triangular representations are dense.

     Also, if the triangles are not degenerate, then the
    representation is triangular by definition. 
    A sufficiently nearby  representation to any triangular representation has nondegenerate triangles, and hence,
    they are triangular. Thus the set of triangular characters
    are open.

        \end{proof}
    

\begin{prop}\label{prop:poprep}
    $\rep(\pi_{1}(P), \SOThr)$ is homeomorphic to the quotient
    of the tetrahedron $G$ by a $\{I, I_A, I_B, I_C\}$-action.
    {\rm (}See Figure \ref{fig:tetr1}{\rm .)}
    \end{prop}
\begin{proof}
    The set of triangular representations up to conjugacy are
    classified by the open tetrahedron of angle parameters determined by equation
    \eqref{e:tri}. The ambiguity is given by the choice of fixed
    points. In the interior, this is precisely the ambiguity.
    We can choose three noncollinear fixed points of $h({\mathcal A}), h({\mathcal B})$, and 
    $h({\mathcal C})$ respectively as long as the vertices in the boundary appear
    in the clockwise-direction. Thus, we obtain that $V$ is
    generated by the maps
    \begin{eqnarray}\label{eqn:iabc}
    I_{{A}}:(v,w,z) \mapsto (v,-w,-z), \\
    I_{B}:(v,w,z) \mapsto (-v,w,-z), \\
    I_{{C}}:(v,w,z) \mapsto (-v,-w,z).
    \end{eqnarray}
    Respectively, these maps make the angles change by
    \begin{eqnarray}\label{eqn:iabc2}
    I_{{A}}:(\theta_{1},\theta_{2},\theta_{3}) \mapsto (\theta_{1},
    \pi-\theta_{2}, \pi-\theta_{3}) \\
    I_{B}:(\theta_{1},\theta_{2},\theta_{3}) \mapsto
    (\pi-\theta_{1},\theta_{2},\pi-\theta_{3})\\
    I_{{C}}:(\theta_{1}, \theta_{2},\theta_{3}) \mapsto
    (\pi-\theta_{1},\pi-\theta_{2},\theta_{3}).
    \end{eqnarray}

    In the boundary of the tetrahedron, all fixed points with angles not
    equal to $\pi$, are either identical or antipodal. Since the
    angle $\pi$, corresponds to the identity transformation, we can
    regard any two such transformations with same angles up to the
    $V$-action to be the same. As they are abelian characters,
    we see that there are no more identification other than ones obtained
    from the $V$-action.

    \end{proof}

   \section{$\SOThr$ and $\SUTw$: geometric relationships} \label{sec:sosu}
    
    Here, we discuss some well-known geometric relationship between $\SOThr$ and $\SUTw$.
    Using this, we will show that the $\SUTw$-character space of the fundamental group of 
    a pair of pants is a tetrahedron 
    branch-covering the $\SOThr$-character space of the same group. 
    This is a result known to Hubueschmann \cite{Hueb} and recently written up by Florentino and Lawton \cite{FL} in some generalities.
    Finally, for the later purposes, we reprove these by considering the fundamental group as a free group of rank two. 
    This has different flavor and will be used later.
    
    \subsection{The covering group}
    We can characterize an element of $\SOThr$ by noting its antipodal pair of 
    fixed points and the rotation angles at the points related by $\theta$ to $2\pi-\theta$. 
    Thus, $\SOThr$ can be identified with $\rp^3$ in the following way: 
    Take $B^3$ of radius $\pi$ in $\bR^3$ and parameterize each segment from the origin to the boundary 
    by the distance to the origin. Then for each nonidentity $g \in \SOThr$, we choose the fixed point with 
    angle $\theta < \pi$ and take the point in the ray to the point in $B^3$ of distance $\theta$ from the origin. 
    If $\theta=\pi$, then we take both points in the boundary $\SI^1$ of $B^3$ in the direction 
    and identified them. 
    
    Taking the quotient space, we obtain $\rp^3$. 
    This correspondence gives a diffeomorphism $\SOThr \ra \rp^3$.
    
The space of quaternions can be describes as the space of 
matrices 
\[ \begin{pmatrix} x & y\\ -\bar y& \bar x \end{pmatrix},
x,y \in \bC. \]
Elements of $\SUTw$ can be represented as elements of the unit $3$-sphere
in the space of quaternions
and the action is given by left-multiplications.
There is a unit sphere in the purely
quaternionic subspace in the space of quaternions,
i.e.,
\[ \begin{pmatrix} i \beta & \gamma + i \delta \\ -\gamma + i \delta & - i \beta \end{pmatrix}, 
\beta, \gamma, \delta \in \bR, \beta^2+\gamma^2+\delta^2 = 1. \]
$\SUTw$ acts on $\SI^2$ by conjugations.
This gives a homomorphism from $\SUTw$ to $\SOThr$
with kernels $\{I, -I\}$ where $I$ is a $2\times 2$ identity matrix, which
provides us with a double cover $\SUTw \cong \SI^3 \ra \SOThr \cong \rp^3$.

    Since $\SUTw$ double-covers $\SOThr$, the Lie group $\SUTw$ is diffeomorphic to $\SI^3$. 
    We realize the diffeomorphism in the following way: we take the ball $B^3_2$ of radius $2\pi$ so 
    that the boundary is identified with a point. Hence, we obtain $\SI^3$ .
    Let $||v||$ denote the norm of a vector $v$ in $B^3_2$.
    Take the map from $B^3_2 \ra B^3$ 
    given by sending a vector $v$ to $v$ if $||v||\leq \pi$ or to $(\pi-||v||)v$ if $||v||> \pi$. 
    This is a double-covering map clearly.

    By lifting, we obtain an identification of $\SUTw$ to $B^3_2/\sim$ by collapsing the sphere 
    at radius $2\pi$ to a point to be regarded as $-I$.
    Thus, every element of $\SUTw$ unless $I$ or $-I$ can be identified by a choice of a fixed point
    in $\SI^2$ and the angle of rotation between $0$ and $2\pi$. Hence, it is to be regarded as an element of 
    $\SOThr$ with a choice of a fixed point and the angle of rotation as seen from that point.    
    Furthermore, $I$ can be considered a rotation with any chosen fixed point and $0$ rotation angle
    and $-I$ a rotation with any fixed point and a $2\pi$ rotation angle.
    
    Since we have $R_{x,\theta} = R_{-x, 4\pi -\theta}$,
    an element of $\SUTw$ can be considered as a fixed point of $\SI^2$ with 
    angles in $[-2\pi,2\pi]$ where $-2\pi$ and $2\pi$ are identified
    or with angles in $[0, 4\pi]$ where $0$ and $4\pi$ are identified.

    Each one parameter family of commutative group $\SUTw$ double-covers 
    a one-parameter family of a commutative group in $\SOThr$. Hence, we have a circle $\SI^1_2$ of 
    twice the circumference in $\SUTw$.
    
     We denote by $R_{x,\theta}$ for $x \in \SI^2$ and $\theta \in [-2\pi,2\pi] \mod 4\pi$ 
     or $\theta \in [0, 4\pi] \mod 4\pi$, the element 
     of $B^3_2/\sim$ represented as above. 
          This element lifts $R_{x,\theta'}$ for $\theta'=\theta \mod 2\pi$. 
     Given an element of $\SUTw$ with a fixed vector and a rotation angle in $(0, 2\pi)$,
the multiplication has the effect of taking the antipodal vector as the fixed vector:
\begin{equation}\label{eqn:minusid} 
-\Idd R_{w, \theta} =  R_{w, 2\pi} R_{w, \theta}= R_{w, 2\pi+\theta} = R_{-w, 4\pi-2\pi-\theta}=R_{-w, 2\pi-\theta}.
\end{equation}

Note that this also holds when $\theta =0, \pi$.

     (Clearly, the meaning of $R_{x,\theta}$ changes depending on the context
     whether we are in $\SOThr$ or in $\SUTw$.)
    
\begin{defn}\label{defn:normalrep}
By choosing $\theta$ to be in $(0,2\pi)$, $R_{x,\theta}$ is now a point in $B^{3,o}_2 -\{O\}$. 
Thus, each point of $\SI^3 -\{I,-I\}$, we obtain a unique rotation $R_{x,\theta}$ for 
$\theta \in (0,2\pi), x \in \SI^2$ and conversely. {\em Hence, we have a normal representation 
of each element of $\SUTw-\{I, -I\}$ as a rotation with angles in $[0, 2\pi]$ with a fixed point in $\SI^2$}.
\end{defn}
    
    The ``multiplication by geometry" is also available for $\SUTw$:
Let $w_0,w_1$, and $w_2$ be vertices of a triangle, possibly degenerate, oriented in
the clockwise direction. 
Let $e_0, e_1$, and $e_2$ denote the opposite edges. 
Let $\theta_0,\theta_1$, and $\theta_2$ 
be the respective angles for $0\leq \theta \leq 2\pi$. Then 
\[R_{w_2,\theta_2} R_{w_1,\theta_1} R_{w_0,\theta_0} = - \Idd.\]
Here the minus sign is needed. We can show this by continuity 
since we can let $w_0=w_1=w_2$ and this holds
(That is, in this case, $\theta_i$ are angles of a Euclidean triangle.
The sum of two times the angles of a Euclidean triangle is $2\pi$.)

We can extend angles further.
An {\em immersed} triangle is a disk bounded by three arcs
meeting at three-points mapping with injective differential to $\SI^2$ 
so that the arcs map as geodesics. Here the angle at the vertex is measured 
by breaking it into smaller angles so that the angles inject and summing up
the values. Actually, we restrict it so that the angles are always $\leq 2\pi$. 

If none of the angles equal to $2\pi$ or $0$, then 
the disk is either 
\begin{itemize}
\item a hemisphere $H$ with a nondegenerate triangle $T$, $T\subset H$, with an edge in the boundary 
of the hemisphere removed, 
\item a hemisphere $H$ with a nondegenerate triangle $T$, $T \subset \SI^2 -H^o$, 
with an edge in the boundary of the hemisphere added, or
\item the exterior of a nondegenerate triangle or a pointed lune,
\end{itemize}
In the first case two angles are $<\pi$ and one $>\pi$,  
in the second case two angles are $>\pi$ and one $< \pi$, and 
in the last case, all three angles are $>\pi$ or two $>\pi$ and one $=\pi$. 
If one of the angles is zero or $2\pi$, we do not have a definition at this stage. 
(There are more we can include as ``immersed triangles'' but we do not need them in this paper)

Thus, for general $\theta_i$ arising from an immersed triangle, we see that 
\[R_{w_2,2\theta_2} R_{w_1,2\theta_1} R_{w_0,2\theta_0} = - \Idd.\]
if $\theta_i$s are angles of a triangle 
oriented counter-clockwise and $w_0,w_1$, and $w_2$ are ordered in the clockwise 
orientation from the disk.
(Of course, the above equation holds with $\Idd$ replacing $-\Idd$ in the $\SOThr$-cases.)

To summarize, denoting the rotation at $w_0,w_1$, and $w_2$ by ${\mathcal A}, {\mathcal B}$, and $\mathcal C$ respectively,
where $w_0,w_1,w_2$ are in clockwise orientation as seen from inside of an immersed triangle,
we obtain 
\begin{equation}\label{eqn:cba2}
{\mathcal{CBA}}=-\Idd, {\mathcal C}^{-1}=-{\mathcal{BA}}, {\mathcal C}=-{\mathcal A}^{-1}{\mathcal B}^{-1}.
\end{equation}
Thus, $\mathcal {BA}$ is obtained from multiplying $-I$ to the geometrically 
constructed one. That is $\mathcal {BA}$ can be considered as the element of $\SUTw$ with 
fixed point the antipode of the vertex of the triangle and the angle $2\pi$ minus 
the constructed angle. 
(See Section \ref{subsec:gmult}.)

Of course, if one is considering only representations, we can always find an 
imbedded triangle with all edge length $\leq \pi$ giving us 
same elements of $\SUTw$.

\begin{prop}\label{prop:multrule2} 
Suppose that we have a generalized triangle with vertices $v_0, v_1$, and $v_2$
occurring in the clockwise order with angles $\phi_0, \phi_1, \phi_2$ respectively.
Let $\mathcal A$ be a transformation with a fixed point $v_0$ and rotation angle $\phi_0$, 
$\mathcal B$ be one with $v_1$ and $\phi_1$, and $\mathcal C$ be one with $v_2$ and $\phi_2$ respectively
satisfying $\mathcal C \mathcal B \mathcal A =-\Idd$.
Then the conclusions of Proposition \ref{prop:multrule} hold.
\end{prop}

    By Proposition \ref{prop:poprep},
    $\rep(\pi_{1}(P), \SUTw)$ can be identified with the tetrahedron: 
    Clearly, $\rep(\pi_{1}(P), \SUTw)$ maps onto $\rep(\pi_{1}(P), \SOThr)$. 
    In the interior of the second space, 
    each element corresponds to the triangle with angles between $0$ and $\pi$.
    This gives a unique character with a representation $h$ in the $\rep(\pi_{1}(P), \SUTw)$ by taking $h(c_0)$ and $h(c_1)$ as given 
    in the triangle and taking as $h(c_2)$ the $-\Idd$ times the one given by the triangle. 
    The tetrahedron $G$ parametrizing $\rep(\pi_{1}(P), \SOThr)$ now lifts to 
    $\rep(\pi_{1}(P), \SUTw)$ by lifting triangles to the $\SUTw$-characters. 
    (We need a bit of care when some of ${\mathcal A}, {\mathcal B}$, and $\mathcal C$ correspond to $I$ or $-I$.) 
    The map $\iota_G$ from the tetrahedron to $\rep(\pi_{1}(P), \SUTw)$ is one-to-one since the distinct pair of angles in $[0, 2\pi]$ 
    in the normal representation imply different $\SUTw$-elements up to conjugations. (See Definition \ref{defn:normalrep})
    The map from the tetrahedron $G$ to $\rep(\pi_{1}(P), \SOThr)$ is a branch-covering map with the branch locus equal to the union of 
    the tree axes. Removing the union of the axes from the tetrahedron, we obtain a map from the space $G'$ to the open subspace $S$
    of $\rep(\pi_{1}(P), \SUTw)$ where the angles are all distinct from $\pi$ or $\pi/2$. We know that $S$ is a cover of 
    the open subset of $\rep(\pi_{1}(P), \SOThr)$ where all the angles are distinct from $\pi$ or $\pi/2$. Since so is $G'$, 
    we see that the map $\iota_G$ is onto. 
    Since this correspondence is one-to-one map from a compact space to a Hausdorff space, we obtain the following:

\begin{prop}\label{prop:pairsu} 
$\rep(\pi_{1}(P), \SUTw)$ is homeomorphic to the tetrahedron, and map to 
$\rep(\pi_1(P),\SOThr)$ as a four to one branched covering map induced by the Klein four-group $V$-action. $\square$
\end{prop}

We note here that given boundary components $c_0, c_1$, and $c_2$ of $P$, 
by letting $c_0$ and $c_1$ be regarded as generators of $\pi_1(P)$, 
we see that $-h(c_2)$ is obtained from triangle constructions, 
i.e., $-h(c_2)^{-1}=-h(c_1)h(c_0)$. 
Thus, if we use, the angles of 
the fixed point of $h(c_i)$ in $\SUTw$, then we must use $2\pi$ minus 
two times the angle of third vertex of the triangle since we need to take 
the antipodes. 

Hence, by true angle parameterization, $\rep(\pi_1(P), \SUTw)$ is parametrized by 
the tetrahedron $G'$ given by sending $(\theta_0,\theta_1,\theta_2)\ra (2\theta_0,2\theta_1,2\pi-2\theta_2)$
where $G'$ lies in in the positive octant of $\bR^{3}$ and is given by
by the equation $x+y \leq z$, $x+z \leq y$,
$y+z \leq x$, and $x+y+z \leq 4\pi$. (See Proposition 3.5 of \cite{JW} and also \cite{JW2}.)

\begin{rem}\label{rem:T}
However, we will use our tetrahedron $G$ to parameterize $\rep(\pi_1(P), \SUTw)$
because one can read off the characters to $\SOThr$ directly
with the understanding that $h(c_2)$s are the antipodal ones to what are obtained 
from the triangle constructions. 

It will be an important point in this paper that we will be ``pretending'' that our $\SUTw$-characters correspond 
to triangles. This will not cause any trouble in fact. 
\end{rem}


\subsection{The $\SOThr$- and $\SUTw$-character spaces of the free group of rank two} 
\label{subsec:free}

This will be a short section: 
Let $F_2$ denote a free group of two generators.
We will define $H(F_2)$ as the space of
homomorphisms $F_2 \ra \SOThr$ quotient by the conjugation 
action by $\SOThr$. That is,
\[H(F_2) = \Hom(F_2,\SOThr)/\SOThr.\]

Of course $F_2$ is isomorphic to $\pi_1(P)$ for a pair of pants $P$
so the above subsections covered what we intend to redo. 
(See Propositions \ref{prop:poprep} and  \ref{prop:pairsu}).
The reason we {\em redo} this is since we definitely need some modification of the description in this subsection later.
(The descriptions as $\pi_1(P)$ using triangles is also needed at other places.)

In this section, we will use the generators directly: 

Let $g_0$ and $g_1$ be the two generators of $F_2$. 
Consider a representation $h: F_2 \ra \SOThr$. 
The $h(g_i)$ has a pair of fixed points $f_i$ and 
$-f_i$ antipodal to each other for each $i=0,1$.
Then $f_0$, $-f_0$, $f_1$, and $-f_1$ lie on a great circle
which is unique if $\{f_0,-f_0\}$ are distinct from 
$\{f_1,-f_1\}$. Since up to $\SOThr$-action, 
the configuration is determined by the distance between 
the pair, we chose a maximal arc in the complement of 
$\{f_0,-f_0,f_1,-f_1\}$. Then this pair is described by 
the length $l$ of the arc and the counter-clockwise respective angles
$\theta_0$ and $\theta_1$ of 
rotations of $g_0$ and $g_1$ at the corresponding endpoints of $l$.
Thus $0\leq l \leq \pi$ and $0 \leq \theta_i \leq 2\pi$
for $i=0,1$. Since the choice of the arc is not unique, 
there are maps
\begin{eqnarray}
T_0: (l, \theta_0, \theta_1) &\mapsto& (\pi-l, 2\pi-\theta_0, \theta_1) \nonumber\\
T_1: (l, \theta_0, \theta_1) &\mapsto & (\pi-l,\theta_0,2\pi-\theta_1) \nonumber \\
T_0T_1 = T_1T_0 : 
(l, \theta_0, \theta_1) &\mapsto & (l, 2\pi-\theta_0, 2\pi-\theta_1) \label{eqn:F21} 
\end{eqnarray}
(This is again called a {\em Klein four-group action} and is identical with 
the Klein four-group action for the character spaces of a pair of pants. 
We will need the maps later in Section \ref{subsec:ABC} and Section \ref{subsec:tildeO}.)
Also, if $\theta_i = 0$, then $l$ is not determined. Therefore,
there is another equivalence relation 
\begin{equation}\label{eqn:F22}
(l,0,\theta_1) \sim (l',0,\theta_1), 
(l,\theta_0, 0) \sim (l',\theta_0,0), 0\leq l,l'\leq \pi, 
\end{equation}
for any $\theta_0$ and $\theta_1$. 
It is straightforward to see that these are all possible equivalence 
classes under the conjugation action of $\SOThr$.

\begin{prop}\label{prop:HF2} 
$H(F_2)$ is homeomorphic to a $3$-ball. 
The boundary of the $3$-ball corresponds to precisely the 
subspace of abelian characters of $F_2$.
\end{prop}
\begin{proof}

Let $I=[0,\pi]$. 
We consider the space $I\times \SI^1 \times \SI^1$ parameterizing 
the above configurations of arc and isometries with fixed points 
at the end. There is a surjective (continuous) map 
$ Q:I\times \SI^1\times \SI^1 \ra H(F_2)$ defined 
by sending the configuration to the equivalence class of representations 
sending $g_0$ and $g_1$ to the isometries with the fixed points
at the end of the arc. 

Let $\sim$ be given by equations \eqref{eqn:F21} and \eqref{eqn:F22}.
We see that the quotient space $B = I\times \SI^1\times \SI^1/\sim$ 
is mapped to $H(F_2)$ in a one-to-one and onto manner.
There is an inverse map defined on $H(F_2)$ by sending 
each character to $(l, \theta_0, \theta_1)$ by 
doing the procedure before this proposition.
Hence $B$ is homeomorphic to $H(F_2)$. 

We now show that the quotient space $B$ is homeomorphic to $B^3$.

All points of $B$ have equivalent points on $[0,\pi/2]\times \SI^1\times \SI^1$.
We take a domain $D= [0,\pi/2]\times[0,\pi] \times [0,\pi]$.
Then by the equivalence relation generated by $T_0T_1=T_1T_0$
acting fiberwise on $[0,\pi/2]\times \SI^1\times \SI^1$, 
the quotient image $D'$ of $D$ is homeomorphic to 
$[0,\pi/2]\times \SI^2$ 

The equivalence relation generated by $T_0$ is 
the same as that of $T_1$. Only the point in $\{\pi/2\}\times \SI^2$ 
are involved here. A circle $S^1$ in $\{\pi/2\}\times \SI^2$
contains the four singular points.
The equivalence can be described by an isometric reflection on a circle
with some metric on a sphere of constant curvature $1$. 
This last identification gives us a ball $B^3$. 

The final identification corresponding to equation \eqref{eqn:F22} 
collapses two rectangular disks in $B_3$ with one edge 
in the boundary sphere $\partial B_3$ 
meeting each other at a connected boundary edge. 
The rectangular disks are foliated by arcs in $I$-directions, and hence can be considered 
strips. 
The two disks are called the $I$-collapsing
quadrilaterals. The collapsing the strips in $I$-direction gives us the final result.

Since abelian characters correspond to the configurations where 
one of the angle is zero, they are in the $I$-collapsed
part, or in where the segment is of length $0$ or $\pi$. 
We see that these points correspond to the points of boundary of the ball $B_3$
when collapsed as above. 

\end{proof}

The last collapse is said to be the {\em $I$-collapse.}

\subsection{The $\SUTw$-character space} 
\label{subsec:free2}

Now we go over to the $\SUTw$ case.
Define
\[H_\SUTw(F_2) := Hom(F_2,\SUTw)/\SUTw.\]
Then since $F_2$ is isomorphic to $\pi_1(P)$, 
the space is homeomorphic to a tetrahedron $G$.

To understand in term of the above picture, 
we look at two points in $\SI^2$ with angles in the range $(0, 2\pi)$. 
This determines a unique pair of points and the distance between them 
if the angles are not $0$ or $2\pi$. 
Parameterizing by the distance and the angle, 
we obtain $[0,\pi]\times [0,2\pi]\times[0,2\pi]$
by picking only the fixed points with angle in $[0, 2\pi]$. 
The equivalence relation is given by $(l,0,\theta_2)\sim (l',0,\theta_2)$ 
for any $l, l'$ and $(l,2\pi,\theta_2)\sim (l',2\pi,\theta_2)$ 
for any $l, l'$ and $(l,\theta_1,0)\sim (l',\theta_1,0)$ for any pair $l,l'$ and
$(l,\theta_1,2\pi)\sim (l',\theta_1,2\pi)$ for any pair $l,l'$.

These lifts the $I$-collapsing above 
and there are no other nontrivial equivalence 
relations.
This contracts sides of the cube in the $l$ direction. 
Thus, we obtain a $3$-ball. 

Since $[0,\pi]\times [0,2\pi]\times[0,2\pi]$ represents $[0,\pi]\times \SI^1\times \SI^1$ 
up to identification, 
we see that there is a four-to-one branch covering map to 
$H(F_2)$ as generated by transformations of forms $T_0$ and $T_1$ as in equation \eqref{eqn:F21}.

The following is clear:
\begin{prop}\label{prop:HFSU2} 
$H_\SUTw(F_2)$ is homeomorphic to a $3$-ball.
The boundary of the $3$-ball precisely corresponds to the 
subspace of abelian characters in $H(F_2)$.
\end{prop}

\section{The character space of a closed surface of genus $2$.}\label{sec:ch}

We will now discuss the character space of the fundamental group of
a closed orientable surface $\Sigma$ of genus $2$.

First, we discuss the two-components  of the character space. 
Next, we discuss how to view a representation as two related representations 
of the fundamental groups of two pairs of pants glued by three pasting maps.

\begin{figure}

\centerline{\includegraphics[height=4cm]{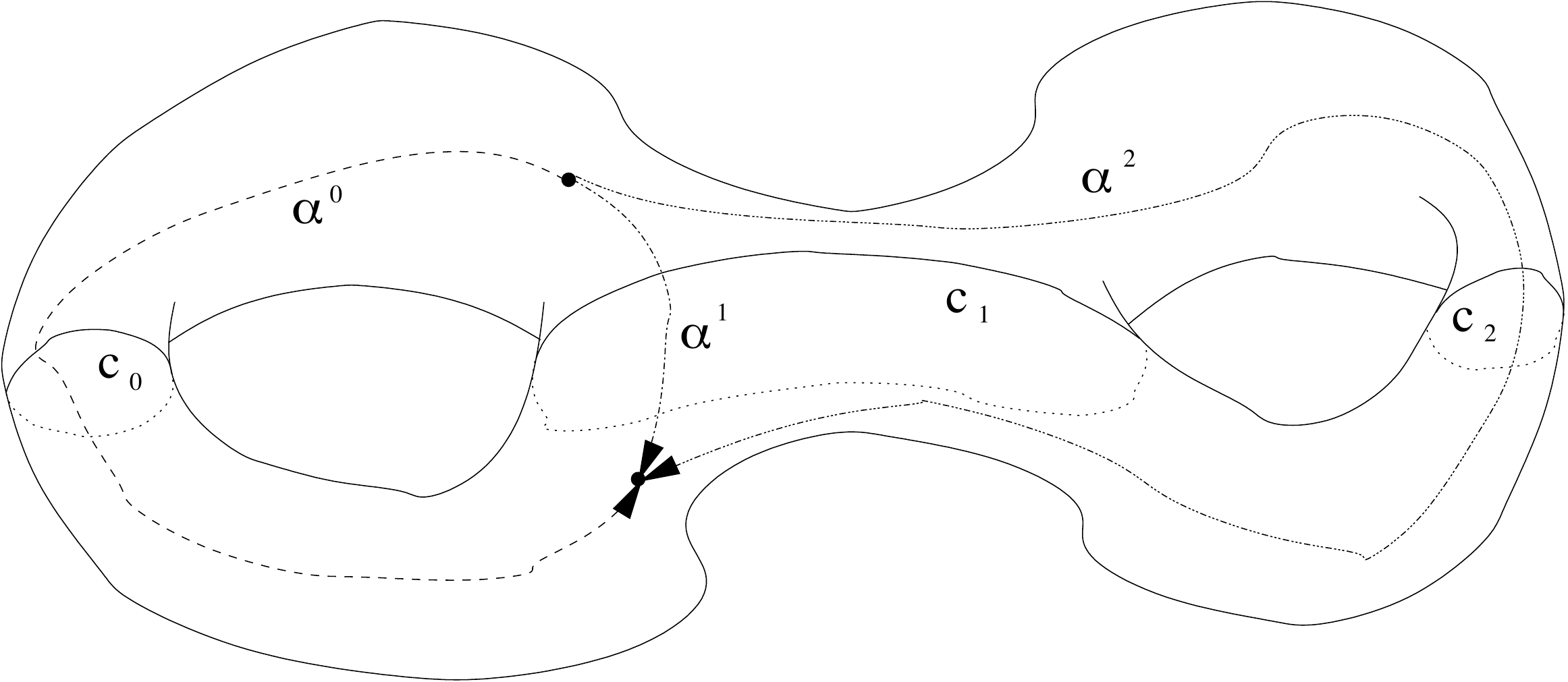}}
\caption{$\Sigma$ and closed curves. }
\label{fig:figsigma}
\end{figure}

\subsection{Two components}\label{subsec:twocomp}
We take three simple closed curves $c_0, c_1$, and $c_2$ on $\Sigma$
so that the closures of the two components of the complement of their union are
two pairs of pants $S_0$ and $S_1$ so that $S_0 \cap S_1 = c_0 \cup c_1 \cup c_2$.
We assume that there is an order two diffeomorphism $\phi$ acting
on $c_0$, $c_1$, and $c_2$ respectively and sending $S_0$ to $S_1$
and vice versa. We assume that $c_i$ are oriented with boundary
orientation from  $S_0$.

There are two components of $\rep(\pi_1(\Sigma), \SOThr)$
as shown by Goldman \cite{Goldman3}. To summarize,
each representation corresponds to a flat $\bR^3$-bundle over
$\Sigma$. The Stiefel-Whitney class in $H^2(\Sigma, \pi_1(\SOThr)) = \bZ_2$ of
the flat bundle classifies the component.
A component ${\mathcal C}_0$ contains the identity representation,
i.e., a representation which send every element to the identity.
Let $c_1,d_1,c_2,$ and $d_2$ are the simple 
closed curves which represents a fixed set of the standard generators of 
$\pi_1(\Sigma)$ so that $c_1, c_2$ are obtained from the boundary 
components of $S_1$ by conjugating by simply disjoint paths in $S_1$.
(There is a slight notational abuse here.)
The other component to be denoted by ${\mathcal C}_1$
contains the representation $h^*$ which sends
generators corresponding to $c_1$ and $d_1$ a simple closed curve
meeting $c_1$ once transversely and disjoint from $c_2, d_2$
to the following matrices respectively
\begin{equation}
A_1 :=
\begin{bmatrix}
1 & 0 & 0 \\
0 & -1 & 0 \\
0 & 0 & -1
\end{bmatrix}  \hbox{ and }
B_1:=
\begin{bmatrix}
-1 & 0 & 0 \\
0 & -1 & 0 \\
0 & 0 &  1
\end{bmatrix}
\end{equation}
And $h^*$ sends $c_2$ and a simple closed curve $d_2$ meeting 
$c_2$ once
transversely and disjoint from $c_1, d_1$ to
the identity matrices respectively



Let this representation be named the {\em other-component basis representation}.

Recall that the three-dimensional sphere $\SI^3$ sits as the unit sphere in the purely
quaternionic subspace in the space of quaternions,
i.e.,
\[ \begin{pmatrix} i b & c+ i d \\ -c + i d & - i b \end{pmatrix},
b^2+c^2+d^2 = 1. \]
Elements of $\SUTw$ can be represented as elements of the unit sphere
in the subspace of quaternions
and the action is given by left-multiplications.
$\SUTw$ acts on $\SI^2$ by conjugations.
This gives a homomorphism from $\SUTw$ to $\SOThr$
with kernels $\{I, -I\}$ where $I$ is a $2\times 2$ identity matrix.

\begin{rem}\label{rem:otherc}
The representation belongs to ${\mathcal C}_1$:

The above matrices $A_1$ and $B_1$ respectively correspond to
\[
\begin{pmatrix} i & 0 \\ 0 & -i \end{pmatrix} \hbox{ and }
\begin{pmatrix} 0 & i \\ i & 0 \end{pmatrix}  
\] 
in $\SUTw$ whose commutator equals $-I$.
This does not give a representation $\pi_1(\Sigma) \ra \SUTw$ but
will induce a representation $\pi_1(\Sigma) \ra \SOThr$ whose
Stiefel-Whitney class is not trivial.
\end{rem}

\subsection{Representations considered with pasting maps}
Let us choose a base point  $x^*$ of $\Sigma$ in the interior of $S_0$.
Given a representation $h: \pi_1(\Sigma) \ra \SOThr$,
we obtain a representations $h_0: \pi_1(S_0) \ra \SOThr$ and
$h_1: \pi_1(S_1) \ra \SOThr$. The second one is determined only
up to conjugation.
A representation $h$ defined on $\pi_1(\Sigma)$ is {\em triangular} 
if the restrictions to $\pi_1(S_0)$ and $\pi_1(S_1)$ are triangular. 
Let us restrict ourselves to triangular representation in $\rep(\pi_1(\Sigma),\SOThr)$ 
for now. 

Let $c^0_0, c^0_1$, and $c^0_2$ denote 
the oriented simple closed curves on $S_0$ with
base point $x_0^*$ that 
are freely homotopic to $c_0,c_1$, and $c_2$ respectively,
and bounding annuli with the curves respectively.
Let us choose a base point $x_1^*$ in $S_1$ and oriented
simple closed curves $c^1_0, c^1_1$, and $c^1_2$ homotopic to
$c_0, c_1$, and $c_2$.
Let $\alpha^i$ denote a simple arc from $x_0^*$ to $x_1^*$ whose interior is
disjoint from $c_j$ for $j \ne i$ and meets $c_i$ exactly once.
We can choose so that $\alpha^0*c_0^1*\alpha^{0,-1}$ is homotopic to
$c_0^0$.

Let $d_1$ denote the closed loop based at $x_0^*$ meeting $c_0$ and $c_1$ only once
transversally and disjoint from $c_2$ oriented from $c_0$ to $c_1$ direction.
Let $d_2$ denote the closed loop based at $x_0^*$ meeting $c_0$ and $c_2$ only once
transversally and disjoint from $c_1$ oriented from $c_2$ to $c_0$ direction.
We assume that $\alpha^1*\alpha^{0,-1}$ is homotopic to $d_1$,
$\alpha^2*\alpha^{0,-1}$ is homotopic to $d_2$, and
$c_1^0, c_2^0, d_1$, and $d_2$ generate the fundamental group
$\pi_1(\Sigma)$ and they satisfy the relation
\[ [c_1^{0}, d_1][c_2^{0}, d_2] = 1. \]
(We identify $\pi_1(S_1)$ with $\alpha_0^{-1}*i_*(\pi_1(S_1))*\alpha_0$.)

If $h_0:\pi(S_0) \ra \SOThr$ and $h_1:\pi_1(S_1) \ra \SOThr$
are induced from a representation $h: \pi_1(\Sigma) \ra \SOThr$,
then $h_0(c_i^0)$ are conjugate to $h_1(c_i^1)$
by an element $P_i$ of $\SOThr$,
i.e., \[P_i h_0(c_i^0) P_i^{-1} = h_1(c_i^1) \hbox{ for } i = 0,1,2.\]
We call $P_i$ the pasting map for $c_i$ for $i=0,1,2$.

If $h_0$ and $h_1$ are obtained from restricting $h$, then $P_0 = I$.
If we change $h_1(\cdot)$ to be $h'_1(\cdot) = k \circ h_1(\cdot) \circ k^{-1}$, then
the new pasting maps become
\begin{equation}\label{eqn:change}
P'_0 = k  P_0, P'_1 = k  P_1 \hbox{ and } P'_2 = k P_2.
\end{equation}
so that
\begin{equation}\label{eqn:conj}
P'_i h_0(c_i^0) P^{\prime -1}_i = h'_1(c_i^1)
\hbox{ for } i=0,1,2.
\end{equation}

\begin{prop}\label{prop:d1d2}
We have 
\[ h(d_1)=P_0^{-1} P_1, h(d_2)=P_0^{-1} P_2. \] 
\end{prop}
\begin{proof} 
When $k=\Idd$, the equation holds. 
By equation \eqref{eqn:change}, it follows.
\end{proof}


The space $\rep(\pi_{1}(P), \SOThr)$ is classified as a quotient space of
a tetrahedron. Thus, a character is determined by angles
$\theta_0, \theta_1$, and $\theta_2$ in $[0,\pi]$ as the halves of the rotation angles.
The equation \eqref{eqn:conj} shows that
if $h_0$ is represented by triple $(\theta_0, \theta_1, \theta_2)$,
the $h_1$ is represented by $(\theta'_0, \theta'_1, \theta'_2)$ so that
$\theta'_i = \theta_i $ or $\theta'_i = \pi - \theta_i$ for $i = 0,1,2$.

We showed:
\begin{prop}\label{prop:h0h1}
Let $h_0$ and $h_1$ be representations of the fundamental groups of 
pairs of pants $S_0$ and $S_1$ from a $\SOThr$-representation $h$.
The angles $(\theta_0, \theta_1, \theta_2)$ of $h_0$
and $(\theta'_0, \theta'_1, \theta'_2)$ of $h_1$ satisfy the equation
\begin{eqnarray}\label{eqn:h0h1angle}
\theta'_i &=& \theta_i \hbox{ or } \\
\theta'_i &=& \pi-\theta_i \hbox{ for } i =0,1,2.
\end{eqnarray}
\end{prop}


\section{Establishing equivalence relations on $\tilde G \times T^3$ to make it equal to ${\mathcal C}_0$.}
\label{sec:C0}

In this major section, we will study the identity component 
${\mathcal C}_0$ by finding the equivalence relation on $\tilde G\times T^3$ giving 
us the same representations up to conjugacy, i.e., to find a point of ${\mathcal C}_0$. 

This will be done by considering each of the regions of $\tilde G$: 
\[G^o, a,b,c,d, {A}, {A}', B,B', C,C',I,II,III,IV\]
individually and considering the action of 
the Klein four-group $\{I, I_{{A}}, I_{{B}}, I_{{C}}\}$ extended to $\tilde G \times T^3$.

We begin by discussing pasting angles. Next, we define a map $\mathcal T$ from 
 $\tilde G\times T^3$ to ${\mathcal C}_0$ corresponding the finding holonomy for the pasting angles. 
 We will show that the map is onto by showing the characters corresponding to nondegenerate triangles 
 are dense. 
 
 Then we discuss the equivalence relation $\sim$ on $\tilde G \times T^3$
 so that the quotient map induces a homeomorphism $\tilde G \times T^3/\sim \ra {\mathcal C}_0$.
 Next, we discuss the quotient relation for the interior and
 move to the boundary regions denoted by $a$, $b$, $c$, and $d$, 
 $I$, $II$, $III$, and $IV$ and finally regions $B$, $B'$,$A$,$A'$,$C$ and $C'$.
 Mainly, we will induce the quotient relations to make the map $\mathcal T$ 
 injective one. The final result of the section is Theorem \ref{thm:main1} establishing 
 the homeomorphism of the appropriate quotient space of $\tilde G \times T^3$ 
 to ${\mathcal C}_0$.
  
 In this section, we will not determine the topology of the quotient spaces but will only 
 explain the equivalence relations. The geometric nature of equivalence 
 relations instead of algebraic equivalences are emphasized here to gain more a geometric 
 point of view of the subject instead of algebraic one. 
 
 The equivalence relation {\em is} 
 the equivalence relation to obtain the identical $\SOThr$-characters.
 However, we chose not to simply state this algebraic fact to preserve the geometric flavor of the paper
 and to gain the topological characterizations of the subspaces in the next section. 
  
  \subsection{The density of triangular characters in ${\mathcal C}_0$.}

\begin{lem}\label{lem:nonc0}
For $i=0,1,2$, define a function 
$f_i:  \rep(\pi_1(\Sigma),\SOThr) \ra \bR$ by 
$f_i: [h] \mapsto \mathrm{tr} h(c_i)$. 
Then the triangular subset $D_0$ of ${\mathcal C}_0$ is not empty
and $f_i$ is not constant.
\end{lem}
\begin{proof} 
The triangular part of ${\mathcal C}_0$ can be constructed by taking 
two identical triangles with identity to construct representations 
corresponding to $\pi_1(S_0)$ and $\pi_1(S_1)$.
Here clearly, $c_i$ has nonconstant rotation angles in $D_0$.
Thus $f_i$ is not constant. 
\end{proof}

\begin{prop}\label{prop:tridense0} 
The subset $D_0$ of ${\mathcal C}_0$ of triangular characters 
is dense and open and moreover, any character can be deformed to 
a triangular one by an analytic path consisting of triangular character 
except at the end.
\end{prop}
\begin{proof} 
The subset where $h(c_i) = \Idd$ for the associated representations $h$ is 
given by an analytic subset of ${\mathcal C}_0$. 
Define a real analytic map $f_i: [h] \mapsto {\mathrm{tr} } h(c_i)$ for each $i=0,1,2$ which is well-defined. 
Since the function is not constant on each component of 
 ${\mathcal C}_0$, by the standard result in 
real algebraic geometry, the set of points where $f_i=3$ is 
an analytic subset whose complement is dense and open.
Finally, the set of abelian ones are of lower dimension $4$ and is an analytic subset, and the result follows. 
\end{proof}

\subsection{Pasting angles}
\begin{prop}\label{prop:h0h12}
If $h$ is in the identity component ${\mathcal C}_0$ of $\rep(\pi_1(\Sigma), \SOThr)$,
and $h_0$ and $h_1$ be obtained as above by restrictions to $S_0$ and $S_1$. 
then
\begin{itemize}
\item[{\rm (a)}] We can conjugate $h_1$ so that $h_0 = h_1$
and corresponding angles are equal.
\item[{\rm (b)}] For each representation $h$ in ${\mathcal C}_0$, we can associate
a pair of identical degenerate or nondegenerate triangles and an element of
$\SI^1 \times \SI^1 \times \SI^1$, i.e., the parameter space of the pasting angles.
{\rm (}This is not normally a unique association for the degenerate triangle cases.{\rm )}
\end{itemize}
\end{prop}
\begin{proof}
(a) This is proved by a continuity argument.
We look at the dense open set $D_0 \subset {\mathcal C}_0$ of triangular characters
from Proposition \ref{prop:tridense0}.
Taking a triangular representation $h_0:\pi_1(S_0) \ra \SOThr$ associated
with a triangle of angles all equal to $2\pi/5$ and $h_1(S_1) \ra \SOThr$
associated with the same triangle, then $h_0=h_1$.
Let $h^0: \pi_1(\Sigma) \ra \SOThr$ be the representation obtained from
$h_0$ and $h_1$ with pasting maps $P_i$s equal to $\Idd$.

There is an open subset of $D_0$ where
induced $h_0$ and $h_1$ have angles
$\theta'_i$ and $\theta_i$ so that $\theta'_i = \theta_i$ for $i=0,1,2$:
At $h^0$, we see that $\theta'_i = \theta_i$ for $i=0,1,2$
by construction. As a function of $D_0$, $\theta_i$ and $\theta'_i$ form
continuous functions since these are obtainable as traces of holonomies of
boundary curves. So by continuity and equation \eqref{eqn:h0h1angle},
it follows that there exists a neighborhood where the angles are equal.
Since triangular representations are determined by angles, we can
conjugate so that $h_0 = h_1$.

By continuity argument again, the equation holds for all of ${\mathcal C}_0$. 

(b) From (a), the equation is true for elements of $D_0$
as the elements of $\SI^1 \times \SI^1 \times \SI^1$ are
determined by the rotation angles of pasting maps $P_0, P_1, P_2$
at the respective vertices of the triangle.

For the element $d$ in the boundary of $D_0$,
we take a sequence $d_i \in D_0$ converging to $d$.
Associated to each $d$, we obtain a triangle, of
$\tilde G$ and an elements $P_i$ with pasting angles $(\phi_0^i, \phi_1^i, \phi_2^i)$ in
$\SI^1\times \SI^1 \times \SI^1$.
We can take a subsequence so that $(\tilde T_i, \phi_0^i, \phi_1^i, \phi_2^i)$
is convergent. Let $(\tilde T_\infty, \phi_0^\infty, \phi_1^\infty,
\phi_2^\infty)$ denote the limit.
The associated triangle $\tilde T_i$ with angles assigned converges
a generalized pointed triangle $\tilde T^\infty$
with assigned angles that are respective limits of
the sequences of angles of $P_i$.

Since $\SOThr$ is compact, each pasting map $P_j^i$, $j=0,1,2$, for $d_i$
converges to an automorphism $P_j^\infty$ with angle $\theta_j^\infty$
for $j=0,1,2$. It is straightforward to see that $d$ can be realized
as a pair of generalized triangles $\tilde T^\infty$ and its copy
and pasting maps $P_j^i$, $j=0,1,2$.

\end{proof}

\subsection{The definition of the map $\mathcal T$.}
We will now define a map from $\tilde G \times T^3$ to ${\mathcal C}_0$ which is
almost everywhere $4$ to $1$.
Let us define a map 
\[\mathcal T: \tilde G \times T^3 \ra {\mathcal C}_0 \subset \rep(\pi_1(\Sigma), \SOThr). \]
Take an element of $\tilde G$. This gives us a triangle perhaps
degenerate with angles assigned.
Actually, the map is induced from a map
\[\mathcal T': \hat G \times T^3 \ra C'_0 \subset \Hom(\pi_1(\Sigma), \SOThr)\]
where $C'_0$ is the component corresponding to ${\mathcal C}_0$ and
there is a commutative diagram:
\[
\begin{CD}
\hat G \times T^3 @>{\mathcal T'}>> C'_0 \subset \Hom(\pi_1(\Sigma), \SOThr)\\
@V{q_0}VV                                     @V{q_1}VV \\
\tilde G \times T^3 @>{\mathcal T}>>  {\mathcal C}_0 \subset \rep(\pi_1(\Sigma), \SOThr).
\end{CD}
\]
Here, $q_0$ is the quotient map under the isometric $\SOThr$-action on the metric space
of generalized triangles and $q_1$ one under the conjugation action of
the space of representations:

We define the map $\mathcal T'$ from $\hat G\times T^3$:
If the triangle is nondegenerate, then
we can glue two copies with vertices removed along the open edges
and obtain a pair of pants, from which we obtain a representation of
$\pi_1(S_0) \ra \SOThr$ for a pair of pants $S_0$.  We obtain the second pair of pants $S_1$
using the exactly the same method. We choose base points $x_i^*$ for
on $S_i$ respectively for $i=0,1$. Let $S_i$ to be a union of 
two triangles $T_i^0, T_i^1$ with vertices removed and where $T_i^0$ has the base-point of $S_i$
for $i=1,2$. We put $T_i^1$ in the standard position and
obtain the elements of $\SOThr$ corresponding to $c_j^i$.
We see that $T_i^0$ are identical with $T_i^1$ and so are the representations.
Using $(\phi_0, \phi_1, \phi_2) \in T^3$, we
assign the gluing maps to be the rotations of those angles
respectively at the fixed points of holonomies of the boundary components of
$S_0$ and $S_1$ as above. 

If the triangle is degenerate, i.e., a pointed hemisphere, a pointed lune,
a pointed segment, or a pointed singleton, we can do the similar construction
as above considering them as having infinitesimal edges if necessary.
From the constructions, one obtains the full-representations.

The map $\mathcal T'$ induces a map $\mathcal T$ since the isometry $\SOThr$-action on
$\hat G$ and pasting map results in conjugation of corresponding holonomies in
$\Hom(\pi_1(\Sigma), \SOThr)$.



\subsection{Extending $I_A, I_B$, and $I_C$}


In fact, $I_{{A}},I_{{B}}$, and $I_{{C}}$ extend to self-homeomorphisms of $\tilde G\times T^3$ since our action is 
just taking different fixed point for two vertices of the triangle and then 
the pasting angle changes correspondingly also.

For example, $I_{{A}}$ can be realized given a generalized triangle with vertices $v_0, v_1, v_2$ 
as taking the opposite triangle
at the vertex $v_0$ and $I_{{B}}$ at $v_1$ and $I_{{C}}$ at $v_2$.
We can assume that our triangle is always one of the four possible ones.

The pasting maps themselves are independent of the choice of
triangles among the four. However, the fixed point used to measure
the angle become antipodal in our picture. In this case, the rotation angle
becomes $2\pi$ minus the rotation angle as measured originally
because of the orientation consideration.

We describe the action 
below where $I_{{A}},I_{{B}},I_{{C}}$ inside are the transformations on $\tilde G$ described above:
 \begin{eqnarray}\label{eqn:iabc4}
    I_{{A}}:(x, \phi_0, \phi_1, \phi_2) \mapsto 
    (I_{{A}}(x), \phi_0, 2\pi-\phi_1, 2\pi-\phi_2)\nonumber \\
    I_{{B}}:(x, \phi_0, \phi_1, \phi_2) \mapsto
    (I_{{B}}(x), 2\pi-\phi_0, \phi_1, 2\pi-\phi_2) \nonumber \\
    I_{{C}}:(x, \phi_0, \phi_1, \phi_2) \mapsto
    (I_{{C}}(x), 2\pi-\phi_0, 2\pi-\phi_1, \phi_2).
    \end{eqnarray}
    Since the actions correspond to changing the fixed points of $c_i$ and hence does not change 
    the associated representations, we have 
    ${\mathcal T}\circ I_{{A}} ={\mathcal T}\circ I_{{B}} = {\mathcal T}\circ I_{{C}} ={\mathcal T}$.
    
   In Sections \ref{subsec:intequiv}, \ref{subsec:abcd}, \ref{subsec:regioni}, and \ref{subsec:ABC}, we will introduce 
   equivalence relation based on making $\mathcal T$ injective on facial regions times $T^3$. 
  

\subsection{The interior equivalence relation}\label{subsec:intequiv}

By the following proposition, a representation or a character is {\em over } a subset $W$ of $\tilde G$ if 
it comes from a point of the subset $W \times T^3$.

\begin{prop}\label{prop:map}
\[\mathcal T: \tilde G \times T^3 \ra {\mathcal C}_0 \subset \rep(\pi_1(\Sigma), \SOThr)\]
is onto and  $4$ to $1$ branch-covering from the interior of $\tilde G \times T^3$ 
to the dense open subset of ${\mathcal C}_0$, i.e., the triangular subspace 
given by the action of the Klein four-group described by equation \eqref{eqn:iabc4}.
\end{prop}
\begin{proof}
First, $\mathcal T| \tilde G^o \times T^3$ is onto
the set of triangular characters in ${\mathcal C}_0$ since
a nondegenerate triangle and three pasting angles determine
a triangular character $\pi_1(\Sigma) \ra \SOThr$ up to conjugations
and hence a point in ${\mathcal C}_0$ as a consequence by Proposition \ref{prop:h0h12}.
Since the set of triangular characters in ${\mathcal C}_0$ are dense,
$\mathcal T$ is onto ${\mathcal C}_0$. 

A triangular holonomy of a pair of pants determines a triangle only
up to the action of the group of order $4$ generated by $I_{{A}}, I_{{B}}, I_{{C}}$.
Given a triangle realized by the holonomy, the other three are given
as follows: Extend the three sides into great circles
and divide the sphere into $8$ triangles. Take only four with
same orientation as the given triangle. These are the ones obtainable
by the action of the group. We take the identical triangles for $\pi_1(S_0)$ and
$\pi_1(S_1)$.
\end{proof}

\begin{rem}\label{rem:uniqueness}
By above, the set of triangular characters and $\tilde G^o\times T^3/\{I, I_{{A}}, I_{{B}}, I_{{C}}\}$ are
in one-to-one correspondence. 

But if $h$ is nontriangular, then there may be more than one way to conjugate 
$h|\pi_1(S_0)$ and $h|i_*\pi_1(S_1)$ so that they correspond to 
generalized triangles in standard position. 
\end{rem}

\subsection{The equivalence over regions $a$, $b$, $c$, and $d$}\label{subsec:abcd}

We can think of the region $a$ as a subset of 
$[0,\pi]\times [0,\pi]\times[0,\pi]$ 
given by equation $\theta_0+\theta_1+\theta_2=\pi$. 
$a$ is to be the region given by $\theta_0,\theta_1,\theta_2\geq 0$.
Consider $a \times T^3$ and the equivalence relation 
\begin{multline}
(x, \phi_0, \phi_1, \phi_2) \sim (y, \phi'_0, \phi'_1, \phi'_2) 
\hbox{ iff } x = y \hbox{ and } \\ (\phi_0, \phi_1, \phi_2) - (\phi'_0, \phi'_1, \phi'_2) = s (2\pi, 2\pi, 2\pi), s \in \bR
\end{multline}
which was determined to have the same abelian characters. 
Thus, we consider this a circle action or translations generating the equivalence.
For $b\times T^3$, the equivalence class is: 
\begin{multline}
(x, \phi_0, \phi_1, \phi_2) \sim (y, \phi'_0, \phi'_1, \phi'_2) 
\hbox{ iff } x = y \hbox{ and } \\ (\phi_0, \phi_1, \phi_2) - (\phi'_0, \phi'_1, \phi'_2) = s (2\pi, -2\pi, 2\pi), s \in \bR
\end{multline}
For $c\times T^3$, it is generated by translations by 
vectors parallel to $(-2\pi, 2\pi, 2\pi)$ as above, 
and finally for $d \times T^3$, it is generated by translations 
by vectors parallel to $(2\pi, 2\pi, -2\pi)$ as above. 
Thus, there is an $\SI^1$-action in each of the regions above these faces $a, b, c$, and $d$. 

Let us look at $a\times T^3/\sim$ in more detail: 
So the equation determining the equivalence class representatives
should be $\phi_0+\phi_1+\phi_2=0$ as $2\pi=0$ in $\SI^1\times \SI^1\times \SI^1$
with an order-three translation-group action:
We take the subset of $T^3$ given by the equation. Looking at $T^3$ 
as a quotient space of the cube $[0,2\pi]^3$, we see that the plane 
corresponds to the union of two triangles in $T^3$ with vertices 
$\tri_1=\tri([2\pi,0,0],[0,2\pi,0],[0,0,2\pi])$ and $\tri_2=\tri([2\pi,2\pi,0],[2\pi,0,2\pi],[0,2\pi,2\pi])$. 
Then two triangles form a $2$-tori $T'$ under the identifications of $T^3$. 
$T'$ is isometric with a torus in the plane $E^2$ quotient out by a lattice generated by 
translation by $(2\pi\sqrt{2},0)$ and $(\pi\sqrt{2}, \pi\sqrt{6})$. 
There is another identification by the above $\SI^1$-action for the union $T'$:
The identification is given by a translation by $(\pi, \pi\sqrt{6}/3)$ in $T'$,
which is of order three. The quotient space $T^2_a$ is homeomorphic to a $2$-torus. 
Thus, the quotient space here is in one-to-one correspondence with 
$a\times T^2_a$.

Similarly, for other faces, we obtain trivial $2$-torus fibrations as well.

\begin{prop}\label{prop:abelian} 
\begin{itemize}
\item[{\rm (i)}] The set of abelian characters is homeomorphic
with $T^4/\sim$ where $\sim$ is given by
\[(\phi_0,\phi_1,\phi_2,\phi_3)
\sim (2\pi-\phi_0,2\pi-\phi_1,2\pi-\phi_2,2\pi-\phi_3)\] 
where coordinates are rotation angles of $h(c_1),h(c_2), h(d_1), h(d_2)$ 
respectively at a common fixed point determined up to the antipodal map.
\item[{\rm (ii)}] Every abelian character $\pi_1(\Sigma) \ra \SOThr$ is 
represented as an element over one of the regions $a$, $b$, $c$, and $d$ {\rm ;}
\item[{\rm (iii)}] 
The subspace of abelian character is homeomorphic with $a \times T^2_a/\sim$ where 
$\sim$ is given on $(a \cap d) \times T^2_a$ by $ I_{{A}}$ and on
$(a\cap c)\times T_a$ by  $I_{{B}}$ and on $(a\cap b)\times T_a$ by  $I_{{C}}$ 
as in equation \eqref{eqn:iabc4}.
\item[{\rm (iv)}] The above quotient space is homeomorphic 
to a generalized Seifert space where the base space is 
a sphere with four singular points of order two 
with fiber homeomorphic to $T^2$ and exceptional fibers 
are spheres with four singular points. 
\end{itemize}
\end{prop}
\begin{proof} 
(i) For abelian representations, $c_1, c_2, d_1, d_2$ will map to an isometry fixing a common pair of points. 
If we choose one fixed point in the antipodal pair of fixed points, 
then the angles are given. If we choose the other fixed point, 
then the angles are given as the complementary one to the former ones. 

The homeomorphism property follows easily. 

(ii) If a representation $h$ is abelian, then $h$ has a common pair of antipodal points 
which are fixed points. The generalized triangle determined 
is then either a point or a segment of length $\pi$,
and the gluing transformations have fixed points in the vertices. 
Thus, $h$ corresponds to an element over one of the regions $a$, $b$, $c$, and $d$. 

(iii) The angles $\theta_i$ of the triangle are half of the rotation angles. 
$a$ is the fundamental region of $\partial G \times T^3$ given 
by the above action of $\{I, I_A, I_B, I_C\}$. 
Thus, the quotient space is obtained from $a \times T^2_a$ by $\sim$.

(iv) When we take the quotient of $a \times T^2_a$, 
the four singular points come from the vertices of $a$ which identifies to a point 
and the middle point of the edges of the triangle $a$.
(See Lee-Raymond \cite{LR}.)
 
\end{proof}


\subsection{Equivalence relations for regions $I,II,III,IV$}\label{subsec:regioni}

The points over regions $I$, $II$, $III$, and $IV$ 
are ones corresponding to the characters $\pi_1(\Sigma)$ that factor into 
$F_2$ sending $c_0, c_1$, and $c_2$ to identity since there have 
triangles with angles either equal to $0$ or $\pi$.
We choose a map from $\Sigma$ to a bouquet of two circles such that
$d_1$ and $d_2$ correspond to the generators of $F_2$. 
We denote those characters by $H(F_2)$ by an abuse of notation
since $h:F_2\ra \SOThr$ corresponds to $h':\pi_1(\Sigma)\ra \SOThr$ 
so that $h'=h\circ f_*$.

The map $\mathcal T_i:  i \times T^3 \ra H(F_2)$
where $i=  I, II, III$ or $IV$ is defined by sending 
the configuration and assigned angles to the equivalence 
class of the associated representation. That is, $\mathcal T_i$ are restrictions
of $\mathcal T$. We are interested in finding the equivalence relation 
to make these maps injective.

The map $\mathcal T_i$ can be constructed as follows for any region $I$, $II$, $III$, and $IV$: 
Let $v_0, v_1,$ and $v_2$ be the three points on a great circle or a segment
corresponding to the vertices of a triangle representing an element of  $I, II, III$, or $IV$
in the clockwise orientation and $\phi_0, \phi_1,$ and $\phi_2$ the associated pasting
angles. To begin with, suppose that none of the angles are zero or $\pi$ and that 
the vertices $v_0, v_1$, and $v_2$ in these cases are mutually distinct and do not coincide with an antipode of 
the other points. 
The vertices $v_0, v_1,$ and $v_2$ give us the orientation of 
the boundary of the generalized triangle. 
For example, in case I, 
we take one triangle $\tri_1$ in the outside of
the great circle and another one $\tri_2$ inside.
In all cases, the two triangles are oriented in the clockwise direction.
They have respective edges sharing a common great circle 
passing through $v_0$. 
We choose $\tri_1$ and $\tri_2$ in the following manner:
$\tri_1$ has two vertices $v_0$ and $v_1$ 
with angles $\pi- \phi_0/2$ and $\phi_1/2$ respectively. 
The remaining vertex $w_1$ has an angle $\eta_1$, $0\leq\eta_1\leq\pi$.
Since $\tri_1$ is outside, the vertices $v_1,v_0,$ and $w_1$ are in the clockwise direction in the boundary. 
$\tri_2$ has two vertices $v_0$ and $v_2$ 
with angles $\pi- \phi_0/2$ and $\phi_2/2$ 
respectively. The remaining vertex $w_2$ has an angle 
$\eta_2$, $0\leq\eta_2\leq\pi$. 
Since $\tri_2$ is inside, $v_2, v_0, w_2$ are in the clockwise direction in the boundary. 
(See Figure \ref{fig:mapF} in the case of region $I$.)

\begin{figure}

\centerline{\includegraphics[height=6cm]{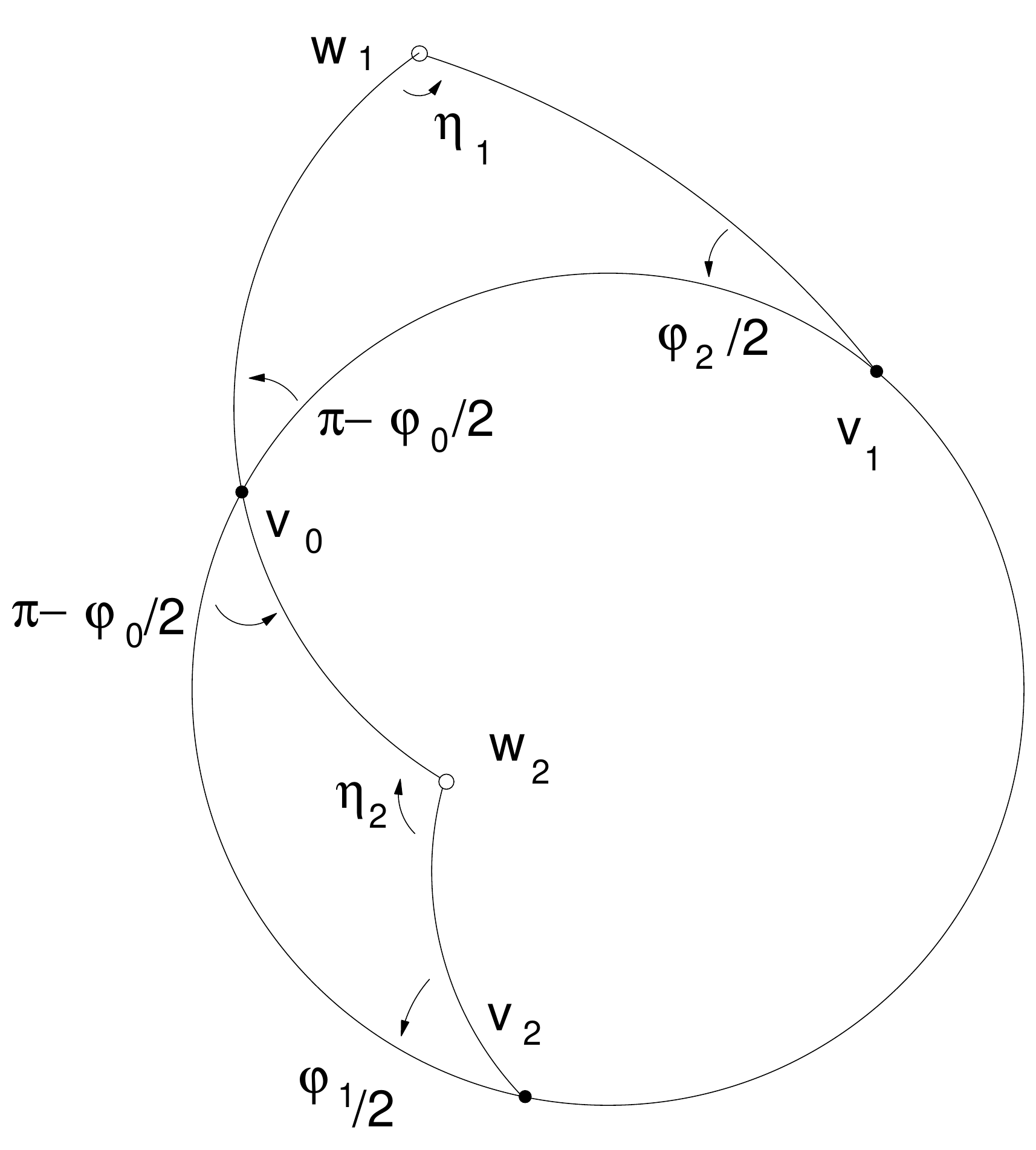}}
\caption{  Finding holonomy of $d_1$ and $d_2$ using triangles for region $I$.}
\label{fig:mapF}
\end{figure}

Recall that $R_{x,\theta}$ denote an isometry with fixed 
points $x, -x$ which rotates about the $x$ in the counter-clockwise direction of angle $\theta$. 

The isometry which is a composition of
$R_{v_0, 2\pi-\phi_0}$ and $R_{v_1, \phi_1}$ 
has a fixed point at $w_1$ and a rotation angle $2\pi-2\eta_1$.
Similarly, the isometry which is a composition of 
$R_{v_1,2\pi-\phi_0}$ and $R_{v_2,\phi_2}$ 
has a fixed point at $w_2$ and a rotation angle $2\pi-2\eta_2$, which follows 
from the consideration as in Section \ref{sec:rpp}.
In this way, we obtain two isometries corresponding to 
$d_1$ and $d_2$.

Let us illustrate more details here for region $I$ case only (The other regions 
are similarly treatable and the result follows from Proposition \ref{prop:multrule}): 
We still suppose that any two of 
$v_0, v_1,$ and $v_2$ are mutually distinct and do not coincide with an antipode of 
the other point:
If $\phi_0$ is zero, then we use $w_1$ to be $-v_1$ and 
$d_1$ is mapped to the rotation about $-v_1$ of angle $2\pi-\phi_1$ 
and $w_2$ is to be $-v_2$ and $d_2$ is a rotation about $-v_2$ of 
angle $2\pi-\phi_2$. (Here we are using lunes as degenerate triangles 
to compute these. )

If $\phi_0$ is $2\pi$ and others are not zero or $2\pi$, then we use $w_1$ to be $v_1$ and $w_2$ be $v_2$ 
and $d_1$ is a rotation about $v_1$ of angle $\phi_1$
and $d_2$ is a rotation about $v_2$ of angle $\phi_2$.

If $\phi_1$ is zero and others are not zero or $2\pi$, then we let $w_1$ be 
$v_0$ and $d_1$ is a rotation about $v_0$ of angle $2\pi-\phi_0$,
and the other triangle is constructed as above by angles 
$2\pi-\phi_0$ and $\phi_2$ and we obtain $d_2$ 
as in the above nondegenerate triangle $T_2$ case. 

If $\phi_1$ is $2\pi$ and others are not zero or $2\pi$, then we let $w_1$ be $-v_0$ and 
$d_1$ be a rotation about $-v_0$ of angle $\phi_0$
and we let $w_2$ and $d_2$ be determined as in the above nondegenerate $T_2$ case. 

If $\phi_2$ is zero and others are not zero or $2\pi$, then we let $w_2$ be 
$v_0$ and $d_2$ is a rotation about $v_0$ of angle $2\pi-\phi_0$, 
and the other information is analogously determined to as above for the nondegenerate $T_1$ case.

If $\phi_2$ is $2\pi$ and others are not zero or $2\pi$, then we let $w_2$ 
be $-v_0$ and $d_2$ be a rotation about $-v_0$ of angle $\phi_0$
and the other information is analogously determined as above for the nondegenerate $T_1$ case.

If $\phi_1=\phi_2=0$ and $\phi_0\ne 0, \pi$, then we let $w_1=w_2=v_0$ and 
$d_1$ and $d_2$ are respectively rotations of angle $2\pi-\phi_0$
at $v_0$.

If $\phi_1=0, \phi_2=2\pi$  and $\phi_0\ne 0, \pi$, then we let $w_1=v_0,w_2=-v_0$ 
and $d_1$ is a rotation about $v_0$ of angle $2\pi-\phi_0$  and $d_2$ is a rotation of angle $\phi_0$. 

If $\phi_1=2\pi, \phi_2=0$, then we let $w_1=-v_0,w_2=v_0$ 
and $d_1$ is a rotation of angle $\phi_0$ about $-v_0$ and $d_2$ is a rotation of angle 
$2\pi-\phi_0$ about $v_0$. 

Finally, suppose that some of  $v_0, v_1, v_2$ coincide or some of them coincide with an antipode of 
the other points. Then we have an abelian representation over one of the regions $a$, $b$, $c$, and $d$, 
which we studied above.




\begin{rem}\label{rem:contin} 
Note that our construction of $d_i$ and the fixed points $w_i$ depends 
smoothly unless $\eta_i=0,\pi$ on the data $v_i, \phi_i$. 
Hence, our construction easily extends 
to a construction on $\SUTw$ where the angles are restricted to be in $[0, 2\pi]$. 
However, to obtain $d_1$ and $d_2$ we have to multiply by $-I$ or to take the antipode to the result 
obtained from the triangle construction. (See Equation \eqref{eqn:cba2} and
Section \ref{subsec:gmult}.)
\end{rem}

\begin{prop}\label{prop:calFonto} 
The map \[\mathcal T: \bigcup_{i=I,II,III,IV} \tri_i \times T^3 \ra H(F_2)\] 
is onto. In fact, $\mathcal T_i$ is onto for each $i$.
\end{prop}
\begin{proof} 
Given a generalized triangle, 
we can control the distances between the vertices denoted by $v_1$ and $v_2$ and put any angles there.
By putting the angle at the vertex denoted by $v_0$ be $2\pi$, we can 
realize all segment with any given angles at vertices up to isometries.
\end{proof}

Under the $\{I, I_A, I_B, I_C\}$-action as in Equation \eqref{eqn:iabc4}, 
the regions $I, II, III, IV$ are actually identical.  
The maps ${\mathcal T}_i$s are also identical since we can easily show that $\mathcal T_i = \mathcal T_j \circ f$ holds 
for $f \in V$  sending $i$ to $j$. (That is, we show that they correspond to the same characters.)
Thus, we introduce the equivalence that element $x \in \tri_i\times T^3$ 
is equivalent to $y \in \tri_j \times T^3$ if $y=f(x)$ under $f\in V$ sending $\tri_i$ to $\tri_j$.

Finally, we introduce the equivalence relation $\sim$ so that two elements of $\tri_i\times T^3$ are equivalent 
if it maps to the same point of $H(F_2)$. Since $\tri_i$ is compact and $H(F_2)$ is a Hausdorff space, 
we obtain
\begin{prop}\label{prop:calFonto2} 
$\mathcal T_i$ induces a homeomorphism $\tri_i\times T^3/\sim \ra H(F_2) $ for each $i$, $i=I, II, III. IV$.
\end{prop}

\subsection{The equivalence relation on ${A}$,${A}'$, $B$, $B'$, $C$, $C'$}\label{subsec:ABC}

Next, we wish to consider the representations for each of which exactly one of $h(c_i)$ is identity 
and does not correspond to representations over the interior of regions $I$, $II$, $III$, and $IV$.
Define $U=I\cup II\cup III \cup IV$.
Then these representations correspond to ones over the regions 
${A}-U$,${A}'-U$,$B-U$, $B'-U$, $C-U$, and $C'-U$. 

Let us define $C^{{A}}, C^{{A}'}, C^{{B}}, C^{{B}'}, C^{{C}}$, and $C^{{C}'}$ as the subspace of 
${\mathcal C}_0$ that can be described as having the restricted $\pi(S_0)$ characters 
as described by the generalized triangles in regions ${A}$,${A}'$,$B,B',C$, and $C'$ respectively. 
In fact given a region $R$ in any of $A, B, C, A', B', C'$, we denote by $C^{R}$ the subset of 
${\mathcal C}_0$ coming from $R \times T^3$. 

We call the subspace over some intersection region, i.e., of form $X \cap y$ where 
$X$ is one of $A, A', B, B', C, C'$ and $y$ is one of $a, b, c, d$, 
the {\em abelian edges} of $X$. They correspond to abelian representations. 
In fact, every abelian representation in one of $C^{{A}}, C^{{A}'}, C^{{B}}, C^{{B}'}, C^{{C}}$, and $C^{{C}'}$
is from the abelian edges as we will see. We call them the {\em boundary abelian representations}. 

\subsubsection{Region $B - I - II$} \label{subsub:B-I-II}
We consider the region $B - I - II$ first and then we consider $B \cap I$ and $B\cap II$.
We choose an element of $(B-I-II)\times T^3$ and analyze the triangles and pasting maps:
Given a triangle here, we use the usual notations $v_i, \theta_i, l_i$s.
Then $\theta_0 = \pi$ and $\theta_1=\theta_2 \ne 0, \pi$ and $v_1$ and $v_2$ are antipodal
and $v_1$ and $v_2$ are vertices of a lune $L$ with angle $\theta_1=\theta_2$ 
at $v_1$ or $v_2$. 

This subsection is divided into two parts and the remaining subsections are similar:
(i) We will first show how to construct the representations over $B-I-II$ by pasting angles using 
geometric constructions. (ii) Next, we will find the equivalence relation on $(B-I -II)\times T^3$ that 
gives precisely the space of characters over $B-I -II$. 

(i) We start the construction process: 
We are given an element of $B-I-II$; hence, we choose a representative
generalized triangle with vertices $v_0, v_1, v_2$.
We now construct a representation $h$ so that $h(c_1)$ fixes $v_1$ and $h(c_2)$ fixes $v_2$
and has the right pasting angles given by $T^3$.  We use Proposition \ref{prop:multrule}. 

For this case, we will write the full details:

Suppose that $v_0$ is equal to $v_1$ or $v_2$. We choose $v_1$ and $v_2$ arbitrarily so that $v_2=-v_1$. 
In the former case, $h(d_1)$ is realized as a rotation about $v_1$ of angle $\phi_1-\phi_0$ 
and $h(d_2)$ as one about $v_2$ of angle $\phi_2+\phi_0$.
In the later case, $h(d_1)$ is realized as a rotation about $v_1$ of angle $\phi_1+\phi_0$ 
and $h(d_2)$ as one about $v_2$ of angle $\phi_2-\phi_0$. 
(These are abelian ones covered also in Section \ref{subsec:abcd}.)

Now suppose that $v_0$ is distinct from $v_1$ and $v_2$.
Suppose now that $0< \phi_0, \phi_1, \phi_2 < 2\pi$. 
To compute the holonomy of $h(d_1)$ and $h(d_2)$, let $l$ be the segment containing 
$v_0,v_1,v_2$ which is a side of the lune $L$. 
Then we draw a clockwise oriented triangle $T_1$ with vertices $v_0,v_1$ 
of angles $\phi_1/2$ and $\pi-\phi_0/2$ respectively. Then $h(d_1)$ is 
the rotation about the third vertex $w_1$ of $T_1$ of angle $\tau_1$
equal to $2\pi$ minus twice the vertex angle $\eta_1$. 

Similarly, we construct a clockwise oriented spherical triangle $T_2$ with 
vertices $v_0, v_2$ of angles $\pi-\phi_0/2$ and $\phi_2/2$ respectively. 
Then $h(d_2)$ is the rotation about the third vertex $w_2$ of $T_2$ of angle 
$\tau_2$ equal to $2\pi$ minus twice the vertex angle $\eta_2$. 

By allowing degenerate triangles, we can let $\phi_0 = 0$ or $2\pi$ and $0< \phi_1, \phi_2 < 2\pi$ as well, 
in which case, $w_1$ becomes $v_1$ or $v_2$ respectively and 
$T_1$ a pointed segment or a pointed lune.
Therefore, our representation is one which send 
$c_1$ and $c_2$ to rotations about $v_1$ and $v_2$ of respective rotation angles 
$\theta_1$ and $\theta_2$ and $d_1$ and $d_2$ to the rotations 
about $v_1$ and $v_2$. Thus, the characters are 
abelian ones and are over regions $a,b,c$ or $d$.
(Actually, these are only abelian ones if $v_0$ does not equal
$v_1$ or $v_2$.) The equivalence over the region were studied above
in Section \ref{subsec:abcd}.

Now see Proposition \ref{prop:multrule} when $\phi_1$ and $\phi_2$ equal $0$ or $\pi$. 


\begin{figure}

\centerline{\includegraphics[height=4cm]{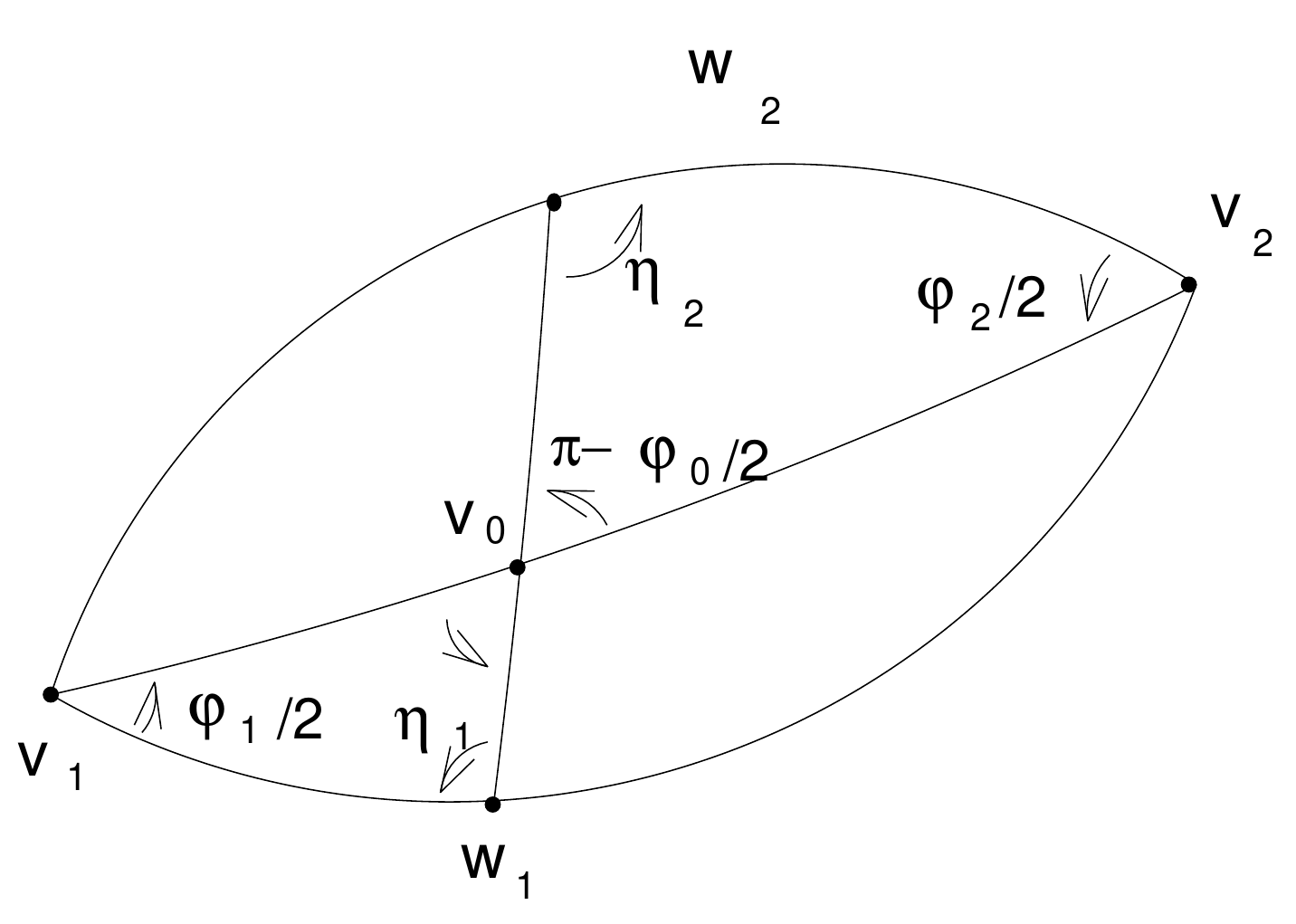}}
\caption{  Finding holonomy of $d_1$ and $d_2$ using triangles over $B-I -II$}
\label{fig:mapFB}
\end{figure}





For a later purpose, we prove: 
\begin{prop}\label{prop:nonabelianB}
\begin{itemize}
\item A representation $h$ corresponding to an element of $C^{B-I-II}$ is nonabelian 
if and only if both of $h(d_1)$ and $h(d_2)$ are not identity and moreover have 
no fixed points in $v_1$ and $v_2$ for any configuration corresponding to an element in $(B-I-II)\times T^3$ representing $h$. 
\item Suppose that $h$ is nonabelian and is in $C^{B-I-II}$. 
Let $\SI^1$ be the great circle passing through fixed points $\pm w_1$ of $h(d_1)$ and 
$\pm w_2$ of $h(d_2)$ or the great circle passing through $w_1$ at an angle $\eta_1=\pi-\tau_1/2$ for 
the angle $\tau_1$ of rotation of $h(d_1)$ when $\pm w_1=\pm w_2$. 
Then $\SI^1$ passes through none of $v_1$ and $v_2$.
\end{itemize}
\end{prop}
\begin{proof} 
To prove the first statement, suppose that $h$ represent an element of $C^{B-I-II}$. 
Since $h$ comes from a triangle construction, we choose an element of $B-I-II$ and a representing 
generalized triangle with vertices $v_0, v_1$, and $v_2$. 
If $v_0$ is not identical with $v_1$ and $v_2$ and $0< \phi_0 < 2\pi$, then 
$h(d_1)$ and $h(d_2)$ are both not identity and the fixed points $\pm w_1$ and $\pm w_2$ 
are distinct from $v_1, v_2$. Thus, $h$ is nonabelian. 

If $v_0$ is not identical with $v_1$ and $v_2$ and $\phi_0=0, 2\pi$. Then if $\phi_1$ 
is not zero or $2\pi$, then $w_1$ equals one of $v_1$ and $v_2$. 
If $\phi_1$ is zero or $2\pi$, then $h(d_1)=\Idd$. 
Since the same consideration holds for $h(d_2)$ as well, $h$ is abelian here.

If $v_0$ is identical with $v_1$ or $v_2$, then $h$ is abelian and $h(d_1)$ and $h(d_2)$ have fixed points 
at $v_1$ or $v_2$. 

This proves the first statement. 

Suppose that $\SI^1$ passes $v_1$ or $v_2$. 
Then it passes both. By construction, the intersection of 
$\SI^1$ meets the great segment with endpoints $v_1$ and $v_2$ at $v_0$. This implies $v_0 $ equals 
$v_1$ or $v_2$ or $\phi_0=0, \pi$ and hence, $h$ is abelian.  
\end{proof}

\begin{rem} \label{rem:L}
For purposes below, we look at the above construction from another angle:
We suppose that $0< \phi_0 < 2\pi$ and $v_0$ is distinct from $v_1$ and $v_2$
and $\pm w_1$ and $\pm w_2$ are distinct pairs.
$d_1$ and $d_2$ correspond to two pairs of fixed points 
$(w_1,-w_1)$ and $(w_2, -w_2)$ in $\SI^2$. 
We draw a great circle $l$ through $v_0$, $v_1$, and $v_2=-v_1$
separating $w_1$ and $-w_1$. The fact that
$0<\phi_0/2<\pi$ implies $\{w_1,-w_1 \}$ and $\{w_2, -w_2\}$ 
are distinct from $\{v_1, v_2\}$.
Let $s_0$ denote the segment of length $\pi$ with endpoints 
$v_1$ and $v_2$ through $v_0$ in $l$. 

We change a definition here temporarily: 
the {\em lune} simply means a disk bounded by two segments of 
length $\pi$ ending at a pair of antipodal points. 
If the vertex angle is $\leq \pi$, then the lune is convex. 

We choose two segments $s_1$ and $s_2$ with vertices $v_1$ and $v_2$ 
through $w_1$ and $w_2$ respectively.
Then we obtain the lune $L$ bounded by two of these bounded by 
a union of $s_1$ and $s_2$.
We constructed a lune $L$ containing $s_0$. 
(See Figures \ref{fig:mapFB} and \ref{fig:Bcircle}.)

We see that the segment $s'$ that is a union of two edges in $T_1$ 
and $T_2$ respectively through $v_0$ ends in $w_1$ and $w_2$ and is in $\SI^1$.
Let $v_0$ is the intersection point of $s_0$ and $s'$.
The boundary of $L$ and $s_0$ and $s'$ bound four triangles 
two of which are $T_1$ with vertices $v_1, v_0,$ and $w_1$ 
and $T_2$ with vertices $v_2, v_0$, and $w_2$.

Since $\phi_1/2$ and $\phi_2/2$ are angles at $v_1$ and $v_2$ 
of triangles $T_1$ and $T_2$ 
with vertices $w_1$ and $w_2$ respectively and 
$\phi_1/2+\phi_2/2$ is the angle of the lune $L$ 
containing both triangles $T_1$ 
and $T_2$ while agreeing with them in parts of edges. 

Notice that we can change $s_0$ to another segment from $v_1$ and $v_2$ in $L$. 
We can even move $s_0$ outside $L$. Then we use  a different choice of plus-minus signs among $\pm w_1$ and 
$\pm w_2$ so that the corresponding lune contains $s_0$ and we can construct 
the triangles as above to obtain the same isometries denoted by $h(d_1)$ and $h(d_2)$.
(See Figure \ref{fig:Bcircle}.)

Now, we suppose that $0< \phi_0 < 2\pi$ and $v_0$ is distinct from $v_1$ and $v_2$
and $\pm w_1 = \pm w_2$.
In this case, we let the great circle $\SI^1$ be the great circle through them 
that has angle $\eta_1$ of $w_1$ with the segment from $w_1$ to $v_1$.
Then $v_0$ is the intersection $s_0\cap \SI^1$. We can choose two 
triangles $T_1$ and $T_2$ as above. 

When $s_0$ is chosen to pass one of $\pm w_1$, we have $\phi_1=\phi_2=0$ and 
the angles $\eta_i$ of $w_i$ satisfy $\eta_1= \eta_2$ and $\eta_i=\pi-\tau_i/2$ for
the rotation angles $\tau_i$ of $h(d_i)$ for each $i=1,2$.
If $s_0$ does not pass one of $\pm w_1$, then $w_2=-w_1$ and 
$\phi_1=\pi-\phi_2$ and $\eta_1= \pi-\eta_2$ also. 
Here, a hemisphere is divided into four triangles. Two of which are
$T_1$ and $T_2$. 

Again, we can move $s_0$ outside the hemisphere and we will 
be choosing a different pair from $\pm w_1, \pm w_2$. 
\end{rem}

We remark that 
the isometry equivalence class of the set of 
the six points $\pm w_1, \pm w_2, v_1$, and $v_2$ and the respective rotation angles of $h(d_1)$ and $h(d_2)$,
$\tau_1$ and $\tau_2$, and $\theta_1=\theta_2$ completely 
determines the conjugacy class of the representations over ${B-I -II}$ 
and vice versa. 
The isometry class space is homeomorphic to the conjugacy class of the characters over ${B-I -II}$.
This will be used later in determining the topology of the space.

(ii) We introduce the equivalence relation on $({B-I -II})\times \SI^1 \times \SI^1\times \SI^1$ by 
demanding the ${\mathcal T}_{B-I -II}$ be injective. 

To see more details, we will introduce it as
equivalence relations in a geometric form. 
The equivalence relation is so that two are equivalent if they correspond 
to the same $\SOThr$-character of $\pi_1(\Sigma)$ with condition 
that the configurations are on $B-I-II$.

By Proposition \ref{prop:nonabelianB}, 
the subspace of $B \times T^3$ giving us abelian characters is precisely
\begin{gather}\label{eqn:abelianB}
(v(1), 0, \phi_0, \phi_1, \phi_2), (v(1), \pi, \phi_0, \phi_1, \phi_2), \nonumber \\
(v(1), l(1), 0, \phi_1, \phi_2),\nonumber \\
(0, l(1), \phi_0, \phi_1, 2\pi-\phi_1), (\pi, l(1), \phi_0, \phi_1, 2\pi-\phi_1) \nonumber \\ 
\mbox{ for } v(1), l(1) \in [0, \pi], \phi_0, \phi_1, \phi_2 \in [0, 2\pi] .
\end{gather}

For the region $B \times T^3$, we first introduce 
an obvious equivalence relation given by
\begin{equation} \label{eqn:obvious}
\begin{split}
& (v(1), 0, \phi_0, \phi_1, \phi_2)  \sim   (v(1), 0, \phi_0-t, \phi_1+t, \phi_2-t), \mbox{ for every } t \in \SI^1 \\
& (v(1), \pi, \phi_0, \phi_1, \phi_2)  \sim   (v(1), \pi, \phi_0-t, \phi_1-t, \phi_2+t)  \mbox{ for every } t \in \SI^1  \\
& (v(1), 0, \phi_0, \phi_1, \phi_2)  \sim   (v(1), \pi, 2\pi-\phi_0, \phi_1, \phi_2) \\
& (v(1), l(1), 0, \phi_1, \phi_2)  \sim   (v(1), l(1)', 0, \phi_1, \phi_2)  \\
& \mbox{ for an arbitrary pair of $l(1)$ and $l(1)'$ and } \\
& (0, l(1), \phi_0, \phi_1, 2\pi-\phi_1)  \sim   (0, l(1)', \phi'_0, \phi'_1, 2\pi-\phi'_1) \\
& \mbox{ for an arbitrary pair of $l(1)$ and $l(1)'$ and  the same third angles }  \\ 
& (\pi, l(1), \phi_0, \phi_1, 2\pi-\phi_1)  \sim   (\pi, l(1)', \phi'_0, \phi'_1, 2\pi-\phi'_1) \\ 
& \mbox{ for an arbitrary pair of $l(1)$ and $l(1)'$ and the same third angles. } 
\end{split}
\end{equation}

Here by the same third angle, we mean that the triangle with a side $l(1)$ and angles $\pi-\phi_0/2, \pi-\phi_1/2$ at the two 
vertices and the triangle with a side $l(1)'$ and angles $\pi-\phi'_0/2, \pi-\phi'_1/2$ has the same third angle. 
This case of course includes the generalized triangle cases, which might be ambiguous but not the third angle itself. 
Or we can use the equation from spherical trigonometry: 
\begin{multline} \label{eqn:thirdangle}
-\cos (\pi-\phi_0/2) \cos (\pi-\phi_1/2)+ \sin (\pi-\phi_0/2)\sin(\pi-\phi_1/2)\cos l(1) = \\
 -\cos (\pi-\phi'_0/2) \cos (\pi-\phi'_1/2)+ \sin (\pi-\phi'_0/2)\sin(\pi-\phi'_1/2)\cos l'(1).
\end{multline}
Note that we include the degenerate triangles and in that case the equation is not an enough condition
and the condition is a geometric one.

Note  that these elements are all of the ones in $B \times T^3$ corresponding to abelian characters 
and the the complete equivalence relations for abelian characters on $B$.
The first two are really restrictions of equivalence relations on $b \times T^3$ and $d\times T^3$, i.e., the 
$\SI^1$-action there. ( Finally, abelian characters are always equivalent to abelian edge ones
by the final three items above.)

There is also an $\SI^1$-action giving us further equivalences on the subset of $B\times T^3$ corresponding
to nonabelian characters: 
Let $v_0, v_1, v_2$, $w_1$ and $w_2$ be obtained for a nonabelian representation $h$ coming from $B \times T^3$.
By Proposition \ref{prop:nonabelianB}, 
$w_1$ and $w_2$ do not coincide with $v_1$ or $v_2$
and there exists a great circle $\SI^1$ containing $w_1$ and $w_2$ determined in Remark \ref{rem:L}
and $\SI^1$ contains none of $v_1, v_2$.
Then there exists $v_0, v_1$, and $v_2$ where $v_0$ is not identical with $v_1$ or $v_2$.
For any choice of $v_0$ in a great circle $\SI^1$ containing $w_1$ and $w_2$, 
we choose a corresponding 
$s_0$ between $v_1$ and $v_2$ in the collection of all great segments between $v_1$ and $v_2$. 
This gives us a generalized triangle with vertices
$v_0, v_1$, and $w_1$ 
and possibly $\geq \pi$ angle at $v_1$.
(See equation \eqref{eqn:cba2}. The generalized 
triangle with one vertex of angle in $[2\pi, 4\pi]$ and the other two angles 
in $[0, 2\pi]$ also needs to be considered here in multiplications by geometry. ) 
 
If the angle at $v_2$ is $\geq \pi$, we subtract $\pi$ from it. This 
amounts to replacing $w_2$ with its antipode and the angle $\eta_2$ to be $\pi-\eta_2$
and thus $h(d_2)$ is unchanged. (See Figure \ref{fig:Bcircle}.
The generalized triangle $v_1v_0w_2$ is replaced by 
$v_1v_0(-w_2)$.)
Now the angle of the triangle at $v_2$ is in $[0, \pi)$. 
The angle of the triangle at $v_0$ is always in $[0, \pi)$, and our operation does not change 
the rotation angle of the pasting map $P_0$ at $v_0$. (This is described also in the last part of 
Remark \ref{rem:L}.)

\begin{figure}

\centerline{\includegraphics[height=7cm]{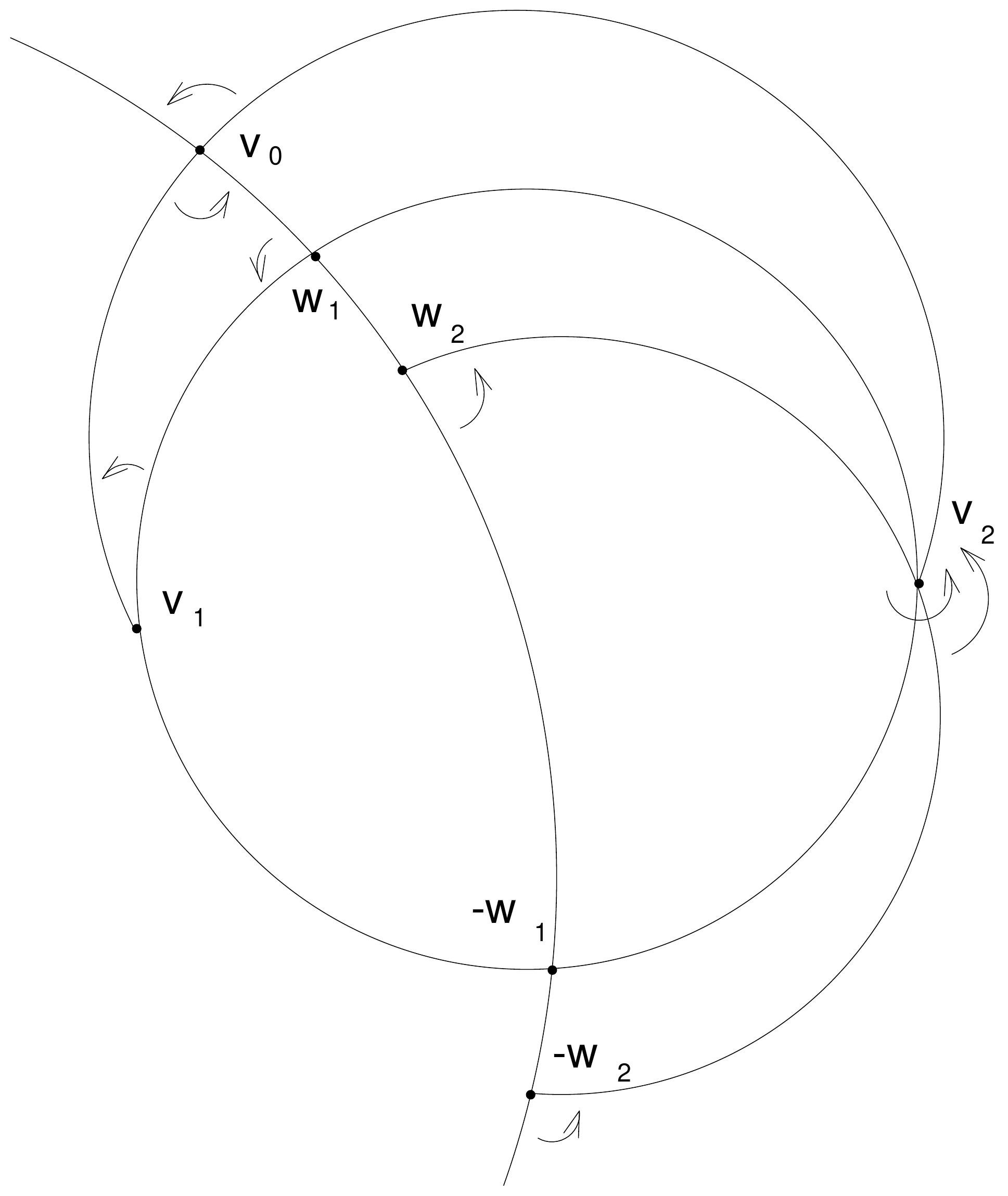}}
\caption{  The circle action.}
\label{fig:Bcircle}
\end{figure}

Similarly, we can repeat the constructions for $v_0,v_2,w_2$.
Therefore, there is an $\SI^1$-action on $(B-I-II)\times T^3$ 
by choosing one among the space of great segments, denoted by $s_0$, 
between $v_1$ and $v_2$ forming $\SI^1$. 

(Actually the choice of fixed point for $h(d_i), i=1,2$ among an antipodal pair 
gives possibilities for $s_0$ and conversely the choice of $s_0$ forces choices of 
fixed points of $h(d_i), i=1,2$. )

Suppose now that we are over the abelian edges $(B-I-II) \cap b$ or $(B-I-II) \cap d$. 
For boundary abelian characters, the angles at points denoted by $w_1$ and 
$w_2$ matter only and the equivalence relations are clear. 
This can be seen as an $\SI^1$-action on the abelian edges given by changing pasting angles
\begin{eqnarray}\label{eqn:S1abB}
(\phi_0, \phi_1, \phi_2) &\mapsto& (\phi_0 -t, \phi_1+t, \phi_2-t) \mbox{ if } l(1)=0 \\
 &\mapsto& (\phi_0 -t, \phi_1+t, \phi_2+t) \mbox{ if } l(1)=\pi \mbox{ for } t\in [0, 2\pi).
 \end{eqnarray} 
This is a continuous extension of the above $\SI^1$-action and is a restriction of 
the $\SI^1$-action on the subspace of abelian characters in Section \ref{subsec:abcd}. 

By the circle action, we can always make our situation so that $\phi_1=0$ or $\phi_2=0$. 
We call these the {\em canonical representatives}, and every representation or character over $B$ has
canonical representatives.
Then the map
\[{ \mathcal T}_{{{B-I -II}}}^1:({B-I -II})\times \SI^1\times \{0\}\times \SI^1 /\sim \ra C^{B-I -II} \hbox{ and }\] 
\[{ \mathcal T}_{{{B-I -II}}}^2:({B-I -II})\times \SI^1\times \SI^1 \times \{0\} /\sim \ra C^{B-I -II}\] 
are easily seen to be surjective.

The two maps are  generically  $4$ to $1$ since there is a Klein four-group 
action on $({B-I -II})\times T^3$ given as follows and which acts on
the domains of the above two maps:

(a) There is a $\bZ_2$-action on the initial spaces given by changing $s_0$ to $-s_0$
which goes to the trivial action in the second space. (This is a transformation in the $\SI^1$-action.)

(b) This is a transformation $I_{{B}}$ on $(B-I-II)\times T^3$ induced by sending the generalized triangle with 
vertices $v_0, v_1, v_2$ to another generalized triangle by sending $v_1$ and $v_2$ to 
their antipodal points and $v_0$ to itself. 
This sends parameters in the following way:
\begin{eqnarray} \label{eqn:ib}
(\theta_1, l(1),\phi_0,\phi_1,\phi_2) & \mapsto & (\pi-\theta_1,\pi-l(1), \phi_0, 2\pi-\phi_1, 2\pi-\phi_2) 
\end{eqnarray}

Note that $\{\Idd, I_B\}$ acts on the union of fibers above mid-tie but sends other ties in $B-I-II$ to different ones.
$I_A$ and $I_C$ do not preserve $B$ and hence do not appear in equivalences
for $B$ only.
For end ties, we have further equivalences introduced by mapping into the regions above $I,II,III$, or $IV$.

\begin{prop}\label{prop:B}
The equivalence relation $\sim$ generated by all of above, i.e, by equations \eqref{eqn:obvious} 
and equation \eqref{eqn:ib} and the $\SI^1$-action, gives us 
a homeomorphism:
\[{\mathcal T}_{{{B-I -II}}}:({B-I -II})\times \SI^1\times \SI^1\times \SI^1 /\sim \ra C^{B-I -II}.\] 
Furthermore, 
\[{\mathcal T}_{{{B-I -II}}}^1:({B-I -II})\times \SI^1\times \{0\}\times \SI^1 /\sim \ra C^{B-I -II} \hbox{ and }\] 
\[{\mathcal T}_{{{B-I -II}}}^2:({B-I -II})\times \SI^1\times \SI^1 \times \{0\} /\sim \ra C^{B-I -II}\] 
are homeomorphisms where $\sim$ is given by the equation \eqref{eqn:obvious} and 
the Klein four-group action above denoted {\rm (a)} and {\rm (b)}.
\end{prop} 
\begin{proof}
The surjectivity follows by definition of $C^{B-I-II}$.

The injectivity will follow if we show that two elements of $(B-I-II) \times T^3$ map to a same element, 
then they are equivalent. For abelian representations, this is easy to show from equation \eqref{eqn:obvious}
and the $\SI^1$-action.
For a nonabelian representation $h$, we have a great circle containing fixed points of $h(d_1)$ and 
$h(d_2)$ with nonzero angles by Proposition \ref{prop:nonabelianB}.
We make both of the elements have angles $0$ at $v_1$ by the $\SI^1$-action. 
Then $w_1= v_0$ in both cases. Then we see that the two give the same isometry class of fixed points 
and rotation angles if and only if they are the same or differ by taking two of $w_i$ to be their antipodes respectively, 
which amounts to changing $s_0$ to $-s_0$ and the move described by equation \eqref{eqn:ib}.

We prove that the map is a homeomorphism: 
Elements of $C^{B-I-II}$ can be grouped into a family of ties $C_{\theta}$ given by
$\theta_1=\theta_2=\theta$
where $0 < \theta < \pi$, i.e., images of the ties. Note that $C_{\theta}=C_{\pi-\theta}$. 
For each tie in $B-I-II$ where $\theta_1=\theta_2=\theta'$ for a constant $\theta'$, 
we obtain a subset $C'_{\theta'}$ in $({B-I -II})\times \SI^1\times \SI^1\times \SI^1 /\sim$
where $C'_{\theta'}=C'_{\pi-\theta'}$. 
We have a map from the compact set $C'_{\theta'}$ to $C_{\theta'}$ which is injective by the above paragraph 
and is surjective by definition.
The map ${\mathcal T}_{B-I-II}$ is a proper map by the above property of preserving $\theta'$. 
By the above injectivity and the Hausdorff property of $C^{B-I-II}$, we have a homeomorphism.

The proofs of the last two statements are from the fact that the $\SI^1$-action and the $\bZ_2$-action restrict to
an action of Klein four-group action in the two spaces. 
\end{proof}


\subsubsection{Regions $B_\pi$ and $B_0$} \label{subsec:Bpi}

Let $B_{\pi}$ be the subset of $B$ where $\theta(1)=\theta(2)=\pi$, i.e., $B_\pi= B \cap I$,
and $B_0$ the subset of $B$ where $\theta(1)=\theta(2)=0$, i.e., $B_0=B \cap II$. 

We consider the restriction ${\mathcal T}_{B_\pi}|B_\pi \times \SI^1 \times \SI^1 \times \SI^1$.
Here the characters are in $H(F_2)$ since $h(c_i)$ are all identity. 

By considering the choice of $s_0$, we have an $\SI^1$-action, which was 
described above for the region $B$ above.
A representation corresponds to the two fixed points of $h(d_1)$ and $h(d_2)$ 
and their rotation angles. 

There is again a Klein four-group action: 
(a) For each map, there is a $\bZ_2$-action on $B_\pi \times \SI^1 \times \SI^1 \times \SI^1$
given by changing $s_0$ to $-s_0$ which goes to trivial action in the $\SOThr$-character space. 
(This is a transformation in the $\SI^1$-action.)

(b) Again there is a $\bZ_2$-action on 
$B_\pi\times \SI^1\times \SI^1\times \SI^1$
by taking the fixed points $v_1$ and $v_2$ to the antipodal ones respectively. 
\begin{eqnarray} 
(\pi, l(1),\phi_0,\phi_1,\phi_2) &\mapsto & (\pi, \pi-l(1),\phi_0, 2\pi-\phi_1, 2\pi-\phi_2) \label{eqn:ib2} 
\end{eqnarray}
We call the map an {\em augmented $I_B$-action} and the two generate
a Klein four group action. (Compare to equation \eqref{eqn:ib}.)

As above, $\sim$ now defines a {\em restricted equivalence relation} on $B_\pi\times \SI^1\times \SI^1\times \SI^1$
given by equation \eqref{eqn:obvious} and the $\SI^1$-action and the above $\bZ_2$-action.

We  define similar equivalence relations for $B_0$ as well, also called 
the {\em restricted equivalence relation} with a $\bZ_2$-action
including the {\em augmented $I_B$-action}:
\begin{eqnarray} 
(0, l(1),\phi_0,\phi_1,\phi_2) &\mapsto & (0, \pi-l(1),\phi_0, 2\pi-\phi_1, 2\pi-\phi_2) \label{eqn:ib3}
\end{eqnarray}


Recall that $B_\pi$ and $B_0$ are intervals homeomorphic to $I=[0,\pi]$.
\begin{prop}\label{prop:Bpi}
We can identify $B_\pi\times \SI^1  \times \SI^1  \times \SI^1/\sim$ with $H(F_2)$ 
by sending elements to the fixed points of $h(d_1)$ and $h(d_2)$ respectively and 
rotation angles and their distances and $\sim$ is given by
equations \eqref{eqn:obvious} and \eqref{eqn:ib2} and the $\SI^1$-action.
Furthermore,  we can identify
$B_\pi\times \SI^1  \times \{0\}  \times \SI^1/\sim$ with $H(F_2)$ 
and similarly
$B_\pi\times \SI^1  \times \SI^1\times \{0\}/\sim$ with $H(F_2)$
where $\sim$ is given by equation \eqref{eqn:obvious} and the Klein four-group action.
The same statements hold with $B_\pi$ replaced with $B_0$. 
\end{prop}
\begin{proof} 
The existence and the continuity and the surjectivity of the map from $B_\pi \times T^3$ 
can be proved by showing every elements of $H(F_2)$ can be 
constructed by triangle constructions.

The injectivity is proved as follows: Suppose first that we have a representation $h$ for 
an element of $H(F_2)$ and $h$ is not abelian so that $h(d_1)$ and $h(d_2)$ are both not $\Idd$
by Proposition \ref{prop:nonabelianB}. Suppose that some element of $B_\pi\times T^3$ corresponds to $h$. 
Let $v_0, v_1, v_2, w_1, w_2$ denote the vertices of the corresponding lune and the respective fixed points of 
$h(d_1)$ and $h(d_2)$. 
Moreover $\pm w_1 \ne \pm w_2$ since otherwise $h$ is abelian. 
The fixed points $\pm w_1$ of $h(d_1)$ and $\pm w_2$ of $h(d_2)$ are well-defined, 
we have four choices by choosing one from each. 
Then the rotation angle at $w_1$ determines a great circle where $v_1$ lies and 
the rotation angle at $w_2$ determine a great circle where $v_2=-v_1$ lies. 
This determines $v_1$ and $v_2$. Once $v_1$ and $v_2$ are determined, then $v_0$ is determined 
to be $w_1$ if we choose a configuration where $\phi_1=0$ . This implies that the Klein four-group action on the spaces of 
quadruples and their angles give the only ambiguities. These are covered by our equivalences 
based on the $\SI^1$-action and the $\bZ_2$-action. 

For the space of abelian characters, the injectivity is followed  by equation \eqref{eqn:obvious}.
Since  $B_\pi \times T^3/\sim$ is compact, it follows that our map is a homeomorphism. 

The final statements on $B_\pi\times \SI^1  \times \SI^1\times \{0\}/\sim$ and
$B_\pi\times \SI^1  \times \SI^1\times \{0\}/\sim$ are proved using the fact that the $\SI^1$-action 
restricts to a Klein four-group action. 

\end{proof}

\subsection{Region $B'-III-IV$}
First, we describe how to obtain the corresponding character from 
an element of $(B'-III-IV) \times T^3$.
Over the region $B'-III-IV$, the triangle is a segment $s$ with a vertex $v_0$ and 
another vertex $v_1=v_2$. Drawing segments $s_1$ and $s_2$ of length $\pi$ from 
$v_1$ to $-v_1$ of angles $\phi_1/2$ and $\phi_2/2$
respectively and drawing a segment $s_0$ from $v_0$ of angle $\pi-\phi_0/2$ to the segment $s$, 
we obtain two triangles $T_1$ and $T_2$ sharing an edge $s$. 
The respective vertices $w_1$ and $w_2$ of $T_1$ and $T_2$ 
are fixed points of $h(d_1)$ and $h(d_2)$ and the respective vertex 
angles are $\pi$ minus the half of those of $h(d_1)$ and $h(d_2)$. 

The configuration space that we have to consider is given by the point $v_1=v_2$ with 
the rotation angles for $h(c_1)$ and $h(c_2)$ and 
the fixed points $w_1$ and $w_2$ and their rotation angles for 
$h(d_1)$ and $h(d_2)$ respectively. 
Again, we will find the topology of the configuration space 
and hence $C^{{B}'-III-IV}$ in Section \ref{sec:topC0}.

Let us discuss the equivalence relation to make $(B'-III-IV)\times T^3$ equal to 
$C^{B'-III-IV}$:

By a similar reason to Proposition \ref{prop:nonabelianB}, 
the subspace of $B' \times T^3$ giving us abelian characters is precisely
\begin{gather}\label{eqn:abelianBp}
(v(1), 0, \phi_0, \phi_1, \phi_2), (v(1), \pi, \phi_0, \phi_1, \phi_2), (v(1), l(1), 0, \phi_1, \phi_2), \nonumber\\
(0, l(1), \phi_0, \phi_1, \phi_1), (\pi, l(1), \phi_0, \phi_1, \phi_1)  \nonumber \\
\mbox{ for } v(1), l(1) \in [0, \pi], \phi_0, \phi_1, \phi_2 \in [0, 2\pi] 
\end{gather}

For the region ${B'} \times T^3$, we introduce 
the equivalence relation given by
\begin{equation}\label{eqn:obvious2}
\begin{split}
& (v(1), 0, \phi_0, \phi_1, \phi_2)  \sim   (v(1), 0, \phi_0+t, \phi_1+t, \phi_2+t), \mbox{ for every } t \in \SI^1  \\
& (v(1), \pi, \phi_0, \phi_1, \phi_2)  \sim   (v(1), \pi, \phi_0+t, \phi_1-t, \phi_2-t)  \mbox{ for every } t \in \SI^1  \\
& (v(1), 0, \phi_0, \phi_1, \phi_2) \sim    (v(1), \pi, 2\pi-\phi_0, \phi_1, \phi_2). \\
& (v(1), l(1), 0, \phi_1, \phi_2)  \sim   (v(1), l(1)', 0, \phi_1, \phi_2) \\
& \mbox{ for an arbitrary pair of $l(1)$ and $l(1)'$ and } \\
& (0, l(1), \phi_0, \phi_1, \phi_1)  \sim   (0, l(1)', \phi'_0, \phi'_1, \phi'_1)  \\
& \mbox{ for an arbitrary pair of $l_1$ and $l'_1$ and the same third angles}  \\ 
& (\pi, l(1), \phi_0, \phi_1, \phi_1)  \sim   (\pi, l(1)', \phi_0, \phi_1, \phi_1) \\
& \mbox{ for an arbitrary pair of $l(1)$ and $l(1)'$ and the same third angles.}
\end{split}
\end{equation}
They correspond to abelian characters. 
The same third angle condition means the triangle angle condition as indicated 
by the equation \eqref{eqn:thirdangle} with $\pi-\phi_1/2$ replaced by $\phi_1/2$ and
$\pi-\phi'_1/2$ replaced by $\phi'_1/2$

From now on, given an element of $(B'-III-IV) \times T^3$ and the corresponding representation $h$, 
we assume that the fixed points of $h(d_i)$ are well-defined and 
the angles are nonzero for $i=1,2$ since otherwise the above equivalence relation is enough. 
As above, we can choose $s_1=s_2$ differently not changing $v_1=v_2$ 
but changing $v_0$ on the great circle $\SI^1$ containing $w_1$ and $w_2$
and this changes $\phi_i$s but 
the representation itself does not change. We may have to consider 
a generalized triangle with one angle $\geq \pi$, where we have to subtract by $\pi$. 
Here $\SI^1$ is determined as in the case B. 
This gives us an $\SI^1$-action, which we will not describe in detail. 
For each abelian edge region, the $\SI^1$-action is defined similarly to equation \eqref{eqn:S1abB}
extending the $\SI^1$-action continuously 
and is a restriction of the $\SI^1$-action in Section \ref{subsec:abcd}.

We also have a notion of {\em canonical representatives} 
for a representation or character over $B'$
by making $\phi_1=0$ or $\phi_2=0$
using the $\SI^1$-action
so that the restriction to $({B'-III-IV})\times \SI^1 \times \{0\} \times \SI^1$ or 
$({B'-III-IV})\times \SI^1 \times \SI^1 \times \{0\}$ are still onto.

We also describe a Klein four-group action: The first one is given by choosing $s_1=s_2$ to 
be $-s_1=-s_2$ or equivalently letting $v_0$ be its antipode $-v_0$. 

Choosing $v_1=v_2$ to be its 
antipode gives us $I_{{B}}$-transformation as above, which is a $\bZ_2$-action as well. 
The formulas are exactly the same.  (See equations \eqref{eqn:ib}.)
The above equation \eqref{eqn:obvious2} and this $\SI^1$-action and 
the $I_{{B}}$-transformation gives us the needed equivalence relation $\sim$. 


Thus, there is an onto map 
${\mathcal T}_{{B}'-III-IV}: ({B'-III-IV}) \times \SI^1 \times \SI^1\times \SI^1 \ra C^{{B'-III-IV}}.$
(Actually, ${B'-III-IV}$ is equivalent to $B-I-II$ under $I_{{A}}$ but we will not need 
this fact explicitly for the equivalences on ${B'-III-IV}$ itself.)


\begin{prop}\label{prop:B'} 
There is a well-defined induced homeomorphism
${\mathcal T}_{{B}'-III-IV}: ({{B'-III-IV}})\times \SI^1\times \SI^1 \times \SI^1/\sim \ra C^{{B'-III-IV}}$. 
Furthermore, the restricted maps
\[{ F}_{{{B-III -IV}}}^1:({B'-III -IV})\times \SI^1\times \{0\}\times \SI^1 /\sim \ra C^{B'-III -IV} \hbox{ and } \] 
\[{ F}_{{{B-III -IV}}}^2:({B'-III -IV})\times \SI^1\times \SI^1 \times \{0\} /\sim \ra C^{B'-III -IV}\] 
are homeomorphisms where $\sim$ is given by the Klein four-group action above. 

\end{prop}

\subsubsection{Regions $B'_\pi$ and $B'_0$} \label{subsec:B'pi}

Let $B'_{\pi}$ be the subset of $B$ where $\theta(1)=\pi$, i.e., $B'_\pi= B \cap IV$,
and $B'_0$ the subset of $B$ where $\theta(1)=0$, i.e., $B'_0=B \cap III$. 
The restricted equivalence relation is given in $B'_\pi \times T^3$ by the ones given by equation \eqref{eqn:obvious2} 
and the circle action and the $\bZ_2$-action of form \eqref{eqn:ib2} and \eqref{eqn:ib3}.  
We can define a restricted equivalence relation on $B'_0\times T^3$ the same way. 

Similarly, we have: 

\begin{prop}\label{prop:B'pi}
\[{\mathcal T}_{B'_\pi}:B'_\pi\times \SI^1\times \SI^1 \times \SI^1/\sim \ra H(F_2)\] 
is a homeomorphism where $\sim$ denotes the restricted equivalence relation given in $B' \times T^3$.
\[{\mathcal T}_{B'_0}:B'_0\times \SI^1\times \SI^1 \times \SI^1/\sim \ra H(F_2)\] 
is a homeomorphism where $\sim$ denotes the restricted equivalence relation given in $B' \times T^3$.
Furthermore,  we can identify
$B'_\pi\times \SI^1  \times \{0\}  \times \SI^1/\sim$ with $H(F_2)$ 
and similarly
$B'_\pi\times \SI^1  \times \SI^1\times \{0\}/\sim$ with $H(F_2)$
where $\sim$ is given by equation \eqref{eqn:obvious2} and the Klein four-group action.
The same statements hold with $B_\pi$ replaced with $B_0$. 
\end{prop}



\subsection{Region ${A}-I-IV$}

Let us now consider $C^{{A-I-IV}}$ and $C^{{A}'-II-III}$. 
(Again, they are equivalent under $I_{{B}}$ but we will not need this fact. )

We consider ${A}-I-IV$ first. Here the corresponding generalized triangle  is
a lune with vertices $v_0$ and $v_1=-v_0$ and the vertex $v_2$ is on an edge. 

Let $v_0$ and $v_1$ be an arbitrary pair of antipodal points. 
First, we discuss how to construct a representation $h$ so 
that $h(c_i)$ fix $v_i$ for $i=0,1$ and has the right pasting angles 
for an element of $C^{A-I-IV}$. 

To obtain the transformations $h(d_1)$ and $h(d_2)$, we consider 
a triangle $T_2$ with vertices $v_2, v_0$ and $w_2$ and with 
angles $\phi_2/2$ and $\pi-\phi_0/2$ at $v_2$ and $v_0$ respectively. 
The vertex $w_2$ is the fixed point of $h(d_2)$ and the vertex angle $\eta_2$ is 
the half of $2\pi$ minus the rotation angle $\tau_2$ of $h(d_2)$, and
$h(d_1)$ is a rotation about $v_1$ with angle $\phi_1+\phi_0$ and 
$w_1=v_1$. Since we defined $h(c_i)$ and $h(d_j)$ for $i=1,2$ and $j=1,2$, we are done.

Let us discuss the equivalence relation:

First, the following is obvious:
\begin{prop}\label{prop:nonabelianA}
A representation $h$ corresponding to an element of $C^{A-I-IV}$ is nonabelian 
if and only if $h(d_2)$ is not identity and does not fix $v_0$ and $v_1$ for any configuration corresponding 
to an element of $(A-I-IV) \times T^3$ representing $h$. 
\end{prop}

By Proposition \ref{prop:nonabelianA}, 
the subset of $A\times T^3$ corresponding to the abelian representations is 
\begin{gather}\label{eqn:abelianBpp}
(v(0), 0, \phi_0, \phi_1, \phi_2), (v(0), \pi, \phi_0, \phi_1, \phi_2), 
(v(0), l(0), \phi_0, \phi_1, 0), \nonumber\\
(0, l(0), \phi_0, 2\pi-\phi_0, \phi_2), (\pi, l(0), \phi_0, 2\pi-\phi_0, \phi_2) \nonumber \\
\mbox{ for } v(0), l(0) \in [0, \pi], \phi_0, \phi_1, \phi_2 \in [0, 2\pi] 
\end{gather}

For the region $A \times T^3$, we first introduce 
the equivalence relation given by
\begin{equation} \label{eqn:obviousA}
\begin{split}
& (v(0), 0, \phi_0, \phi_1, \phi_2) \sim (v(0), 0, \phi_0+t, \phi_1-t, \phi_2-t) \\
& (v(0), \pi, \phi_0, \phi_1, \phi_2) \sim (v(0), \pi, \phi_0+t, \phi_1-t, \phi_2+t) \\
& (v(0), 0, \phi_0, \phi_1, \phi_2) \sim (v(0), \pi, \phi_0, \phi_1, 2\pi-\phi_2)\\
& (v(0), l(0), \phi_0, \phi_1, 0) \sim  (v(0), l(0)', \phi_0, \phi_1, 0) \\
& \mbox{ for an arbitrary pair of $l(0)$ and $l(0)'$ and } \\
& (0, l(0), \phi_0, 2\pi-\phi_0, \phi_2) \sim  (0, l(0)', \phi'_0, 2\pi-\phi'_0, \phi'_2) \\
& \mbox{ for an arbitrary pair of $l(0)$ and $l(0)'$ and the same third angles}  \\
& (\pi, l(0), \phi_0, 2\pi-\phi_0, \phi_2)  \sim   (\pi, l(0)', \phi'_0, 2\pi-\phi'_0, \phi'_2) \\
& \mbox{ for an arbitrary pair of $l(0)$ and $l(0)'$ and the same third angles. }
\end{split}
\end{equation}
By the same third angles, we mean that the third angle of the triangle with a side of length $l(1)=\pi-l(0)$ and angles $\pi-\phi_0/2$ and $\phi_2/2$ at 
the end is same as the third angle of the triangle with a side of length $l(1)'=\pi-l(0)'$ and the angles 
$\pi-\phi'_0/2$ and $\phi'_2/2$ at the end. 
\begin{multline}\label{eqn:thirdangleA}
-\cos (\pi-\phi_0/2) \cos \phi_2/2 + \sin (\pi-\phi_0/2) \sin \phi_2/2 \cos(\pi- l(0)) = \\
-\cos (\pi-\phi'_0/2) \cos \phi'_2/2 + \sin (\pi-\phi'_0/2) \sin \phi'_2/2 \cos(\pi- l(0)').
\end{multline} 
Note again that we include the degenerate triangles and in that case the equation is not an enough condition
and the condition is a geometric one.

From now on, we may assume that the fixed points $\pm w_2$ of $h(d_2)$ are well defined and 
the angle $\tau_2$ at $w_2$ is not $0$ or $2\pi$ since otherwise the above equivalences are enough. 
Also, each of $\pm w_2$ does not coincide with $v_0, v_1$.

An $\SI^1$-action is given by taking a different segment $s_2$ with endpoints $v_0$ and $v_1=-v_0$. 
Then $v_2$ is obtained as an intersection of $s_2$ with the great circle through $w_2$ 
that has angle $\eta_2= \pi-\tau_2/2$ with respect to the shortest segment to $v_0$ from $w_2$.
This again involves generalized triangles and changing the angle $\theta_1/2$ to $\theta_1/2 - \pi$ if 
$\theta_1/2 \geq \pi$. 
Again there is an $\SI^1$-action on the abelian edge regions given by restricting the $\SI^1$-action in
Section \ref{subsec:abcd}, which extends the previous one continuously. 

If we choose a segment $s_2$ with endpoints $v_1$ and $v_0$ passing 
through $w_2$, then $v_2=w_2$, $h(d_2)$ has the rotation angle 
$\phi_2$, $\phi_0=0$, and $h(d_1)$ is now a rotation about $v_1$ with angle 
$\phi_1$. We can also choose so that $\phi_1=0$ by making $\phi_0$ be equal to $2\pi$ minus 
the rotation angle of $h(d_1)$. We call these the {\em canonical} representatives for 
a representation or character.
Thus, every representation over $A$ has {\em canonical} representatives. 

Thus, the configurations that we need to classify are ones where each is given by an antipodal pair $(v_0, v_1)$ 
and a choice of fixed point $w_2, -w_2$ and the rotation angles of $h(d_2)$ and $h(d_1)$. 

Also, there is a Klein four-group action again: First, there is a $\bZ_2$-action
given by choosing $-s_2$ instead of $s_2$. This is in the $\SI^1$-action. 

Second, the different antipodal choices of $v_0, v_1$ give us an $I_{{A}}$-action: 
\begin{eqnarray} \label{eqn:ia}
(v(0), l(0), \phi_0, \phi_1, \phi_2) &\mapsto & (\pi-v(0), \pi-l(0), 2\pi-\phi_0, 2\pi-\phi_1, \phi_2).
\end{eqnarray}

We define the equivalence relation $\sim$ on $({A-I-IV}) \times \SI^1\times \SI^1\times \SI^1$
to be one generated by the equivalence relations given by equation \eqref{eqn:obviousA} 
and the $\SI^1$-action and equation \eqref{eqn:ia}.

\begin{figure}

\centerline{\includegraphics[height=7cm]{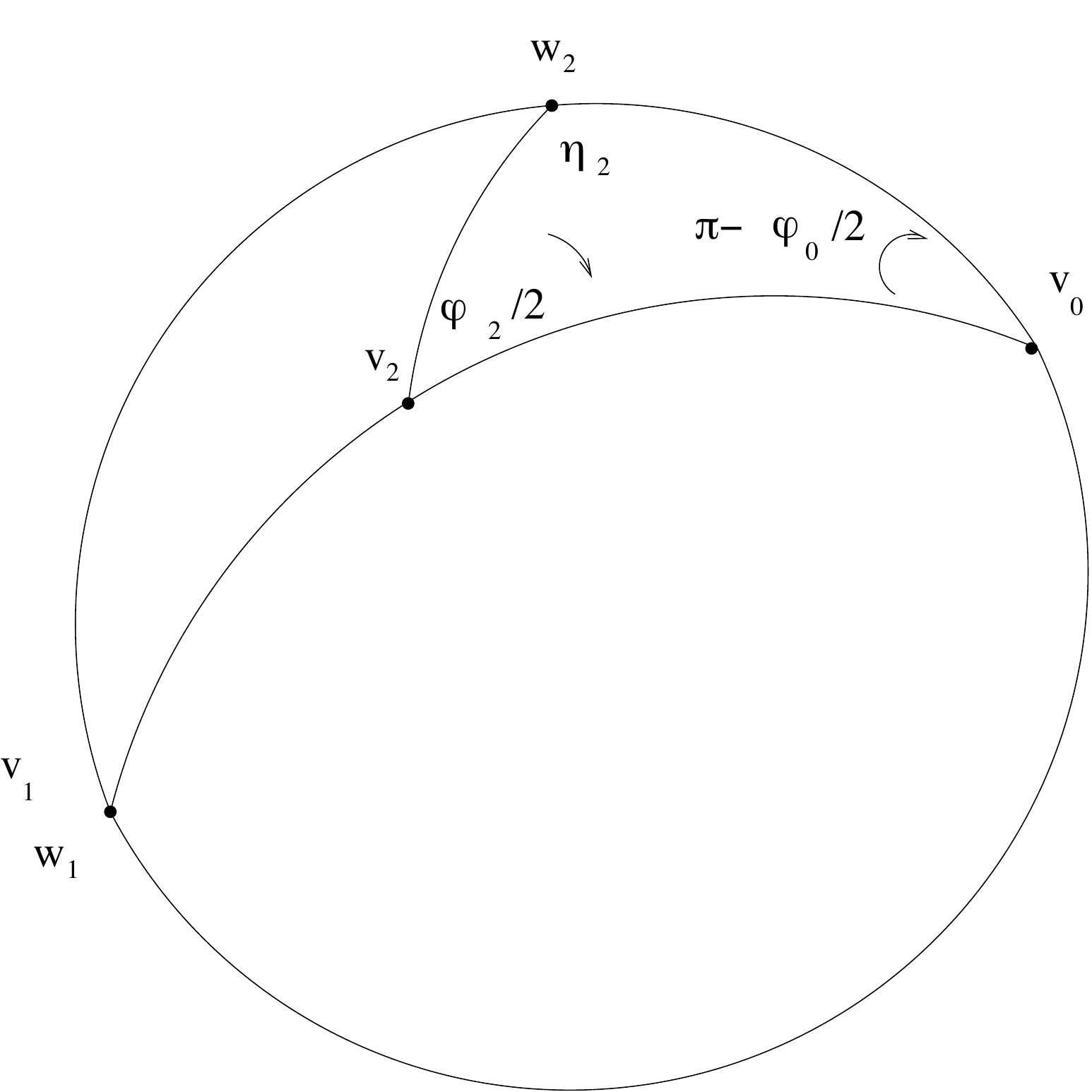}}
\caption{  Finding holonomy of $d_1$ and $d_2$ using triangles for case ${A-I-IV}$.}
\label{fig:mapFA}
\end{figure}

Therefore we define a map 
\[{\mathcal T}_{{A-I-IV}}: ({A-I-IV})\times \SI^1\times \SI^1\times \SI^1 \ra C^{{A-I-IV}},\]
which is onto by definition.
We note that $({A-I-IV})\times \{0\} \times \SI^1\times \SI^1 \ra C^{{A-I-IV}}$
and ${A-I-IV} \times \SI^1 \times  \{0\}\times \SI^1 \ra C^{{A-I-IV}}$
are onto since we can let $\phi_0=0$ or let $\phi_1=0$ by the $\SI^1$-action.

\begin{prop}\label{prop:A}
The induced map 
\[{\mathcal T}_{{A-I-IV}}: ({A-I-IV}) \times \SI^1\times \SI^1\times \SI^1/\sim \ra C^{{A-I-IV}}\] 
is a homeomorphism.
Furthermore,  the restrictions
\[{\mathcal T}_{{A-I-IV}}^0: ({A-I-IV}) \times \{0\}\times \SI^1\times \SI^1/\sim \ra C^{{A-I-IV}}\]
\[{\mathcal T}_{{A-I-IV}}^1: ({A-I-IV}) \times \SI^1 \times \{0\} \times \SI^1/\sim \ra C^{{A-I-IV}}\]
 are homeomorphisms
where $\sim$ is given by the equation \eqref{eqn:obviousA} and the Klein four-group action.
\end{prop}
\begin{proof} 
The proof is similar to Proposition \ref{prop:B}. That is, the $\SI^1$-action gives the equivalences for
the nonabelian parts and the abelian parts are studied by equation \eqref{eqn:obviousA}.
\end{proof}

\subsubsection{Regions $A_\pi$ and $A_0$} \label{subsec:Api}

Again let ${A}_\pi$ and ${A}_0$ denote the subsets of ${A}$ defined by $\theta(0)=\pi$
and $\theta(0)=0$ respectively. 

Also, on $A_\pi\times T^3$ and $A_0\times T^3$, 
there is a $\bZ_2$-action sending $v_0, v_1$ to their antipodes respectively and not changing $v_2$.
\begin{eqnarray} \label{eqn:ia2}
(\pi, l(0), \phi_0, \phi_1, \phi_2) &\mapsto &(\pi, \pi-l(0), 2\pi-\phi_0, 2\pi-\phi_1, \phi_2). \nonumber \\
(0, l(0), \phi_0, \phi_1, \phi_2) &\mapsto &(0, \pi-l(0), 2\pi-\phi_0, 2\pi-\phi_1, \phi_2). 
\end{eqnarray}
(These can be considered an augmented $I_A$-transformation.)
${A}_{\pi}\times T^3$ has an equivalence relation given by 
equation \eqref{eqn:obviousA} and the $\SI^1$-action and the $\bZ_2$-group action 
given by equation \eqref{eqn:ia2}. Similar statements hold for $A_0\times T^3$ as well.





\begin{prop}\label{prop:Api} 
\[ {\mathcal T}_{A_\pi} : {A}_{\pi}\times \SI^1 \times \SI^1\times \SI^1 /\sim \ra H(F_2)\] is 
a homeomorphism.
Furthermore, so are the following restrictions
 \[ {\mathcal T}_{A_\pi}^0 : {A}_{\pi}\times \{0\}\times \SI^1\times \SI^1/\sim  \ra H(F_2) \hbox{ and } \] 
\[ {\mathcal T}_{A_\pi}^1 : {A}_{\pi} \times \SI^1 \times  \{0\}\times \SI^1/\sim \ra  H(F_2).\]
We can replace $A_\pi$ with $A_0$ and the above statements hold again. 
\end{prop}
\begin{proof}
This can be proved as in Proposition \ref{prop:Bpi}. 

However, the second statement can be shown by 
the direct expression of coordinates and the fact that the Klein four-group
acts on both coordinate spaces can be used: 
The map \[ {\mathcal T}_{A_\pi} : {A}_{\pi}\times \{0\}\times \SI^1\times \SI^1  \ra H(F_2)\] 
is given by \[(\pi, l(0), 0,\phi_1,\phi_2) \ra (l(0),\phi_1,\phi_2),\] 
which is clearly onto.
(Recall from Section \ref{subsec:free} for the second coordinates.)
\end{proof}


\subsection{Region ${A}'-II-III$}
Now, consider the region ${A}'$ where a corresponding triangle is a segment 
with endpoints  $v_0=v_1$ and $v_2$. 

First, we find the representation $h$ corresponding to an element of $(A'-II-III)\times T^3$:
To obtain $h(d_1)$ and $h(d_2)$,  we note first that 
$h(d_1)$ is a rotation about $v_1$ with angle $\phi_1-\phi_0$. 
We next consider 
a triangle $T_2$ with vertices $v_2, v_0$ and $w_2$ and with 
angles $\phi_2/2$ and $\pi-\phi_0/2$ at $v_2$ and $v_0$ respectively. 
The vertex $w_2$ is the fixed point of $h(d_2)$ and the vertex angle is 
$\pi$ minus the one half of the rotation angle of $h(d_2)$. 

Similarly to Proposition \ref{prop:nonabelianA}, we show that 
the subset of $A'\times T^3$ corresponding to the abelian characters is 
\begin{gather}\label{eqn:abelianBppp}
(v(0), 0, \phi_0, \phi_1, \phi_2), (v(0), \pi, \phi_0, \phi_1, \phi_2), 
(v(0), l(0), \phi_0, \phi_1, 0), \nonumber\\
(0, l(0), \phi_0, \phi_0, \phi_2), (\pi, l(0), \phi_0, \phi_0, \phi_2)\nonumber \\
 \mbox{ for } v(0), l(0) \in [0, \pi], \phi_0, \phi_1, \phi_2 \in [0, 2\pi] 
\end{gather}

For the region $A' \times T^3$, we first introduce 
the equivalence relation given by
\begin{equation}\label{eqn:obviousA'}
\begin{split}
& (v(0), 0, \phi_0, \phi_1, \phi_2) \sim (v(0), 0, \phi_0+t, \phi_1+t, \phi_2+t) \\
& (v(0), \pi, \phi_0, \phi_1, \phi_2) \sim (v(0), \pi, \phi_0+t, \phi_1+t, \phi_2-t) \\
& (v(0), 0, \phi_0, \phi_1, \phi_2) \sim (v(0), \pi, \phi_0, \phi_1, 2\pi-\phi_2) \\
& (v(0), l(0), \phi_0, \phi_1, 0) \sim  (v(0), l(0)', \phi_0, \phi_1, 0)  \\ 
& \mbox{ for an arbitrary pair of $l(0)$ and $l(0)'$, } \\
& (0, l(0), \phi_0, \phi_0, \phi_2) \sim  (0, l(0)', \phi'_0, \phi'_0, \phi'_2)  \\ 
& \mbox{ for an  arbitrary pair of $l(0)$ and $l(0)'$ and the same third angles}  \\
& (\pi, l(0), \phi_0, \phi_0, \phi_2)  \sim   (\pi, l(0)', \phi'_0, \phi'_0, \phi'_2)  \\ 
& \mbox{ for an  arbitrary pair of $l(0)$ and $l(0)'$ and the same third angles.}
\end{split}
\end{equation}
The same third angle condition means the triangle angle condition as indicated by equation \eqref{eqn:thirdangleA}.

From now on, we may assume that the fixed points of $h(d_i)$ are well-defined and 
the angles are nonzero for $i=1,2$ since otherwise the above equivalences are enough. 

There is an $\SI^1$-action also given by choosing a different segment from $v_0$ to $-v_0$
and taking $v_2$ as intersection of a great circle $\SI^1$ through $w_2$ having $\pi$ minus the $1/2$ of 
the angle of $h(d_2)$ with the segment from $v_0$ to $w_2$. 

If we choose a segment $s$ with endpoints $v_1$ and $v_0$ passing 
through $w_2$, then $v_2=w_2$ and $h(d_2)$ has the rotation angle 
$\phi_2$. Now, we obtain $\phi_0=0$. We can also make $\phi_1=0$. 
We call these the {\em canonical} representatives. 
Thus, every representation over $A'$ has canonical representatives. 
For each abelian edge region, the $\SI^1$-action is given by restricting one in Section \ref{subsec:abcd}.

Therefore we define a map 
${\mathcal T}_{{A}'}:{{A}'} \times  \SI^I\times \SI^1\times \SI^1\ra C^{{A}'}$. 
The map is onto by same reasoning as above. 


Again, we define $\sim$ on $({{A}'-II-III}) \times  \SI^1\times \SI^1\times \SI^1 $
by the above equivalence relations, i.e., equation \eqref{eqn:obviousA'} and the $\SI^1$-action 
and the $\{\Idd, I_A\}$-action, given by same formula as equation \eqref{eqn:ia}.
The $\{\Idd, I_A\}$-action and the $\SI^1$-action restrict to the respective Klein four-group actions on 
$({{A}'-II-III}) \times  \{0\}\times \SI^1\times \SI^1$ and 
$({{A}'-II-III}) \times  \SI^1 \times \{0\} \times \SI^1$.
We define equivalence relations  on these by equation \eqref{eqn:obviousA'} and 
the Klein four-group action.

We obmit the proof of the following: 
\begin{prop}\label{prop:A'}
The induced map \[{\mathcal T}_{{A}'-II-III}:({{A}'-II-III}) \times  \SI^1\times \SI^1\times \SI^1/\sim \ra C^{{A}'-II-III}\] is 
a homeomorphism.
The restrictions \[{\mathcal T}_{{A}'-II-III}:({{A}'-II-III}) \times  \{0\} \times \SI^1\times \SI^1/\sim \ra C^{{A}'-II-III} \hbox{ and }\]
 \[{\mathcal T}_{{A}'-II-III}:({{A}'-II-III}) \times  \SI^1\times  \{0\} \times \SI^1/\sim \ra C^{{A}'-II-III}\] are
homeomorphisms.
\end{prop}

\subsubsection{Regions $A'_\pi$ and $A'_0$} \label{subsec:A'pi}

Again let ${A'}_\pi$ and ${A'}_0$ denote the subsets of $A'$ defined by $\theta_0=\pi$
and $\theta_0=0$ respectively. For these, $I_A$ has an {\em augmented action} as given by  \eqref{eqn:ia2}.
The equivalence relation $\sim$ on $A'_\pi\times T^3$ is given 
by equation \eqref{eqn:obviousA'} and the $\SI^1$-action and the above augmented $I_A$-action. 
The actions restrict to a Klein four-group action on 
${{A}'_\pi} \times  \{0\}\times \SI^1\times \SI^1$ and 
${{A}'_\pi} \times  \SI^1 \times \{0\} \times \SI^1$.
We define equivalence relations  on these by equation \eqref{eqn:obviousA'} and 
the Klein four-group action. Similarly, we define equivalence relations on $A'_0\times T^3$ 
and $A'_0 \times \{0\}\times \SI^1\times \SI^1$ and $A'_0 \times \{0\}\times \SI^1\times \SI^1$.

By introducing appropriate $\sim$ from equation \eqref{eqn:obviousA'} and 
the $\SI^1$-action and the augmented $\{\Idd, I_A\}$-action, we obtain:
\begin{prop}\label{prop:A'pi} 
\[ {\mathcal T}_{A'_\pi}: {A'}_{\pi} \times  \SI^1\times \SI^1\times \SI^1/\sim \ra H(F_2)\] is 
a homeomorphism. 
Furthermore, the restriction are homeomorphisms
\[{\mathcal T}_{{A'_\pi}}^0: {A'_\pi} \times \{0\}\times \SI^1\times \SI^1/\sim \ra C^{{A'_\pi}}\]
\[{\mathcal T}_{{A'_\pi}}^1: {A'_\pi} \times \SI^1 \times \{0\} \times \SI^1/\sim \ra C^{{A'_\pi}}.\]
Also, the same statements are true with $A'_0$ replacing $A'_\pi$. 
\end{prop}
\begin{proof}
The proof is similar to the case $B'_\pi$. 

The second statement can be also proved by 
the map \[ {\mathcal T}_{A'_\pi}: {A'}_{\pi} \times  \{0\}\times \SI^1\times \SI^1/\sim \ra H(F_2)\] 
is given by \[(\pi, l(0), 0,\phi_1,\phi_2) \mapsto (l(0),\phi_1,\phi_2)\]
\end{proof}


\subsection{Region $C-I-III$}
Finally, we go to $C^{{C-I-III}}$ and $C^{{C}'-II-IV}$.
Suppose that we are over $C$. Then an element of $C$ corresponds to 
a lune with vertices $v_0$ and $v_2=-v_0$ 
and $v_1$ is on an edge.

First, we obtain a map $(C-I-III)\times T^3 \ra C^{C-I-III}$:  Given an element of $(C-I-III)\times T^3$, 
to construct $h(d_1)$ and $h(d_2)$, we consider 
a triangle $T_1$ with vertices $v_1, v_0$ and $w_1$ and with 
angles $\phi_1/2$ and $\pi-\phi_0/2$ at $v_1$ and $v_0$ respectively
for the lune with  vertices $v_0$ and $v_2$ and $v_1$ on the edge.
The vertex $w_1$ is the fixed point of $h(d_1)$ and the vertex angle $\eta_1$ is 
$\pi$ minus the half of the rotation angle of $h(d_1)$ and 
$h(d_2)$ is a rotation about $v_2$ with angle $\phi_2-\phi_0$ and $w_2=v_2$.

\begin{figure}
\centerline{\includegraphics[height=7cm]{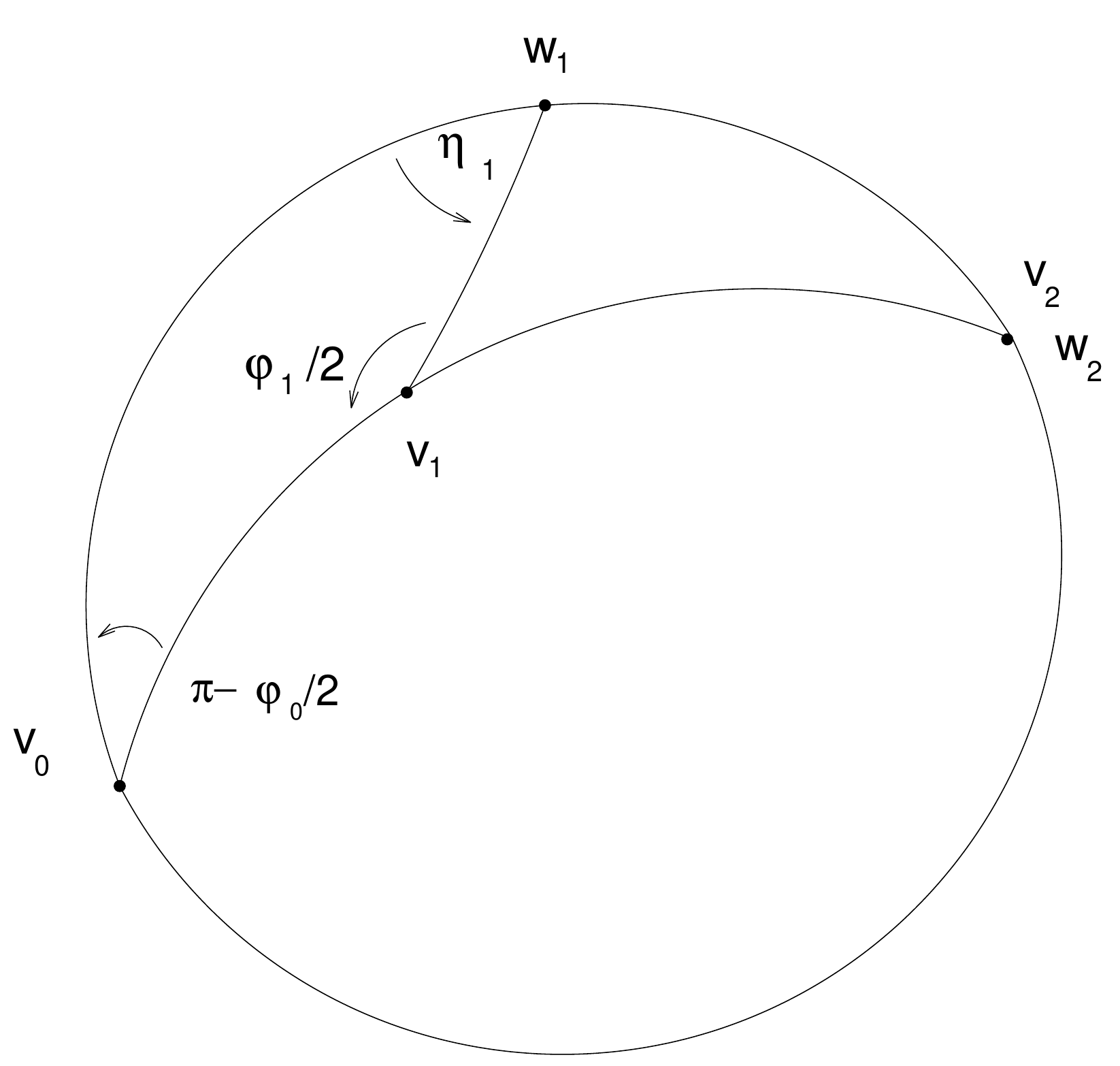}}
\caption{  Finding holonomy of $d_1$ and $d_2$ using triangles.}
\label{fig:mapFC}
\end{figure}

We now discuss the equivalence relation on $C\times T^3$: 
\begin{prop}\label{prop:nonabelianC}
A representation $h$ corresponding to an element of $C \times T^3$ is nonabelian 
if and only if $h(d_1)$ are not identity and does not fix $v_0$ and $v_2$ for any configuration corresponding 
to an element of $C \times T^3$ representing $h$. 
\end{prop}

The subset of $C\times T^3$ corresponding to the abelian representations is 
\begin{gather}\label{eqn:abelianC}
(v(2), 0, \phi_0, \phi_1, \phi_2), (v(2), \pi, \phi_0, \phi_1, \phi_2),
(v(2), l(2), \phi_0, 0,  \phi_2), \nonumber \\
(0, l(2), \phi_0, \phi_1, 2\pi-\phi_0), (\pi, l(2), \phi_0,  \phi_1, 2\pi-\phi_0) \nonumber \\
\mbox{ for } v(2), l(2) \in [0, \pi], \phi_0, \phi_1, \phi_2 \in [0, 2\pi].
\end{gather}

For the region $C \times T^3$, we introduce
the equivalence relation  given by
\begin{equation}\label{eqn:obviousC}
\begin{split}
& (v(2), 0, \phi_0, \phi_1, \phi_2) \sim (v(2), 0, \phi_0+t, \phi_1+t, \phi_2-t) \\
& (v(2), \pi, \phi_0, \phi_1, \phi_2) \sim (v(2), \pi, \phi_0+t, \phi_1-t, \phi_2-t) \\
& (v(2), 0, \phi_0, \phi_1, \phi_2) \sim (v(2), \pi, \phi_0, 2\pi- \phi_1, \phi_2).\\
& (v(2), l(2), \phi_0, 0, \phi_2) \sim  (v(2), l(2)', \phi_0, 0, \phi_2)  \\
& \mbox{ for an arbitrary pair of $l(2)$ and $l(2)'$ and } \\
& (0, l(2), \phi_0, \phi_1, 2\pi-\phi_0) \sim  (0, l(2)', \phi'_0, \phi'_1, 2\pi-\phi'_0) \\ 
& \mbox{ for an arbitrary pair of $l(2)$ and $l(2)'$ and the same third angles} \\
& (\pi, l(2), \phi_0, \phi_1, 2\pi-\phi_0)  \sim   (\pi, l(2)', \phi'_0, \phi'_1, 2\pi-\phi'_0) \\ 
& \mbox{ for an arbitrary pair of $l(2)$ and $l(2)'$ and the same third angles.}
\end{split}
\end{equation}
The same third angle means the triangle angle condition as indicated by
\begin{multline}\label{eqn:thirdangleAp}
-\cos (\pi- \phi_0/2) \cos \phi_1/2 + \sin (\pi- \phi_0/2) \sin \phi_1/2 \cos l(2) = \\
-\cos (\pi-\phi'_0/2) \cos \phi'_1/2 + \sin (\pi-\phi'_0/2) \sin \phi'_1/2 \cos l(2)'.
\end{multline} 


Choose an element of $C \times T^3$. We have a lune with vertices $v_0$ and $v_2=-v_0$ 
and $v_1$ is on an edge.
Again there is an $\SI^1$-action given by choosing a different segment 
$s$ with endpoints $v_0$ and $v_2$ and taking $v_1$ as an intersection 
point of $s$ with a great circle $\SI^1$ through $w_1$ which has $\pi$ minus 
the angle $1/2$ of the rotation angle of $h(d_1)$ with 
respect to the shortest segment from $w_1$ to $v_0$. 
For each abelian edge region, the $\SI^1$-action is given by the restriction of ones in Section \ref{subsec:abcd}
and continuously extends the above. 

If we choose a segment $s$ with endpoints $v_2$ and $v_0$ passing 
through $w_1$, then $v_1=w_1$ and $h(d_1)$ has the rotation angle 
$\phi_1$. Here we obtain $\phi_0=0$. 

We can also make $\phi_2=0$ by making $\phi_0$ have angle 
$2\pi$ minus the angle of $h(d_2)$. 
We call these the {\em canonical} representatives. Thus, every representation over $C$ has
canonical representatives. 



Again $\{\Idd, I_C\}$ acts on these spaces. 
We define the equivalence relation $\sim$ using the above ones by equation \eqref{eqn:obviousC} and 
the $\SI^1$-action and the $\{\Idd, I_C\}$-action.

We obmit the proof of the following: 
\begin{prop}\label{prop:C}
The induced map 
\[{\mathcal T}_{{C-I-III}}: ({C-I-III})\times  \SI^1 \times \SI^1\times \SI^1/\sim \ra C^{{C-I-III}}\] 
is a homeomorphism.
Furthermore, the restrictions 
\[{\mathcal T}_{{C-I-III}}^0: ({C-I-III})\times  \{0\} \times \SI^1\times \SI^1/\sim \ra C^{{C-I-III}} \hbox{ and } \]
\[{\mathcal T}_{{C-I-III}}^2: ({C-I-III})\times  \SI^1 \times \SI^1\times \{0\}/\sim \ra C^{{C-I-III}}\]
where $\sim$ is given by equation \eqref{eqn:obviousC} and the Klein four-group given 
by restricting the $\SI^1$-action and $\{\Idd, I_C\}$-action. 
\end{prop}

\subsubsection{Regions $C_\pi$ and $C_0$} \label{subsec:Cpi}

Again let ${C}_\pi$ and ${C}_0$ denote the subsets of 
$C$ defined by $\theta_0=\pi$ and $\theta_0=0$ respectively. 
The equivalence relations on $C_\pi\times T^3$ and $C_0\times T^3$ are 
defined by equation \eqref{eqn:obviousC} and by the $\SI^1$-action and the {\em augmented} $\{\Idd, I_C\}$-action: 
\begin{eqnarray}\label{eqn:auic}
(0, l(2), \phi_0, \phi_1, \phi_2) & \mapsto & (0, \pi-l(2), 2\pi-\phi_0, \phi_1, 2\pi-\phi_2) \nonumber \\ 
(\pi, l(2), \phi_0, \phi_1, \phi_2) &\mapsto & (\pi, \pi-l(2), 2\pi-\phi_0, \phi_1, 2\pi-\phi_2).  
\end{eqnarray}
Define $\sim$ on ${C}_{\pi}\times \{0\} \times \SI^1\times \SI^1$
by equation \eqref{eqn:obviousC} and an induced Klein four-group action on.
Similarly we define $\sim$ on ${C}_{\pi} \times \SI^1\times \SI^1\times \{0\}$.
We define respective equivalence relations on $C_0$ cases.

\begin{prop}\label{prop:Cpi} 
\[ {\mathcal T}_{C_\pi}: {C}_{\pi}\times \SI^1\times  \SI^1\times \SI^1/\sim \ra H(F_2)\] is 
a homeomorphism. 
Similarly, \[ {\mathcal T}_{C_\pi}: {C}_{\pi}\times \{0\} \times \SI^1\times \SI^1/\sim \ra H(F_2) \hbox{ and }\] 
\[ {\mathcal T}_{C_\pi}: {C}_{\pi} \times \SI^1\times \SI^1\times \{0\}/\sim \ra H(F_2)\] are 
homeomorphisms. The same statements hold with $C_\pi$ replaced by $C_0$. 
\end{prop}

The proof is similar to the above $B_\pi$ case. 
Again, the second statement can be proved using
the onto map \[ {\mathcal T}_{C_\pi}: {C}_{\pi} \times \{0\}\times \SI^1 \times \SI^1 \ra H(F_2)\] 
given by \[(\pi, l(2), 0,\phi_1,\phi_2) \ra (\pi-l(2),\phi_1,\phi_2).\]

\subsection{Region $C'-II-III$}
We are over $C'$ finally. 
Here, for our configuration, a triangle is a segment with endpoints 
$v_1$ and $v_0=v_2$. 

First, we construct a map $C'\times T^3 \ra C^{C'}$:
To construct $h(d_1)$ and $h(d_2)$ for the representation corresponding to a given configuration, 
we consider a triangle $T_1$ with vertices $v_1, v_0$ and $w_1$ and with 
angles $\phi_1/2$ and $\pi-\phi_0/2$ at $v_2$ and $v_0$ respectively
for a pointed segment with vertices $v_1$ and $v_0=v_2$.
The vertex $w_1$ is the fixed point of $h(d_1)$ and the vertex angle is 
$\pi$ minus the half of the rotation angle of $h(d_2)$. 
$h(d_2)$ is a rotation about $v_2$ with angle $\phi_2-\phi_0$. 

Similarly to Proposition \ref{prop:nonabelianC},
the subset of $C'\times T^3$ corresponding to the abelian characters is 
\begin{gather}\label{eqn:abelianC'}
(v(2), 0, \phi_0, \phi_1, \phi_2), (v(2), \pi, \phi_0, \phi_1, \phi_2), \nonumber\\
(v(2), l(2), \phi_0, 0,  \phi_2), v(2), l(2), \in [0, \pi] ,\phi_0, \phi_1, \phi_2\in [0, 2\pi]. \nonumber\\
(0, l(2), \phi_0, \phi_1, \phi_0), (\pi, l(2), \phi_0,  \phi_1, \phi_0) \nonumber\\
\mbox{ for } v(2), l(2) \in [0, \pi], \phi_0, \phi_1, \phi_2 \in [0, 2\pi].
\end{gather}

For the region $C' \times T^3$, we introduce
the equivalence relation  given by
\begin{equation}\label{eqn:obviousCp}
\begin{split}
& (v(2), 0, \phi_0, \phi_1, \phi_2) \sim (v(2), 0, \phi_0+t, \phi_1+t, \phi_2+t) \\
& (v(2), \pi, \phi_0, \phi_1, \phi_2) \sim (v(2), \pi, \phi_0+t, \phi_1-t, \phi_2+t) \\
& (v(2), 0, \phi_0, \phi_1, \phi_2) \sim (v(2), \pi, \phi_0, 2\pi-\phi_1, \phi_2) \\
& (v(2), l(2), \phi_0, 0, \phi_2) \sim  (v(2), l(2)', \phi_0, 0, \phi_2) \\
& \mbox{ for an arbitrary pair of $l(2)$ and $l(2)'$ and } \\
& (0, l(2), \phi_0, \phi_1, \phi_0) \sim  (0, l(2)', \phi'_0, \phi'_1, \phi'_0) \\ 
& \mbox{ for an arbitrary pair of $l(2)$ and $l(2)'$ and the same third angles}  \\
& (\pi, l(2), \phi_0, \phi_1, \phi_0)  \sim   (\pi, l(2)', \phi'_0, \phi'_1, \phi'_0) \\
& \mbox{ for an arbitrary pair of $l(2)$ and $l(2)'$ and the same third angles.}
\end{split}
\end{equation}
The same third angle condition means the triangle angle condition as indicated by
\begin{multline}\label{eqn:thirdangleApp}
-\cos (\pi- \phi_0/2) \cos \phi_1/2 + \sin (\pi- \phi_0/2) \sin \phi_1/2 \cos l(2) = \\
-\cos (\pi-\phi'_0/2) \cos \phi'_1/2 + \sin (\pi-\phi'_0/2) \sin \phi'_1/2 \cos l(2)'.
\end{multline} 



We may assume without loss of generality that fixed points of $h(d_1)$ and $h(d_2)$ are well-defined.
Again, there is an $\SI^1$-action. The equivalence relation is given on $C'\times T^3$ by 
the above equation \eqref{eqn:obviousCp} and the $\SI^1$-action and the $\{\Idd, I_C\}$-action.

If we choose a segment $s$ with endpoints $v_2$ and $v_0$ passing 
through $w_1$, then $v_1=w_1$ and $h(d_1)$ has the rotation angle 
$\phi_1$. Here $\phi_0=0$. We can also make $\phi_2=0$. 
We call these the {\em canonical} representatives. 
Thus, every representation over $C'$ has {\em canonical} representatives.

\begin{prop}\label{prop:C'}
The induced map 
\[{\mathcal T}_{{C}'-II-IV}:({C'-II-IV})\SI^1\times \SI^1\times \SI^1\times /\sim \ra C^{{C}'-II-IV}\] 
is a homeomorphism.
Furthermore, the restrictions 
\[{\mathcal T}_{{C'}-II-IV}^0: ({C'}-II-IV)\times  \{0\} \times \SI^1\times \SI^1/\sim \ra C^{{C'}-II-IV} \hbox{ and } \]
\[{\mathcal T}_{{C'}-II-IV}^2: ({C'}-II-IV)\times  \SI^1 \times \SI^1\times \{0\}/\sim \ra C^{{C'}-II-IV}\]
where $\sim$ is given by equation \eqref{eqn:obviousCp} and the Klein four-group given 
by restricting the $\SI^1$-action and the $\{\Idd, I_C\}$-action. 
\end{prop}

\subsubsection{${C'}_\pi$ and ${C'}_0$} 

Again let ${C'}_\pi$ and ${C'}_0$ denote the subsets of $C$ defined by $\theta_0=\pi$
and $\theta_0=0$ respectively. 
Again equivalence relations on $C'_\pi\times T^3$ and $C'_0\times T^3$ are given by 
equation \eqref{eqn:obviousCp} and the $\SI^1$-action and augmented $\{\Idd, I_C\}$-actions.
as given by equation \eqref{eqn:auic}.

The equivalence relations on
${C'}_{\pi}\times \{0\} \times \SI^1\times \SI^1$ and 
${C'}_{\pi} \times \SI^1\times \SI^1\times \{0\}$ are given by equation \eqref{eqn:obviousCp} 
and the Klein four-group action given by restricting the $\SI^1$-action and the $\{\Idd, I_C\}$-action. 

\begin{prop}\label{prop:C'pi} 
\[ {\mathcal T}_{C'_\pi}: {C'}_{\pi}\times \SI^1\times \SI^1\times \SI^1/\sim \ra H(F_2)\] is 
a homeomorphism. Similarly, 
\[ {\mathcal T}_{C'_\pi}: {C'}_{\pi}\times \{0\} \times \SI^1\times \SI^1/\sim \ra H(F_2) \hbox{ and }\] 
\[ {\mathcal T}_{C'_\pi}: {C'}_{\pi} \times \SI^1\times \SI^1\times \{0\}/\sim \ra H(F_2)\] are 
homeomorphisms. The same statements hold with $C'_\pi$ replaced by $C'_0$. 
\end{prop}

The proof is similar to the above $C_\pi$ case. 
Again, the second statement can be proved using
the map \[ {\mathcal T}_{C'_\pi}: {C'}_{\pi}\times \{0\} \times \SI^1\times \SI^1/\sim \ra H(F_2)\] 
given by \[(\pi, l(2), 0,\phi_1,\phi_2) \ra (l(2),\phi_1,\phi_2).\]

\subsection{Obtaining the homeomorphism}\label{subsec:obh}

We will now find the topological structure of ${\mathcal C}_0$ as
an identification space obtained from $\tilde G \times T^3$. 
The equivalence relation is namely the ones we obtained for $a, b, c, d$, $I$, $II$, $III$, $IV$, 
${A}$, $A'$, $B$, $B'$, $C$, and $C'$ and by the action of $\{I,I_{{A}},I_{{B}},I_{{C}}\}$.

One can characterize the characters over regions $I$, $II$, $III$, and $IV$, 
as ones which factor into a free group of rank two character
with a canonical map sending $c_i$s to the trivial 
element and the characters over regions ${A}$, $A'$, $B$, $B'$, $C$, and $C'$ as 
one where two of $c_i$s share the fixed points, 
and characters over regions $a$, $b$, $c$, and $d$ as abelian characters. 
(The representation where two of $c_i$s share the common antipodal pair
of fixed points are said to be the {\em degenerated} representation. )

We will describe the topology in the next section.

\begin{thm}\label{thm:main1}
The identity component ${\mathcal C}_0$ of $\rep(\pi_1(\Sigma), \SOThr)$ is
homeomorphic to $\tilde G \times T^3/\sim$.
Thus ${\mathcal C}_0$ is a topological complex consisting of $3$-dimensional 
$H(F_2)$, the $4$-dimensional complexes $C^A,C^{B},$ and $C^{C}$, the space of 
abelian representations,  coming from the boundary of the $3$-ball $\tilde G$, and 
the $6$-dimensional complex from the interior of ${\mathcal C}_0$
\end{thm}
\begin{proof}

The proof is divided into three parts:
We show that 
\[\mathcal T: \tilde G \times T^3 \ra {\mathcal C}_0 \subset \rep(\pi_1(\Sigma), \SOThr)\]
induces one-to-one onto map from the quotient 
space of $\tilde G \times T^3$ by the equivalence 
relation described above.

We showed that the map $\mathcal T$ is onto in Proposition \ref{prop:map}.
We now verify the injectivity which was partially done in
the above arguments:

First, let us look at the interior:
By Proposition \ref{prop:map}, $\mathcal T|\tilde G^o\times T^3/\sim$ is injective 
onto the set of conjugacy classes of triangular characters. 

We now restrict our attention to $\mathcal T|\partial \tilde G\times T^3/\sim$:
The characters coming from $\partial\tilde G\times T^3/\sim$
do not coincide with any from $\tilde G^o\times T^3/\sim$ since 
the former ones are not triangular.

The injectivity over each of $a\times T^3/\sim$, $b\times T^3/\sim$, $c\times T^3/\sim$, and $d\times T^3/\sim$ 
to the set of abelian characters was shown above. 
The injectivity over each of  $I\times T^3/\sim$, $II\times T^3/\sim$, $III\times T^3/\sim$, 
and $IV\times T^3/\sim$ to $H(F_2)$ is proved 
above also. The injectivity over each ${A}\times T^3/\sim$, $A'\times T^3/\sim$, $B\times T^3/\sim$, $B'\times T^3/\sim$, 
$C\times T^3/\sim$, and $C'\times T^3/\sim$ to the space of 
degenerate characters is also clear and done above. 
(Up to the action of the Klein four-group, these are only three distinct maps.)


Recall $U=I\cup II\cup III\cup IV$.
The characters coming from 
the regions $I\times T^3,II\times T^3,III\times T^3,IV \times T^3$ 
do not coincide with ones coming from the elements 
in $(A-U)\times T^3$, $(A'-U)\times T^3$, $(B-U)\times T^3$, $(B'-U)\times T^3$, 
$(C-U)\times T^3$, and $(C'-U)\times T^3$.
The reason is that two of $h(c_i)$ are not identity for the 
later regions. Also, the region $(A\cup A' -U)\times T^3$, 
the region  $(B\cup B'-U)\times T^3$, and
the region $(C\cup C'-U)\times T^3$ map to mutually 
disjoint subspaces of the character space. This follows since the characters 
have $h(c_i)=\Idd$ and $h(c_j)\ne \Idd$ for $j \ne i$ and the index denoted by $i$ are different 
for $A, A'$ and $B, B'$ and $C, C'$. 

Also, $a^o\times T^3$, $b^o\times T^3$, $c^o\times T^3$, and $d^o \times T^3$ do not produce 
characters from the union of the regions $I$, $II$, $III$, $IV$, ${A}$, $A'$, $B$, $B'$, $C$, and $C'$ times 
$T^3$. This is so since all $h(c_i)$ is not identity in the former case 
while in the other $h(c_i)$ is identity for at least one $i$. 

Because of the identification by the $\{I, I_A, I_B, I_C\}$-action, we need to consider maps from 
$a\times T^3$, $I\times T^3$, $A\times T^3,B\times T^3,C\times T^3$, and 
$G^o\times T^3$. We can verify that the respective regions $I^o\times T^3$, $(A-U)\times T^3, (B-U)\times T^3, (C-U)\times T^3$, 
and $a^o\times T^3$ correspond to disjoint character spaces by above.
Since we know that over each of these the map is injective, 
this completes the proof of injectivity. 

Since $\SOThr$ is a compact Lie group, the space $\rep(\pi_1(\Sigma), \SOThr)$ is Hausdorff.
Since $\tilde G \times T^3$ is compact, it follows that $\mathcal T$ is a homeomorphism.

\end{proof}

%
%


\section{The $\SUTw$-character space}
\label{sec:SUT}

We will first find the topological type of the $\SUTw$-character space of $\pi$
and then in the next section do the same for the $\SOThr$-character space.
The following is the classical result of Narashimhan, Ramanan, and Seshadri \cite{NR}, \cite{NS} which 
we prove using a different method. However, we will need the smoothness result of Hubueschmann \cite{Hueb}.

We first study $\CP^3$ as a quotient space of a tetrahedron times a $3$-torus.
Then we represent each $\SUTw$-character by a generalized triangle and pasting angles 
as in the $\SOThr$-case. 
An $\SUTw$-character of a pair of pants corresponds to a generalized triangle in a one-to-one 
manner except for the degenerate ones. The space of pasting maps in $\SUTw$ is still homeomorphic to 
$\SI^1$ but it double covers the space of pasting maps in $\SOThr$.
Thus, we parametrize the pasting angles by $[0, 4\pi]$ where $0=4\pi$ correspond to $I$
and $2\pi$ represents $-I$ for $I \in \SUTw$. We denote by $\SI^1_2$ the quotient of $[0, 4\pi]$ by 
identifying the endpoints. 

We denote by $T^3_2$ the product of the three copies of $\SI^1_2$. 
We find the equivalence relation on $\tilde G \times T^3_2$ to make 
it homeomorphic to the character space. These equivalence relation ``lifts'' 
the equivalence relation in $\SOThr$-case. We discuss abelian characters 
and discuss the characters over the fattened edges of $\tilde G$
showing that they form solid torus fibrations. 
The main result of this section Theorem \ref{thm:cp3-2} will be proved finally.
We will see that our space is homeomorphic to 
$\CP^3$ with some parts such as points and annuli ``blown-up'' to balls and solid tori.
(Here the term ``blown-up'' is by no means precise and we invite the readers to make this precise.)

The arguments are basically the same as in the $\SOThr$-case in  above Section \ref{sec:C0}.
The only difference is that we do not need to consider the Klein four-group $\{I, I_{{A}},I_{{B}},I_{{C}}\}$-action
although we need a $Z_2$-action on the pasting angle space $T^3_2$. 

\subsection{The complex projective space}\label{subsec:cp3}
Let $\CP^3$ denote the complex projective space. 
According to the toric manifold theory, $\CP^3$ admits a $T^3$-action 
given by 
\begin{equation}\label{eqn:cp2} 
(e^{i\theta_1}, e^{i\theta_2}, e^{i\theta_3})\cdot [z_0, z_1, z_2, z_3] 
= [ e^{i\theta_1}z_0, e^{i\theta_2}z_1, e^{i\theta_3}z_2, z_3] 
\end{equation}
and the quotient map is given by 
\begin{equation}\label{eqn:cp3} 
[z_0,z_1,z_2,z_3] \mapsto \pi (|z_0|^2,  |z_1|^2, |z_2|^2)/\sum_{i=0}^3 |z_i|^2, z_i \in \bC
\end{equation}
The image is a standard $3$-simplex $\tri^*$ in the positive quadrant of $\bR^3$ given by
$x_0+x_1+x_2 \leq \pi$ and the fibers are the orbits of $T^3$-action. 
The fibers are given by $\bR^3$ quotient out by the standard lattice $L^*$ with 
generators  $(2\pi, 0, 0), (0, 2\pi, 0), (0,0,2\pi)$. 
(See Masuda \cite{Mas} for the standard facts about toric manifolds.)

We can realize $\CP^3$ as a quotient space of $\tri^* \times \bR^3$ under 
the equivalence relation given by $L^*$-action on $\bR^3$ fibers 
and over the face given by $x_0 =0$, $(\phi_0,\phi_1,\phi_2) \sim (\phi_0+t,\phi_1,\phi_1)$ 
for any real $t$ and over the face given by $x_1 =0$, 
$(\phi_0,\phi_1,\phi_2) \sim (\phi_0,\phi_1+t,\phi_1)$, $t\in \bR$ and 
over the face $x_2=0$, 
$(\phi_0,\phi_1,\phi_2) \sim (\phi_0,\phi_1,\phi_1+t)$, $t\in \bR$
and over the face $x_0+x_1+x_2=\pi$, 
$(\phi_0,\phi_1,\phi_2) \sim (\phi_0+t,\phi_1+t,\phi_1+t)$, $t\in \bR$.
We prove the homeomorphism of $\CP^3$ with 
the quotient space by using Hausdorff compactness property of both spaces. 

We give another realization of $\CP^3$ for our purposes. 

\begin{prop}\label{prop:cp3}
By considering fibers of faces of $G$, 
we can realize $\CP^3$ as the quotient space $G\times T^3_2$ of under an equivalence 
relation given as follows{\rm :} 
\begin{itemize}
\item In the interior, the equivalence is trivially given.
\item For the face $a$, the equivalence relation on $a\times T^3_2$ is given by 
\[(v,\phi_0,\phi_1,\phi_2) \sim (v', \phi'_0,\phi'_1,\phi'_2)\] if and only if 
$v=v'$ and two vectors 
\[(\phi_0,\phi_1,\phi_2)-(\phi'_0,\phi'_1,\phi'_2) = s(2\pi, 2\pi, 2\pi) \hbox{ for } s \in \bR.\] 
This is given by the $\SI^1$-action generated by a vector normal to $a$.
\item For faces $b$, $c$, and $d$, the equivalence relation is defined \hfill\break
again using the respective $\SI^1_2$-action generated by vectors parallel  to
 \[(-2\pi, 2\pi, 2\pi), (2\pi, -2\pi, 2\pi), (2\pi, 2\pi, -2\pi)\] perpendicular to 
$b, c, d$ respectively.
\item In the edges and the vertices, the equivalence relation is induced 
from the facial ones. 
\item Finally an equivalence relation is given by the $\bZ_2$-action 
sending $(\phi_0, \phi_1, \phi_2)$ to $(\phi_0+2\pi, \phi_1+2\pi, \phi_2+2\pi)$. 
\end{itemize}
There is a natural $T^3_2$-action of $\CP^3$ making it a toric variety
which is induced by the natural $T^3_2$-action on $G\times T^3_2$ which behaves well-under 
the quotient map since the equivalence classes are mapped to equivalence classes under the action.
\end{prop}
\begin{proof} 
We basically induce a map from the first description of $\CP^3$ to the description of the proposition 
and show the equivalence. 

We first note that the lattice $L$ in $\bR^3$ generated by \[(4\pi,0,0), (0,4\pi,0), (0,0,4\pi), (2\pi, 2\pi, 2\pi)\]
has a basis \[(-2\pi, 2\pi, 2\pi), (2\pi, -2\pi, 2\pi), (2\pi, 2\pi, -2\pi).\] Note that the fiber over 
the interior point in the above is $\bR^3$ quotient out by this lattice.

We define a map $\mathcal D:\tri^* \ra G$ given by the affine map 
\[ \begin{bmatrix} x \\ y \\ z \end{bmatrix} \mapsto 
\begin{bmatrix} 0 & 1 & 1 \\ 1 & 0 & 1 \\ 1 & 1 & 0 \end{bmatrix} \begin{bmatrix}  x \\ y \\ z \end{bmatrix}
+ \begin{bmatrix} \pi \\ \pi \\ \pi \end{bmatrix}\]
Then $L^*$ is sent to $L$ by $2$ times the inverse transpose of the linear part of $\mathcal D$.
We now verify that the quotient space of the toric bundle over $G$ and the quotient space of the toric bundle over $\tri^*$ are 
in one-to-one correspondence by using $\mathcal D$ and the inverse transpose of the linear part for 
the fibers. The compactness and Hausdorff property give us
the homeomorphism property.
\end{proof}

\subsection{The Klein four-group action}

There is a natural Klein four-group 
action generated by 
three transformations $L$ sending a point $[z_0,z_1,z_2,z_3]$ to $[z_1,z_0,z_3,z_2]$ 
and $M$ sending $[z_0,z_1,z_2,z_3]$ to $[z_2,z_3,z_0,z_2]$ 
and $N$ sending $[z_0,z_1,z_2,z_3]$ to $[z_3,z_2,z_1,z_0]$
where $z_i$ are complex numbers.

The above action of $\{I, L, M, N\}$ 
corresponds to our $\{I, I_{{A}},I_{{B}},I_{{C}}\}$-action exactly under this map as given in the following equations
 \begin{eqnarray}\label{eqn:iabc5}
    I_{{A}}:(x, \phi_0, \phi_1, \phi_2) \mapsto 
    (I_{{A}}(x), \phi_0, 4\pi-\phi_1, 4\pi-\phi_2) \nonumber \\
    I_{{B}}:(x, \phi_0, \phi_1, \phi_2) \mapsto
    (I_{{B}}(x), 4\pi-\theta_0, \phi_1, 4\pi-\phi_2) \nonumber \\
    I_{{C}}:(x, \phi_0, \phi_1, \phi_2) \mapsto
    (I_{{C}}(x), 4\pi-\phi_0, 4\pi-\phi_1, \phi_2).
    \end{eqnarray}
(This will be used in the next section.)


\subsection{The classical theorem recovered}
\begin{thm}\label{thm:cp3-2} 
$\rep(\pi_1(\Sigma),\SUTw)$ is diffeomorphic to $\CP^3$ considered as a $T^3$-fibration over $G$
with the following properties: 
\begin{itemize} 
\item Each edge of $G$ corresponding to the region $A, A', B, B', C, C'$ of $\tilde G$
correspond a solid torus fibration over the interior of edges of $\tilde G$, the closure of which 
is homeomorphic to the quotient space of the solid torus times the interval with a $\bZ_2$-action 
at each boundary component giving an identification. Here, the  solid torus end is identified to a $3$-ball. 
\item Three of them  meet in a $3$-ball over each vertex of $\tilde G$ according to 
the pattern of the edges of $\tilde G$. 
\item The set of abelian characters $\chi_2(\Sigma)$ forms a subspace with an orbifold 
structure with 16 singularities.  It consists of the two-torus fibrations over faces of $G$ which meet
at the boundary components of the above solid torus fibration. 
\item The boundary torus of the solid torus fibration corresponding to an endpoint of the intervals
identifies to the boundary $2$-spheres of vertex $3$-balls by the $\bZ_2$-action.
\end{itemize}
\end{thm} 

We begin the proof divided into many subsections:

\subsection{Triangular characters}
A representation of $\pi_1(\Sigma)$ into $\SUTw$ gives us representations of two pairs of pants  
which are identical by the reason similar to Proposition \ref{prop:h0h1}.
By Proposition \ref{prop:pairsu}, the pair of pants representations correspond to generalized triangles. 
The triangles for the restricted representation for the two pairs of pants 
have the same angles in this case since we know that our space $\rep(\pi_1(\Sigma), \SUTw)$ 
is connected and it holds for at least one example. 
Therefore, we can form $\tilde G \times T^3_2$ as in the $\SOThr$-case
and introduce equivalence relations, where the pasting maps are rotations fixing the vertices of the parametrizing triangles
and they give us pasting angles in $T^3_2$. 
(Here, we are using triangles to denote a representation. The actual fixed point of $c_2$ is actually
the opposite the third vertex. But since everywhere, it is reversed, there are no problems.
We use triangles since they are concrete and easy to visualize.)

As an aside, we can consider triangles as a representation of punctured pair of pants with 
the puncture holonomy $-\Idd$. This approach can be followed but we will not pursue it here.

Here given three pasting angles in $T^3_2$, we see that 
the transformation \[(\theta_0,\theta_1,\theta_2) \ra (2\pi+\theta_0,2\pi+\theta_1,2\pi+\theta_2)\]
does not change characters as $d_1$ and $d_2$ do not change while the pasting maps $P_0, P_1, P_2$ are all 
changed by $-P_0,-P_1,-P_2$. 
Thus, we must use the generated $\bZ_2$-action to $\tilde G \times T^3_2$ first.
Actually, we limit ourselves to $\tilde G \times (T^3_2/\bZ_2)$ from now on. 



We find a description of $\rep(\pi_1(\Sigma),\SUTw)$ as a quotient space of $\tilde G \times T^3_2/\bZ_2$: 
For the open domain of triangular characters, a representation of $\pi_1(\Sigma)$ gives us 
a unique triangle on $\SI^2$ by Proposition \ref{prop:pairsu} and 
hence a unique set of pasting maps up to $\bZ_2$. Thus, the space of triangular characters is homeomorphic 
to $\tilde G^o \times T^3_2/\bZ_2$.
By density, the map \[\tilde G \times T^3_2/\bZ_2 \ra \rep(\pi_1(\Sigma),\SUTw)\] is onto.


Here, the equivalence relation is simply defined as follows, two points of $\tilde G \times T^3_2/\bZ_2$ is 
equivalent if they induce the same $\SUTw$-characters. But we will go through each face of $G$ below. 
Thus $\tilde G\times T^3_2/\bZ_2/\sim$ is homeomorphic to $\rep(\pi_1(\Sigma),\SUTw)$.

\subsection{The space of abelian $\SUTw$-characters}\label{subsec:ab}

We study the faces $a$, $b$, $c$, and $d$ of $G$, which correspond to abelian characters.
We denote by $\chi_2(\Sigma)$ the subspace of abelian $\SUTw$-characters in 
$\rep(\pi_1(\Sigma), \SUTw)$.
Given an abelian $\SUTw$-representation $h$, $h(c_1), h(c_2), -h(d_1),$ and $-h(d_2)$ can be represented 
as a point in $\SI^2$ with rotation angles in $T^4_2$. Since we can choose the fixed point 
to be a common one or its antipode, there is a $\bZ_2$-action on $T^3_2$
given by $(\theta_1,\theta_2,\phi_1,\phi_2)$ to $(4\pi-\theta_1, 4\pi-\theta_2, 4\pi-\phi_2,4\pi-\phi_2)$. 
Thus, the space $\chi_2(\Sigma)$ is homeomorphic to $T^3_2/\bZ_2$, where the 16 singular points are
given by $(\pm 2\pi, \pm 2\pi, \pm 2\pi, \pm 2\pi)$.

Reinterpreting in terms of pairs of pants decomposition, we see that 
each abelian character corresponds to a point of $\partial \tilde G\times T^3_2/\bZ_2/\sim$ by choosing 
$h(c_i)$ to have a fixed point with angles in $[0, 2\pi]$.

For regions, $a$, $b$, $c$, and $d$, 
the circle action is a simply lifted one from the circle action above for the $\SOThr$-case.
\begin{itemize}
\item For the face $a$, the equivalence relation on $a\times T^3_2$ is given by 
$(v,\phi_0,\phi_1,\phi_2) \sim (v', \phi'_0,\phi'_1,\phi'_2)$ if and only if 
$v=v'$ and two vectors \[(\phi_0,\phi_1,\phi_2) - (\phi'_0,\phi'_1,\phi'_2) = s(2\pi, 2\pi, 2\pi) \hbox{ for } s \in \bR.\] 
\item For faces $b$, $c$, and $d$, the equivalence relation is defined again using the respective $\SI^1$-action
generated by vectors parallel  to $(-2\pi, 2\pi, 2\pi), (2\pi, -2\pi, -2\pi), (2\pi, 2\pi, -2\pi)$ perpendicular to 
$b, c, d$ respectively.
\end{itemize} 

The interiors of $a\times T^3, b\times T^3, c\times T^3$, and $d\times T^3$ are mutually disjoint. 

By an analogous reasoning to Section \ref{subsec:abcd}, 
the fibers over the points of faces are $2$-tori. (The only we have to double the size of 
every vector in the argument here):
Looking at $a\times T^3_2/\sim$ in more detail, 
the equation determining the equivalence class representatives
should be $\phi_0+\phi_1+\phi_2=0$ as $4\pi=0$ in $\SI^1_2\times \SI^1_2\times \SI^1_2$
with an order-three translation action. 
We take the subset of $T^3_2$ given by this equation. Contemplating $T^3_2$ 
as a quotient space of the cube $[0,4\pi]^3$, we see that the plane 
corresponds to the union of two triangles with vertices 
$\tri_1=\tri([4\pi,0,0],[0,4\pi,0],[0,0,4\pi])$ and $\tri_2=\tri([4\pi,4\pi,0],[4\pi,0,4\pi],[0,4\pi,4\pi])$. 
Then two triangles form a $2$-tori $T'$ under the identifications of $T^3$. 
There is an identification by the above $\SI^1$-action which gives us an order-three 
translation action. 
$T'$ is isometric with a torus in the plane quotient out by a lattice generated by 
translation by $(4\pi\sqrt{2},0)$ and $(2\pi\sqrt{2}, 2\pi\sqrt{6})$. 
Then the order-three identification group is generated by a translation by $(2\pi, 2\pi\sqrt{6}/3)$. 
The quotient space $T^2_{2, a}$ is homeomorphic to a $2$-torus. 
Thus, the character space here is in one-to-one correspondence with 
$a\times T^2_{2, a}$. Similarly, we obtain $T^2_{2, b}$, $T^2_{2, c}$, and $T^2_{2, d}$
for respective faces $b$, $c$, and $d$. 

Recall that for regions, $A,B,C,A',B'$, and $C'$, note that fixing all three angles determines a segment, 
a tie, which joins a point of an edge of a face to a point in the edge of another face. 

We take a union of $a\times T^2_{2, a}$, $b\times T^2_{2, b}$, $c\times T^2_{2, c}$, and $d\times T^2_{2, d}$. 
Note that as we cross an edge through a tie from a face to another face, 
we change one of the vertex of a lune triangle to its antipode.
We have four faces of $\tilde G \times T^3_2/\bZ_2$ glued over the edges by isomorphisms
with mappings such as $(\phi_0,\phi_1,\phi_2) \ra (4\pi-\phi_0, \phi_1, \phi_2)$, 
$(\phi_0,\phi_1,\phi_2) \ra (\phi_0, 4\pi-\phi_1, \phi_2)$, and
$(\phi_0,\phi_1,\phi_2) \ra (\phi_0, \phi_1, 4\pi-\phi_2)$,
which occurs if we change one of the vertex to its antipode as we cross an edge.
(See third ones in each set of equations \eqref{eqn:obvious}, \eqref{eqn:obvious2}, \eqref{eqn:obviousA}, \eqref{eqn:obviousA'}, 
\eqref{eqn:obviousC} and \eqref{eqn:obviousCp}.) 
Therefore, the $2$-tori get identified across the edges by homeomorphisms.
(Over the interiors of edges the fibers are still homeomorphic to a $2$-torus.)
Around each vertex the gluing gives $-\Idd$ as the identification holonomy around the vertex, 
we have a $2$-sphere with four singularities which is obtained as $T^3_2/\sim$ by the $\bZ_2$-action 
generated by $-\Idd$. (Here by, $-\Idd$, we mean sending $(x,y)$ to $(4\pi-x,4\pi - y)$,
that is $-\Idd$ mod $4\pi$.)

Hence, we can consider the union as a fibration over $\partial \tilde G$ with fibers homeomorphic to 
$T^2$ except at vertices where the fibers are homeomorphic to a $2$-sphere which is 
the boundary of a three-ball parameterizing $\SUTw$-characters factored by the free group $F_2$ of rank $2$
up to changing some of $h(c_i)=\pm \Idd$ to $\Idd$.
(See Proposition \ref{prop:HFSU2} and Section \ref{subsec:aabbcc}) 
Here the sphere can be considered a quotient of a torus by $\bZ_2$-action generated by $-\Idd$ 
as in the above paragraph.

Actually, we take a quotient of the first two factors $T^2_2$ of $T^4_2$ first, 
we obtain a sphere with four singularities, corresponding to the vertices of $G$, 
which is identical with $\partial G$ in a natural way. 
In this way, we see naturally that $\chi_2(\Sigma)$ is a fibered space as described above. 
We will see how $\chi_2(\Sigma)$ imbeds in $\rep(\pi_1(\Sigma), \SUTw)$ in later subsections.


\subsection{Regions above $A, A', B, B', C, C'$.}\label{subsec:aabbcc}

We continue the proof from the above:
Now, we go over the regions $A$, $A'$, $B$, $B'$, $C$, and $C'$. 
Recall $U=I\cup II\cup III\cup IV$. 
We reason as follows:

\begin{lem}\label{lem:ABCtop} 
The subspace over a tie in one of the regions $A$, $B$, $C$, $A'$, $B'$, and $C'$ but 
not in $U$ is homeomorphic to $\SI^1\times B^2$.
Thus, over the interior of each of $A, A', B,B', C, C'$, there is a bundle over an open interval 
with fibers homeomorphic to the solid tori. 
If a tie is in $U$, the subspace over it is identical with the subspace over 
$I$, $II$, $III$, or $IV$ respectively and hence is homeomorphic to a $3$-ball
and can be considered as having been obtained from a $\bZ_2$-action on the solid torus.
Hence, the region above each of  $A, A', B,B', C, C'$ is homeomorphic to the quotient space of a solid torus 
times an interval with the solid torus over each end identified with a $3$-ball. 
\end{lem} 
\begin{proof} 

We work similarly to Section \ref{subsec:ABC}. 
We form $h(d_1)$ and $h(d_2)$ for a representation $h$ 
using triangle constructions from a configuration where we use $-P_0^{-1}$ instead of $P_0$ 
so that the construction is the same. 
Sometimes, we need to use a generalized triangle since we have angles in $[0, 4\pi)$
(See equation \eqref{eqn:cba2}. The generalized 
triangle with one vertex of angle in $[2\pi, 4\pi]$ and the other two angles 
in $[0, 2\pi]$ also need to be considered here in multiplications by geometry. ) 
Here, their fixed points and angles are determined from 
the geometric constructions. The circle action exists and here we do not need to change 
the fixed points of $h(d_1)$ and $h(d_2)$ up to antipodes.

First, suppose that we are on $B-U$. Let $h$ be a representation over a tie in $B-U$.
Then $h(c_1)$ and $h(c_2)$ have a pair of respective fixed points antipodal to each other.
Suppose that $h$ is nonabelian. 
Then $h(d_1)$ and $h(d_2)$ respectively 
have well-defined pairs of fixed points $\pm w_1$ and $\pm w_2$ distinct from this pair. 
Let $w_1$ and $w_2$ be ones with respective rotation angles for $h(d_1)$ and $h(d_2)$ in $[0, 2\pi]$.
(See Figure \ref{fig:topabc}.)

 \begin{figure}
\centerline{\includegraphics[height=5cm]{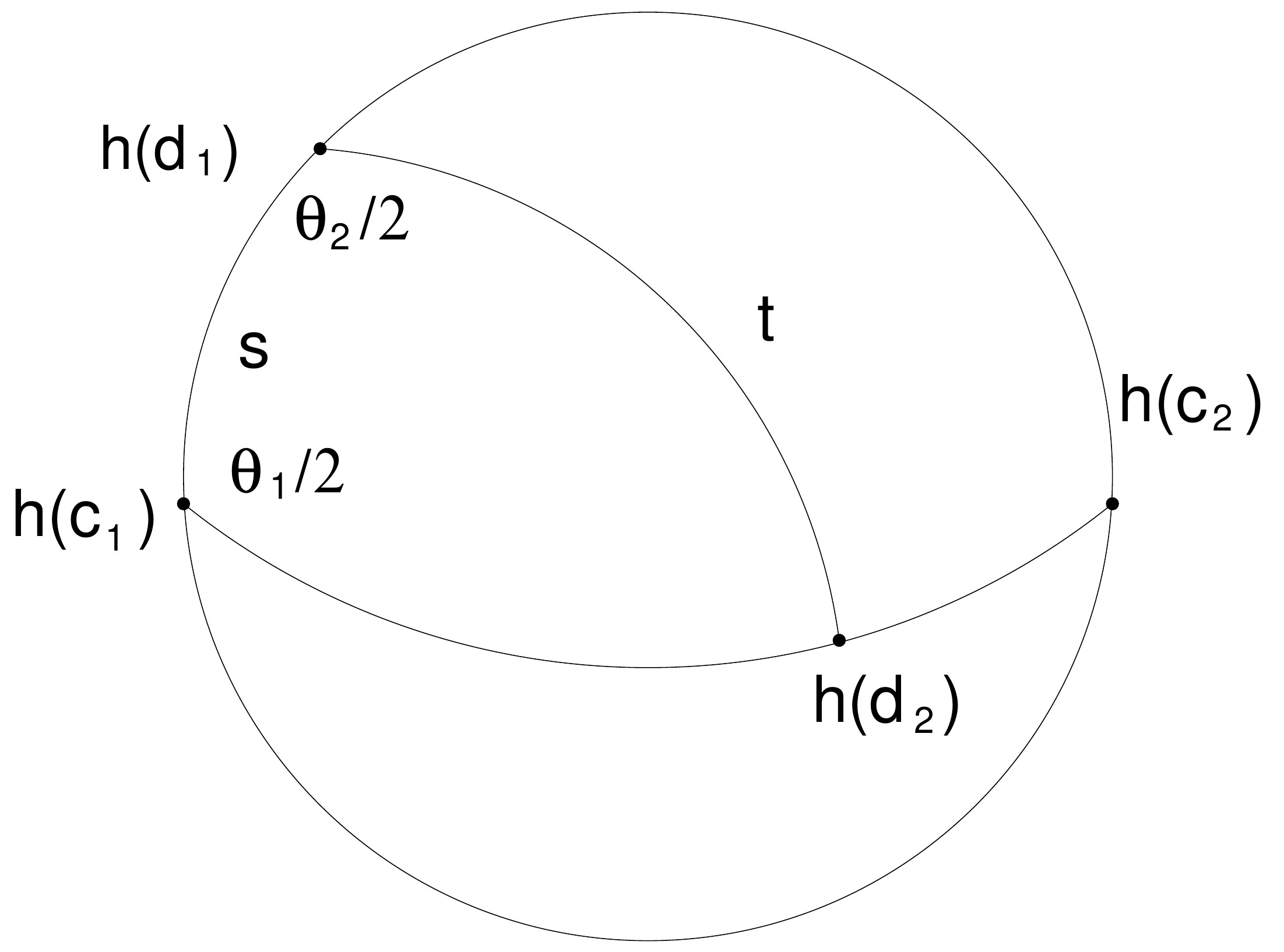}}
\caption{  Finding topology of space over regions $A-U$.}
\label{fig:topabc}
\end{figure}

The subspace of such characters is dense. 
We can parameterize such configurations by 
$(l, \theta_1,\theta_2)$ where $l$ is the length of 
the shortest segment $s$ between the fixed point $v_1$ of 
$h(c_1)$ to $w_1$ and hence $l \leq \pi$
and $\theta_1/2$ is the counter-clockwise angle at $v_1$ between $s$ and 
the shortest segment to the fixed point $w_2$ of $h(d_2)$ for $\theta_1 \in [0, 4\pi] \mod 4\pi$.
And $\theta_2$, $0 \leq \theta_2 \leq 2\pi$, is the angle of rotation of $h(d_1)$ and equals $2\pi$ minus the
twice of the angle at $w_1$ between $s$ and one of the geodesic segment $t$ between
the fixed points $w_1$ and $w_2$ corresponding to the angle of $\theta_1/2$ between $l$ and $s$. 
(The segment $t$ is unique if $w_1$ and $w_2$ are not antipodal.) 

Since the segments $l, s,$ and  $t$ bound a generalized triangle with angles $\theta_1/2$ at a vertex $v_1$ 
and $\theta_2/2$ at $w_1$ and $\theta_3/2$ at $w_2$. 
Then it follows that $2\pi-\theta_2$ is the angle of rotation of $h(d_1)$:
To see this, recall our holonomy constructions for $B$ and we are letting 
the segment $s_0$ between $v_1$ and $v_2$ 
pass through $w_2$, i.e., we are in a canonical position. See Section \ref{subsub:B-I-II}.
(See Figure \ref{fig:topabc}.) 

Since the fixed point of $h(d_2)$ may lie anywhere in $\SI^2$, 
it follows that $\theta_1$ can vary within $\SI^1_2$. 

The abelian characters over the tie in $B -U$ are in the limit set of the nonabelian characters. 
Thus, our space is $[0,\pi]\times \SI^1_2 \times [0,2\pi]$ quotient by some equivalence relation.
Here the equivalence relation, we have to consider, is when $\theta_2 =0, 2\pi$, i.e., when $h(d_1)$ 
is $I$ or $-I$ and thus the length $l$ is not defined well. 
\[(l,\theta_1,0) \sim (l', \theta_1, 0) \hbox{ and } (l,\theta_1,2\pi) \sim (l', \theta_1, 2\pi). \]
The identification takes place over the abelian subspace and 
we verify easily that there is no other identification. 


The quotient space is homeomorphic to a disk times a circle.
The circle corresponds to the second factor and the disk the other two factors.
The boundary region corresponds to the subspace of abelian characters, forming a two-torus. 

For the region $B'-U$, the reasoning is identical and the corresponding results hold.

For the region $A-U$, let $h$ be a representation over it.  
The fixed point $v_1$ of $h(c_1)$  is antipodal to the fixed point $v_0$ of $h(c_0)$. 
Thus, we  choose a fixed point of 
$h(d_1)$ be $v_1$ and the rotation angle to vary in $\SI^1_2$. 
Then the distance from $v_1$ to the fixed point of $h(d_2)$ and 
the rotation angle of $h(d_2)$ give us a parameterization 
$[0,\pi]\times \SI^1_2\times [0,2\pi]$. Again the equivalence relations 
occur when the third coordinate is zero or $2\pi$. The quotient space is 
homeomorphic to a solid torus with boundary containing all of the abelian 
representations in $C^A$. 

For the region $A'-U$, the reasonings are similar and the corresponding results hold. 

For the region $C-U$, let $h$ be a representation over it. 
The fixed point $v_2$ of $h(c_2)$ is antipodal to the fixed point $v_0$ of $h(c_0)$ 
for a triangle with vertices $v_0, v_1, v_2$ representing 
an element of $C-U$.
We choose a fixed point of $h(d_2)$ be $v_2$ and the rotation angle to 
vary in $\SI_2$. The distance from $v_0$ to the fixed 
point of $h(d_1)$ and the rotation angle of $h(d_1)$ and the rotation angle of $h(d_2)$ 
give us a parameterization $[0,\pi]\times \SI^1_2\times [0,2\pi]$.
The equivalence relation occurs when the third factor is equal to $0$ or $2\pi$.

For region $C'-U$, the reasoning is similar. 
 
For each of regions $I$, $II$, $III$, and $IV$, we obtain $3$-dimensional balls by Section \ref{subsec:free}:
For regions $I$, $II$, $III$, and $IV,$ the subspace of $\rep(\pi_1(\Sigma),\SUTw)$ over each 
corresponds to the set of representations factoring to homomorphisms from $F_2$ to $\SUTw$
by changing each representation $h$ by changing $h(c_i)$ to $-h(c_i)$ for $i=0,1,2$ to make it $\Idd$. 
Therefore, each region above $I, II, III,$ or $IV$ is homeomorphic to a $3$-ball. 

They are mutually disjoint and their boundary sets are the subspaces of abelian characters.
Moreover, each of this is parameterized by $(l, \theta_1, \theta_2)$ where $l$ is the distance 
between the fixed points of $h(d_1)$ and $h(d_2)$ and $\theta_i$ is the angle of rotation of 
$h(d_i)$ for $i=1,2$. See Section \ref{subsec:free2} for details. 

Thus, we showed that over each of the regions $A-U, A'-U, B-U, B'-U, C-U, C'-U$, 
we have a bundle over an open interval with solid torus fibers. 

If a tie is in an edge region intersected with one of a vertex region $I$, $II$, $III$, or $IV$, 
then the subspace over the tie is actually onto the subspace over the vertex region. 
This follows since given a representation $h$ over one of the ties,  
$h(c_i)$ are $\pm \Idd$  and just $h(d_1)$ and $h(d_2)$ are important 
and we obtain all possibilities of $h(d_1)$ and $h(d_2)$ on the tie by construction. 

For region $B$ intersected with region $I$,
from the information $(l, \theta_1,\theta_2) \in [0,\pi]\times [0,4\pi]\times [0,2\pi]$, 
we can construct a pair of fixed points  of $h(c_1)$ and $h(c_2)$
provided that we are given an orientation on the great circle containing the fixed point of 
$h(d_1)$ and $h(d_2)$.
Changing the orientation of the great circle 
has the effect $(l, \theta_1,\theta_2) \ra (\pi-l, 4\pi-\theta_1, \theta_2)$. 
Thus, the axis of this transformation is a union of two segments and the quotient space is 
a $3$-ball. 
(This is related to Proposition \ref{prop:Bpi} and the analogous propositions.)
Over the region $B$, the solid torus bundle over the interval gets identified with a $3$-ball by an involution over 
the end. (We call the phenomenon the {\em clasping}.)

For region $A$ intersected with the vertex region $I$, we also
have a transformation $(l, \theta_1,\theta_2) \ra (\pi-l, 4\pi-\theta_1, \theta_2)$. 
This gives a quotient space homeomorphic to $B^3$. 
For region $A'$, the reasoning is similar. 

For regions above ties in regions $C, C'$ intersected with the vertex region $I$, we can again switch $v_0$ and $v_2$. 
The transformation can be written in the same way as above. 

For intersections with vertex regions $II$, $III$, or $IV$, we omit the obvious generalizations of 
the proofs. 
\end{proof} 


Finally, we need (see Section \ref{subsec:obh} also):
\begin{prop}\label{prop:SU} 
$\rep(\pi_1(\Sigma),\SUTw)$  is homeomorphic to the quotient space of $\tilde G\times T^3_2/\bZ_2$ 
by the above equivalence relations over the regions above faces $a, b, c, d$, 
$A$, $A'$, $B$, $B'$, $C$, $C'$, $I$, $II$, $III$, and $IV$. 
\end{prop}
\begin{proof} 
Again there is a one-to-one map from $\tilde G\times T^3_2/\bZ_2/\sim$ 
to $\rep(\pi_1(\Sigma),\SUTw)$ defined by sending the configuration to the character it produces  from 
the generalized multiplication by geometry. 

The map is clearly onto 
since any character is reached by a triangular path. 
Thus, every character can be achieved by a configuration in $\SI^2$. 

The one-to-one property can be achieved by 
first considering the interior of $\tilde G \times  T^3_2/\bZ_2$ which corresponds
to triangular characters. This is achieved as in the $\SOThr$-case. 
For the boundary regions, the injectivity for each region times $T^3_2$ quotient by $\sim$ 
is clear. We note that the map restricted to $(a^o\cup b^o \cup c^o \cup d^o) \times T^3_2/\sim$ is injective
since the angles of $c_0, c_1, c_2$ are all in different ranges in the boundary of $G$. 
The restricted maps from 
\begin{multline}
(A-U)\times T^3_2/\sim, (A'-U)\times T^3_2/\sim, (B-U)\times T^3_2/\sim, (B'-U)\times T^3_2/\sim, \\
(C-U)\times T^3_2/\sim \hbox{ and } (C'-U)\times T^3_2/\sim
\end{multline}
have disjoint images since they are characterized 
by angles of $c_0, c_1, c_2$ are in disjoint ranges in $\partial G$.
The restricted maps from 
\[I\times T^3_2/\sim, II\times T^3_2/\sim, III\times T^3_2/\sim, \hbox{ and } IV\times T^3_2/\sim\]
have mutually disjoint images since 
$h(c_i)=\pm \Idd$ for every representation $h$ in each of them, and the signs are different for these regions. 
Furthermore, we see that all these regions list here map to disjoint subsets of the character space.
By the injectivity over $G^o\times T^3_2$, we obtain the injectivity of the map. 

As in the $\SOThr$-case, 
$\tilde G \times T^3_2/\bZ_2/\sim$ is compact and $\rep(\pi_1(\Sigma), \SUTw)$ is Hausdorff
since $\SUTw$ is a compact Lie group. 
Thus, the map is a homeomorphism. 
\end{proof}



\subsection{The proof of Theorem \ref{thm:cp3-2}.} 

We begin by noting that $\CP^3$ is a toric variety of a tetrahedron $G$. 
We described it as a quotient space of $G\times T^3_2/\bZ_2$ as  in Section \ref{subsec:cp3}.
The equivalence relation is the one given by 
the formula same as one on $G \times T^3_2/\bZ_2$ restricted to the faces $a$, $b$, $c$, and $d$ and 
there are no more. (See Proposition \ref{prop:cp3}.)
The interiors of the edges inherit two different equivalence relations and the fibers are homeomorphic to 
$\SI^1$ over each point of the interiors of the edges and the vertex inherit three different equivalence 
relations and the fiber is a point over each vertex. We call these the {\em vertices} of $\CP^3$.

We will now prove Theorem \ref{thm:cp3-2} by comparing it to $\CP^3$:
First, we will realize $\CP^3$ as a $T^3_2/\bZ_2$-fibration over $\tilde G$. 
Then we will remove the regular neighborhoods of the vertices. 
Compare the result to the $\SUTw$-character space 
with regular neighborhoods of vertex regions removed. 
This is accomplished by removing the regular neighborhoods of 
edge regions and comparing and comparing the topological 
types of the edge regions.


The quotient image of the four faces regions $a,b,c,$ and $d$ times $T^3_2/\bZ_2$ are four complex projective planes respectively
since each face corresponds to a standard simplex.
That is, we see that the standard $2$-simplex in the standard $3$-simplex $\tri^*$ 
correspond to copies of $\CP^2$ by the quotient relations (see equation \eqref{eqn:cp2}).

These faces homeomorphic to $\CP^2$ meet transversally at six $2$-spheres that fiber over the edges. 
Any three of the complex projective 
planes meet at a vertex which lie over a vertex of $G$. Any two of the six $2$-spheres meet at a vertex also. 
Moreover, three of the six $2$-spheres will meet at a vertex if their corresponding edges do.
(These are standard facts from Fulton \cite{Fulton} and Masuda \cite{Mas})



Let $F$ be the set of four points over the vertex of $G$. Then we choose $B$ to be an open regular $T^3$-invariant 
neighborhood of $F$ homeomorphic to $6$-dimensional balls as was chosen above. 
Thus $B$ maps to a union $B'$ of four three-balls that are neighborhoods of 
the vertices of $G$. Remove from $\CP^3$ the $T^3_2$-invariant regular neighborhood $B$.

Denote the resulting manifold by $N$ where $\partial N$ is a union of four $5$-spheres. 
This corresponds to removing regular open neighborhoods of four corners of the tetrahedron, the result of which we denote by $D$.
The edges of the tetrahedron correspond to six $2$-spheres meeting at four 
vertices. If we remove points in $B$, the six subspaces are homeomorphic to $\SI^1 \times I$. 
An open $T^3$-invariant regular neighborhood of each of the subspaces in $N$ is homeomorphic to 
$\SI^1\times B^4 \times I$
where the triviality follows since the neighborhood is orientable. 
And hence their boundary is homeomorphic to $\SI^1 \times \SI^3 \times I$. 
This corresponds to removing regular neighborhood of the six edges in $D$.

The above can be seen by the theory of tubes under a compact group action (See Bredon \cite{Bredon}.)

Denote by $M$ the space $\tilde G\times T^3/\bZ_2/\sim$. 
Finally, each of the regions over $I$, $II$, $III$, and $IV$ in $M$
is homeomorphic to $B^3$. 
Let $L$ be their union. Again, we find a neighborhood $N(L)$ of $L$ in $\tilde G\times T^3/\sim$ 
so that $N(L)$ with $L$ removed coincides with $B-V$ over the points 
of $G$ with edges removed. (This can be done by choosing a neighborhood $B''$ of the union of 
the vertex regions in $\tilde G$ corresponding to $B'$ 
under $\tilde G \ra G$  and taking the inverse image.)
Since $L$ has contractible components and any map into $N(L)$ can be
pushed into a contractible regular neighborhood of $L$, it follows that $N(L)$ is homotopy trivial.
Therefore, $N(L)$ is a union of four six-dimensional balls. 

Remove from $M$ the union $N(L)$ and call the resulting space $N'$.

By Lemma \ref{lem:ABCtop}, each of the quotient spaces over the strips $A$, $A'$, $B$, $B'$, $C$, and $C'$ in $N'$ 
is homeomorphic to $\SI^1\times B^2\times I$. Denote by $J_{{A}}$, $J_{{A}'}$, $J_{{B}}$, $J_{{B}'}$, $J_{{C}}$, 
and  $J_{{C}'}$ the corresponding spaces. Let $J$ be their union. 
By identifying $G$ with edges and vertices removed with 
 $\tilde G$ with the strips $A$, ${{A}'}$, $B$, $B'$, $C$, and $C'$ and regions $I$, $II$, $III$, and $IV$ removed, we
 see that $T^3_2$-fibrations are the same ones where defined. 
 
 Being homeomorphic to a character space $\rep(\pi_1(\Sigma), \SUTw)$ by Proposition \ref{prop:SU},  
 our space $\tilde G\times T^3_2/\bZ_2/\sim$ is an orientable manifold. 
 Let $N'(J)$ be the $T^3_2$-invariant set that is the neighborhood of $J$. 
Since we can homotopy any map into $N'(J)$ to a regular neighborhood of $J$ homotopy-equivalent to $J$, 
$N'(J)$ is homotopy equivalent to $J$, and thus,
$N'(J)$ is homeomorphic to the union of six copies of $\SI^1 \times B^4 \times I$

Let us denote the union $V$ of the above six subspaces over the edges homeomorphic to $\SI^1\times I$ in $N$ by $V$,
and let $N(V)$ denote the regular $T^3_2$-invariant neighborhood of $V$,
which is homeomorphic to $\SI^1\times I\times B^4$.
We can choose $N'(J)-J$ and $N(V)-V$ homeomorphic by choosing their basis images 
under the above $T^3_2$-actions to be the same.
Moreover, $N(V)$ and $N'(J)$ are homeomorphic
since they are six-dimensional compact orientable manifolds homotopy equivalent to circles.

Remove from $N$ the regular $T^3_2$-invariant neighborhood $N(V)$ 
and remove from $N'$ the regular $T^3_2$-invariant neighborhood $N'(J)$. 
Now consider the description of $G\times T^3_2/\sim$ and $N-N(V)$ as a torus fibration 
with equivalence relations that are same. (See Section \ref{subsec:ab}.)
By our description, it follows that $N -N(V)$ and $N'-N'(J)$ are homeomorphic. 
We conclude that $N'$ is homeomorphic to $N$.
(We remark that the homeomorphism preserves the parameter $I$. )

Note that $\chi_2(\Sigma)$ in $N'-N'(J)$ correspond to $a''\cup b''\cup c''\cup d'' \times T^3_2/\sim$
where $a''$, $b''$, $c''$, and $d''$ are disks in the faces $a$, $b$, $c$, and $d$ respectively
and hence it is homeomorphic to four $2$-cells times $2$-tori.
They correspond to four copies of $\CP^2$ in $N$.

Adding $N(L)$ to $N'$, we recover $M$. Therefore $M$ is homeomorphic 
to $N$ with four six-balls reattached and hence $M$ is homeomorphic to $\CP^3$. 

Let us discuss about the subspaces over the edge regions $A$, ${{A}'}$, $B$, $B'$, $C$, and $C'$. 
By Lemma \ref{lem:ABCtop} each of the subspace is a union of solid torus time intervals 
over the $1$-skeleton of a tetrahedron meeting each other at a $3$-ball 
over a vertex, where the boundary of the torus fiber and 
$3$-ball fiber here consists of abelian characters. 


Let $S$ be one of the $5$-spheres that are a boundary component of $N'$.
Assume without loss of generality it meets $A, B, C$.  
Then $S \cap (N' -N'(J))$ correspond to the circle $k$ in $\partial G$ with three closed 
arcs in $A, B, C$ in the boundary removed, and
$\chi_2(\Sigma) \cap S \cap (N' -N'(J))$ 
is a union of three imbedded copies of $T^2 \times I$. 
The abelian part over the each tie in one of the edge regions
is the boundary $2$-torus of the solid torus. 
As we decrease $N'(J)$ to $J$ by thinning the regular neighborhood, 
three imbedded copies of $T^2 \times I$ will meet on these boundary $2$-tori. 
The two copies get identified across ties in the edge regions by 
a reflection in one of the three gluing parameter coordinates.
Hence, their composition equals $-\Idd$ in $T^3_2$ in gluing parameter coordinates before taking 
the quotients. 
The union is then homeomorphic to $T^2_2 \times I$ by identifying 
the top and the bottom where the identification homeomorphism is $-\Idd$. 
We call the result {\em a twisted $3$-dimensional torus}. 
(See Subsection \ref{subsec:ab} also.)

Now we can vary $S$. 
If we choose a parameter of boundary spheres $S_t$ of $N$ so that the corresponding image circle 
bounds disk $D_t$ in $\tilde G$ and $D_t$
converges to one of the regions $I$, $II$, $III$, and $IV$
as $t \ra 0$. Then the parameter of $5$-spheres $S_t$ converges to the set above each of the regions, i.e., a three-ball
representing the homomorphism factoring into a free group $F_2$-characters. 
The parameter of the twisted $2$-dimensional tori in $S_t$ will geometrically converge as $t \ra 0$ 
to the boundary of the three-ball since the twisted
torus corresponds to abelian characters which must accumulate to abelian 
characters in the subspace of free group characters, i.e., the boundary 
of the three-ball.
Each of the parameters of the three solid tori in $S_t$ geometrically converges to the three-ball as $t\ra 0$
since the solid torus corresponds to the space of characters over the edges which 
must converge to a free group character modulo changing $h(c_i)$ to $\pm h(c_i)$. 
(See Lemma \ref{lem:ABCtop} also for the clasping phenomena.)
This identification of course extends to the boundary 
tori also, i.e., the abelian characters. 

This completes the proof of Theorem \ref{thm:cp3-2}. 

\begin{rem}\label{rem:ab}
The characters over $A$, ${{A}'}$, $B$, $B'$, $C$, and $C' $ form a 4-dimensional CW-complex 
meeting in a $3$-dimensional cell over the regions $I$, $II$, $III$, and $IV$. They are twistedly imbedded. 
We can picture them as a solid torus times interval ``clasping'' itself and two or three of them meet 
at the clasped $3$-balls if their regions meet. 

\end{rem}

\section{The topology of the quotient space $\tilde G \times T^3/\sim$ or ${\mathcal C}_0$} 
\label{sec:topC0}

In this section, we will determine the topology of the quotient space
$\tilde G \times T^3/\sim$. We will show that it is branch-covered by 
$\CP^3$. 

First, we discuss the group action by $\bZ_2^4$ by considering 
the Klein four-group action on the fibers and the Klein group action on $\tilde G$. 

Clearly, there is a group $V'$ of order 16 action on $\tilde G\times T^3_2/\bZ_2/\sim$ generated by
the $\{\Idd, I_{{A}},I_{{B}},I_{{C}}\}$-action as given by equations \eqref{eqn:iabc5}
and the Klein four-group acting on each of the fibers $\SI^1_2\times \SI^1_2\times \SI^1_2/\bZ_2$ 
defined by 
\begin{eqnarray}
 i_a &: &(\theta_0,\theta_1,\theta_2) \mapsto (\theta_0+2\pi,\theta_1+2\pi,\theta_2) \nonumber \\
i_b &: & (\theta_0,\theta_1,\theta_2) \mapsto (\theta_0,\theta_1+2\pi,\theta_2+2\pi) \nonumber \\ 
i_c &: & (\theta_0,\theta_1,\theta_2) \mapsto (\theta_0+2\pi,\theta_1,\theta_2+2\pi).
\end{eqnarray}

These transformations preserve equivalence classes:
Each of $I_{{A}}, I_{{B}}, I_{{C}}$ is described as a geometric transformation 
of taking two vertex points by their antipodes (see equation \eqref{eqn:iabc5})
and each of $i_a, i_b$, and $i_c$ amounts to multiplying two of the pasting maps by $-\Idd$ is $\SUTw$.
Hence, they are transformations on the set of the equivalence classes of $\SUTw$-representations
whose quotient is the space of $\SOThr$-representations in ${\mathcal C}_0$. 

The group $V'$ is a Klein four-group extended by a Klein four-group:
We show that $\{I, I_{{A}}i_a, I_{{B}}i_b, I_{{C}}i_c\}$ is a subgroup
by considering the coordinate form of the action on 
$G^o\times T^3_2/\bZ_2/\sim$, we see that the group
is isomorphic to a product group $V\times V$ and hence isomorphic to $\bZ_2^4$.


\subsection{The fiberwise group action}

We discuss the fiberwise Klein four-group action over each of the regions. 
In the regions over the interiors of the faces of $G$, the action is 
free as we can see from the quotient formula.

Each region over $I, II, III$, or $IV$ is a $3$-ball and can be identified with a tetrahedron $G$
since we can change $h(c_i)=\pm \Idd$ to $\Idd$ and obtain a representation of $F_2$ to $\SUTw$
which gives a homeomorphism between these and $H(F_2)$. 
(See Section \ref{subsec:free} and Proposition \ref{prop:pairsu}.)
\begin{prop}\label{prop:actions}
The action of $\{I, i_a, i_b, i_c\}$ on the regions over each $I$, $II$, $III$, and $IV$, can be seen 
as the Klein group action on a tetrahedron if we identify it with a tetrahedron $G$. 
That is, $\{I, i_a, i_b, i_c\}$ that acts fiberwise 
on  regions above $I$, $II$, $III$, and $IV$ as involutary rotations about 
axes parallel to the coordinate direction through $(\pi/2,\pi/2,\pi/2)$.
\end{prop}
\begin{proof}
Consider a region $I$. 
The transformation $i_a$ adds $2\pi$ to the pasting angle for $c_1$ and $c_2$. 
This means taking the antipodal fixed point in
$\SI^2$ and changing the respective rotation angles to be $2\pi$ minus the original ones.
The third coordinate is the angle of rotation about a product, which does not change in 
this case. This is an involutary rotation about the one of the axis. For $i_b$ and $i_c$, the arguments are similar.
By an analogous consideration for each of the regions $II$, $III$, and $IV$, we are done. 
\end{proof}

Recall from Lemma \ref{lem:ABCtop} that the subspace over a tie determined by 
the angles in each of $A$, $B$, $C$, $A'$, $B'$, and $C'$ with $U=I\cup II\cup III\cup IV$ removed 
admits the quotient space under $\sim$ homeomorphic to $\SI^1\times B^2$.

By a {\em circle} we mean a subspace homeomorphic to the circle in a solid torus where 
the imbedding is a homotopy equivalence. 
An {\em axis} circle is a subset of  fixed points of an involution homeomorphic to a circle. 

\begin{lem}\label{lem:ABCtop2} 
For the Klein four-group $\{I, i_a, i_b, i_c\}$ which acts on the solid torus over each tie in
of  $A$, $B$, $C$, $A'$, $B'$, and $C'$  with $U=I\cup II\cup III\cup IV$ removed.
The maps  are involutions. 
\begin{itemize}
\item $i_a$ has the axis circle in 
the solid torus over each tie in $A-U$ and $A'-U$
where the other two $i_b$ and $i_c$ are without fixed points.
\item $i_b$ has  the axis circle over the solid torus over each tie in $B-U$ and $B'-U$ 
where the other two $i_a$ and $i_c$ have no fixed points.
\item $i_c$ has the axis circle over the solid torus over each tie in $C-U$ and $C'-U$
where  the other two $i_a$ and $i_b$ have no fixed points.
\item The result of the quotient of the each region by 
the action of $\{I, i_a, i_b, i_c\}$ is homeomorphic 
to $\SI^1 \times B^2$, i.e., a solid torus. 
\end{itemize}
\end{lem} 
\begin{proof} 

First, we discuss the region $B-U$. Let $h$ be a representation over $B-U$.
By the circle action described above, we can find a great circle 
in $c^1_1$ containing the two fixed points of 
$h(d_1)$ and $h(d_2)$ and a great circle $c^1_2$ containing the pair $P$ 
of antipodal points fixed by $h(c_i)$ for $i=1,2$. 


We recall Lemma \ref{lem:ABCtop}.
Suppose first that $h(d_1)$ and $h(d_2)$ have the pairs of fixed points distinct from $P$. 
Then their respective fixed points $f_1$ and $f_2$ are on a great circle $c^1_1$.
We suppose that the circle does not contain $P$ as well. 
The other cases are in the closure of subspaces satisfying this condition as before 
and hence we can use the condition to parameterize our subspace of $B$.

Since we are on a particular segment in $B$ determined by $\theta_1$, we can
choose a fixed point $v_1$ of $h(c_1)$
so that the angle of rotation at $v_1$ is the one given by $\theta_1$.

We parameterize such configurations by 
$(l, \theta_1,\theta_2)$ where $l$ is the length of one of
the shortest segment $s$ between the fixed points of 
$v_1$ to the fixed point of $f_1$ and hence $l \leq \pi$
and $\theta_1/2$ is the angle between the great circle containing $s$
the shortest segment to $f_2$
and $\theta_2$ is the angle of rotation of $h(d_1)$. 
(See Figure \ref{fig:topabc}.)
Recall that $\theta_1$ is in $\SI^1_2$ and $\theta_2$ in $[0,2\pi]$.
Thus, our space is parameterized by $([0,\pi]\times \SI^1_2\times [0,2\pi])/\sim$ where 
$\sim$ is given by Lemma \ref{lem:ABCtop} when $\theta_2 = 0, 2\pi$.

Now consider the action by $\{I, i_a, i_b, i_c\}$. 
This amounts to multiplying the pasting angles by $-\Idd$. We start with $i_b$:
Again $i_b$ will transform by multiplying $h(d_1)$ and $h(d_2)$ by $-\Idd$. 
This amounts to changing the fixed point to its antipode and changing the angles 
to $2\pi$ minus the original one by equation \ref{eqn:minusid}. This is identical with sending $f_1$ and $f_2$ to 
its antipodal points. Thus $(l,\theta_1,\theta_2)$ is mapped to $(\pi-l, \theta_1, 2\pi-\theta_2)$. 
Thus an axis circle of fixed point of $i_b$ is a circle given by $l=\pi/2$ and $\theta_2=\pi$. 

Next, $i_a$ transforms by multiplying $h(d_2)$ by $-\Idd$, 
and $(l,\theta_1,\theta_2)$ is mapped to $(l, \theta_1+2\pi, \theta_2)$. 
Finally, $i_c$ transforms by multiplying $h(d_1)$ by $-\Idd$, and 
$(l,\theta_1,\theta_2)$ is mapped to $(\pi-l, \theta_1+2\pi, 2\pi-\theta_2)$
and $i_a$ and $i_c$ do not have any fixed points. 

The fundamental domain of the action is given by
\[[0,\pi/2]\times [0, 2\pi]\times [0, 2\pi]/\sim,\]
and an equivalence
\begin{eqnarray}
(\pi/2,\theta_1,\theta_2) &\sim & (\pi/2, \theta_1, 2\pi-\theta_2) \nonumber \\
(l, 0, \theta_2) &\sim & (l, 2\pi, \theta_2)
\end{eqnarray}
Thus, the quotient space is homeomorphic to $\SI^1\times B^2$.

For region $B'$, the reasoning is very similar, and we also have that $i_b$ is 
a rotation about an axis circle of the solid torus of angle $\pi$ 
and $i_a$ s a translation by $\pi$-amount. $i_c$ is a composition of the two.

For region $A$, we have the vertex denoted by $v_0$ is antipodal to $v_1$ and $v_2$ is on the segment between them.
Let $h$ be a representation over $A$. 
Again, we can assume that the fixed point $f_1$ of $h(d_1)$ and the one $f_2$ of $h(d_2)$ are well-defined
on a dense set. 
Thus, we can choose the fixed point $f_1$ to lie on $v_1$ and 
the fixed point $f_2$ of angle in $[0, 2\pi]$ is on the sphere. We coordinatize by 
taking $l$ to be the distance between $v_0$ and $f_2$ and $\theta_1$ to be 
the rotation angle of $h(d_1)$ at $v_1$ and $\theta_2$ to be the rotation 
angle of $h(d_2)$ at $f_2$. 
Here, since $i_a$ sends $h(d_2)$ to $-h(d_2)$, we have that 
$i_a(l, \theta_1, \theta_2)=(\pi-l, \theta_1, 2\pi-\theta_2)$, and $i_a$ has an axis circle of 
fixed points given by $l=\pi/2, \theta_2 = \pi$. 
Since $i_c$ sends $h(d_1)$ to $-h(d_1)$, it follows that $i_c$ is a translation
$i_c(l, \theta_1, \theta_2)=(l, \theta_1+2\pi, \theta_2)$. 
Again $i_b$ is a composition of the two and hence has no fixed point. 
We have a same result for the region ${{A}'}$. 

For regions $C$ and $C'$, we have that $i_c$ is a rotation about an axis circle of angle $\pi$ and $i_a$ is a translation in
the axis direction and $i_c$ is the composition of the two, similarly. 

From the description of the actions, we see that each of the quotient spaces 
is homeomorphic to a solid torus.
\end{proof} 



\begin{prop}\label{prop:prequot} 
The fiberwise Klein four-group action on $\tilde G\times T^3_2/\bZ_2/\sim$ is described as follows:
\begin{itemize}
\item The fiberwise Klein four-group action consists of order-two translations
in the open domain above the interior of $\tilde G$ and have no fixed points. 
So is the action on the fibers over the faces of $\tilde G$. 
\item $i_a$ is an involutionary rotation about the union of two disjoint $2$-sphere $\SI^2_{A}$ 
and $\SI^2_{A'}$. Similarly $i_b$ is a rotation about 
the union of two disjoint $2$-spheres $\SI^2_{B}$ and $\SI^2_{B'}$, and 
$i_c$ is one about the union of $2$-spheres $\SI^2_{C}$ and $\SI^2_{C'}$.
\item The spheres $\SI^2_{{A}}$, $\SI^2_{{A}'}$, $\SI^2_{{B}}$, $\SI^2_{{B}'}$, $\SI^2_{{C}}$, and $\SI^2_{{C}'}$ 
 will be above each edge region with the same label. 
\item The spheres meet in the region above $I$, $II$, $III$, and $IV$ 
in the axes of the corresponding action in the $3$-ball. 
\item The intersection pattern is given by the intersection pattern of the axes of the action. 
That is, the two meet transversally at a point if and only if the corresponding edges in $G$ do. 
The three meet transversally at a point if and only if the corresponding edges in $G$ do.
\item The quotient space has an orbifold structure with singular points the images of the above 
spheres which are again spheres with same incidence relations. The local group 
of the singular points is $\bZ_2$ in the nonvertex point and is a Klein four-group at the vertices where 
any two of them meet. This can be made into a smooth orbifold structure.
\end{itemize}
\end{prop}
\begin{proof}

For the interior part, the first item is clear. 

In the regions over the faces, the $\{I, i_a, i_b, i_c\}$-action is also free: consider these actions as 
translations in the tori $T^3_2/\sim$ where $\sim$ is generated by vectors parallel to 
the normal vectors of form $(\pm 2\pi, \pm 2\pi, \pm 2\pi)$ where the signs depend on which 
faces you are on:

Consider the face $a$ of $\tilde G$ for now. Then $i_a, i_b,$ and $i_c$ correspond to 
translations by $(0, 2\pi, 2\pi)$, $(2\pi, 0, 2\pi)$, and $(2\pi, 2\pi, 0)$ respectively. 
Then these vectors can be seen to be parallel to the plane of the quotient tori 
with generators $(4\pi,0,-4\pi)$ and $(0, 4\pi, -4\pi)$ since we can change 
the signs of translation vectors mod $4\pi$. 
We see easily that these translations are not trivial and the fiberwise 
Klein four-group acts freely here.  (See Section \ref{subsec:ab} for details.)
Similar considerations show that the fiberwise Klein four-group action is free over 
the faces $b, c$ and $d$. 

Let us now go over to the regions $A, B,C,A',B'$, and $C'$:
We construct the two sphere $\SI^2_{{B}}$ by taking the axis circle of $i_b$ over each solid tori in $B$ removed with 
the vertex regions, as obtained from the ties. The parameter of axis circles in the parameter of the solid tori 
over the ties in $B -U$ will geometrically converge to 
an axis, a segment in the vertex $3$-ball fixed by $i_b$, by approaching the vertex region and finally 
will double branch-cover the segment with two singular points (see Theorem \ref{thm:cp3-2}). 
Thus, we obtain a sphere by compactifying the two ends of an annulus consisting of axis 
by two segments by Conner's work \cite{Co}. 
($i_a$ and $i_c$ do not have fixed points here and they are not considered. )

The same constructions will give us all other spheres $\SI^2_{{B}'}$, $\SI^2_{{A}},\SI^2_{{A}'},\SI^2_{{C}},$ and 
$\SI^2_{{C}'}$.

Note that the two spheres will meet if their compactifying edges will meet. Hence, the result follows. 

We can assume that they are smooth and have intersections so that their tangent planes meet only at 
the origin. (Thus, we can give an invariant Riemannian structure.)

Now we go over the action. For $i_a$ will fix axes circles above A and {{A}'}. 
Since they fix the axis in the vertex $3$-balls as well,
$\SI^2_{{A}}, \SI^2_{{A}'}$ will be fixed. The action is a rotation of angle $\pi$ in the solid tori and will 
be identity in the complementary $2$-dimensions to the sphere and the solid tori. 
The similar reasoning also hold for $i_b$ and $i_c$. 

Since the quotient of the actions are always locally linear and the local groups are finite, 
we obtain that the quotient has an orbifold structure. 
Since these actions are finite, we can put the smooth orbifold structure on the quotient space.
\end{proof}


\subsection{The topology}

\begin{thm}\label{thm:cp3}
$\tilde G\times T^3/\sim$ is homeomorphic to a quotient of $\CP^3$ under 
the product of the two Klein four-group actions generated by fiberwise and axial action. 
The branch loci of $I_{{A}}, I_{{B}}, I_{{C}}$ are given as follows:
six $2$-spheres corresponding to the axes of $I_A,I_B,$ and $I_C$. 
There are two $2$-spheres over each axis, and 
over each axis, the two $2$-spheres are disjoint. 
Any two $2$-spheres over respectively different axes meet at 
 unique points over the center  $(\pi/2, \pi/2, \pi/2)$ of $\tilde G$.
All three $2$-spheres over different axis meet at the same point as above.  
\end{thm}
\begin{proof} 


Now we consider $\{I, I_A, I_B, I_C\}$-action on $\tilde G\times T^3_2/\bZ_2/\sim$.

Let us consider $I_{{B}}$ first.  Recall that a mid-tie is a tie in the middle of the strip, i.e., the half way locus along 
the angular parameterization.
The solid tori in the edge regions are permuted or if the edge region is marked by $B$ or $B'$, 
then $\{\Idd, I_{{B}}\}$ acts on it and in particular $\{\Idd, I_{{B}}\}$ acts on the solid torus over the middle tie.

$\{\Idd, I_{{B}}\}$ acts on the triangles corresponding to a representation above 
the middle tie by sending a triangle to another triangle by replacing $v_1$ by $-v_1$ and $v_2$  by $-v_2$ and fixing $v_0$. 
Also, the pasting angles behave $(\theta_0,\theta_1,\theta_2) \mapsto (\theta_0, 4\pi-\theta_1,4\pi-\theta_2)$. 
Thus, the solid torus coordinate introduced above the mid-tie in B or B' behaves 
$(l,\theta_1,\theta_2) \ra (l, 4\pi - \theta_1, 2\pi-\theta_2)$. 
Hence, the two fixed segments are given by equations $\theta_1 = 0, 2\pi$ and $\theta_2 = \pi$ 
(as $\theta_2 \in [0, 2\pi]$ an interval).

On the $3$-torus over the fixed axis in the interior of $\tilde G$, the pasting maps are also sent the same way. 
Thus $(\theta_0,\theta_1,\theta_2)$ is sent to $(\theta_0, 4\pi-\theta_1,4\pi-\theta_2)$, 
and there are four fixed axes given by $\theta_1=0,2\pi, \theta_2=0,2\pi$, which are four circles. 
By the $\bZ_2$-action, there are actually only two circles of fixed axes over these points.
Thus, the union of entire fixed axes over the axis of $I_{{B}}$ in $\tilde G$ correspond to 
a union of two annuli compactified by segments to form two $2$-spheres by Conner's work \cite{Co}
as $I_B$ is a smooth involution.


The action of $\{\Idd, I_{{B}}\}$ in $M$ permutes the four $6$-dimensional balls which 
are regular neighborhoods of the subspaces corresponding to the regions $I$, $II$, $III$, and $IV$
In $\CP^3$, this is true with the four $6$-dimensional balls described in the above proof. 
Thus, the action of $\{\Idd, I_{{B}}\}$ on $\CP^3$ is conjugate to that on $M$ by our identification above. 
Therefore, $I_{{B}}$ is an involution with a fixed point a union of two disjoint $2$-sphere.

Similar conclusions can be established about $I_{{A}}$ and $I_{{C}}$. The corresponding spheres meet at a $3$-torus over
at the center $(\pi/2, \pi/2, \pi/2)$ of $\tilde G$. Since they meet the $3$-torus in orthogonal circles, 
we see that they meet at unique points:

\end{proof}

\begin{rem} \label{rem:finalorbstr}
Thus, we see that the final orbifold structure of ${\mathcal C}_0$ is simply given by $\CP^3$ with 
six $2$-spheres denoted by $s_{{A}}, s'_{{A}}, s_{{B}}, s'_{{B}}, s_{{C}}$, and $s'_{{C}}$ corresponding to involutions $I_{{A}}, I_{{B}}$,
and $I_{{C}}$ 
and the six $2$-spheres  $\SI^2_{{B}}$, $\SI^2_{{B}'}$, $\SI^2_{{A}},\SI^2_{{A}'},\SI^2_{{C}}$, and $\SI^2_{{C}'}$
in $\tilde G\times T^3_2/\sim$ which form the branch  locus of the $\{I, i_a, i_b, i_c\}$-action. 
Denote the spheres by $S_{{A}}, S'_{{A}}, S_{{B}}, S'_{{B}}, S_{{C}},$ and $S'_{{C}}$ respectively. 

The three spheres $s_{{A}}, s'_{{A}}, s_{{B}}, s'_{{B}}, s_{{C}},$ and $s'_{{C}}$ all meet at a point in the center of the tetrahedron.
$S_{{A}}, S'_{{A}}, S_{{B}}, S'_{{B}}, S_{{C}},$ and $S'_{{C}}$ also meet at a point in the ``vertex" region. 
We see that $s_{{A}}$ and $S_{{A}}$ meet at a point in the middle of an ``edge" region.
and $s_{{B}}$ and $S_{{B}}$ do and so do $s_{{C}}$ and $S_{{C}}$. Also, we can insert 
primes any where to get the incidences.
There are no other incidence type. 

We can now describe the orbifold structure on ${\mathcal C}_0$.
The twelve spheres above have order two singularities 
everywhere. There are two points with local group $\bZ^3_2$ and there are three
with local group $\bZ^2_2$.

Finally, the image of the first set of six $2$-spheres map to three $2$-spheres and 
and the second set of six $2$-spheres map to another set of three $2$-spheres.

\end{rem}


\section{The other component}\label{sec:other}

In this section, we study the other component ${\mathcal C}_1$. 
We follow the basic strategy as in ${\mathcal C}_0$ case. 
First, we introduce an octahedron $O$ and its blown-up compactification $\tilde O$
to parameterize the characters. This octahedron characterizes the characters of 
the two pairs of pants $S_0$ and $S_1$.
Again, we compactify the octahedron by blowing up vertices into squares
to prepare for the study of the characters of the fundamental group of
the surface itself. 
Then we introduce equivalence relation 
so that $\tilde O \times T^3/\sim$ becomes homeomorphic to ${\mathcal C}_1$. 
This will be done by considering the interior and each of the boundary regions as in the previous sections.
In the next sections, we will show that the quotient space is homeomorphic to an octahedral 
manifold.

\subsection{The triangle region for the characters}

We will now discuss the other component, i.e., the component 
containing a representation $h^*$ in Section \ref{subsec:twocomp} which correspond to 
one with an upper triangle and a lower triangle that are
identical standard lunes with vertices $[1,0,0]$ and $[-1,0,0]$ and two 
edges passing through $[0,-1,0]$ and $[0,0,1]$ where the vertices $v_0, v_1, v_2$ of 
the upper triangle and the vertices $v_0', v'_1, v'_2$ of the lower triangle 
satisfy $v_0=v'_0=[1,0,0], v_1=v'_1=[-1,0,0]$ and $v_2=v'_2=[0,-1,0]$
and pasting maps $P_0$ and $P_2$ identity and 
$P_1$ the rotation of angle $\pi$ about $[0,0,1]$ sending $v_1$ to $-v'_1$. 

A {\em triangular} representation is a representation such that the restriction to 
the fundamental groups of the pairs of pants $S_0$ and $S_1$ are triangular.

We can characterize the characters in triangular representations in ${\mathcal C}_1$ to be ones where 
each corresponding representation gives two triangles where each pasting map 
send a vertex $v_i$ of the first triangle $\tri_1$ to the corresponding vertex $v'_i$
of the second triangle $\tri_2$ or to the antipode $-v'_i$ of the vertex of the second 
triangle and at least one pasting map $P_i$ sending $v_i$ to $-v_i$ exists always.

A Klein group $V$ acts on the upper triangle and the lower triangle 
by sending the vertices of the upper triangle $v_0, v_1, v_2$ to  the vertices
$\pm v_0, \pm v_1, \pm v_2$ for a triangle where 
there are two negatives or no negatives and sending 
the vertices of the lower triangle $v'_0, v'_1, v'_2$ to 
$\pm v'_0, \pm v'_1, \pm v'_2$ with corresponding sign as the upper triangle case. 
In $G^o$, we have the formula as in the $\SOThr$-case.


By the Klein group action, we can make the triangular representation be so that 
the pasting map $P_0$ sends $v_0$ to $v'_0$, $P_1$ sends $v_1$ to $-v'_1$ 
and $P_2$ sends $v_2$ to $v'_2$ and hence satisfy the angle condition \eqref{eqn:anglecon}.


\begin{lem}\label{lem:nearh} 
The subset of triangular characters in ${\mathcal C}_1$ is a dense open subset. 
Any character whose associated triangles 
is degenerate can be pushed to triangular ones by a path of 
deformations. 
Every character in the component ${\mathcal C}_1$ is associated 
with generalized triangles $(\tri_0, \tri_1)$ whose associated angles are 
\begin{equation}\label{eqn:anglecon}
(\theta_0, \theta_1, \theta_2) \hbox{ and } 
(\theta_0, \pi-\theta_1, \theta_2).
\end{equation} 
\end{lem}
\begin{proof} 

Choose a one near $h^*$ that is triangular
and hence it is represented by two generalized triangles 
$(\tri_0, \tri_1)$ with the angles $(\pi/2, \pi/2, \pi-\eps)$ and 
$(\pi/2, \pi/2, \pi-\eps)$ respectively and with vertices 
$v_0, v_1, v_2$ of $\tri_0$ and $v_0', v'_1, v'_2$ of $\tri_1$. 
Let us put them into standard positions. 
Thus, a pasting map must send the corresponding vertex to 
a vertex, say $v$, in the other triangle or to the antipodal point 
$-v$. 
Use the pasting maps $P_0$ and $P_2$ identity and 
$P_1$ a rotation at $[0,0,1]$ sending $v_1$ to $-v'_1$. 
The associated representation $h_1$ is clearly near $h^*$. 

For some real numbers $\delta, \delta'$, and $\delta''$ with $|\delta|,|\delta'|,|\delta''| << \eps$, we can form 
two triangles with respective angles 
$(\pi/2+\delta, \pi/2-\delta', \pi-\eps+\delta'')$ and $(\pi/2+\delta, \pi/2+\delta', \pi-\eps+\delta'')$
with respective vertices $(v_0, v_1, v_2)$ and $(v'_0, v'_1, v'_2)$.
Put these triangles 
in standard positions with the pasting map $P_0$ fixing $v_0$ and the corresponding $P_1$ 
sending $v_1$ to $-v'_1$ and $P_2$ sending $v_2$ to $v'_2$. 
If the pasting angles are different, the representation is different. 
One can form a near triangular representation $h'$ to $h_1$ with triangles 
of angle $(\pi/2+\delta, \pi/2-\delta', \pi-\eps+\delta'')$ and $(\pi/2+\delta, \pi+\delta', \pi-\eps+\delta'')$
respectively. Therefore, we obtain six parameters of characters near the character of $h_1$.

Therefore, the subset of triangular characters is not empty in ${\mathcal C}_1$. 
Moreover, the function $f_i:{\mathcal C}_1 \ra R$ given by 
$[h] \mapsto \mathrm{tr} h(c_i)$ for each $i=0,1,2$ is not constant. 
Therefore, we see that the subset defined by $f_i=3$ is an algebraic subset 
whose complement is dense and is locally path-connected. 
The subspace of the abelian characters are also $4$-dimensional algebraic subset. 
Thus, any character can be pushed to a triangular character 
by a real analytic arc which is in triangular part except at the end. 
This proves the first part. 

The set of triangular characters with the angle conditions forms an open subset of ${\mathcal C}_1$
since any nearby triangular character should have uniquely defined triangles up to isometry 
and uniquely pasting maps nearby. 

Since all triangular characters satisfy the above angle condition, and 
these form a dense set, this completes the proof of the second part. 
\end{proof}

The set of possible nondegenerate triangles for $\tri_0$ and $\tri_1$ is then 
described as the intersection of 
$\tilde G^o \cap \kappa(\tilde G^o)$ where $\kappa$ is the map 
sending $(\theta_0,\theta_1,\theta_2)$ to 
$(\theta_0, \pi-\theta_1, \theta_2)$. 
Since $\tilde G^o$ is given by
\begin{eqnarray}
    \theta_{0}+\theta_{1}+\theta_{2} & >&  \pi \nonumber \\
    \theta_{0} & < & \theta_{1} + \theta_{2} - \pi, \nonumber \\
    \theta_1 &< & \theta_2 + \theta_0 -\pi \nonumber \\
    \theta_2 &<& \theta_0 +\theta_1 -\pi
    \label{e:tri2}
    \end{eqnarray}
it follows that our domain is an octahedron $O$
given by eight equations 
\begin{eqnarray}
    \theta_{0}+\theta_{1}+\theta_{2} & >&  \pi : (a) \nonumber \\
    \theta_0+\theta_2 &>& \theta_1 : (a') \nonumber \\
    \theta_{0} & < & \theta_{1} + \theta_{2} - \pi, : (c) \nonumber \\
    \theta_0+\theta_1 &<& \theta_2 : (c') \nonumber \\ 
    \theta_1 &< & \theta_2 + \theta_0 -\pi : (b) \nonumber \\
    2\pi &< & \theta_0+\theta_1+\theta_2 : (b') \nonumber \\ 
    \theta_2 &<& \theta_0 +\theta_1 -\pi : (d) \nonumber \\
    \theta_1+\theta_2 &<& \theta_0 : (d')
    \label{e:oct}
    \end{eqnarray}
    The paranthesis indicates the name of the corresponding sides. See Section \ref{subsec:comp}.
      
    For a character $h$ in ${\mathcal C}_1$, 
 $\tri_0$ and $\tri_1$ are in the octahedron and they are related by $\kappa$.


\subsection{The compactification $\tilde O$}\label{subsec:comp}

We define $\tilde O$ as the subset of $\clo(O)$ by 
additional requirements 
\[\eps \leq \theta_0 \leq \pi-\eps, \eps \leq \theta_1 \leq \pi-\eps, \eps \leq \theta_2 \leq \pi-\eps.\]
We obtain six additonal faces corresponding to vertices
\begin{multline}
B_B: \theta_0 = \eps, B'_{B'}: \theta_0=\pi-\eps, C_{C'}: \theta_1 =\eps, C'_C: \theta_1 = \pi-\eps, \\
A_A: \theta_2 = \eps, A'_{A'}: \theta_2 = \pi-\eps, \hbox{ for some } 0< \eps < 1/10
\end{multline}
where respective faces are given by $\clo(O)$ intersected with the planes of the respective equations. 
Actually, by $\tilde O$ we only mean the combinatorial facial structure of the above space.

We compactify $O$ into $\tilde O$: We map $O$ into $\tilde G \times \tilde G$ 
by sending $\tri_0$ to $(\tri_0, \tri_1)$  by a map $(\Idd, \kappa)$ 
where $\tri_1$ is the associated triangle 
with angles obtained by $\kappa$. 
Again the compactification is given by taking the closure of 
the image of $O$ under the map $(v(0), v(1), v(2), l(0), l(1), l(2),v(0)', v(1)', v(2)', l'(0), l'(1), l'(2)) $
where $v(i)$ and $l(i)$ denote the angle at $v_i$ and the length of $l_i$ of the upper triangle, 
and $v(i)'$ and $l(i)'$ denote the angle at $v_i$ and the length of $l_i$ of the lower triangle.
We show that the result has the same combinatorial structures as $\tilde O$. 
(See Figure {fig:tildeO}.) 

Since geodesic rays in $\tilde G^o$ correspond to points of 
the boundary of $\tilde G$, and only the edge points of $G$ or $\kappa(G)$ are blown-up in $\tilde G$ or $\kappa(\tilde G)$, 
we see that $O$ compactifies in $\tilde G\times \kappa(\tilde G)$ 
so that each vertex gets expanded to a quadrilateral and no other point of the boundary of $O$ is changed:
A ray in $O$ gives us a ray in $\tilde G$ and a ray in $\kappa(\tilde G)$. 
Therefore, a sequence of upper triangles and the corresponding sequence of 
the lower triangles in a given 
ray may converge to two different generalized triangles respectively.
Recall that a ray $a$ in $\tilde G^o$ ending at the interior of an edge
has a limit in a segment times the edge in $\tilde G$. That is, a point 
is expanded to a segment.
Considering a ray $a$ and $\kappa(a)$, we see that at each vertex of $O$, 
a ray can converge to any point in an interval times an interval
where the corresponding parameters of 
upper triangles and lower triangles converges respectively
coordinate-wise. (See the proof of Theorem \ref{thm:T}.)
Thus $O$ compactifies to 
a blown-up octahedron, i.e., an octahedron where each vertex is 
expanded to a square. The proof is similar to Theorem \ref{thm:T} and
we omit it.

\begin{figure}

\centerline{\includegraphics[height=5cm]{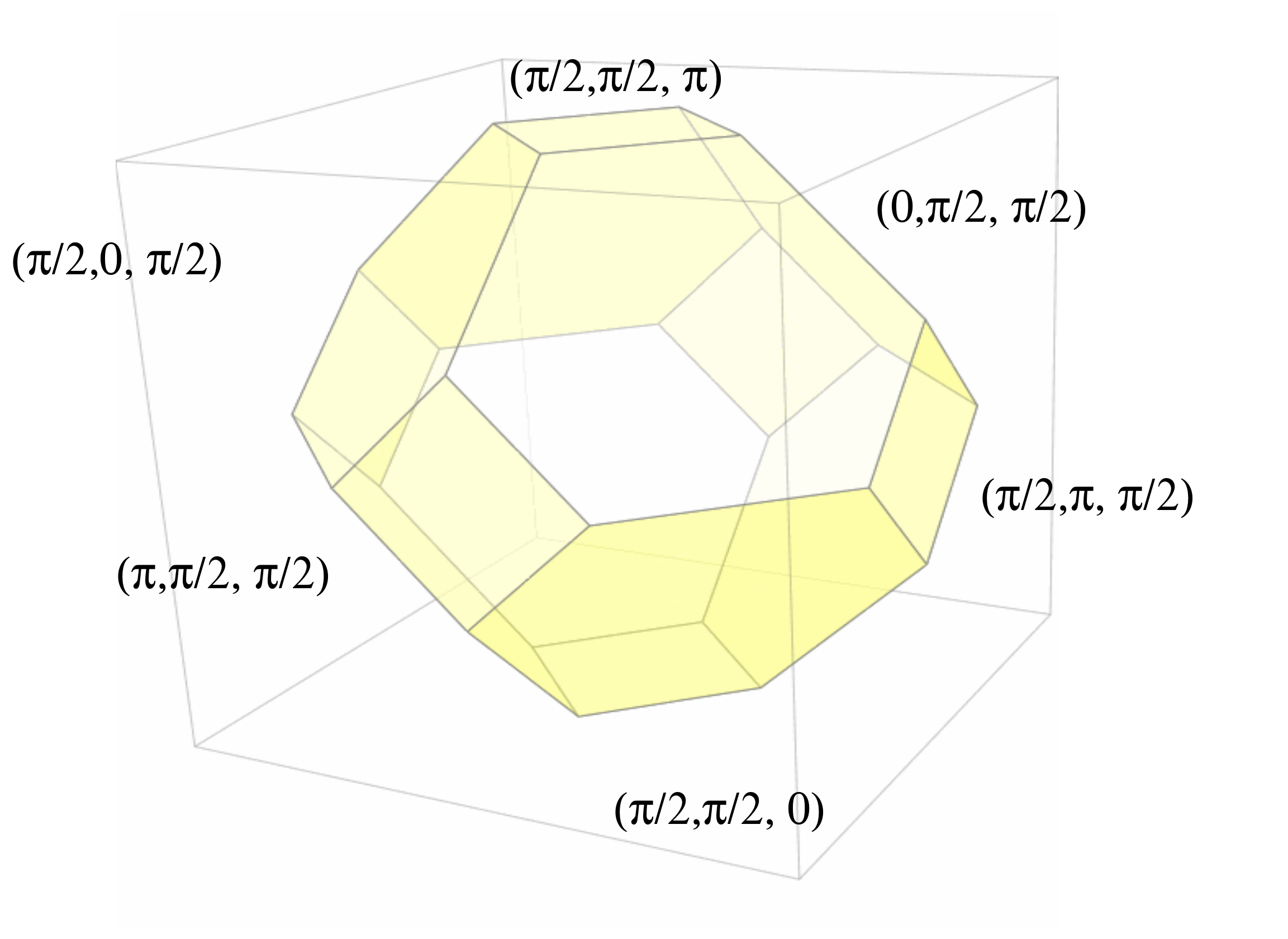}}

\centerline{\includegraphics[height=9cm]{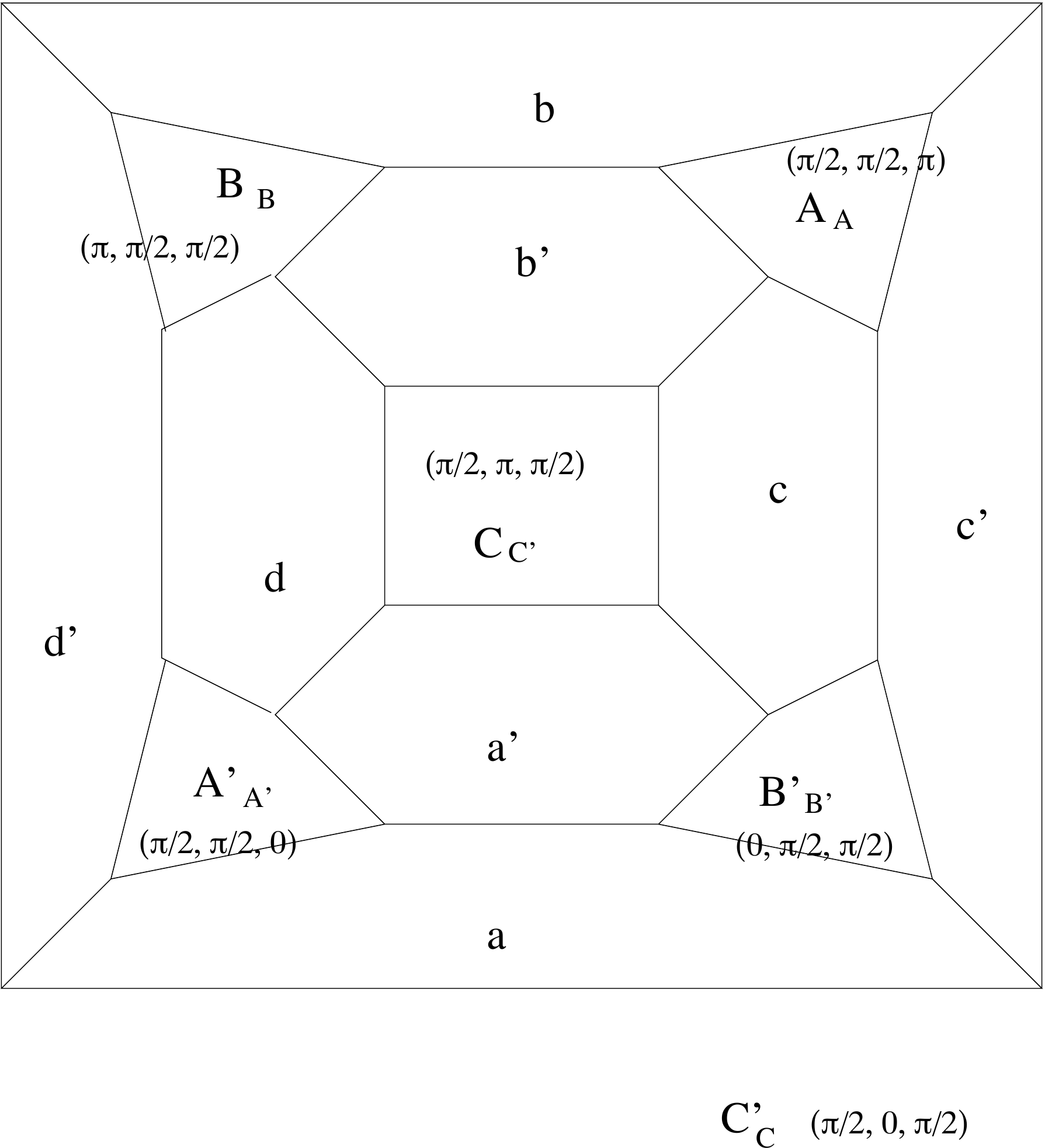}}
\caption{ $\tilde O$ and 
the faces of $\tilde O$ drawn topologically and on a plane using stereographic 
projection. 
}
\label{fig:tildeO}
\end{figure}


The Klein four-group $\{I, I_{{A}}, I_{{B}}, I_{{C}}\}$ acts on the resulting polyhedron $\tilde O$ isometrically.
They are obtained by replacing $v_i$ to $-v_i$ for some $i=0,1,2$, and they have the same formula
as in the ${\mathcal C}_0$ case.
Of course, we need to verify these extensions but are analogously proven 
to the $\tilde G$ case. See Figure \ref{fig:tildeO} to see the topological actions.

They are as follows in terms of coordinates
\begin{multline}
I_A: (v(0), v(1), v(2), l(0), l(1), l(2),v(0)', v(1)', v(2)', l'(0), l'(1), l'(2)) \mapsto \\
(\pi-v(0),\pi-v(1), v(2),  \pi-l(0),\pi-l(1), \\ l(2), \pi-v(0)',\pi-v(1)', v(2)', \pi-l(0)',\pi-l(1)',l(2)')
\end{multline} 
where $v(i)$ and $l(i)$ denote the angle at $v_i$ and the length of $l_i$ of the upper triangle, 
and $v(i)'$ and $l(i)'$ denote the angle at $v_i$ and the length of $l_i$ of the lower triangle.
The map $I_{{B}}$ changes the triangle with vertices $v_0, v_1$, and $v_2$ 
to one with $v_0, -v_1$, and $-v_2$:
\begin{multline}
(v(0), v(1), v(2), l(0), l(1), l(2), v(0)', v(1)', v(2)', l(0)', l(1)', l(2)') \mapsto \\
(v(0),\pi-v(1), \pi-v(2), l(0),\pi-l(1), \pi-l(2), \\
v(0)',\pi-v(1)', \pi-v(2)', l(0)',\pi-l(1)', \pi-l(2)').
\end{multline} 
The map $I_{{C}}$ changes the triangle 
with vertices $v_0, v_1$, and $v_2$ to one with 
$-v_0, v_1,$ and $-v_2$:
\begin{multline}
(v(0), v(1), v(2), l(0), l(1), l(2), v(0)', v(1)', v(2)', l(0)', l(1)', l(2)') \mapsto \\
(\pi-v(0),v(1), \pi-v(2), \pi-l(0),l(1),\pi-l(2),\\
 \pi-v(0)',v(1)', \pi-v(2)', \pi-l(0)',l(1)',\pi-l(2)').
\end{multline} 

Since $\kappa$ is an involution, we have $\kappa^{-1}=\kappa$. 
We label the regions by first what region in $\tilde G$  it is from 
and second what region in $\kappa(\tilde G)$  it is from.
For faces, from $\tilde G$, we will simply denote it by $a$, $b$, $c$, or $d$ according 
to where it came from. For faces from $\kappa(\tilde G)$, we denote it 
by $a'$, $b'$, $c'$, or $d'$ according to where it came from among $a$, $b$, $c$, or $d$ 
by $\kappa$ respectively.
Of course, here the orderings in labels are important.
In the equation \eqref{e:oct}, the equations correspond to 
$a$, $a'$, $c$, $c'$, $b$, $b'$, $d$, and $d'$ respectively.

We describe the faces briefly. The parameters of each quadrilateral 
is given by tie parameters in $\tilde G$ and $\iota(\tilde G)$
since the rectangles are really a product of the two middle ties. 

\begin{itemize}
\item[($A_A$)] This region is a quadrilateral. 
The top triangle of a point over
this region is a lune with vertices $v_0$ and $v_1$, of angle $\pi/2$,
and $v_2$ is on the segment between $v_0$ and $v_1$. The bottom triangle is
a lune with vertices $v_0$ and $v_1=-v_0$ and $v_2$
on the segment between $v_0$ and $v_2$.
A point in corresponds to two generalized triangles 
$(\tri_0, \kappa(\tri_0))$. The functions $l(1)$ measures the length of the edge 
between $v_0$ and $v_2$ and $l'(1)$ measures the length of the edge 
between $v'_0$ and $v'_2$ parameterize the square with 
values in $[0,\pi]\times [0,\pi]$. 
Here the lune and the map above for $h^*$ correspond to the center of the square. 
\item[(${{A'_{A'}}}$)] The upper triangle is a segment with one endpoint $v_0=v_1$ 
both of angle $\pi/2$ and $v_2$ the other endpoint. 
The lower triangle is a segment with one endpoint $v_0=v_1$ both of angle $\pi/2$ 
and $v_2$ the other endpoint. 
\item[($B_B$)] This region is again a quadrilateral. 
The top triangle is a lune with vertices $v_1$ and $v_2$ antipodal to each other 
and $v_0$ is on an edge. The bottom one is a lune of same type.
\item[($B'_{B'}$)] This is again a quadrilateral and the top and the bottom
triangles are segments with vertices $v_1=v_2$ and $v_0$.
\item[($C_{C'}$)] This is again a quadrilateral and the top triangle is 
a lune with vertices $v_0$ and $v_2$ antipodal to each other. The bottom
triangle is a segment with an endpoint $v_0=v_1$ and the other endpoint $v_1$.
\item[($C'_C$)] This is again a quadrilateral and the top triangle and the bottom
one are reversed ones for $C_{C'}$.
\item[($b$)] This region is a hexagon. 
A point correspond to a representation where 
the top triangle is a pointed lune 
and $v_0=v_2$ and $v_1$ is antipodal to them. 
The bottom is a nondegenerate triangle except when we are at the boundary of $b$.
The boundary consists of segments in $A_{A}$, $B_{B}$, $C_{C'}$,and in $b',c',d'$.
This means that the bottom triangles are as in $A$, $B$, $C'$,and $b,c,d$
respectively.
\item[($c$)] The region is a hexagon. Again analogous statements to (b) hold.
\item[($d$)] The region is a hexagon. Again analogous statements to (b) hold.
\item[($a$)] The region is a hexagon. Analogous statements to (b) hold.
\item[($a',b',c',d'$)] These are hexagons and the additional edges 
are among $A_A$, $B_B'$, and $C_{C'}$. The triangles of the points are 
easily learned from from $a$, $b$, $c$, and $d$ respectively.
\end{itemize}

The regions from the faces of $O$ are said to be {\em faces} still, and 
the regions from the vertices of $O$ are said to be {\em square regions.}
Notice that edges are not blown-up here. 

Note that $\{\Idd, I_{{A}}\}$ acts on $A_A$ and $A'_{A'}$ but does not act on other rectangular regions. 
Similarly $\{\Idd, I_{{B}}\}$ acts on $B_B$ and $B'_{B'}$ but does not act on other regions
and $\{\Idd, I_{{C}}\}$ acts on $C_{C'}$ and $C'_C$ but does not act on other regions.
Finally, $I_{{A}}$ sends faces  $b$ to $c$ and $b'$ to $c'$ and $a$ to $d$ and $a'$ to $d'$ and vice versa.
$I_{{B}}$ sends faces $a$ to $c$ and $a'$ to $c'$ and $b$ to $d$ and $b'$ to $d'$ and vice versa.
 $I_{{C}}$ sends faces $c$ to $d$ and $a'$ to $b'$ and $a$ to $b$ and $c'$ to $d'$ and vice versa. 
 The pasting angles are sent as in the ${\mathcal C}_0$ cases. One can verify this using 
 a straightforward process.

\subsection{Equivalence relations on $\tilde O$}\label{subsec:tildeO}

We will now introduce equivalence relations on $\tilde O \times T^3$ so that 
the result is homeomorphic to the other component ${\mathcal C}_1$. 

First we define a continuous map $ \mathcal T:\tilde O\times T^3 \ra {\mathcal C}_1$:
We are given two triangles $(\tri_0,\tri_1)$ and 
three pasting angles, we define a representation. 
Put $\tri_0$ and $\tri_1$ into standard positions in the closed hemisphere $H$ 
as defined above. That is, for vertices $v_0, v_1, v_2$, we have $v_0 =[1,0,0]$ and $v_1$ on the closure $L$ of 
the component of $\partial H - \{[1,0,0],[-1,0,0]\}$ containing $[0,-1,0]$
and $v_2 \in H$. 

Let $v_0,v_1$, and $v_2$ be the vertices of $\tri_0$ and 
$v'_0,v'_1$, and $v'_2$ be those of $\tri_1$. 
Let $(\theta_0, \theta_1, \theta_2)$ denote the angles of $\tri_0$.

We always have $v_0 =[1,0,0]$ and $v_1 \in L$. 
Thus, given a pasting angle $(\phi_0, \phi_1, \phi_2)$, we define the pasting map $P_0$ 
to be $R_{v_0, \phi_0}$. 
Given $\phi_1$, we define the pasting map $P_1$ sending $v_1$ to $-v'_1$
to be the \[R_{-v_1',\phi_1} R_{[0,0,1],d'(v_1,-v_1')}
= R_{[0,0,1],d'(v_1,-v_1')} R_{v_1,\phi_1}\]
where $d(v_1, -v_1')$ is the counter-clockwise arclength between $v_1$ and $-v_1'$.  

Denote by $L_0$ the segment in $H$ with endpoints $v_0$ and $v_2$ in all cases
that has angle $\theta_0$ with $L$ at $v_0$. Let $\SI^1_0$ be the great circle containing $L_0$.
(Note that $\theta_0 = \pi/2$ when we are on $A_A, A'_{A'}, C_{C'},$ or $C'_C$ and 
$\theta_0=\pi$ when we are on $B_B$ or $B'_{B'}$.)

Suppose that at least one of the triangles is nondegenerate. Then
one of $v_2$ and $v_2'$ is in $H^o$ and 
hence $v_2$ and $v'_2$ are not antipodal, and hence there exists 
a unique isometry $R$ sending $v_2$ to $v_2'$ preserving 
the great circle through $v_2$ and $v'_2$ if $v_2 \ne v'_2$. 
If $v_2=v'_2$ we let $R$ be the identity. 
We define $P_2=R_{v'_2,\phi_2} R = R R_{v_2, \phi_2}$. 
Since $v_2$ and $v'_2$ lie on $L_0$, $R$ preserves  $\SI^1_0$. 


Suppose that both triangles are degenerate, and they are on 
one of the interiors of the six squares corresponding to the six vertices of $O$.

Suppose that $(\tri_0, \tri_1)$ is in the square $A_A$.
Put the two triangles into standard positions.
They are both lunes with vertices on a segment $L_0$. Then $R$ must 
preserve $\SI^1_0$. Therefore, we define the pasting maps $P_0, P_1, P_2$ as in the two paragraphs ago. 
If  $(\tri_0, \tri_1)$ is in the square ${{A'_{A'}}}$, we do similarly to $A_A$. 

Suppose that $(\tri_0, \tri_1)$ is in the square $B_B$. 
In this case, $v_0$ has a $\pi$-angle and 
$v_1$ is in $L$ and $v_2$ is in $-L$.  The pasting maps $P_0$ and $P_1$ are defined as above. 
For the pasting map $P_2$, we take as $R$ the rotation preserving $\SI^1_0=\partial H$ as prescribed two paragraphs ago. 
For case $B'_{B'}$, the arguments are similar. 

For case $C_{C'}$, $\tri_0$ is a lune with vertices $v_0=[1,0,0]$ and $v_2=-v_0$ 
and $v_1$ is in $L$. We have that $\tri_1$ is a segment with an endpoint $v'_0=v'_2=[1,0,0]$ 
and $v'_1$ is in $L$. Then the pasting maps $P_0$ and $P_1$ are defined as above. 
For the pasting map $P_2$, we use the composition of $R$ preserving the great circle $\SI^1_0$
with a rotation at $v_2=[-1,0,0]$. 
For $C'_C$, the definitions are similar. 

The continuous dependences of $P_0, P_1,$ and $P_2$ on $\tilde O\times T^3$ follows from Proposition \ref{prop:stdp}.
Since $h^*$ is represented by two triangles with vertices
$[1,0,0]$, $[0,1,0],$ and $[0,0,1]$ where $\phi_0,\phi_1$, and $\phi_2$ are 
all zero, it follows that our map $\mathcal T$ is into ${\mathcal C}_1$. 
Since the subspace of triangular characters are dense in ${\mathcal C}_1$ 
and $\tilde O$ is compact, and our map is 
onto the triangular subspace, $\mathcal T$ is onto ${\mathcal C}_1$.

We now extend the action of $\{\Idd, I_A, I_B, I_C\}$ on $\tilde O \times T^3$ 
from the action on $\tilde O$ in Section \ref{subsec:comp}
as in equation \eqref{eqn:iabc4}. This is our first equivalence relation.

\begin{prop}\label{prop:calt} 
Using pasting map construction as above, we have a continuous onto map
\[ \mathcal T: \tilde O \times T^3 \ra {\mathcal C}_1.\]
Furthermore, we have 
\[\mathcal T \circ I_A = \mathcal T \circ I_B = \mathcal T \circ I_C =\mathcal T.\]
\end{prop}
\begin{proof} 
The continuity and the surjectivity were proved above. 
Consider the transformation $I_A$ on $\tilde O$. It acts simultaneously on the upper and the lower
triangles and on the pasting angles. By geometry, we see that this does not change the character constructed. 
The similar arguments hold for $I_B$ and $I_C$. 
$\quad\quad$
\end{proof}

We aim to introduce an equivalence relation on $\tilde O\times T^3$ so that 
the induced map is one-to-one onto.

\subsubsection{The equivalence relations on faces  $a,b,c,d,a',b',c',$ and $d'$}
\label{subsub:abcda'b'}

There are equivalence relations on
the faces $a,b,c,d,a',b',c',$ and $d'$. 
We will consider $a$ first. Then here the top triangle is a pointed point and 
the bottom triangle a nondegenerate triangle.

($a$) We consider $a$ first. 
The equivalence relation is given by the fact that a representation of 
a pair of pants corresponding to a pointed point 
has a stabilizer in $\SOThr$ isomorphic 
to a circle. Thus the pasting maps $P_0$, $P_1, P_2$ are determined only up to 
a right multiplication by an isometry fixing the point. This changes the pasting 
angles to be added by  the same angle. 
This changes $P_i$ to $P_i R_{[1,0,0],\theta}$ as the pointed point is 
at $[1,0,0]$ in the standard position. We rewrite it as 
$P_i R_{[1,0,0],\theta} P_i^{-1} P_i$. The first three terms 
give us a rotation of angle $\theta$ at $P_i([1,0,0])$. This changes 
the rotation angles by $\theta$, and gives us an $\SI^1$-action 
given by sending \[(\phi_0,\phi_1,\phi_2) \mapsto (\phi_0+\theta,\phi_1+\theta,\phi_2+\theta)
\hbox{ for } \theta \in \SI^1.\]  This action gives us the equivalence relation on $a$. 
(This section is similar to Section \ref{subsec:abcd}.)
The quotient space is homeomorphic to a $2$-torus, which we denote by $T^2_{a, 2}$
as in the $\mathcal C_0$ case.
Thus, the character space here is in one-to-one correspondence with 
$a\times T^2_{a, 2}$. There are no more equivalences here since the lower triangle 
is nondegenerate.

Furthermore, the $\SI^1$-action extends to $a \times T^3$ and 
still gives equivalent pastings. 

For regions $b,c$, and $d$, these are all equivalent to $a$ by the
the action of the group generated by $I_{{A}},I_{{B}},I_{{C}}$. 
Thus, similar conclusions hold. 
\begin{itemize}
\item[($b$)] The $\SI^1$-action
\[(\phi_0,\phi_1,\phi_2) \mapsto (\phi_0+\theta,\phi_1-\theta,\phi_2+\theta)
\hbox{ for } \theta \in \SI^1.\]
The equation for the equivalence class representatives is $\phi_0 -\phi_1 +\phi_2=0$.
\item[($c$)] The $\SI^1$-action
\[(\phi_0,\phi_1,\phi_2) \mapsto (\phi_0-\theta,\phi_1+\theta,\phi_2+\theta)
\hbox{ for } \theta \in \SI^1.\]
The equation for the equivalence class representatives is $-\phi_0 +\phi_1 +\phi_2=0$.
\item[($d$)] The $\SI^1$-action
\[(\phi_0,\phi_1,\phi_2) \mapsto (\phi_0+\theta,\phi_1+\theta,\phi_2-\theta)
\hbox{ for } \theta \in \SI^1.\] 
The equation for the equivalence class representatives is $\phi_0 +\phi_1 -\phi_2=0$.
\end{itemize}
There is again the order $3$ translation group action giving us a $2$-torus as a quotient in each case. 
Here the $\pm$ of $\pm \theta$ corresponds to the action of $\SI^1$ on the pasting 
angles. (These $\SI^1$-actions all extend to the boundary where they all give 
equivalent pastings. )

The upper and lower triangles are related by the relation in $O$:
\[(\theta_0, \theta_1, \theta_2) \leftrightarrow
(\theta_0, \pi-\theta_1, \theta_2).\]
For regions $a^{\prime}, b^{\prime}, c^{\prime}$, and $d^{\prime}$, these 
are related to $a$, $b$, $c$, and $d$ by switching the upper triangle 
with the lower ones and taking inverses of the pasting maps $P_0,P_1,P_2$. 
\begin{itemize}
\item[($a'$)] Here the upper triangle is nondegenerate in the interior of $a'$ 
and the lower one is a pointed point, say $v$. Let $v_0, v_1, v_2$ denote 
the vertices of the upper triangle. 
The pasting map $P_1$ maps $v_1$ to the antipodal point of $v$ and $P_0,P_2$ sends $v_0, v_2$ to $v$. 
By reasoning like the case $a$, we obtain that
$\SI^1$-action fixing $v$, gives us an $\SI^1$-action 
\[(\phi_0,\phi_1,\phi_2) \mapsto (\phi_0+\theta,\phi_1-\theta,\phi_2+\theta)
\hbox{ for } \theta \in \SI^1.\] 
The equation for the equivalence class representatives is $\phi_0 -\phi_1 +\phi_2=0$
\item[($b'$)] By a similar reason to above,
the $\SI^1$-action is given by 
\[(\phi_0,\phi_1,\phi_2) \mapsto (\phi_0+\theta,\phi_1+\theta,\phi_2+\theta)
\hbox{ for } \theta \in \SI^1.\] 
The equation for the equivalence class representatives is $\phi_0 +\phi_1 +\phi_2=0$
\item[($c'$)] The $\SI^1$-action is given by 
\[(\phi_0,\phi_1,\phi_2) \mapsto (\phi_0-\theta,\phi_1-\theta,\phi_2+\theta)
\hbox{ for } \theta \in \SI^1.\] 
The equation for the equivalence class representatives is $-\phi_0 -\phi_1 +\phi_2=0$
\item[($d'$)] The $\SI^1$-action is given by 
\[(\phi_0,\phi_1,\phi_2) \mapsto (\phi_0+\theta,\phi_1-\theta,\phi_2-\theta)
\hbox{ for } \theta \in \SI^1.\] 
The equation for the equivalence class representatives is $\phi_0 -\phi_1 -\phi_2=0$.
\end{itemize}
Again, there is an order-$3$ translation group action in each case giving us a $2$-torus as a quotient. 

The $\SI^1$-actions can be described as orthogonal directional one to the faces 
and the representative equations are tangent directions to the faces. 
Also, $I_{{A}},I_{{B}},I_{{C}}$ permute $\SI^1$-actions
and the representative directions equations. 

In the interior of $a, b, c, a', b', c'$, the equivalence relation is precisely the ones given above since
one of the upper and lower triangles is nondegenerate and the vertices are well-defined. 

Up to the Klein group action, there are only two faces, say $a'$ and $c$.


\subsubsection{Equivalence relations over the edges.}\label{subsub:edge}

When two faces meet, $\SI^1$-actions induced from the two faces 
are different ones and they generate a $T^2$-action which gives equivalent points
over the edges. 
In the interior of each edge, these two will be the only relations we need to consider:  
The upper triangle and the lower triangle are either a pointed-point or a pointed segment of 
length $\pi$ with all the vertices equal to the endpoints. Hence there are two 
$\SI^1$-actions corresponding to the upper triangle and the lower one. 
In the interior of the edge, these point or segments are well-defined 
up to isometries of $\SI^2$. 

Here we obtain as the quotient space a circle over each point of the interiors of edges as the quotient space, and 
the circles form annuli over the interiors of the edges. We obtain $12$ open annuli. 

Up to the Klein four group action, we only have three edges and three open annuli.



\subsubsection{The equivalence relations on faces  $A_A$,$A'_{A'}$,$B_B,B'_{B'},C_{C'},C'_C$.}
\label{subsub:AA}
We will introduce equivalences over the six regions, $A_A$,$A'_{A'}$,$B_B,B'_{B'},C_{C'},C'_C$.
Each equivalence relation is simply the identification producing 
the same characters. We show that each of the quotient spaces is homeomorphic to $\SI^3$ 
and its tubular neighborhood is homeomorphic to $\SI^3\times B^3$. 
Also, we show that 
there is a branch double-cover by $\SI^3$ and by $\SI^3\times B^3$.

First, we go to the square $A_A$.  
Here the top triangle is a lune with vertices $v_0$ and $v_1$ and 
$v_2$ on an edge of a lune. The bottom triangle is a lune with
vertices $v'_0$ and $v'_1=-v'_0$ and $v'_2$ on an edge of the lune. 
The vertex angles are all $\pi/2$.
These are put into standard positions.

We note that $d_1=P_0^{-1}P_1$ and $d_2=P_0^{-1}P_2$ for the pasting maps $P_0, P_1,$ and $P_2$.
We will find preferred fixed points of $d_1$ and $d_2$ respectively on $\SI^2$.

\begin{figure}
\centerline{\includegraphics[height=7cm]{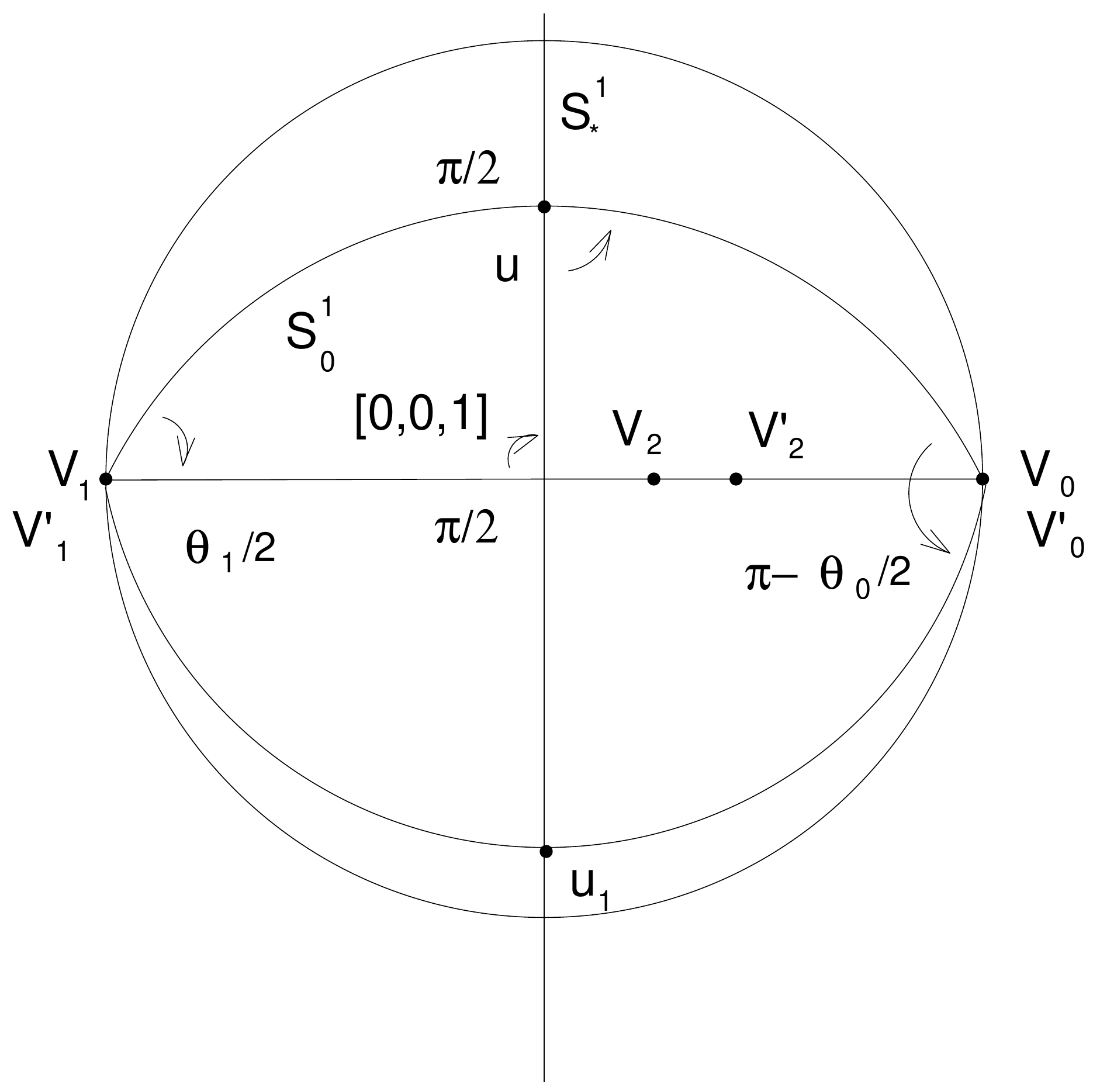}}
\caption{How to find a fixed point of $h(d_1)$ for region $A_A$.}
\label{fig:figA}
\end{figure}

First, the pasting map $P_0$ is a rotation $R([1,0,0],\theta_0)$. 

And $P_1$ is a composition of rotation of angle $\theta_1$ at $[-1,0,0]$ 
with $R([0,0,1],\pi)$. 
Thus, $P_1$ sends $[1,0,0]$ to $[-1,0,0]$ 
and has $-1$ as an eigenvalue. 
A fixed point is given as a vertex of a triangle with vertex 
$[0,0,1]$ and $[-1,0,0]$ with respective angles $\pi/2$ and 
$\theta_1/2$ respectively. We let $u$ be this vertex
which lies on the great circle $\SI^1_o$ passing through $[\pm 1,0,0]$ 
and perpendicular to $\SI^1_*$ passing through $[0,1,0]$ and $[0,0,1]$.


To obtain the fixed point of $d_1$, we draw a triangle with vertex $u$ 
and $[1,0,0]$ with respective angles $\pi/2$ and $\pi-\theta_0/2$. 
The third vertex $u_1$ of this triangle is a fixed point of $d_1$. 
Again $u_1$ is on $\SI^1_*$. The rotation angles is $\pi$ clearly 
as the angle of the the vertex is $\pi/2$.
Also, $u_1$ is the intersection point of $\SI^1_*$ and the segment 
with vertices $[1,0,0]$ and $[-1,0,0]$ having 
angle $\pi-\theta_0/2-\theta_1/2$ with the segment 
connecting $[1,0,0]$ with $[0,0,1]$. 
Thus, $u_1$ could be any point on $\SI^1_*$. (See Figure \ref{fig:figA}.)
From construction, the rotation angle of $h(d_1)$ is $\pi$ always. 

Now, $[1,0,0]$ and $u_1$ will give us a coordinate system $\kappa_{u_1}$ on $\SI^2$ so that 
$[1,0,0]$ has coordinates $[1,0,0]$ and $u_1$ the coordinates $[0,0,1]$.

Consider $\SI^3$ as a double cover of $\SOThr$, i.e., as a ball of radius $2\pi$ in $\bR^3$ 
with the outer sphere identified to a point. 
The coordinate expression of $d_2$ with a choice of the set of fixed points and the angle
correspond to a point of $\SI^3$.
The coordinates of $d_2$ under the coordinate system $\kappa_{u_1}$ give us 
a two-valued function on ${A_A}\times T^3 \ra \SI^3$. 
The function depends on the choice of the fixed point $u_1, -u_1$ for $h(d_1)$ 
and the choice among the pair in the set of fixed points of $h(d_2)$ as well. 
This generates a Klein four-group $V$ action on $\SI^3$. Therefore, we obtain a map
$A_A\times T^3 \ra \SI^3/V$. 
(We remark that if we are on $\SUTw$, we can find a canonical choice.)


By an easy topological argument, $V$ acts by two transformations $S$ and $T$:
$S$ on $\SI^3$ is induced by a rotation of angle $\pi$ at $[1,0,0]$; 
that is, sending an isometry to the isometry of same rotation angle but with the fixed point 
changed by the rotation of angle $\pi$ at $[1,0,0]$  and by the transformation $T$ equal to 
the antipodal map on $\SI^3$. 
Then the set of fixed points of $S$ is the circle through the identity, consisting of all rotations 
with fixed points $[1,0,0]$ and $[-1,0,0]$. The set of fixed points of $T$ is empty. 
The set of fixed points of $S T$ is given by isometries with fixed points 
at the circle though $[0,1,0]$ and $[0,0,1]$ of rotation angle $\pi$.

\begin{rem}\label{rem:4d} 
Here, it is useful to see $\SI^3$ as a unit sphere in $\bR^4$. 
Then $S$ is represented as a diagonal matrix with entries $1,-1,-1,1$
and $T$ equals $-I$ and $ST$ is a diagonal matrix with entries $-1,1,1,-1$. 
Thus, the set of fixed points of $S$ is given by the $xt$-plane intersected with $\SI^3$ and 
the set of fixed points of $S T$ is given by the $yz$-plane intersected with $\SI^3$. 
Hence, the branch locus is a link consisting of two trivial knots 
with one crossing point under a projection. 
\end{rem}

Finally, $\{I, I_A\}$ acts on $A_A$. This is the action sending $v_0$ to $v_1$ and $v_1$ to $v_0$ and fixing $v_2$ 
and $v'_0$ to $v'_1$ and $v'_1$ to $v'_0$ and fixing $v'_2$. 
The angles change $\theta_0 \ra 2\pi-\theta_0$ and $\theta_1 \ra 2\pi-\theta_1$ but $\theta_2$ is not changed. 
(See Section \ref{subsec:comp} for equations of $I_A$.)

\begin{lem}\label{lem:IA}
$I_A$ acts on $\SI^3$ as an involution with a circle as a fixed point set. 
The fixed points form a circle meeting each of the above two fixed circles of $V$ twice. 
Let $V'$ be the group generated by $V$ and $I_A$, the resulting space $\SI^3/V'$ is homeomorphic to $\SI^3$. 
\end{lem}
\begin{proof}
For the proof we work geometrically and the maps here are maps on $\SI^2$ and 
hence in $\SUTw = \SI^3$ for the duration of the proof.

If we now construct the representation while the two triangles are changed as above. Then $P_2$ is not changed 
but $P_0$ is changed to be 
\[R_{2\pi-\theta_0}(-v_0) = -R_{\theta_0}(v_0)=-P_0\] by equation \ref{eqn:minusid}.
and $P_1$ is changed to be 
\[R_{2\pi-\theta_1}(v_1)  R_\pi([0,0,1])= -R_{\theta_1}(-v_1) R_\pi([0,0,1])=-P_1\] now for the new $v_1$ that is $-v_1$.
 Hence $h(d_1)$ is not changed  and $h(d_2)$  changed to $-h(d_2)$.
Define $\tilde R = R_\pi([0,0,1])$ with $\tilde R^{-1}=\tilde R$. 
Then we obtain $\tilde R  h(d_1)  \tilde R$ and $-\tilde R  h(d_2)  \tilde R$ 
to find coordinates. Our new $u_1$ is $\tilde R (u_1)$.

Suppose that we need an isometry $L$ to move old $u_1$ to $[0,0,1]$ fixing $[1,0,0]$. To move $\tilde R(u_1)$ to $[0,0,1]$,
we need $L^{-1}$. To find the coordinate of $-\tilde R  h(d_2)  \tilde R^{-1}$, we conjugate 
by $L^{-1}$ to obtain $-L^{-2} L \tilde R  L^{-1}  L h(d_2)  L^{-1} L \tilde R^{-1} L^{-1} L^2$. 
Thus, we act by $L^{-2} L \tilde R L^{-1}$ to obtain the fixed point of $-h(d_2)$ read from the new $u_1$ coordinates. 
Since $L \tilde R L^{-1}$ has a fixed point $L([0,0,1])$ and angle $\pi$ and $L^{-2}$ is a rotation at $[1,0,0]$ of an angel $2l$ for 
the distance $l$ between old $u_1$ and $[0,1,0]$. We can use a triangle with angle $\pi/2$ at $L([0,0,1])$ and angle $l$ 
at $[1,0,0]$ to show that $L^{-2}  L \tilde R L^{-1} = \tilde R$ in fact. 
Thus the coordinate of new $h(d_2)$ is changed by $\tilde R$ from the old one. 
Thus, $I_A$ is a left multiplication by $-\tilde R$ in this case. 
In the above coordinates, $I_A$ corresponds to the diagonal matrix with entries $1,1, -1, -1$ in the notation of 
 Remark \ref{rem:4d}.
\end{proof}


Hence, we obtain a map \[{A_A}\times T^3 \ra \SI^3/V'\cong \SI^3.\]
This map is seen to be surjective, which follows from 
the geometric constructions; essentially, any element of $\SOThr$ can arise 
as $h(d_2)$. We show this by showing that the set of $h(d_2)$ with $h(d_1)$ fixed that arises is an open set
by controlling $P_2$ in an open set fixing the angles $\theta_0, \theta_1$.
The control of $P_2$ in an open set can be achieved by the multiplication by geometry construction. 
The set of $h(d_2)$ that arises with $h(d_1)$ fixed is a closed set since this is the image of a continuous map. 
(See the proof of Theorem \ref{thm:filledSU} also.)
Furthermore, under the equivalence relation we assigned, 
\[{A_A}\times T^3/\sim \ra \SI^3/V'\cong \SI^3\] is a homeomorphism. 





For the region $A'_{A'}\times T^3$, we can do the similar argument:
Hence, we obtain a map \[A'_{A'}\times T^3 \ra \SI^3/V\cong \SI^3.\]
This map is seen to be surjective, and 
 \[A'_{A'}\times T^3/\sim \ra \SI^3/V'\cong \SI^3\] is a homeomorphism 
 for our equivalence relation $\sim$. 

We discuss $B_B$ where $I_B$ acts on. Here the top triangle is a lune with vertices 
$v_1$ and $v_2$ with $v_0$ on a segment. The vertex angle of the lune is $\pi/2$.
The bottom triangle is a lune with an endpoint $v'_1$ and $v'_2=-v'_1$ 
and $v'_0$ in the segment. 



We change generators so that we use $c_0$ and $c_2$ for generators and 
$d'_1=P_1^{-1}P_2$ and $d'_2=P_1^{-1}P_0$.
Then $d'_1=d_2^{-1}$ and $d'_2=d_1^{-1}d_2$.
This is a change of index sending $(0,1,2)$ to $(2,0,1)$. 
Thus, the situation again reduces to the case $A_A\times T^3$ with $I_B$ corresponding to $I_A$. 

Let $V''$ be generated by $V$ and $I_B$.
We obtain a map \[B_{B}\times T^3/\sim \ra \SI^3/V''\cong \SI^3,\]
and  \[B_{B}\times T^3/\sim \ra \SI^3/V''\cong \SI^3\]
is a homeomorphism.

For the region $B'_{B'}\times T^3$, we can do the similar argument:
Hence, we obtain a map \[B'_{B'}\times T^3/\bZ_2 \ra \SI^3/V\cong \SI^3.\]
This map is seen to be surjective, and 
\[B'_{B'}\times T^3/\sim \ra \SI^3/V''\cong \SI^3.\]

Next, we go to the region $C_{C'}$.  Here the vertex $v_0$ is antipodal to $v_2$ and
$v_0$ and $v_2$ are the vertices of a lune, and  
$v_1$ is on the segment of the lune.
The lower triangle is a segment with vertices $v'_0=v'_2$ and $v'_1$. 

In this case, the pasting map $P_2$ sends $v_2$ to $v'_2=-v_2$
which is a composition of rotation of angle $\pi$ at $[0,1,0]$
and a rotation at $[1,0,0]=v'_2$.
Since $P_0$ is a rotation at $[1,0,0]$, 
it follows that $h(d_2)$ has a fixed point in the great circle $\SI^1_*$ containing 
$[0,1,0]$ and $[0,0,1]$ by an argument similar to $A_A$ case. 
$h(d_1)$ can be any isometry. 
Again $I_C$ can be studied as in the $I_A$ case.
Let $V'''$ denote the group generated by $V$ and $I_C$. 
 \[C_{C'}\times T^3 \ra \SI^3/V'''\cong \SI^3,\]
is surjective also, and  \[C_{C'}\times T^3/\sim \ra \SI^3/V'''\cong \SI^3\]
is a homeomorphism. 

Thus, we obtained the desired conclusions. 
The same can be said for the region $C'_C$.

We define the equivalence relation over the faces $A_A$,$A'_{A'}$,$B_B$, $B'_{B'}$, $C_{C'}$, $C'_C$
so that the above maps are injective. Here, we really have only three regions up to the $V$-action, say 
$A_A, B_B, C_{C'}$. 

Also, we use $\{\Idd, I_{A}, I_{B}, I_{C}\}$-action 
to identify characters. 

\begin{thm}\label{thm:C1} 
The map $ \mathcal{T}:\tilde O \times T^3_2 \ra {\mathcal C}_1$ induces 
a homeomorphism $\tilde O \times T^3_2/\sim \ra {\mathcal C}_1$
where $\sim$ is given as above. 
\end{thm}
\begin{proof}
We showed that $ \mathcal T$ is onto by Proposition \ref{prop:calt}.
Let us show that $\mathcal T$ is one-to-one.
The interior $O\times T^3/\sim$ injects into 
triangular characters
since triangular characters uniquely determine
the triangles up to isometry and the choices of 
fixed points up to antipodal pairs
and rotation angles are also determined by the standard positions. 

The map from the boundary of $O$ with vertex rectangles removed 
is also injective by Subsection \ref{subsub:abcda'b'}. 
The maps from the regions above the square regions are injective up to the $\{I, I_A, I_B, I_C\}$-action by
Section \ref{subsub:AA}.
The injectivity of the total map follows from these results and the fact that the images of 
$O\times T^3$ and $\partial \tilde O -( A_A \cup A'_{A'} \cup B_B \cup B'_{B'} \cup C_{C'} \cup C'_C )\times T^3$
and $A_A \times T^3, A'_{A'}\times T^3, B_B \times T^3, B'_{B'} \times T^3, C_{C'} \times T^3, C'_C )\times T^3$
are mutually disjoint up to action of $\{I, I_A, I_B, I_C\}$ from considering the angles of the triangles.

Since $\SOThr$ is compact, the quotient space $\rep(\pi_1(\Sigma_1),\SOThr)$ is Hausdorff. 
Since $\tilde O\times T^3/\sim$ is compact and ${\mathcal C}_1$ is Hausdorff, 
we have a homeomorphism. 

\end{proof}

\section{The $\SUTw$-pseudo-characters for the above component} \label{sec:su2o}

We define the $\SUTw$-character space $\rep_{-I}(\pi_1(\Sigma_1),\SUTw)$ of 
a punctured genus 2 surface $\Sigma_1$
with the puncture holonomy $-I$ as the quotient space 
of the subspace of $\Hom(\pi_1(\Sigma_1),\SUTw)$ where $h(c)= -I$ 
by conjugations where $c$ is a simple closed curve around the puncture. 

A representation corresponding to a point of $\rep_{-I}(\pi_1(\Sigma_1),\SUTw)$ has 
a centralizer if and only if it is abelian. But no abelian character is in it. 
Thus, the $\SUTw$-character space is a manifold by Proposition 3.7 and 
the modified Proposition 1.2 of \cite{Gsymp}, where the relation 
is modified to equal $-I$. 

Let $T^3_2$ be the product of three copies of $\SI^1_2$.
Let $\bZ_2$ acts on it as before and we obtain $T^3_2/\bZ_2$.
We will show that this space is homeomorphic to an octahedral manifold. 

\begin{thm}\label{thm:filledSU} 
$\rep_{-I}(\pi_1(\Sigma_1),\SUTw)$ is homeomorphic 
to the filled octahedral manifold. 
\end{thm}
\begin{proof} 

As before, we need to show that $\rep_{-I}(\pi_1(\Sigma_1),\SUTw)$
is homeomorphic to a quotient space of $\tilde O\times T^3_2/\bZ_2$. 
Then we need to show that the quotient space is homeomorphic to 
the desired space.  

Realize $\Sigma_1$ as a union of two pairs of pants $S_0$ and $S_1$ and a punctured annulus $A$
inserted in place of $c_1$. Thus, the angles for $c_1$ in $S_0$ and $c'_1$ in $S_1$ are related 
by $\theta$ to $2\pi-\theta$. 

The triangular subset $D_1$ of ${\mathcal C}_1$ can be represented by 
$O^o\times T^3_2/\bZ_2$. 
We define a triangular part $D_2$ of $\rep_{-I}(\pi_1(\Sigma_1),\SUTw)$
as the subset where the restrictions of each element to $\pi_1(S_0)$ and $\pi_1(S_1)$ are both 
triangular characters to $\SUTw$, i.e., represented by a nondegenerate
triangle. Using the trace functions for $c_i$, we show that $D_2$ is a dense open set
with the complement equal to an algebraic subset. 
More precisely, there is a 16 to 1 map 
from $O^o \times T^3_2/\bZ_2$ to $D_1$ and $D_2$ is an open connected subset.
It follows that we can lift $O^o\times T^3_2/\bZ_2$ into the dense open set $D_2$ in $\rep_{-I}(\pi_1(\Sigma_1),\SUTw)$.
The lift is injective since in $O^o$ the two triangles are not degenerate and hence 
the isometry classes of the two triangles and the pasting angles 
are uniquely determined. (See Proposition \ref{prop:calt} for details.)



By a geometric limit argument, we see that 
each character in $\rep_{-I}(\pi_1(\Sigma_1),\SUTw)$ induces 
a representation where two pairs of pants are represented as two triangles (generalized) 
related by the angle conditions. Conversely, such a pair corresponds to 
a character of $\pi_1(\Sigma_1)$ to $\SUTw$. 
Therefore, we have a surjective map from $\tilde O\times T^3/\bZ_2$ to $\rep_{-I}(\pi_1(\Sigma_1),\SUTw)$.
(Note that as before, we are using triangles, not the fixed point themselves, to represent these.)

Now we introduce equivalence relation $\sim$ on this space, which does not 
include ones given by $I_{{A}},I_{{B}},I_{{C}}$. 

We will use the same equivalence relation on regions above $a$, $b$, $c$, $d$, $a'$, $b'$, $c'$, and $d'$ 
as in the $\SOThr$-case except that now 
the fibers are $T^3_2/\bZ_2$. The equivalence relation is given by identifying all 
elements proportional to $(2\pi,2\pi,2\pi)$-vector in $T^3_2$ to $O$ for face $a \times T^3_2$, 
by identifying all 
elements proportional to $(2\pi,-2\pi,2\pi)$-vector in $T^3_2$ to $O$ for face $a' \times T^3_2$, 
and so on. (See \ref{subsub:abcda'b'} for more details.)
Thus, the character space over the interior of a face $a^o$ is in one-to-one correspondence with 
$a^o\times T^2$. There are no more equivalences here since the lower triangle 
is nondegenerate. 
For all other faces, the similar consideration gives us that over the interior of a face, there is a $2$-torus fibration.
Over the interior of edges, the equivalence relation is one given by the two adjacent faces. 
Let $\tilde O'$ be the $\tilde O$ with six rectangular regions removed.
Now, $\tilde O' \times T^3_2/\bZ_2/\sim$ with the same equivalence relation as induced by faces 
extended to edges and vertices is the octahedral variety:
This follows since this is $\tilde O' \times \bR^3$ quotient by the lattice 
generated by $(4\pi, 0,0), (0, 4\pi, 0), (0,0,4\pi), (2\pi, 2\pi, 2\pi)$, 
which has generators $(\pm 2\pi, \pm 2\pi, \pm 2\pi)$. 
By this description, we see that the quotient space is homeomorphic to the octahedral variety with six singular 
points removed.
(See the page 50 in Fulton \cite{Fulton}).

This space is actually described as follows. 
$\CP^3$ is a toric variety over a tetrahedron $G$. 
Then there are six spheres mapping to six edges of the tetrahedron.
Over each vertex of $G$, the fiber is a singleton.  
We blow-up each of the four singleton mapping to a vertex of $G$ to 
a two-dimensional sphere and now the six spheres do not meet. 
Now blow-down each of the  spheres to a point. 
Then remove the six points, we obtain a manifold $\mathcal O$.
(See page 50, Section 2.6, of Fulton \cite{Fulton}.)

We see from this picture that
if we remove the union of contractible $T^3_2/\bZ_2$-invariant 
cone-neighborhoods of the six points, then 
the remaining is a manifold  so that the closure in $\mathcal O$ is compact 
and has boundary a union of six $5$-manifolds each homeomorphic to 
an $\SI^3$-fibration over a $2$-sphere.

For regions, $A_A$, $B_B, C_{C'}$, $A'_{A'}$, $B'_{B'}$, and $C'_C$, 
we will use a different but similar identification to the $\SOThr$-cases.
Let us take the case $A_A$ first. Here the pasting angles are in $\SI^1_2$. 
By an extended multiplication 
by geometry for $\SUTw$, we obtain the fixed point 
$u_1$ of $-h(d_1)$ and $-h(d_2)$ uniquely determined in this case. 
Therefore, by reading the coordinates of $-h(d_2)$ in terms of $u_1$, we obtain 
a map $A_A\times T^3_2 \ra {\mathcal C}_1$ whose image is an imbedded $3$-sphere since given all points of 
$\SI^2$ arise as a point and all rotation angles occur by our constructions, 
where again we used our control of $P_2$ with pasting angles fixed at $v_0$ and $v_2$.
Moreover, $A_A\times T^3_2/\sim \ra {\mathcal C}_1$ is an embedding
since $\sim$ is chosen for this to hold.


Since $\rep_{-I}(\pi_1(\Sigma_1),\SUTw)$ is topologically a manifold, 
the tubular neighborhood of the region above $A_A$ is homeomorphic to $\SI^3\times B^3$ and covers 
the tubular neighborhood of the $3$-sphere in $\SOThr$ case $4$ to $1$. 

In all other cases $B_B$, $C_{C'}$,${{A'_{A'}}}$, $B'_{B'}$, and $C'_C$, similar arguments show that 
the rectangle times $T^3_2$ maps to a $3$-sphere in ${\mathcal C}_1$. 
and its tubular neighborhood homeomorphic to $\SI^3\times B^3$ branch covers 
the tubular neighborhood of the $3$-sphere in ${\mathcal C}_1$ in $4$ to $1$ manner. 

Finally, $\tilde O\times T^3_2/\bZ_2/\sim$ is homeomorphic to $\rep_{-I}(\pi_1(\Sigma_1),\SUTw)$. 
Again, this follows from the injectivity and compactness and the Hausdorff property. 
(See Theorem \ref{thm:C1} for the analogous proof.)

Now, we go over to the topology of  $\tilde O\times T^3_2/\bZ_2/\sim$: 
The space $\tilde O'\times T^3_2/\bZ_2/\sim$ is a toric variety ${\mathcal O}'$ over the octahedron with vertices removed.
By removing a thin neighborhood of square regions 
$A_{A}$,$B_{B}$,$C_{C'}$,$A'_{A'}$, $B'_{B'},$ and $C'_C$ from $\mathcal O$, 
and considering the region $M$ in our quotient space over it, we obtain a toric variety over an octahedron with 
neighborhoods of the ends removed. 

Since the boundary of the closure $M$ is a union of six 
five-dimensional manifolds homeomorphic to $\SI^3\times \SI^2$.
Thus, we can glue six neighborhoods of the three spheres over 
$A_{A}$,$B_{B}$,$C_{C'}$,$A'_{A'}$, $B'_{B'},$ and $C'_C$,
homeomorphic to $\SI^3\times B^3$ to obtain the octahedral manifold. 
Hence, it follows that $\tilde O \times T^3_2/\bZ_2/\sim$ is homeomorphic to the octahedral manifold. 

\end{proof}

\section{The topology of the other component ${\mathcal C}_1$}\label{sec:C1}

In this section, we will describe the topology of the other component ${\mathcal C}_1$. 

We first discuss the fiberwise action and then the axial action induced from that of the octahedron $O$. 
We consider our space $T^3_2$-fibers over the octahedron with 
the equivalence relations $\tilde O\times T^3_2/\bZ_2/\sim$ as determined in Section \ref{sec:other}.

\subsection{The fiberwise action}
Again, we define $i_a, i_b,$ and $i_c$ on $T^3_2/\bZ_2$:
\begin{eqnarray}
i_a &:& (\phi_0, \phi_1, \phi_2) \mapsto (\phi_0+2\pi, \phi_1+2\pi, \phi_2) \nonumber \\
i_b &:& (\phi_0, \phi_1, \phi_2) \mapsto (\phi_0, \phi_1+2\pi, \phi_2+2\pi) \nonumber \\
i_c &:& (\phi_0, \phi_1, \phi_2) \mapsto (\phi_0+2\pi, \phi_1, \phi_2+2\pi).
\end{eqnarray}
Again, this amounts to changing some of the pasting maps by multiplications by $-\Idd$. 
The action sends equivalence classes to equivalence classes for 
$\rep_{-I}(\pi_1(\Sigma_1),\SUTw)$. 



In the domain over the interior of $\tilde O$, the $\{I, i_a, i_b, i_c\}$-action is free and there is no branch locus. 

In the regions over the interiors of faces, the $\{I, i_a, i_b, i_c\}$-action is also free: consider this action as 
translations in the tori $T^3_2/\sim$ where $\sim$ is generated by vectors parallel to 
the normal vectors of form $(\pm 2\pi, \pm 2\pi, \pm 2\pi)$ where the signs depend on which 
faces you are on. 

Consider $a^o$ for now. Then $i_a, i_b$, and $i_c$ correspond to 
translations by $(2\pi, 2\pi, 0)$, $(0, 2\pi, 2\pi)$, and $(2\pi, 0, 2\pi)$ respectively. 
Then these vectors can be seen to be parallel to the plane of the quotient tori 
with generators $(4\pi,0,-4\pi)$ and $(0, 4\pi, -4\pi)$ since we can change 
the signs of translation vectors mod $4\pi$. 
We see easily that these translations are not trivial even up to the order $3$ translation group
action. Therefore, the fiberwise Klein four-group acts freely over each point of $a^o$.

Similar considerations shows that the fiberwise Klein four-group action is free over the interiors of the other faces
$b, c, d, a', b', c', d'$.

The twelve edges of $O$ form three classes given by three respective equations $\theta_0=\pi/2, \theta_1=\pi/2, \theta_2=\pi/2$.
That is, four edges satisfy $\theta_0=\pi/2$, the next four edges satisfy $\theta_1=\pi/2$ and the last four edges 
$\theta_2=\pi/2$.
 Over the interiors of edges, the equivalence relation gives us circles. Thus, we have twelve annuli over 
 the interiors of the edges. 

Over the interiors of edges, by considering the equivalence classes, we deduce that
if the edge is in the class given by $\theta_0=\pi/2$, then the circle over the interior is fixed under $i_b$ and 
each of $i_a$ and $i_c$ has no fixed points. This follows by looking at the equivalence classes. 

Similarly, if the edge is in the class given by $\theta_1=\pi/2$, then the circle over the interior of it is a fixed under $i_c$ 
and each of $i_a$ and $i_b$ has no fixed points. 

If the edge is in the class given by $\theta_2=\pi/2$, then 
the circle over the interior of it is fixed under $i_a$ and each of $i_b$ and $i_c$ has no fixed points. 

For $i_b$, the four open annuli over the interiors of the edges given by $\theta_0=\pi/2$, 
will extend to a circle in each of the four vertex $3$-spheres where $\theta_0=\pi/2$: 
$A_A, A'_{A'}, C_{C'}, C'_C$.
(See Section \ref{subsub:AA} and the proof of 
Theorem \ref{thm:filledSU}.)
Hence, we see that the four annuli patch up to 
a $2$-torus since $i_b$ is a smooth involution and the fixed points form a submanifold 
(See Conner \cite{Co}) as our manifolds and the maps are semialgebraic. 


Similarly, $i_a$ fixes a union of four annuli forming a torus over the edges given 
by $\theta_1=\pi/2$.  Finally, $i_c$ fixes a union of four annuli forming a torus over the edges given 
by $\theta_2=\pi/2$.  Denote these torus by $t_0, t_1, t_2$ respectively.

\begin{lem}\label{lem:STST}
Each of $i_a, i_b,$ and $i_c$ equals $S, T, ST$ in some coordinates of 
each of the square regions $A_A, {A'}_{A'}, B_B, B'_{B'}, C_{C'},$ and $C_{C'}$. 
\end{lem}
\begin{proof} 
The action correspond to changing $h(d_i)$ to $\pm h(d_i)$ for $i=1,2$. 
See Remark \ref{rem:4d}. 
\end{proof}

The tori are mutually disjoint in 
each of the $3$-spheres over the vertex regions: 
First, the respective sets of fixed points of $S, T,$ and $S T$ are disjoint. 
$i_a, i_b,$ and $i_c$ are distinguished over each of the spheres and equal one of
$S, T, S T$. Hence, the respective fixed sets of $i_a, i_b, $ and $i_c$ are mutually disjoint.







\subsection{The axes action} \label{subsec:axes}

Finally, let us study what are the branch loci of $I_{{A}}, I_{{B}}, I_{{C}}$. 
This is defined by equations \eqref{eqn:iabc5}. We note that these send equivalence classes to
equivalence classes in $\tilde O \times T^3_2$. 

The branch locus is as follows: $I_{{A}}$ has a fixed axis in $\tilde O$. 
$\{\Idd, I_{{A}}\}$ acts by taking the upper triangle with vertices $v_0, v_1, v_2$ and lower triangle with vertices $v'_0, v'_1, v'_2$ 
so that $v_0$ and $v_1$ are changed to their antipodes 
and $v'_0$ and $v'_1$ are changed to their antipodes also and the corresponding pasting angles 
are changed by the map $\phi_i \ra 4\pi-\phi_i$ for $i=0,2$. 
We aim to find an expression of $I_A$:

We need to consider the point over its axis in $O^o$ and $A_A$ and $A'_{A'}$ since 
other points on $\tilde O$ are sent to different points. 

In a fiber in $O^o \times T^3_2/\bZ_2$ over the axis part in $O^o$, the fixed points of $I_{{A}}$ is a union of 
two circles given by $\phi_0= 0, 2\pi$ and $\phi_2=0, 2\pi$ mod the $\bZ_2$-action.  
The circle is parameterized by $\phi_2$. 
In $A_A$, we show that $I_A$ fixes a circle as we approach $A_A$,
that is transversal to the fixed circles of $S$ and $S T$
and meeting each twice as in Section \ref{subsub:AA}. This is also true in $A'_{A'}$. 
Thus, in $O^o\times T^3/\bZ_2$, the set of fixed points is a union of two annuli 
and extends in $A_A$ to the circle  and similarly, we can show it extends in $A'_{A'}$ to a circle. 
Hence, the fixed subspace is homeomorphic to a $2$-torus 
again using the smoothness of the involution and the manifold by Conner's work \cite{Co}.

Similarly, the fixed set of $I_{{B}}$ is a $2$-torus meeting the fixed torus of $I_A$ 
at the fiber over $(\pi/2,\pi/2,\pi/2)$ at two points $(0, 0, 0)$ and $(2\pi, 0, 2\pi)$ modulo $4\pi$.
The fixed set of $I_{{C}}$ is a $2$-torus meeting the two other tori at  the fiber 
over $(\pi/2, \pi/2, \pi/2)$ at the two pairs of points. The three $2$-tori meet 
 in the fiber torus over $(\pi/2, \pi/2, \pi/2)$ in pairs. 

Denote these tori by $T_{{A}}, T_{{B}}, T_{{C}}$ respectively.
Then $T_{{A}}$ meets each of $t_0$ and $t_1$ at a pair of points respectively in regions $A_A$ and ${{A'_{A'}}}$
by Lemmas \ref{lem:IA} and \ref{lem:STST}.
Similarly, $T_{{B}}$ meets each of $t_1$ and $t_2$ at a pair of points respectively in regions $B_B$ and $B'_{B'}$, 
and $T_{{C}}$ meets each of $t_2$ and $t_0$ at a pair of points respectively in regions $C_{C'}$ and $C'_C$.
(We now have the complete intersection pattern.) 




\begin{thm}\label{thm:filled} 
The other component ${\mathcal C}_1$ is homeomorphic to the quotient of filled octahedral manifold with 
the product Klein four-group action given by the action of $I_{{A}},I_{{B}},I_{{C}}$, $i_a$, $i_b,$ and $i_c$. 
\end{thm}

\section*{Acknowledgments}

We thank Oscar Garcia-Prada, Bill Goldman, and Sean Lawton for helping us to understand the many issues in this area. 
We also thank the AIM, Palo-Alto for organizing a conference in March, 2007 where we were able to 
work on some of the ideas presented here.  We also thank Sadayoshi Kojima and 
the Department of Mathematical and Computing Sciences 
at Tokyo Institute of Technology for their great hospitality while this work was being carried out. 
We should say that this work was inspired by the view points of the configuration points from 
the works of many people including Thurston, Kojima, Kapovich, and Millson. We are also indebted to the referees 
for many suggestion to make this article better. 

This work was supported by the Korea Research Foundation Grant funded 
by the Korean Government (MOEHRD)" (KRF-2007-314-C00025).

\bibliographystyle{plain}

\end{document}